\documentclass[10pt, reqno, twoside
, final
]{amsart}

\pdfoutput=1
\usepackage{graphicx}
\usepackage[utf8]{inputenc}
\usepackage{array,epsfig}
\usepackage{amsmath}
\usepackage{amsfonts}
\usepackage{amssymb}
\usepackage{dsfont}
\usepackage{amsxtra}
\usepackage{amsthm}
\usepackage{mathtools}
\usepackage{adjustbox}
\usepackage{leftindex}
\usepackage{leftidx}
\usepackage{mathrsfs}
\usepackage{color}
\let\originallhook\lhook
\usepackage{MnSymbol}
\let\lhook\originallhook
\usepackage{tikz-cd}
\usepackage{enumerate}
\usepackage{todonotes}
\usepackage{comment}
\usepackage{hyperref}    
\usepackage{quiver}
\usepackage{cleveref}
\usepackage{standalone}
\usetikzlibrary{arrows.meta}
\usetikzlibrary{3d}
\usepackage{xcolor}

\usepackage[
  a4paper
, centering
, margin=3cm
]{geometry}
\usepackage{fancyhdr}
\fancypagestyle{titlefooter}{%
  \fancyhf{} 

}

\numberwithin{equation}{subsection}

\makeatletter
\newtheorem*{rep@theorem}{\rep@title}
\newcommand{\newreptheorem}[2]{%
\newenvironment{rep#1}[1]{%
 \def\rep@title{#2 \ref{##1}}%
 \begin{rep@theorem}}%
 {\end{rep@theorem}}}
\makeatother

\newtheorem{thmx}{Theorem}
\newtheorem{corx}{Corollary}

\theoremstyle{definition}
\newtheorem{defn}{Definition}[subsection]
\crefname{defn}{Definition}{Definitions}
\theoremstyle{definition}

\crefname{defnprop}{Definition-Proposition}{Definition-Propositions}
\theoremstyle{definition}
\newtheorem{nota}[defn]{Notation}
\crefname{nota}{Notation}{Notations}
\theoremstyle{definition}
\newtheorem{cons}[defn]{Construction}
\crefname{cons}{Construction}{Constructions} 
\theoremstyle{definition}

\crefname{incdef}{Incorrect Definition}{Incorrect Definitions}
\theoremstyle{definition}

\crefname{fact}{Fact}{Facts}
\theoremstyle{definition}
\newtheorem{var}[defn]{Variant}
\crefname{var}{Variant}{Variants}
\theoremstyle{definition}
\newtheorem{warn}[defn]{Warning}
\crefname{warn}{Warning}{Warnings}
\theoremstyle{definition}

\crefname{prob}{Problem}{Problems}
\theoremstyle{plain}
\newtheorem{thm}[defn]{Theorem}
\crefname{thm}{Theorem}{Theorems}
\theoremstyle{definition}
\newtheorem{obs}[defn]{Observation}
\crefname{obs}{Observation}{Observations}
\newtheorem*{thm*}{Theorem}
\theoremstyle{plain}

\theoremstyle{definition}

\crefname{defthm}{Definition/Theorem}{Definition/Theorem}
\theoremstyle{plain}
\newtheorem{cor}[defn]{Corollary}
\crefname{cor}{Corollary}{Corollaries}
\theoremstyle{definition}
\newtheorem{rmk}[defn]{Remark}
\crefname{rmk}{Remark}{Remarks}
\theoremstyle{definition}

\crefname{vari}{Variant}{Variants}
\theoremstyle{plain}
\newtheorem{lem}[defn]{Lemma}
\crefname{lem}{Lemma}{Lemmas}
\theoremstyle{plain}

\crefname{claim}{Claim}{Claim}
\theoremstyle{definition}
\newtheorem{ex}[defn]{Example}
\crefname{ex}{Example}{Examples}
\theoremstyle{plain}

\crefname{soln}{Solution}{Solutions}
\theoremstyle{plain}
\newtheorem{prop}[defn]{Proposition}
\crefname{prop}{Proposition}{Propositions}
\theoremstyle{plain}

\crefname{conj}{Conjecture}{Conjectures}

\newcommand{\bb}[1]{\mathbb{#1}}
\newcommand{\Z}{\bb{Z}}

\newcommand{\R}{\bb{R}}

\newcommand{\E}{\bb{E}}

\newcommand{\F}{\mathcal{F}}

\newcommand{\Of}{\mathcal{O}}
\newcommand{\M}{\mathbf{M}}

\newcommand{\del}{\partial}

\newcommand{\hooklongrightarrow}{\lhook\joinrel\longrightarrow}

\newcommand{\simp}{\boldsymbol{\Delta}}

\newcommand{\set}{\mathsf{Set}}
\newcommand{\spa}{\mathcal{S}}

\newcommand{\Hom}{\mathrm{Hom}}

\newcommand{\adj}{\dashv}
\newcommand{\colim}{\mathrm{colim}\,}

\newcommand{\Mod}{\mathsf{Mod}}

\newcommand{\icat}{\mathcal{C}}
\newcommand{\icatd}{\mathcal{D}}
\newcommand{\icate}{\mathcal{E}}

\newcommand{\iop}{\mathcal{O}^{\otimes}}

\newcommand{\spect}{\mathcal{S}\mathrm{p}}
\newcommand{\fun}{\mathrm{Fun}}

\newcommand{\simpop}{\boldsymbol{\Delta}^{op}}
\newcommand{\simpopn}{\boldsymbol{\Delta}^{n,op}}
\newcommand{\simpopnin}{\boldsymbol{\Delta}^{n,op}_{\mathrm{inj}}}
\newcommand{\ner}{\mathbf{N}}

\newcommand{\fin}{\mathsf{Fin}_*}

\newcommand{\pshv}{\mathsf{PShv}}
\newcommand{\shv}{\mathsf{Shv}}

\newcommand{\Fact}{\mathsf{Fact}}

\newcommand{\prl}{\mathsf{Pr}^{\mathrm{L}}}

\newcommand{\infcat}{$\infty$-category }
\newcommand{\infop}{$\infty$-operad }
\newcommand{\infcats}{$\infty$-categories }
\newcommand{\infops}{$\infty$-operads }
\newcommand{\infcatt}{$\infty$-category}
\newcommand{\infopt}{$\infty$-operad}
\newcommand{\infcatst}{$\infty$-categories}
\newcommand{\infopst}{$\infty$-operads}

\newcommand{\cat}{\mathsf{Cat}}
\newcommand{\catinf}{\mathsf{Cat}_{\infty}}
\newcommand{\catinfh}{\widehat{\mathsf{Cat}}_{\infty}}

\newcommand{\opinf}{\mathsf{Op}_{\infty}}
\newcommand{\topo}{\mathsf{Top}}

\newcommand{\alg}{\mathsf{Alg}}
\newcommand{\mor}{\mathsf{Mor}}
\newcommand{\umor}{\mathfrak{u}\mathsf{Mor}}

\newcommand{\calg}{\mathsf{CAlg}}

\newcommand{\ev}{\mathrm{ev}}
\newcommand{\coev}{\mathrm{coev}}

\newcommand{\disk}{\mathsf{Disk}}
\newcommand{\disks}{\mathsf{Disks}}
\newcommand{\infdisk}{\mathcal{D}\mathsf{isk}}
\newcommand{\infdisks}{\mathcal{D}\mathsf{isks}}
\newcommand{\mfd}{\mathsf{Mfd}}
\newcommand{\infmfd}{\mathcal{M}\mathsf{fd}}
\newcommand{\bsc}{\mathsf{Bsc}}
\newcommand{\infbsc}{\mathcal{B}\mathsf{sc}}
\newcommand{\cube}{\mathsf{diskCube}}

\newcommand{\open}{\mathsf{Open}}
\newcommand{\pfact}{\mathsf{PFact}}
\newcommand{\rev}{\text{rev}}

\newcommand{\pt}{\text{pt}}
\newcommand{\Bord}{\mathsf{Bord}}
\newcommand{\fr}{\text{fr}}

\newcommand{\Id}{\text{Id}}
\newcommand{\inv}{\mathrm{inv}}
\usepackage{bbold}
\newcommand{\unit}{\mathbb{1}}

\newcommand{\Ss}{\mathcal{S}}
\newcommand{\Tt}{\mathcal{T}}

\newcommand{\mytodos}[1]{}

\newcommand{\collresc}{\varrho}

\usepackage{xstring}
\newcommand\cdsquareOpt[9][]{%
  \def\obja{#2}%
  \def\objb{#3}%
  \def\objc{#4}%
  \def\objd{#5}%
  \def\mora{\pgfkeysalso{#6}}
  \def\morb{\pgfkeysalso{#7}}
  \def\morc{\pgfkeysalso{#8}}
  \def\mord{\pgfkeysalso{#9}}
  \begin{tikzcd}[ampersand replacement=\&]
    \obja\ar[r,/utils/exec=\mora]\ar[d,/utils/exec=\morb]%
    \IfSubStr{#1} {pb}{\isCartesian}{}
    \IfSubStr{#1} {po}{\iscoCartesian}{}
    \IfSubStr{#1} {bicart}{\isbiCartesian}{}
    \&%
    \objb\ar[d,/utils/exec=\morc]%
    \\%
    \objc\ar[r,/utils/exec=\mord]%
    \&%
    \objd%
  \end{tikzcd}%
}

\newcommand\isCartesian[1][dr]{\ar[{#1}, phantom, description, very near start,"\lrcorner"]}
\newcommand\iscoCartesian[1][dr]{\ar[{#1}, phantom, description, very near end,"\ulcorner"]}
\newcommand\isbiCartesian[1][dr]{\ar[{#1},phantom, description, "\square"]}

\usepackage{xifthen} 		
\newcommand\maybeparens[1]{\ifthenelse{\isempty{#1}}{}{(#1)}}

\newcommand\cdcubeNA[9][normal]{%
  \def\obja{#2}%
  \def\objb{#3}%
  \def\objc{#4}%
  \def\objd{#5}%
  \def\obje{#6}
  \def\objf{#7}
  \def\objg{#8}
  \def\objh{#9}
  \begin{tikzcd}[ampersand replacement=\&, column sep=#1, row sep = #1]
    \obja\ar[rr]\ar[dr]\ar[dd]\&
    \&\objb\ar[dr]\ar[dd]\&\\
    \&\objc\ar[rr,crossing over]\&
    \&\objd\ar[dd]\\
    \obje\ar[rr]\ar[dr]\&
    \&\objf\ar[dr]\&\\
    \&\objg\ar[rr]\ar[from=uu,crossing over]\&
    \&\objh\\
  \end{tikzcd}%
}

\title{On dualizability and invertibility in the higher Morita category}
\author[Scheimbauer]{Claudia I. Scheimbauer}
\author[Steffens]{Pelle Steffens}
\author[Stewart]{William Stewart}
\address{TU Munich, Boltzmannstr. 3, 85746 Garching}
\email{scheimbauer@tum.de}
\email{pelle.steffens@tum.de}
\email{will.stewart@tum.de}

\date{}

\begin{document}
\setcounter{tocdepth}{2}
\setcounter{page}{1}

\begin{abstract}
 We prove a conjecture of Lurie characterizing $(n+1)$-dualizability in higher Morita categories of $\bb{E}_n$-algebras in terms of dualizability over certain factorization homologies.
 A key ingredient is a higher Morita category based on the recently developed framework of pointless factorization algebras of Karlsson and the first author.
We also verify an invertibility conjecture of Brochier--Jordan--Safranov--Snyder as an immediate corollary of our main result.

Moreover, we prove a relative version of the dualizability conjecture, yielding a new criterion for relative/twisted field theories.
We give some examples, including Dirichlet and Neumann relative theories.
\end{abstract}

\maketitle
\thispagestyle{titlefooter}

\tableofcontents

\section{Introduction}

This paper establishes dualizability results for higher Morita categories valued in any suitable higher category.
Namely, we prove a conjecture by Lurie \cite[Remark 4.1.27]{LurieCobordism} characterizing $(n+1)$-dualizability in the Morita
$(\infty, n+1)$-category\footnote{The $\mathfrak{u}$ in $\umor_n$ stands for \emph{univalent}. For most of this article, it will be more convenient for us to work with the \emph{valent} $(\infty,n+1)$-category $\mor_n$.} $\umor_n$ informally described as follows: objects are $\E_n$-algebras, 1-morphisms are bimodules in $\E_{n-1}$-algebras, 2-morphisms are bimodules of bimodules in $\E_{n-2}$-algebras, etc.; the pattern changes for $(n+1)$-morphisms, namely, those are homomorphisms.

Finding highly dualizable objects is desirable in light of the Baez--Dolan-Lurie {\em Cobordism Hypothesis} \cite{BaezDolan, LurieCobordism}, which roughly says the following:  $k$-dualizable objects in a higher category lead to framed fully extended $k$-dimensional topological field theories (TFTs) with that higher category as a target (and vice versa).
Here {\em fully extended} means that they assign values to manifolds of all dimensions $\leq k$.
One interpretation of the Cobordism Hypothesis is that fully extended topological field theories are {\em local}, i.e.~fully determined by their value at a point (or rather an arbitrarily  small $n$-dimensional neighborhood thereof).

When studying algebras, in light of representation theory and their categories of modules, it is most natural to study them with \emph{Morita equivalences} rather than just with isomorphisms.
This can be reformulated in terms of equivalences in a certain Morita 2-category: objects are algebras, 1-morphisms are bimodules with composition the relative tensor product, and 2-morphisms are homomorphisms of bimodules. This higher category is a natural target for TFTs: it is a straightforward exercise to see that every algebra is 1-dualizable and 2-dualizable objects are precisely the finite dimensional separable ones.

Generalizing this situation, higher Morita categories of $\bb{E}_n$-algebras are commonly used as targets of (extended) topological field theories.
An $\bb{E}_n$-algebra object in a category $\mathcal{C}$ can be thought of as an object with $n$ compatible associative multiplications, or alternatively, with associative multiplications parametrized by configurations of two points in $\mathbb{R}^n$.
Examples include coherently associative algebra objects such as $A_\infty$-algebras or loop spaces ($n=1$, $\mathcal{C}=\mathrm{Chain}\,\mathrm{complexes}$ or $\mathcal{C}=\mathrm{Spaces}$), possibly braided tensor categories ($n=1, 2$, $\mathcal{C}=\mathrm{Linear}\,\mathrm{categories}$), and observables of a topological quantum field theory.

Higher Morita categories are relevant for approaching many of the well-studied TFTs via dualizability: Turaev--Viro theories \cite{DSPS,BJS}, Reshetikhin--Turaev and Crane--Yetter theories \cite{BJS,BJSS,Haioun,Kinnear,FST}; they are the natural home for (braided) fusion 2-categories and  their higher cousins \cite{DouglasReutter, Decoppet_dualizability, classificationfusion2cat, Decoppet_relative}.

\subsection{Main results and applications to topological field theories}

In this article we systematically approach dualizability questions in higher Morita categories, building upon previous work in \cite{GS}.
Our first main result is the proof of a characterization of $(n+1)$-dualizability in $\umor_n$ as conjectured by Lurie \cite[Remark 4.1.27]{LurieCobordism} in terms of dualizability over the factorization homologies over spheres.
\begin{thmx}\label{thmx:mainthm}
    (Theorem \ref{thm:mainthm})
An $\bb{E}_n$-algebra
in the Morita $(\infty,n+1)$-category $\umor_n(\icat)$ of a nice symmetric monoidal $\infty$-category $\icat^{\otimes}$ is $(n+1)$-dualizable if and only if for all $0\leq k\leq n$, the object $A$ is left dualizable as a left module over the factorization homology 
\[  \int_{S^{k-1}\times \R^{n-k+1}}A. \]
Assuming the Cobordism Hypothesis, we obtain a fully extended framed $(n+1)$-dimensional topological field theory 
valued in $\umor_n(\icat)$, which assigns $A$ to a framed point.
\end{thmx}

This result was previously proven in the cases $n=1,2$ and for the  $(2,2)$-category of locally presentable $k$-linear categories in \cite{BJS} (see Section \ref{sec:intro_examples}, and also Section \ref{sec:evenhigher} for the generalization of Theorem \ref{thmx:mainthm} to Morita categories of higher categories with non-invertible higher morphisms).

The difficulty in proving  Theorem A lies in the details: we obtain dualizability data in the Morita category via the geometry of Euclidean spaces, and one needs a model for the higher Morita category flexible enough to incorporate such geometric reasoning. We will get back to more details on the issue and the model we develop in Section \ref{sec:intro_Morita}.

As a straightforward corollary we prove a conjecture by Brochier--Jordan--Safronov--Snyder \cite{BJSS} about invertibility:
\begin{corx}[Corollary \ref{cor:invertibility}]
For $\icat^{\otimes}$ a nice symmetric monoidal \infcatt, an $\bb{E}_n$-algebra $A$ is invertible in $\umor_n(\icat)$ if, and only if, it is $(n+1)$-dualizable and for all $k=0,\ldots n$, the canonical map
\[\int_{S^{k-1}\times \mathbb{R}^{n-k+1}} A \longrightarrow Z_{n-k}(A)\]
to the \emph{$\bb{E}_{n-k}$-center} is an equivalence.
\end{corx}

In recent years, in particular related to studying symmetries of quantum field theories, relative a.k.a.~twisted versions of field theories have become ubiquitous. They come in two flavors: lax and oplax.
We prove a relative version of our main theorem in terms of relative adjointability using the criterion of \cite{JFS}, see Section \ref{sec:relative} for terminology.
\begin{thmx}[Theorem \ref{thm:main-oplax-thm} and Corollary \ref{cor:main-oplax-thm}, ``even'' = oplax case]\label{thmx:main-oplax-thm}
 Let $M:A\to B$ be a 1-morphism between two $\bb{E}_n$-algebras in the $(\infty,n+1)$-category $\umor_n(\icat)$ 
Then $M$ is $n$-times right adjointable if and only if for $0\leq k \leq n-1$ the object $M$ is right dualizable as a right module over the factorization homologies 
\[  \int_{\widetilde{S}^{k}_{+} \times \R^{n-k}} M_{A\to B} \, ,\] 
 where $\widetilde{S}^{k}_{+}$ is a certain submanifold of a stratified $k$-sphere (see Construction \ref{cons:stratified-spheres}).

Assuming the Cobordism Hypothesis and using \cite{JFS}, in this case we obtain an oplax relative fully extended framed $n$-dimensional topological  field theory $\mathcal{Z}_M$, relative to the absolute, once-categorified $n$-dimensional topological field theories $\mathcal{Z}_A$ and $\mathcal{Z}_B$.
\end{thmx}

Some low-dimensional examples of $\widetilde{S}^k_+$ are:

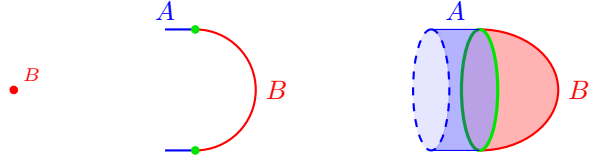
\begin{equation*}
\begin{tikzpicture}[scale=0.8]
\node[
    draw=red,
    fill=red,
    circle,
    inner sep=1pt
  ] (A) at (0,0) {};

\node[above right,red] at (A) {\scriptsize $B$};
\begin{scope}[xshift=3cm]
\draw[thick,red] (0,1) arc (90:-90:1) node[midway,right,red] {$B$};

\draw[thick,blue] (0,1) -- (-0.5,1)  node[above,blue] {$A$};;
\draw[thick,blue] (0,-1) -- (-0.5,-1);

\node[
    draw=green!90!black,
    fill=green!90!black,
    circle,
    inner sep=1pt
  ] at (0,1) {};
  \node[
    draw=green!90!black,
    fill=green!90!black,
    circle,
    inner sep=1pt
  ] at (0,-1) {};
\end{scope}

\begin{scope}[xshift=8cm]
    \draw[thick,blue,dashed]  (-1.1,1) arc (90:270:0.3 and 1);
    \draw[very thick, green!90!black] (-0.3,1) arc (90:270:0.3 and 1);

    \filldraw[blue,opacity=0.1] (-0.8,0) arc (0:360:0.3 and 1);
    
    \filldraw[blue,opacity=0.3] (-1.1,1) -- (-0.3,1) arc (90:-90:0.3 and 1) -- (-1.1,-1) arc (-90:90:0.3 and 1) ;
    \draw[thick,blue,dashed] (-1.1,1) arc (90:-90:0.3 and 1);
    \draw[blue] (-1.1,1) -- (-0.3,1) node[midway,above,blue] {$A$};;
    \draw[blue] (-1.1,-1) -- (-0.3,-1);

    \draw[thick,red]  (-0.3,1) arc (90:-90:1.3 and 1) node[midway,right,red] {$B$};
    \filldraw[red,opacity=0.3]  (-0.3,1) arc (90:-90:1.3 and 1)  arc (-90:90:0.3 and 1);
    \filldraw[red,opacity=0.1] (0,0) arc (0:360:0.3 and 1);

    \draw[very thick, green!90!black] (-0.3,1) arc (90:-90:0.3 and 1); 
\end{scope}
\end{tikzpicture}
\end{equation*}

Using this Theorem, we clarify the existence of Neumann and Dirichlet relative theories. In particular, viewing an $\bb{E}_n$-algebra $A$ as a regular module, the 1-morphism $A\colon \mathbb{1} \to A$ is always $R^n$-adjointable, hence via the Cobordism Hypothesis gives an oplax relative fully extended framed $(n+1)$-dimensional topological  field theory. In contrast, a 1-morphism $N \colon A \to \unit$ in $\umor_n(\icat)$ is $R^{n}$-adjointable if and only if $N$ is $n$-dualizable as an object in $\umor_{n-1}(\icat)$, hence via Cobordism Hypothesis gives an oplax relative fully extended framed $(n+1)$-dimensional topological field theory. 

In case $n=1$, Theorem \ref{thm:main-oplax-thm} reduces to the statement that a bimodule $M:A\to B$ of $\E_1$-algebras admits a right adjoint if and only if $M$ is dualizable over $B$ \cite[Proposition 4.6.2.13]{LurHA}.

\subsection{Upcoming examples: rigidity}\label{sec:intro_examples}

To apply our main theorem, we need some way of verifying that its criteria hold. In the (very rich and much-studied) setting of Morita categories valued in categories of (presentable) \emph{enriched categories}, it has been long-expected folklore that {\em rigidity} should imply that presentably $\bb{E}_n$-monoidal $\mathcal{V}$-\infcats (for $\mathcal{V}^{\otimes}$ an enriching presentably symmetric monoidal \infcatt) are $(n+1)$-dualizable in the higher Morita $(n+2)$-category valued in $\prl_{\mathcal{V}}$, the \infcat of presentable $\mathcal{V}$-enriched \infcatst. This well known for (braided) tensor categories, that is, in the Morita categories  $\umor_1\left(\prl_{\Mod_k^{\heartsuit}}\right)$ (and $\umor_2\left(\prl_{\Mod_k^{\heartsuit}}\right)$) where $\Mod_k^{\heartsuit}$ is the ordinary 1-category of $k$-modules for $k$ a commutative ring, as proven by Brochier-Jordan-Snyder (\cite[Theorem 5.16]{BJS}). In the derived setting, rigidity is leveraged to prove 2-dualizability of presentably monoidal stable \infcats inside the $(\infty,3)$-category $\umor_1\left(\prl_{\spect}\right)$ (where $\spect^{\otimes}$ is the presentably symmetric monoidal \infcat of spectra) by Ben-Zvi-Nadler \cite{BZN-complexgroupcharacter,BZGN-characterfieldthy}. The property of rigidity\footnote{We will refrain from spelling out a definition here, but the notion of rigidity we are referring to corresponds to what is called \emph{semi-rigidity} in \cite{BZN-complexgroupcharacter} and \emph{local rigidity} in \cite{Krause-Nikolaus-sheavesonmfds}.} of a presentably monoidal $\mathcal{V}$-category ensures that many questions concerning dualizability have particularly simple answers; indeed, in follow-up work (joint with Jackson Van Dyke), we will confirm the expectation articulated above and prove, using Theorem \ref{thmx:mainthm}, that rigidity implies being $(n+1)$-dualizable in the Morita category of $\bb{E}_n$-algebras in $\prl_{\mathcal{V}}$ for any presentably symmetric monoidal $\mathcal{V}^{\otimes}$. In particular, this recovers the results of Brochier-Jordan-Snyer and Ben-Zvi-Nadler. This has particular significance when combined with a suitable choice of $\mathcal{V}$. For $\mathcal{V}$ a suitable category of vector spaces with functional analytic flavor like those considered in \cite{BenBasKelKrem} or \cite{ClauScholze}, the second author shows in \cite{Steffens} that \infcats of \emph{smooth representations} of compact Lie groups (more generally, of proper Lie groupoids) are indeed rigid.
 
 On the field theory side, this leads to interesting novel (Sym)TFTs in arbitrary dimensions; for example, letting $G$ be a compact lie group, the presentable \infcat of $G$-representations determines, for every $n\geq 1$, a once-categorified theory with target the Morita category $\umor_n\left(\prl_{\mathcal{V}}\right)$ for $\mathcal{V}$ a suitable \infcat of linear spaces of functional analytic flavor. It is expected that this is the quantum BF-theory of, for instance, \cite{BrennanSun-BFthy,BonettiDelZottoMinasian-BFthy} which in \emph{loc. cit.} is identified as the SymTFT for theories with continuous $G$-symmetries.

\subsection{Strategy of proof: modifying the model of the higher Morita category}\label{sec:intro_Morita}

In the literature, there are two rather different approaches to the construction of higher Morita categories: one algebro-combinatorial in \cite{HaugsengEn} and the other more geometrically flavored, using factorization algebras in \cite{Scheimbauer}; both were bootstrapped to ``even higher'' Morita categories in \cite{JFS} to include higher morphisms.

The former approach is neat and elegant, and comes with many desirable properties: its objects and all $k$-morphisms (for $k\leq n$) enjoy additivity by definition, the input higher category $\icat$ can be weakened to be merely $\bb{E}_n$-monoidal rather than symmetric monoidal, and it requires the existence of fewer colimits.

However, for questions regarding dualizability, Haugseng's combinatorial model is too rigid. An object in a higher category is highly dualizable if it has a dual, the evaluation and coevaluation maps
thereof have left and right adjoints, the units and counits exhibiting those themselves have left and right
adjoints, and so on. 
These (co)evaluation and (co)unit morphisms cannot be realized as maps among the simple operads that feature in the combinatorial model, and hence constructing them appears to be difficult.

On the other hand, factorization algebras are flexible and allow for geometric manipulations.
These were exploited in \cite{GS} to prove $n$-dualizability and adjointability of all $k$-morphisms for $k<n$ in the $(\infty,n)$-Morita category $\umor^{\mathrm{ptd}}$ constructed using factorization algebras. 
In this latter approach $\bb{E}_n$-algebras are modeled as locally constant factorization algebras on $\mathbb{R}^n$, and all $k$-morphisms for $k\leq n$ are (essentially) constructible factorization algebras on $\mathbb{R}^n$ with a stratification by a flagging of hyperplanes.
Factorization algebras can be pushed forward along continuous maps, and the factorization algebras underlying the (co)evaluation and (co)unit morphisms exhibiting $n$-dualizability and adjointability of all $k$-morphisms for $k<n$ are given as pushforwards along certain self-maps of $\mathbb{R}^n$ (where care has to be taken with regards to constructibility); see Section \ref{sec:n-duals} for a recollection.

However, factorization algebras model {\em pointed} objects explaining the notation $\umor^{\mathrm{ptd}}$; for instance, a constructible factorization algebra on $\mathbb{R} \subset \{0\}$ (in vector spaces, say) is precisely the data of a bimodule $M$ together with a chosen element $\mathbb{1} \to M$.
These pointings prevent $(n+1)$-dualizability of objects in $\umor^{\mathrm{ptd}}$ unless they are trivial, as was first discussed in \cite[Theorem 7.3]{Johnson-Freyd}, and later expanded upon in \cite{GS}. This phenomenon can be seen already at low categorical levels: consider the category of pointed vector spaces, then it is a good exercise to test one's understanding of the definitions to verify that the only dualizable objects therein are one-dimensional. Note that Haugseng's model does not have this problem, but in this setting we cannot even prove $n$-dualizability as we have no way of exploiting the geometry of $\R^n$. 

Faced with these issues, one might imagine two strategies to overcome them:
\begin{enumerate}
\item Define a symmetric monoidal ``unpointing'' map from the pointed factorization model for the Morita category to Haugseng's operadic model, and then prove the final dualizability step there.
\item Modify the definition of factorization algebra to get rid of the undesired pointings, reprove $n$-dualizability therein, and then prove the final dualizability step.
\end{enumerate}
The first appears rather challenging, as the factorization model is more flexible than the operadic one, and hence defining a morphism of symmetric monoidal higher categories seems like a difficult undertaking. We have chosen to follow the second strategy in a series of articles
\begin{enumerate}[$(I)$]
    \item The article \cite{KSW} of the first author with Eilind Karlsson and Tashi Walde lays solid foundations of the theory of constructible factorization algebras. Among its main results is a gluing theorem asserting that the formation of constructible factorization algebras determines a sheaf of \infcats on the category of smooth conical manifolds. As a consequence of this result, or rather its pointless version, the factorization Morita category constructed in this paper satisfies the Segal condition. 
    \item In the article \cite{KS} extracted from Karlsson's thesis \cite{Karlssonthesis}, Karlsson and the first author defined a notion of {\em pointless factorization algebras}, reproved the statements in \cite{KSW} in the pointless setting, and provided many ingredients necessary for assembling a higher Morita category.
    \item In this article, we construct a higher Morita category based on this notion of pointless factorization algebras and prove the desired dualizability statements. Along the way, we characterize pointless factorization algebras as algebras for a suitable $\infty$-operad of pointless disks.    
\end{enumerate}

In future work joint between the first two authors and Anja \v{S}vraka (\cite{ScStSv}), we will pursue a more general version of the  approach (1) above: we will show that the pointless factorization model developed in this work is in fact equivalent to Haugseng's model. This will rely on a version of Dunn--Lurie additivity for pointless factorization algebras, which is based on modifying the results of Carmona-\v{S}vraka in \cite{Carmona-Svraka} and \v{S}vraka's thesis \cite{Svraka}.

\subsection{Pointless factorization algebras and the pointless higher Morita category}

As explained above, our strategy to put ourselves in a situation where we can potentially prove a statement like Theorem \ref{thmx:mainthm} is to modify the construction of the higher Morita category $\umor_n(\icat)$ by basing it on {\em pointless} constructible factorization algebras.

\subsubsection{Pointless factorization algebras.}
To define a pointless constructible factorization algebra, we need to start with a smooth conical manifold, or conically smooth stratified space, $X$, together with a discrete subset $E$, called a {\em marking}.

A pointless factorization algebra $\F$ assigns to every open set $U$ an object $\F(U)\in \icat$ and to an inclusion $U_1\amalg \cdots \amalg U_k \subset V$ a map
\(\F(U_1)  \otimes \cdots \otimes \F(U_k) \longrightarrow \F(V),\)
{\em but only if} 
\[\left(U_1\amalg \cdots \amalg U_k \right) \cap E = V\cap E .\]
Of course, this assignment must satisfy several conditions, namely multiplicativity, Weiss descent, and constructibility; they are analogous to the conditions satisfied by usual factorization algebras.

For the purposes of illustration in this introduction, let us consider the marked (stratified) space
\[ X\coloneqq\mathbb{R} \supset E\coloneqq\{0\}.\]

We have three different ``types'' of intervals in $X$: those containing 0, those contained in $\mathbb{R}_{-}$, and those contained in $\mathbb{R}_{+}$ and $\F$ assigns to them objects (equivalent to) $M$,  $A$, and $B$; respectively.
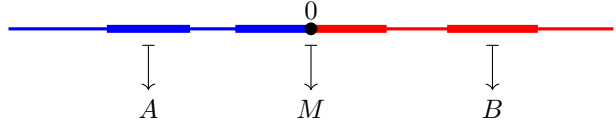
\begin{center}
\begin{tikzpicture}[scale=2]
\draw[blue, very thick] (-2,0) -- (0,0);
\draw[red, very thick] (0,0) -- (2,0);

\draw[blue, line width=3pt] (-1.35,0) -- coordinate (midA) (-0.8,0);
\draw[blue, line width=3pt] (-0.5,0) -- (0,0) coordinate (midM);
\draw[red, line width=3pt] (0,0) -- (0.5,0);
\draw[red, line width=3pt] (0.9,0) -- coordinate (midB) (1.5,0);

\draw (2, 0) node [anchor=north] {};
\draw (-2, 0) node [anchor=north] {};
\fill (0,0) circle (0.12em) node [anchor=south] {0};

\draw [|->] ([yshift=-3pt]midA) -- ++(0,-0.3) node [anchor=north] {$A$};
\draw [|->] ([yshift=-3pt]midM) -- ++(0,-0.3) node [anchor=north] {$M$};
\draw [|->] ([yshift=-3pt]midB) -- ++(0,-0.3) node [anchor=north] {$B$};
\end{tikzpicture}
\end{center}
Moreover, we assign maps to the following inclusions $U_1\amalg \cdots \amalg U_k \subseteq V$ of intervals:
\begin{enumerate}
    \item Any inclusion into an interval not containing the point 0;
    \item if the target $V$ contains 0, we must ensure that 0 is contained in the source, that is, $0\in U_1\amalg \cdots \amalg U_k$.
\end{enumerate}

The inclusions in (1) lead to $\E_1$-algebra structures on $A$ and $B$. The inclusions in (2) lead to an $(A,B)$-bimodule structure on $M$. Crucially, since we {\em exclude} the inclusion $\emptyset \hookrightarrow \mathbb{R}$, we do {\em not} have a pointing of $M$.

Constructible pointless factorization algebras enjoy the same nice properties as usual factorization algebras: we can glue them over open covers, they have a symmetric monoidal structure, and satisfy a cone decomposition property, all proven in \cite{KS}.

In this work, we add to this list a localization result making the connection to an $\infty$-operad $\infdisk_{/(X,E)}^{\otimes}$ of pointless disks (for $(X,E)$ a smooth conical manifold equipped with a marking as in Definition \ref{defn:markedsmcon}).
This implies that that (in good situations that are essentially always satisfied in practice) constructible pointless factorization algebras on a marked smooth conical manifold $(X,E)$ are just algebras for this $\infty$-operad (see Corollary \ref{cor:fact=alg}):
\begin{equation*}
    \mathsf{Fact}^{\mathrm{cstr}}_{(X,E)}(\icat) \xrightarrow{\simeq}  \alg_{\infdisk_{/(X,E)}}(\icat) \,.  
\end{equation*} 
An immediate consequence (Corollary \ref{cor:Fact-bimod}) is that constructible pointless factorization algebras on $(\mathbb{R}, 0)$ are {\em equivalent} to bimodules, hence justifying our approach to the Morita category.

Additivity of constructible factorization algebras exhibits constructible factorization algebras on product spaces $X\times Y$ as constructible factorization algebras on $X$ valued in constructible factorization algebras on $Y$. This was recently proven by Carmona--\v{S}vraka \cite{Carmona-Svraka} and \v{S}vraka \cite{Svraka} for usual constructible factorization algebras and will be extended to the pointless case shortly (\cite{ScStSv}).

\subsubsection{Assembling the pointless higher Morita category}
Equipped with these results we define a higher Morita category $\umor_n(\icat)$ as an $(\infty,n+1)$-category.

Informally, objects are locally constant factorization algebras on $(0,1)^n \cong \mathbb{R}^n$, which Lurie proved to be equivalent to $\E_n$-algebras.
The $k$-morphisms are constructible {\em pointless} factorization algebras on $(0,1)^n$ equipped with a suitable stratification given by a flag of length $k+1$ of hyperplanes.
Note that only for $k=n$ it is crucial that we are considering pointless factorization algebras and hence differs from the construction in \cite{Scheimbauer}.
The example explained above relating bimodules and constructible pointless factorization algebras on $\mathbb{R} \supset \{0\}$ together with additivity implies that a $k$-morphism indeed is equivalent to a bimodule of bimodules ... in $\E_{n-k}$-algebras for $k<n$ and a bimodule of bimodules ... for $k=n$; hence, modeling precisely the higher Morita category of $\E_n$-algebras.

To make this precise, it is convenient for us to use the model 
of $n$-fold semisimplicial objects in $\catinf$, which satisfy the Segal condition and are quasi-unital, following \cite{Haugseng_quasi}.
First, we define an $n$-fold simplicial object of conical manifolds. All of the appearing conical manifolds are the cube $(0,1)^n$ with varying stratifications given by intersecting (multiple) hyperplanes. The semisimplicial face maps are stratified piecewise rectilinear maps - the so-called \emph{collapse-rescale maps} of \cite{Scheimbauer}. The functoriality in the semisimplex category is a consequence of the basic fact that there is a \emph{unique} orientation-preserving and boundary-preserving rectilinear map between any two bounded closed intervals of $\R$. The higher Morita category is given by postcomposing with the functor $\mathsf{Fact}_{(\square_{(\_)},E)}$ that takes marked spaces to pointless factorization algebras, and then imposing constructiblity; it must be checked that pushing forward pointless factorization algebras along collapse-rescale maps preserves constructibility. We prove the Segal condition for this $n$-fold semisimplicial object (Proposition \ref{prop:moritasegal}) and show that the resulting $n$-fold non-unital category object is quasi-unital (Proposition \ref{prop:moritaquasiunit}). This results in an $n$-fold category object of $\catinfh$, the \infcat of (large) \infcatst, or put differently, an $(n+1)$-fold category object of $\widehat{\spa}$, the \infcat of (large) spaces; the pointless higher Morita category $\umor_n(\icat)$ is then the univalent completion of the underlying $(n+1)$-fold Segal space $\mor_n(\icat)$, which comes equipped with a completion map $\mor_n(\icat)\rightarrow \umor_n(\icat)$. This yields a complete $(n+1)$-fold Segal space: an $(\infty,n+1)$-category. In fact, it is even a symmetric monoidal  $(\infty,n+1)$-category, crucial for studying dualizability.

\subsection{Outline of the proofs}

Now that we have an $(\infty,n+1)$-Morita category $\umor_n(\icat)$ tailored to our goals, the proof of the main theorems consist of two parts:
\begin{enumerate}
    \item An argument for lifting adjoints, allowing to reduce the existence of adjoints of $n$-morphisms in $\umor_n(\icat)$ to the existence of an adjoint of a certain 1-morphism in $\umor_1(\icat)$. This is proven in Theorem \ref{sec:lifting_lemma}.
    \item A geometric argument, computing the manifolds over which the factorization homologies in the statements of Theorem \ref{thmx:mainthm} and Theorem \ref{thmx:main-oplax-thm} are taken. These are completed in Proposition \ref{prop:mod-formula}, Theorem \ref{thm:mainthm}, and Theorem \ref{thm:main-oplax-thm}.
\end{enumerate}

Let us elaborate on each of the steps.

\subsubsection{Lifting of adjoints}

As explained above, an $n$-morphism of $\umor_n(\icat)$ is a bimodule of bimodules of ... of bimodules. We can forget the first $(n-1)$ bimodule structures and remember only the last one, i.e.~the structure of a bimodule, denoted $\pi_*\F$, between $\E_1$-algebras.

Checking for the existence of adjoints of $\F$ can be tedious, because one needs to construct many bimodule structures.
Instead, we prove in Theorem \ref{thm:main-lifting-lemma} that the existence of an adjoint of $\F$ can be tested by the existence of an adjoint for $\pi_*\F$. For instance, in the case $n=2$, the $2$-morphism $\F$ can be thought of as a pointless constructible factorization algebra on the picture on the left, where the strata are labeled by all the sources and targets of $\F$: $A$ and $B$ are $\bb{E}_2$-algebras, $M$ and $N$ are both $\bb{E}_1$-algebras in $A$-$B$-bimodules and $\F$ is an $M$-$N$-bimodule internal to $A$-$B$-bimodules.
\begin{equation*}
\begin{tikzpicture}[scale=1.3]
\filldraw[opacity=0.1] (1,0) -- (0,0) -- (0,2)--(1,2);
\filldraw[opacity=0.1] (1,0) -- (2,0) -- (2,2)--(1,2);
\draw[thick] (1,0) -- (1,2);
\fill (1,1) circle (0.05);
\node at (1.2,0.2) {$M$};
\node at (1.2,1.8) {$N$};
\node at (1.2,1) {$\F$};
\node at (0.2,1) {$A$};
\node at (1.8,1) {$B$};
\end{tikzpicture} \qquad \qquad\begin{tikzpicture}[scale=1.3]
\draw[thick] (1,0) -- (1,2);
\fill (1,1) circle (0.05);
\node at (1.2,0.2) { $M$};
\node at (1.2,1.8) {$N$};
\node at (1.5,1) {$\pi_*(\F)$};
\end{tikzpicture} 
\end{equation*}
Theorem \ref{thm:main-lifting-lemma} states that the 2-morphism $\F$ has an adjoint if and only if the 1-morphism $\pi_*(\F)$ which is merely an $M$-$N$-bimodule in the underlying higher category $\icat$ has an adjoint. In the setting of braided tensor categories, the analog of this result is Proposition 5.17, 5.18, and Lemma 5.19 in \cite{BJS}. Since we work fully coherently, our proof of this fact is more involved.

The proof of Theorem \ref{thm:main-lifting-lemma} is somewhat technical, but it relies on the observation that $\pi_*(\F)$ is also obtained by composing with the maps of $\bb{E}_2$-algebras $\bb{1}_{\icat}\rightarrow A$ and $\bb{1}_{\icat}\rightarrow B$. In the Morita category, these maps correspond to \emph{regular bimodules}. In  Section \ref{sec:reg-bimod}, we prove that if a $k$-morphism in the Morita category happens to be a regular bimodule, then it is adjointable, and its unit and counit are again regular bimodules; in particular, all regular bimodules are maximally adjointable. This allows us to reduce the proof of  Theorem \ref{thm:main-lifting-lemma} to a game with adjointable cubes and squares that we play in  Section \ref{sec:lifting_lemma}.

\subsubsection{The geometric argument} In \cite{GS} it was shown that every $\E_n$-algebra $A$ is $n$-dualizable as an object in the pointed Morita category $\umor^{\mathrm{ptd}}$. This result applies equally to the unpointed Morita category (the units and counits appearing happen to have pointings, but every usual factorization algebra is also a pointless one). In light of Theorem \ref{thm:main-lifting-lemma}, the lifting of adjoints result, an $\E_n$-algebra $A$ is $(n+1)$-dualizable if and only if the bimodules underlying the $n$-morphisms witnessing the $n$-dualizability of $A$ admit adjoints as bimodules of $\E_1$-algebras.

The geometric argument provides a computation of the bimodules underlying the $n$-morphisms witnessing the $n$-dualizability of $A$, using the constructions of duals and adjoints given in \cite{GS}. In Proposition \ref{prop:mod-formula}, we identify these bimodules with $A \simeq \int_{D^k\times D^{n-k}} A$ equipped with canonical left and right actions by the factorization homologies
\begin{equation*}
    \int_{S^{k-1} \times \R \times D^{n-k}} A \qquad \text{and}\qquad \int_{D^{k} \times S^{n-k-1} \times \R}A \ ,
\end{equation*}
where the index $k$ takes values $0,...,n$. Theorem \ref{thm:mainthm} follows by applying the result from \cite[Proposition 4.6.2.13]{LurHA} that a bimodule ${}_{A}M_B$ of $\E_1$-algebras admits a left adjoint if and only if $M$ is dualizable over $A$.

We briefly sketch the key details of the proof of Proposition \ref{prop:mod-formula}. The main technical ingredient is Morse theory. The units and counits of each adjunction, as well as the evaluation and coevaluation 1-morphisms, are constructed in \cite{GS} by pushing forward constructible factorization algebras along self maps of $(0,1)^n$. The underlying bimodules are given by pushing forward along the composition of these self-maps with the projection onto the last coordinate. Starting with a constructible factorization algebra $A$ on $(0,1)^n$ we identify this pushforward with the pushforward of $A$ along a Morse function with a single critical point. The proof is performed inductively, passing to the unit of an adjunction increases the index of the Morse function by one while passing to the counit of an adjunction keeps the index constant. The Morse lemma controls the local behavior of a Morse function at a critical point, giving the desired identification of the bimodules underlying the $n$-morphisms witnessing the $n$-dualizability of $A$.

The arguments using Morse theory provide the technical details behind a more conceptual argument, which we include for comparison in Section \ref{sec:geometric_part_old}. This argument uses excision of factorization homology to prove that the source and target algebras of the underlying bimodules witnessing $n$-dualizability are the factorization homologies of $A$ over spheres of the desired index. However, this argument does not keep track of the precise action of these factorization homologies on $A$. This is in contrast to the Morse theoretic arguments, in which the actions are controlled by the Morse lemma.

\subsection{A comment on `even higher' Morita categories}\label{sec:evenhigher}

In the body of this work, we will assume that our Morita categories take values in an $(\infty,1)$-category. The construction of the factorization higher Morita category we expose in this article can be extended to the `even higher' Morita categories of \cite{JFS}: given a suitable symmetric monoidal $(\infty,d)$-category, the conjunction of the results of Section \ref{sec:moritacat} and \cite[Section 8]{JFS} produce an $(\infty,n+d)$-category $\umor_n(\icat)$. Our dualizablity and invertibility results apply to these `even higher' Morita categories as well, for the following reason: $(n+1)$-dualizability and invertibility in a symmetric monoidal $(\infty,n+d)$-category $\icatd$ (for $d\geq 1$) is detected in the underlying symmetric monoidal $(\infty,n+1)$-category $\tau_{\leq (\infty,n+1)}\icatd$, and we have a canonical equivalence 
\[  \umor_n(\tau_{\leq (\infty,1)}\icat) \simeq \tau_{\leq (\infty,n+1)}\umor_n(\icat) \]
of symmetric monoidal $(\infty,n+1)$-categories, by \cite[Proposition A.20]{GS}.

\subsection{History of this project}

This article is a result of research over the past decade. The idea of unpointing factorization algebras stems from Peter Teichner and Stephan Stolz, who suggested allowing only dense inclusions of open sets, which leads to {\em firm} algebras and unpointed bimodules.
The first author and Theo Johnson-Freyd first started discussing a higher Morita category based on these ideas after realizing that pointings prevent higher dualizability.
The ultimate idea for pointless factorization algebras came from Eilind Karlsson, the first author's first PhD student.
Had she decided to continue her career in academia, she would have become a coauthor of this article as she had worked out many ingredients and pictures for the case of $n=2$.
Eilind and the first author have given several talks about this topic these past years, for instance in October 2022 at \href{https://pirsa.org/22100102}{Perimeter Institute}, in May 2023 at Peter Teichner's birthday conference {\em Interactions of Low-dimensional Topology and Quantum Field Theory} in Les Diablerets, Switzerland, and in March 2026 at the Oberwolfach Workshop {\em Higher Structures from Symmetries in Quantum Field Theory} (see also \cite{OWR});
and a minicourse in June 2024 at the ICMS Summer School on \href{https://icms.ac.uk/archive/workshop/categorical-symmetries-in-quantum-field-theory-summer-school/}{\em Categorical symmetries in Quantum Field Theory}. The second author has given talks about this work in April and May 2025 in Toulouse and at LMU Munich, respectively. The third author also gave a talk about this work in July 2025 at the conference {\em Quantum Field Theory and Topological Phases via Homotopy Theory and Operator Algebras} at MPIM Bonn, and presented a poster in July 2025 at the {\em Spaces of Tensor Categories Workshop} at ICMS.

During the completion of this article, the work of Bustillo Vazquez \cite{Vazquez} appeared which also treats these conjectures. As far as we understand, Bustillo Vazquez' strategy is to argue directly in Haugseng's combinatorial model of the higher Morita category. However, to reprove $n$-dualizability established in \cite{GS}, he uses similar geometric arguments involving disk categories \`{a} la Ayala--Francis--Rozenblyum--Tanaka \cite{AFT-fh-stratified, AFR-stratifiedhtpyhypothesis,AFR-facthomI}.

\subsection{Acknowledgements}

CS is grateful for many inspiring conversations with Theo Johnson-Freyd, Stephan Stolz, and Peter Teichner about unpointing factorization algebras, twisted field theories, and other subjects in Berkeley and other beautiful locations. We thank Dan Freed for inspiration and encouragement in pursuing this project.

CS would like to thank Eilind Karlsson for being a wonderful student and a great coauthor, and for tackling the unpointing problem together.
We all thank her for being a fun friend always excited to embark on adventures (regardless of the weather) and for her sense of humor (as well as her appreciation of ours).
We thank Tashi Walde and Anja \v{S}vraka for many hours of gluing and assembling, and figuring out conical problems. We thank Ben Ha\"{i}oun and Jackson van Dyke for sharing their expertise on braided tensor categories and 3D TQFTs, providing invaluable context for this project.

All authors were supported by the Simons Collaboration on Global Categorical Symmetries (1013836 and 8528-03).
All authors were supported by the Deutsche Forschungsgemeinschaft (DFG, German Research Foundation) through the Collaborative Research Centers SFB 1085 Higher invariants - 224262486 and
SFB 1624 Higher structures, moduli spaces and integrability - 506632645.

\subsection{Conventions and Notation}
\begin{enumerate}
    \item When writing ``category'' we  mean an ordinary, discrete category.
    We will freely use the technology developed e.g.~in \cite{Joyal, LurHTT, LurHA} and use the terminology ``$\infty$-category'' for $(\infty,1)$-categories.
    \item We write $\spa$ for the \infcat of (small) spaces and $\catinf$ for the \infcat of (small) \infcatst. We write $\widehat{\spa}$ and $\catinfh$ for the \infcats of \emph{large} spaces and \infcats respectively. We write $\prl\subset \catinfh$ for the \infcat of presentable \infcats and functors admitting right adjoints between them.
    \item We use the notation $\Hom^k_{\icat}(X,Y)$ for the $(\infty,n-k)$-category of $k$-morphisms in an $(\infty,n)$-category~$\icat$. More generally, let $\icatd$ be an \infcat that admits finite limits and let $\icate=\icate_{\bullet\ldots\bullet}$ be an $n$-fold Segal object in $\icatd$ (see e.g. \cite[Section 4]{Haugseng-iteratedspans}). We say that a \emph{parallel pair of $0$-morphisms} of $\icate$ is a pair $A,B$ of objects of $\icat_{0\ldots 0}$, and we write $\Hom^1_{\icate}(A,B)$ for the $(n-1)$-fold Segal object defined by the pullback
    \[
        \begin{tikzcd}
        \Hom^{1}_{\icate}(A,B)\ar[d] \ar[r]& \icate_{1\bullet\ldots\bullet} \ar[d] \\
        \{A,B\} \ar[r] &  \icate_{0\bullet\ldots\bullet}  \times \icate_{0\bullet\ldots\bullet}. 
        \end{tikzcd}
    \]
    Now we make the following inductive definition for $1\leq k\leq n-1$.
    \begin{itemize}
        \item Given a parallel pair of $(k-1)$-morphisms of $\icate$, the data of which is a pair $(A,B)\in \icate_{1\ldots 10\ldots0}$ (first $(k-1)$ entries are 1's) and a map of $(n-k)$-fold Segal objects $\Hom_{\icate}^{k}(A,B)_{\bullet\ldots\bullet}\rightarrow \icate_{1\ldots 11\bullet\ldots\bullet}$, a \emph{parallel pair of $k$-morphisms of $\icate$} is a pair of objects of $\Hom_{\icate}^{k}(A,B)_{0\ldots 0}$ determining a pair $X,Y$ of objects of $\icate_{{1\ldots 110\ldots0}}$.
        \item Given such a parallel pair of $k$-morphisms of $\icate$, we will write $\Hom^{k+1}_{\icate}(X,Y)$ for the $(n-k-1)$-fold Segal object defined by the pullback 
        \[
        \begin{tikzcd}
        \Hom^{k+1}_{\icate}(X,Y)\ar[d] \ar[r] &\Hom_{\icate}^{k}(A,B)_{1\bullet\ldots\bullet} \ar[d] \\
        \{X,Y\} \ar[r] &  \Hom_{\icate}^{k-1}(A,B)_{0\bullet\ldots\bullet}  \times \Hom_{\icate}^{k}(A,B)_{0\bullet\ldots\bullet}. 
        \end{tikzcd}
        \]        
        which comes with a map $\Hom^{k+1}_{\icate}(X,Y)_{\bullet\ldots\bullet}\rightarrow \icate_{1\ldots 111\bullet\ldots \bullet}$ of $(n-k-1)$-fold Segal objects by construction.
    \end{itemize}

\end{enumerate}

\section{The pointless higher Morita category}\label{sec:Morita-cat}

In this section we recall and extend the results of \cite{KS} on pointless factorization algebras and construct the pointless higher Morita category.
This model for the higher Morita category is the opposite of pointless -- as explained in the introduction, it is essential for the proof of the main Theorem \ref{thm:mainthm}.

\subsection{Recollection of stratified spaces}

In the pointless factorization model for the Morita category, objects, morphisms, higher morphisms, and composable sequences of morphisms are all factorization algebras on very simple spaces -products of the interval $(0,1)$- that are required to be \emph{constructible} with respect to various stratifications.
\subsubsection{Poset-stratified spaces and conical manifolds}
We quickly recall the basics notions of stratified topology, after Lurie \cite[Appendix B]{LurHA} and Ayala-Francis-Tanaka \cite{AFT-local-structures}. 
\begin{defn}[Poset-stratified spaces]
Let $P$ be a partially ordered set, then we may regard $P$ as a topological space with the \emph{Alexandrov topology}: we declare a subset $U\subset P$ open just in case it is upward closed, that is, if $x\in U$ and $y\geq x$, then $y\in U$. A map $f:P\rightarrow Q$ between posets is order-preserving if and only if it is continuous for the Alexandrov topologies on $P$ and $Q$, so we have a fully faithful embedding
\[  \mathsf{Poset}\hooklongrightarrow \topo, \]
identifying the category of posets with a full subcategory of the category $\topo$ of topological spaces that is moreover stable under limits. We will write $\mathsf{StTop}$ for the pullback
\[  \mathsf{StTop} := \fun(\Delta^1,\topo)\times_{\fun(\{1\},\topo)}\mathsf{Poset} \]
of categories; this is the category of \emph{(poset) stratified spaces}. 
\end{defn}
In particular, poset stratified spaces are continuous maps $X\rightarrow P$ for $P$ a poset, and a map $(X\rightarrow P)\rightarrow (Y\rightarrow Q)$ is a commutative square in $\mathsf{Top}$. Since $\mathsf{Poset}\subset \topo$ is stable under limits, $\mathsf{StTop}\subset \fun(\Delta^1,\topo)$ is stable under limits; in particular, $\mathsf{StTop}$ admits all limits. 
\begin{defn}\label{defn:openemb}
A map
\[
\begin{tikzcd}
X \ar[d] \ar[r,"f"] & Y\ar[d] \\
P \ar[r,"g"] & Q
\end{tikzcd}
\]
of poset stratified spaces is an \emph{open embedding} if both $f$ and $g$ are topological open embeddings; that is, $f$ identifies $X$ with an open subset of $Y$ and $g$ identifies $P$ with an upward closed subset of $Q$. We let $\mathsf{StTop}_{\mathsf{open}}\subset \mathsf{StTop}$ denote the subcategory on the open embeddings. A collection $\{(X_i\rightarrow P_i)\rightarrow (Y\rightarrow Q)\}_i$ is an \emph{open cover} if each map $(X_i\rightarrow P_i)\rightarrow (Y\rightarrow Q)$ is an open embedding of stratified spaces and the maps $\coprod_iX_i\rightarrow Y$ and $\coprod_i P_i\rightarrow Q$ are surjections. The collection of open covers generates an evident Grothendieck topology on the category $\mathsf{StTop}$.
\end{defn}

A fundamental operation that produces more complicated stratified spaces from simpler ones is the formation of \emph{cones}.
\begin{cons}[Cones]
For $P$ a poset, we let $P^{\rhd}$ and $P^{\lhd}$ denote the posets obtained from $P$ by adjoining a new maximal element $\infty$ respectively minimal element $-\infty$. Let $p:X\rightarrow P$ be a stratified space, then we define a new stratified space $q:\mathsf{C}(X)\rightarrow P^{\lhd}$ as follows: the space $\mathsf{C}(X)$ is the topological cone
\[  \mathsf{C}(X):= * \coprod_{X\times \{0\}}  X\times [0,\infty)  \]
while the map $q$ is given by the formula
\[
q(y) = \begin{cases}
   p(x)& \text{if } y=(x,t)\in X\times (0,\infty), \\
    -\infty & \text{if $y=*$, the cone point}.  
\end{cases}
\]
One readily verifies that $q$ is continuous, so $\mathsf{C}(X)\rightarrow P^{\lhd}$ is a poset stratified space.
\end{cons}
Among (paracompact Hausdorff) topological spaces, the topological manifolds are those spaces that can be obtained from Euclidean spaces via gluing along open subspaces. Similarly, among (paracompact Hausdorff) stratified spaces, the \emph{conical manifolds} are those that can be obtained from Euclidean spaces via gluing along open subspaces \emph{and} the formation of cones. 
\begin{defn}[Conical manifolds (\cite{AFT-local-structures})]\label{defn:conmfd}
Consider full subcategories $\mathcal{T}$ of $\mathsf{StTop}_{\mathsf{open}}$ with the following properties. 
\begin{enumerate}[$(1)$]
    \item The empty space $(\emptyset\rightarrow \emptyset)$ is contained in $\mathcal{T}$.
    \item If $(X\rightarrow P)$ lies in $\mathcal{T}$ and both $X$ and $P$ are compact, then $(\mathsf{C}(X)\rightarrow P^{\lhd})$ lies in $\mathcal{T}$.
    \item If $(X\rightarrow P)$ lies in $\mathcal{T}$, then $(X\times \R\rightarrow X\rightarrow P)$ lies in $\mathcal{T}$.
    \item If $(X\rightarrow P)$ lies in $\mathcal{T}$ and $(U\rightarrow Q)\hookrightarrow (X\rightarrow P)$ is a map in $\mathsf{StTop}_{\mathsf{open}}$, then $(U\rightarrow Q)$ lies in $\mathcal{T}$.
    \item If $(X\rightarrow P)$ admits an open cover $\{(U_i\rightarrow Q_i)\hookrightarrow (P\rightarrow X)\}_i$ for which all $(U_i\rightarrow Q_i)$ lie in $\mathcal{T}$, then $(X\rightarrow P)$ lies in $\mathcal{T}$.
\end{enumerate}
We let $\mfd_0$ be the full subcategory of $\bigcap_{\mathcal{T}}\mathcal{T}\subset \mathsf{StTop}_{\mathsf{open}}$ spanned by those $(X\rightarrow P)$ for which $X$ is paracompact Hausdorff. These are \emph{conical manifolds}\footnote{These are called $C^0$-\emph{stratified spaces} in \cite{AFT-local-structures}. We follow the terminological conventions of \cite{KSW}.}.
\end{defn}
Let us perform the following routine verification.
\begin{lem}
Let $X\rightarrow P$ be an object of $\mfd_0$ and suppose that $P$ is a discrete set. Then $X$ is a topological manifold. 
\end{lem}
\begin{proof}
If $P$ is discrete, then $X$ is a disjoint union $\coprod_pX_p$ so we may suppose that $P$ contains a single element. Now let $\mathcal{T}\subset \mathsf{StTop}_{\mathsf{open}}$ be the full subcategory spanned by objects $(Y\rightarrow Q)$ having the following property: if $Q$ is a singleton, then $Y$ is locally equivalent to a Euclidean space. It suffices to check that $\mathcal{T}$ satisfies properties $(1)$ through $(5)$ of Definition \ref{defn:conmfd}. Note that $(1)$ is obvious. The condition $(2)$ is satisfied because the cone on a poset $S$ is a singleton if and only if $S=\emptyset$ and the only space with stratifying poset $\emptyset$ is $\emptyset$, the cone on which is the singleton, which is a topological manifold. Conditions $(3)$, $(4)$ and $(5)$ are satisfied because the property of being locally equivalent to a Euclidean space is stable under products with Euclidean spaces, stable under taking open subspaces, and a local property. 
\end{proof}

As noted in \cite{AFT-local-structures}, it is easy to check that for $(p:X\rightarrow P)\in \mfd_0$, the map $p$ is surjective, by verifying that the full subcategory of $\mathsf{StTop}_{\mathsf{open}}$ spanned by objects for which this is true satisfies $(1)$ through $(5)$. Similarly, we have the following useful property of conical manifolds that we could not locate in \cite{AFT-local-structures}.
\begin{lem}\label{lem:structuremapisopen}
For a conical manifold $X$, the structure map $p:X\rightarrow P$ to the stratifying poset is an open map.      
\end{lem}
\begin{proof}
Let $\mathcal{T}\subset \mathsf{StTop}_{\mathsf{open}}$ be the full subcategory spanned by objects $(p:X\rightarrow P)$ for which $p$ is an open map. Clearly, $(1)$ is satisfied, and $(3)$ is satisfied because projections from products are open. The property $(4)$ is satisfied because in a composite $A\overset{f}{\rightarrow}B\overset{g}{\rightarrow}C$ of spaces, if $g$ and $g\circ f$ are open embeddings, then so is $f$. The property $(5)$ is satisfied because being an open map of topological spaces is local on the source and on the target. Suppose that $p:X\rightarrow P$ is open (and both $X$ and $P$ are compact), then we show that $q:\mathsf{C}(X)\rightarrow P^{\lhd}$ is open. We consider two cases: opens $U\subset \mathsf{C}(X)$ that do not intersect the cone point and ones that do. In the first case, $U$ is an open of $X\times \R_{>0}$. Every such open is a union of product opens $V\times W$ for $V\subset X$ and $W\subset \R_{>0}$, so we may assume that $U$ is such a product. Then $q(U=V\times W)=p(V)$, which is open (that is, upward closed) in $P$ by assumption, and therefore also upward closed in $P^{\lhd}$. If $U$ does intersect the cone point, then $q(U)=P^{\lhd}$, which is open.
\end{proof}

\begin{cor}\label{cor:openemb}
Let $(X\rightarrow P)$ be a conical manifold, then the functor 
\[ (\mfd_0)_{/X} \longrightarrow \mathsf{Open}(X),\qquad (Y\rightarrow Q)\overset{f}{\rightarrow} (X\rightarrow P)\longmapsto f(Y)\subset X\]
is an equivalence\footnote{In \cite{AFT-local-structures}, a different definition of an open embedding of poset-stratified spaces is given (\cite[Definition 2.1.11]{AFT-local-structures}). Even though Definition \ref{defn:conmfd} makes reference to open embeddings of stratified spaces, it is not hard to see that \cite[Definition 2.1.11]{AFT-local-structures} and our notion of open embedding produce the same collection of conical manifolds. However, Corollary \ref{cor:openemb} is false for open embeddings in the sense of \cite{AFT-local-structures}: let $M=M_1\coprod M_2$ be a disjoint union of topological manifolds, then $M$ can be made into an object of $\mfd_0$ in two ways: as stratified by $*$ or as stratified by $*\coprod *$. For \cite{AFT-local-structures}, the identity map $M\rightarrow M$ over the map $*\coprod *\rightarrow *$ is a valid open embedding, which spoils the equivalence of Corollary \ref{cor:openemb}. Since this equivalence is used implicitly throughout \cite{AFT-local-structures} (and throughout this work as well), we settle on Definition \ref{defn:openemb}.}.
\end{cor}
\begin{proof}
It follows from Lemma \ref{lem:structuremapisopen} that the functor is essentially surjective. By definition of an open embedding of stratified spaces, there is, for $U\hookrightarrow X$ and $V\hookrightarrow X$ open embeddings, a \emph{unique} map $U\rightarrow V$ over $X$ if the image of $U$ in $X$ is contained in the image of $V$ in $X$, and no map otherwise. 
\end{proof}

\begin{rmk}[Basics]
In \cite{AFT-local-structures}, it is shown that the full subcategory $\mfd_0$ of conical manifolds can be characterized as those stratified spaces $(X\rightarrow P)$ (with $X$ paracompact Hausdorff) for which the images of stratified open embeddings 
\[  \mathsf{C}(Z)\times \R^i\hooklongrightarrow X \]
for $Z$ a compact conical manifold is a basis for the topology of $X$. The objects $\mathsf{C}(Z)\times \R^i$ are so called (topological) \emph{basics}, the building blocks of conical manifolds.
\end{rmk}

\begin{rmk}[Strata]\label{rmk:strata}
If $(X\rightarrow P)$ is a stratified space and $p\in P$, then the fiber over $p$ is the \emph{$p$'th stratum of $X$}. In \cite[Section 2.3]{AFT-local-structures} it is shown that strata of conical manifolds are again conical manifolds; in fact, whenever $Q\subset P$ is a consecutive full subposet (that is, for every $x,y \in Q$ and any $t\in P$ such that $x\leq t\leq y$, then $t\in Q$), then the pullback of $X$ to along $Q\subset P$ is a conical manifold.
\end{rmk}

\subsubsection{Conical smoothness and smooth conical manifolds}

We now partially recall from \cite{AFT-local-structures} the notions of conical smoothness and smooth conical manifolds. This definition is subtle and inducts on a parameter defined for conical manifolds called the \emph{depth}, which keeps track of how many times a cone has been taken; we refer to \cite[Section 2.4]{AFT-local-structures} for the definition and basic results on the depth. The notion of conical smoothness and smooth conical manifolds at a given value of the depth depends explicitly on these notions for lower values of the depth. For the base case of the induction\footnote{We could also have started the induction with smooth conical manifolds of depth $-1$, the category of which is trivial, containing the empty space; then Definition \ref{defn:consmfmfd} would have produced the category described here at depth $0$, but we find it more illuminating to spell this category out explicitly.}, we say that
\begin{itemize}
    \item A \emph{smooth conical manifold} of depth $\leq 0$ consists of a smooth manifold $M$ together with a surjective continuous (equivalently, smooth) function $M\rightarrow S$ to a set equipped with the discrete topology; that is, an object of $\mfd_0$ with discrete stratifying poset, together with a smooth manifold structure on the space is itself. This is nothing but a manifold together with a choice of disjoint decomposition into components (which need not be connected).
    \item A \emph{conically smooth open embedding} among two such is a stratified open embedding that is also a smooth embedding, that is, a smooth open embedding respecting the disjoint decompositions. 
\end{itemize}
Now make the following definition by \cite{AFT-local-structures}.
\begin{defn}[Smooth conical manifolds and conically smooth open embeddings \cite{AFT-local-structures}]\label{defn:consmfmfd}
Let $k\in \Z_{\geq 0}$. Given the notion of smooth conical manifolds and conically smooth open embeddings among them at depth $\leq 0$, we declare the following. 
\begin{enumerate}[$(1)$]
    \item A \emph{(smooth) basic} of depth $\leq k+1$ is a pair of a (topological) basic $U=\R^i \times \mathsf{C}(Z)$ for $Z$ a compact conical manifold of depth $\leq k$, together with a smooth conical manifold structure on $Z$, as defined in $(3)$ below (or just above, in case $k=0$). 
    \item A \emph{conically smooth open embedding} between two basics $U=\R^j\times \mathsf{C}(Y)$, $V=\R^i\times \mathsf{C}(Z)$ of depth $\leq k+1$ is a stratified open embedding $f:V\hookrightarrow U$ such that the following is satisfied.
    \begin{enumerate}[$(i)$]
        \item If the underlying map on posets does not carry the cone point to the cone point, then $f$ factors as a stratified open embedding 
        \[ f':\R^i\times \mathsf{C}(Z) \hooklongrightarrow \R^j\times\R_{>0}\times Y, \]
        where $\R^j\times\R_{>0}\times Y$ has a canonical smooth conical manifold structure (at depth $k$) by $(3)$ below. Since the depth of $Y$ is $\leq k$, this means that the depth of $Z$ must be $\leq k-1$ (since the formation of cones increases depth \emph{strictly}). Thus $V$ admits a canonical conically smooth manifold structure (at depth $\leq k$) by $(3)$ below, and we ask that $f'$ is a conically smooth open embedding, well defined by induction.
        \item If the underlying map of posets carries the cone point to the cone point, we impose the following.
        \begin{itemize}
            \item $f$ is \emph{conically smooth along} $\R^i$, as defined and explained in \cite[Section 3.1]{AFT-local-structures}, with injective derivatives\footnote{Starting the induction at depth $-1$, this injectivity is the condition that guarantees that maps between smooth conical manifolds at depth $0$ are \emph{smooth} (instead of merely topological) embeddings.}. 
            \item The conical manifold $\R^j\times \R_{>0}\times Y$  and its inverse image under $f$ admit canonical smooth conical manifold structures (at depth $\leq k$) by $(3)$ below and we ask that the map
            \[  f^{-1}(\R^j\times \R_{>0}\times Y)\longrightarrow \R^j\times \R_{>0}\times Y \]
            is a conically smooth open embedding, well defined by induction.
        \end{itemize}
   \end{enumerate}
    \item Let $X$ be a conical manifold of depth $\leq k+1$, then an \emph{atlas} on $X$ is a collection of stratified open embeddings $\{f_{\alpha}:U_{\alpha}\hookrightarrow X\}_{\alpha}$ called \emph{charts}, where $U_{\alpha}$ is a smooth basic of depth $\leq k+1$, such that the following are satisfied.
    \begin{itemize}
        \item The stratified open embeddings $\{f_{\alpha}:U_{\alpha}\hookrightarrow X\}_{\alpha}$ are a cover of $X$ in the sense of Definition \ref{defn:openemb}.
        \item For every pair of charts $f_{\alpha},f_{\beta}$ with intersecting image in $X$ and every $x\in f_{\alpha}(U_{\alpha})\cap f_{\beta}(U_{\beta})$, there is another chart $f_{\gamma}$ whose image contains $x$ and conically smooth open embeddings 
        \[  U_{\alpha} \longleftarrow  U_{\gamma} \longrightarrow  U_{\beta}\]
        such that the diagram 
        \[
        \begin{tikzcd}
            U_{\gamma}\ar[d,hook] \ar[r,hook] \ar[dr,"f_{\gamma}"] & U_{\beta} \ar[d,hook,"f_{\beta}"] \\
            U_{\alpha} \ar[r,hook,"f_{\alpha}"'] & X
        \end{tikzcd} 
        \]
        of stratified open embeddings commutes.
    \end{itemize}
    Two atlases $\mathcal{A}$ and $\mathcal{A}'$ are \emph{equivalent} if their union is an atlas; this defines an equivalence relation on the set of atlases (\cite[Lemma 3.2.11]{AFT-local-structures}). Every equivalence class has a canonical representative given by the maximal atlas, the union of all atlases in a given equivalence class\footnote{But some set-theoretic care must be exercised; see \cite[Footnote 9]{KSW}.}. A \emph{smooth conical manifold (at depth $\leq k+1$)} is a conical manifold of depth $\leq k+1$ together with an equivalence class of atlases. Given a smooth conical manifold $X$ of depth $\leq k+1$, there are canonical smooth conical manifold structures of depth $\leq k+1$ on 
    \[ U\subset X\,\text{ open}\qquad \text{and}\qquad X\times\R^p,\, p\in \Z_{\geq 1},\]
    by pulling back atlases and by taking products of basics with Euclidean spaces respectively (\cite[Definition 3.2.12, Definition 3.2.17]{AFT-local-structures}).
    \item Let $X$ and $Y$ be two smooth conical manifolds at depth $\leq k+1$, then a \emph{conically smooth open embedding} $f:X\rightarrow Y$ is a stratified open embedding with the following property: for every basic $U\hookrightarrow X$ in the maximal atlas of $X$, the composite
    \[  U\hooklongrightarrow X\hooklongrightarrow Y  \]
    lies in the maximal atlas of $Y$.
\end{enumerate}
If $X$ is a conical manifold of unbounded depth, then a smooth conical manifold structure on $X$ is simply a collection of compatible smooth conical structures on all open subspaces of bounded depth, and conically smooth open embeddings are stratified open embeddings that locally are conically smooth open embeddings as in $(4)$. This yields the category 
\[  \mfd \]
of \emph{smooth conical manifolds} and \emph{conically smooth open embeddings}. The basics $\R^i\times\mathsf{C}(Z)$ for $Z$ a compact smooth conical manifold of arbitrary (finite) depth form a full subcategory 
\[  \bsc \subset \mfd.\]
\end{defn}
We take care to spell out this definition as we will explicitly endow some (very simple) conical manifolds with smooth conical manifold structures in Section \ref{sec:moritacat}.
\begin{rmk}
Given Definition \ref{defn:consmfmfd}, there is an evident notion of a \emph{conically smooth map} (which need not be an embedding) between smooth conical manifolds essentially by omitting the injectivity condition in $(ii)$ above (\cite[Definition 3.3.1]{AFT-local-structures}). We will not need this general notion in this work (in fact, somewhat paradoxically, the maps between smooth conical manifolds that are not open embeddings that we will consider are usually \emph{not} conically smooth; see Warning \ref{warn:consmf}).
\end{rmk}
\begin{rmk}
Suppose $(X\rightarrow P)$ is a conical manifold equipped with a conically smooth structure. For $Q\subset P$ a \emph{consecutive} full subposet, then the pullback $X_Q$ of $X$ along $Q\subset P$ (which is a conical manifold; see Remark \ref{rmk:strata}) admits a canonical conically smooth structure and the inclusion $X_Q\subset X$ is conically smooth (but not a conically smooth open embedding unless $Q\subset P$ is upward closed), by \cite[Lemma 3.4.5]{AFT-local-structures}. We will put this observation to use to construct the \infops of interest to us in Section \ref{sec:plessdisks}.
\end{rmk}
The results of \cite[Section 4.1]{AFT-local-structures} yield a locally Kan simplicial category $\infmfd_{\simp}$ whose objects are smooth conical manifolds and whose simplicial morphism sets are families of conically smooth open embeddings smoothly parameterized by the simplices. We let 
\(  \infmfd \)
denote the \infcat obtained by taking the coherent nerve of $\infmfd_{\simp}$, and we have a full subcategory 
\[ \infbsc \subset \infmfd\]
spanned by (smooth) basics.
\subsection{Pointless (pre)factorization algebras}
In this section we recall the notion of (pre)factorization algebras on marked spaces, and constructible (pre)factorization algebras on marked smooth conical manifolds from \cite{Karlssonthesis, KS}.
\subsubsection{(Pre)factorization algebras on marked spaces}
Henceforth, all topological spaces will be assumed Hausdorff, and $\topo$ will denote the category of \emph{Hausdorff} topological spaces.
\begin{defn}[Marked spaces]
Let $X$ be a topological space and let $E\subset X$ be a subset whose induced topology is discrete. We will say that the pair $(X,E)$ is a \emph{marked space}. For $(X,E)$ and $(Y,F)$ marked spaces, a \emph{continuous map of marked spaces} $f:(X,E)\rightarrow (Y,F)$ is a continuous map $f:X\rightarrow Y$ such that $f(E)\subset F$; this yields a category $\topo^+$ of marked spaces equipped with an evident functor 
\[ G:\topo^+\longrightarrow \topo  \]
forgetting the marking.
\end{defn}
\begin{rmk}
The fiber of $G:\topo^+\rightarrow \topo$ over some topological space $X$ is the set $P_{\mathrm{disc}}(X)$ of discrete subspaces of $X$ partially ordered by inclusion. The functor $G$ is neither a coCartesian fibration nor a Cartesian fibration (neither the image nor the preimage of a discrete set is discrete in general). However $G$ behaves better on natural subcategories of $\topo$. For instance, the functor $G_{\mathrm{inj}}:\topo^+_{\mathrm{inj}}\rightarrow \topo_{\mathrm{inj}}$ between the subcategories $\topo^+_{\mathrm{inj}}\subset \topo^+$ and $\topo_{\mathrm{inj}}\subset \topo$ whose morphisms are \emph{injective} continuous maps is a Cartesian fibration, since the preimage of a discrete set along an injective continuous map is discrete. In particular, the functor $\topo_{\mathsf{open}}^+\rightarrow \topo_{\mathsf{open}}$ on subcategories on \emph{open topological embeddings} is a Cartesian fibration. We let $\topo_{\mathsf{open}}^{+,\simeq}\subset \topo_{\mathsf{open}}^+$ denote the subcategory on the Cartesian morphisms. 
\end{rmk}
Prefactorization algebras on some topological space $X$ are algebras for a natural operad $\mathsf{Open}(X)^{\otimes}$ of opens of $X$. We now introduce the analogous operads for marked spaces.
\begin{defn}[Operads of marked opens \cite{Karlssonthesis}]
Let $(X,E)$ be a marked space. We will write $\mathsf{Open}(X,E)\subset\mathsf{Open}(X)$ for the wide subposet on those inclusions of opens $U\subset V$ that have the following property: the inclusion $U \cap E\subset V\cap E$ is a bijection. We refer to the inclusions of $\mathsf{Open}(X)$ that belong to $\mathsf{Open}(X,E)$ as \emph{marked inclusions}. We now define the operator category of a colored operad $\mathsf{Open}(X,E)^{\otimes}$ with underlying category $\mathsf{Open}(X,E)$ as follows. 
\begin{enumerate}
    \item[$(O)$] Objects of $\mathsf{Open}(X,E)^{\otimes}$ are pairs $\left(\langle n\rangle, \left\{U_1,\ldots,U_n\right\}\right)$ of a pointed finite set $\langle n\rangle$ together with a collection of opens of $X$.
    \item[$(M)$] A morphism $\left(\langle n\rangle, U_1\oplus\ldots \oplus U_n\right)\rightarrow\left(\langle m\rangle,V_1\oplus\ldots \oplus V_m\right)$ is a map $f:\langle n\rangle \rightarrow \langle m\rangle$ of pointed finite sets such that for all $j\in \langle m\rangle^{\circ}$, the opens $\{U_i\}_{i\in f^{-1}(j)}$ are disjoint and lie in $V_j$ \emph{and} the inclusion 
    \[  \left( \bigcup_{i\in f^{-1}(j)}U_i \right) \cap E  \hooklongrightarrow V_j \cap E \]
    is a bijection.
\end{enumerate}
We will also write $\mathsf{Open}(X,E)^{\otimes}$ for the associated \infop through \cite[Definition 2.1.1.23]{LurHA}. 
\end{defn}
\begin{rmk}
If the set $E$ is empty, the above definition recovers the usual operad of opens $\mathsf{Open}(X)^{\otimes}$ of, for instance, \cite[Definition 5.5.2.1]{LurHA}.    
\end{rmk}
\begin{rmk}[Pullback functoriality of operads of marked opens]\label{rmk:functormarkedopens}
For $f:(X,E)\rightarrow (Y,F)$ a continuous map of marked spaces and $U\subset V\subset Y$ an inclusion such that $U\cap F\hookrightarrow V\cap F $ is a bijection, then the inclusion $f^{-1}(U)\cap E\hookrightarrow f^{-1}(V)\cap E$ is also a bijection. It follows right away that pulling back opens along continuous maps of marked spaces determines a functor 
\[  \left(\topo^+\right)^{op} \longrightarrow \mathsf{Op},\quad\quad (X,E)\longmapsto \mathsf{Open}(X,E)^{\otimes} \]
to the 1-category of ordinary operads carrying a continuous map $f:(X,E)\rightarrow (Y,F)$ to the pullback 
\[ f^{-1}:\open(Y,F)^{\otimes}\longrightarrow\open(X,E)^{\otimes}. \]    
\end{rmk}
\begin{rmk}[Pushforward functoriality of operads of marked opens]\label{rmk:functormarkedopenspush}
For $i:(U,E)\hookrightarrow (Y,F)$ a continuous map of marked spaces such that $U\rightarrow Y$ is an open topological embedding. Suppose that $i$ is $G$-Cartesian, that is, the map $E\hookrightarrow i^{-1}(F)=F\cap i(U)$ is an equality, then for $V\subset W\subset U$ an inclusion such that $V\cap E\hookrightarrow W\cap E$ is a bijection, the inclusion $i(V)\cap F\subset i(W)\cap F$ is also a bijection. It follows right away that taking images of opens by $G$-Cartesian open embeddings of marked spaces determines a functor 
\[ \topo^{+,\simeq}_{\mathsf{open}} \longrightarrow \mathsf{Op},\quad\quad (X,E)\longmapsto \open(X,E)^{\otimes}  \]
that takes a $G$-Cartesian open embedding $i:(U,E)\hookrightarrow (Y,F)$ to the pushforward
\[ i_*:\open(U,E)^{\otimes}\longrightarrow \open(Y,F)^{\otimes}.  \]
\end{rmk}
\begin{rmk}
The poset $\open(X,E)$ does not admit finite products in general, but it does admit pullbacks, and the inclusion $\open(X,E)\subset \open(X)$ preserves them. Unlike $\open(X)$, $\open(X,E)$ generally does not have colimits.
\end{rmk}

\begin{defn}[Prefactorization algebras on marked spaces]
Let $(X,E)$ be a marked space and let $\icat^{\otimes}$ be a symmetric monoidal \infcatt. The \infcat of \emph{prefactorization algebras on $(X,E)$ with values in $\icat$} is
\[ \pfact_{(X,E)}(\icat):= \alg_{\open(X,E)}(\icat), \]
the \infcat of $\open(X,E)^{\otimes}$-algebras in $\icat$.
\end{defn}
\begin{rmk}[Functoriality of prefactorization algebras]\label{rmk:pushfwd}
The formation of prefactorization algebras is evidently functorial in $(X,E)$ and $\icat^{\otimes}$: we have a functor
\[  \topo^+\times \calg(\catinf) \xrightarrow{\open(\_)^{\otimes}\times \mathrm{id}} \mathsf{Op}^{op} \times \calg(\catinf) \xrightarrow{\alg_{(\_)}(\_)}\catinf  \]
where the first map is the functor of Remark \ref{rmk:functormarkedopens}. For $f:(X,E)\rightarrow (Y,F)$ a continuous map of marked spaces and $\icat^{\otimes}$ a fixed symmetric monoidal \infcatt, we let $f_{\sharp}$ denote the induced functor 
\[  f_{\sharp}:\pfact_{(X,E)}(\icat)\longrightarrow \pfact_{(Y,F)}(\icat), \]
the \emph{pushforward}. Similarly, we have a functor
\[  \left(\topo_{\mathsf{open}}^{+,\simeq}\right)^{op}\times \calg(\catinf) \xrightarrow{\open(\_)^{\otimes}\times \mathrm{id}} \mathsf{Op}^{op} \times \calg(\catinf) \xrightarrow{\alg_{(\_)}(\_)}\catinf  \]
where the first map is the functor of Remark \ref{rmk:functormarkedopenspush}. For $i:(U,E)\hookrightarrow (Y,F)$ a $G$-Cartesian open embedding of marked spaces and $\icat^{\otimes}$ a fixed symmetric monoidal \infcatt, we let $i^*$ denote the induced functor 
\[  i^*:\pfact_{(Y,F)}(\icat)\longrightarrow \pfact_{(U,E)}(\icat), \]
the \emph{pullback}. 
\end{rmk}
A factorization algebra in addition should satisfy two conditions: multiplicativity and descent for Weiss covers, both of which we explain below.
\begin{obs}
Let $(X,E)$ be a marked space, then a morphism 
\[  \left(\langle n\rangle, U_1\oplus\ldots \oplus U_n \right) \longrightarrow  \left(\langle m\rangle, V_1\oplus\ldots \oplus V_m \right)\]
of $\open(X,E)^{\otimes}$ is coCartesian for the structure functor $\open(X,E)^{\otimes}\rightarrow \fin$ just in case for all $j\in \langle m\rangle^{\circ}$, the map 
\[  \coprod_{i\in f^{-1}(j)} U_i\longrightarrow V_j \]
of $\mathsf{Open}(X,E)$ is the identity.
\end{obs}

\begin{defn}[Multiplicativity]
We say that a prefactorization algebra $\F$ on a marked space $(X,E)$ with values in a symmetric monoidal \infcat $\icat^{\otimes}$ is \emph{multiplicative} if it carries coCartesian morphisms of $\open(X,E)^{\otimes}$ to coCartesian morphisms of $\icat^{\otimes}$. We denote the full subcategory spanned by multiplicative prefactorization algebras by $\pfact^{\mathfrak{m}}_{(X,E)}(\icat)$.	
\end{defn}

If $\F$ is a multiplicative factorization algebra, then the structure maps associated to finite disjoint unions is an equivalence, in particular,
\[\F(U)\otimes \F(V) \overset{\simeq}\longrightarrow \F\left(U\coprod V\right)\]
for every disjoint pair $U,V\subset X$ of opens. By a straightforward inductive argument, $\F$ is multiplicative if and only if for every such pair $U,V\subset X$ of disjoint opens of $X$, the structure map above is an equivalence. 

\begin{rmk}\label{rmk:pushfwdmultiplicative}
Let $f:(X,E)\rightarrow (Y,F)$ be a continuous map of marked spaces, then the induced map of operads $f^{-1}:\mathsf{Open}(Y,F)^{\otimes}\rightarrow \mathsf{Open}(X,E)^{\otimes}$ preserves coCartesian morphisms, so it follows that the induced functor $\pfact_{(X,E)}(\icat)\rightarrow \pfact_{(Y,F)}(\icat)$ carries multiplicative algebras to multiplicative algebras. It follows that the functor 
\[ \topo^+\longrightarrow \catinf,\quad\quad (X,E)\longmapsto \pfact_{(X,E)}(\icat)  \]
has a subfunctor $\pfact^{\mathfrak{m}}_{(\_)}(\icat)$.
\end{rmk}
\begin{rmk}\label{rmk:pbmultiplicative}
Let $f:(U,E)\hookrightarrow (Y,F)$ be a $G$-Cartesian open embedding of marked spaces, then the induced functor $i_*:\mathsf{Open}(U,E)\rightarrow \mathsf{Open}(Y,F)$ preserves coCartesian morphisms, so it follows that the induced functor $\pfact_{(Y,F)}(\icat)\rightarrow \pfact_{(U,E)}(\icat)$ carries multiplicative algebras to multiplicative algebras. It follows that the functor 
\[ \left(\topo_{\mathsf{open}}^{+,\simeq}\right)^{op}\longrightarrow \catinf,\quad\quad (X,E)\longmapsto \pfact_{(X,E)}(\icat)  \]
has a subfunctor $\pfact^{\mathfrak{m}}_{(\_)}(\icat)$.
\end{rmk}

The second condition for factorization algebras is a gluing condition involving certain covers which are compatible with the marking.
\begin{defn}[Weiss topology]
Let $(X,E)$ be a marked space and let $U\subset X$ be an open subset.
\begin{itemize}
	\item An open cover $\{V_i\subset U\}_{i\in I}$ is called a \emph{Weiss cover} if, for each finite subset $S\subset U$, there is some $V_i$ such that $S\in V_i$.
	\item A \emph{marked Weiss cover} $\{V_i\subset U\}_{i\in I}$ of \(U\) is a Weiss cover such that every inclusion $V_i\subset U$ is a marked inclusion. 
\end{itemize}
The collection of marked Weiss covers constitutes a Grothendieck pretopology on $\open(X,E)$ generating the \emph{Weiss (Grothendieck) topology}; this is simply the restriction of the Weiss Grothendieck topology on $\open(X)$ to $\open(X,E)$. We say that a copresheaf 
\[ \F:\open(X,E)\longrightarrow \icat \]
valued in an \infcat $\icat$ is a \emph{Weiss cosheaf} (or \emph{satisfies  Weiss codescent}) if the functor
\[ \open(X,E)^{op}\longrightarrow \icat^{op}   \]
on opposite \infcats is a sheaf for the Weiss topology on $\open(X,E)$; this yields an \infcat 
\[\mathsf{coShv}^{\mathrm{Weiss}}_{(X,E)}(\icat):=\shv^{\mathrm{Weiss}}_{\mathsf{Open}(X,E)}(\icat^{op})^{op}.\]
A \emph{Weiss (prefactorization) algebra} on the marked space $(X,E)$ is a prefactorization algebra on $(X,E)$ whose underlying precosheaf is a Weiss cosheaf. The \infcat of such fits into a pullback diagram 
\[
\begin{tikzcd}
    \pfact^{\mathrm{Weiss}}_{(X,E)}(\icat) \ar[d] \ar[r,hook] &\pfact_{(X,E)}(\icat) \ar[d] \\
    \mathsf{coShv}^{\mathrm{Weiss}}_{(X,E)}(\icat) \ar[r,hook] & \fun(\mathsf{Open}(X,E),\icat) 
\end{tikzcd}
\]
of \infcatst.
\end{defn}
\begin{rmk}\label{rmk:pushfwdweiss}
Let $f:(X,E)\rightarrow (Y,F)$ be a continuous map of marked spaces, then the induced functor $f^{-1}:\mathsf{Open}(Y,F)\rightarrow \mathsf{Open}(X,E)$ preserves pullbacks and carries marked Weiss covers to marked Weiss covers, so it follows that the induced functor $\fun(\open(X,E),\icat)\rightarrow \fun(\open(Y,F),\icat)$ carries Weiss cosheaves to Weiss cosheaves, and thereafter that the pushforward $f_{\sharp}$ carries Weiss algebras to Weiss algebras. It follows that the functor 
\[ \topo^+\longrightarrow \catinf,\quad\quad (X,E)\longmapsto \pfact_{(X,E)}(\icat)  \]
has a subfunctor $\pfact^{\mathrm{Weiss}}_{(\_)}(\icat)$.
\end{rmk}
\begin{rmk}\label{rmk:pbweiss}
Let $i:(U,E)\hookrightarrow (Y,F)$ be a $G$-Cartesian open embedding of marked spaces, then the induced functor $i_*:\mathsf{Open}(U,E)\rightarrow \mathsf{Open}(Y,F)$ preserves pullbacks and carries marked Weiss covers to marked Weiss covers, so it follows that the induced functor $\fun(\open(Y,F),\icat)\rightarrow \fun(\open(U,E),\icat)$ carries Weiss cosheaves to Weiss cosheaves, and thereafter that the pullback $i^*$ carries Weiss algebras to Weiss algebras. It follows that the functor 
\[ \left(\topo_{\mathsf{open}}^{+,\simeq}\right)^{op}\longrightarrow \catinf,\quad\quad (X,E)\longmapsto \pfact_{(X,E)}(\icat)  \]
has a subfunctor $\pfact^{\mathrm{Weiss}}_{(\_)}(\icat)$.
\end{rmk}

Factorization algebras are obtained from prefactorization algebras by simultaneously imposing the two conditions introduced just above.
\begin{defn}\label{defn:pointless-fact-algs} 
Let $(X,E)$ be a marked space and let $\icat^{\otimes}$ be a symmetric monoidal \infcatt. The \emph{$\infty$-category of factorization algebras on $(X,E)$ (with values in $\icat$)} is the pullback
\[  
\begin{tikzcd}
    \Fact_{(X,E)}(\icat) \ar[d,hook] \ar[r,hook] &\pfact^{\mathrm{Weiss}}_{(X,E)}(\icat)  \ar[d,hook] \\
   \pfact^{\mathrm{m}}_{(X,E)}(\icat) \ar[r,hook] & \pfact_{(X,E)}(\icat) 
\end{tikzcd}
\]
among \infcatst.
\end{defn}
 \begin{rmk}[Functoriality of factorization algebras]
Combining Remarks \ref{rmk:pushfwdmultiplicative} and \ref{rmk:pushfwdweiss}, we conclude that for a fixed symmetric monoidal \infcat $\icat^{\otimes}$, we have a functor 
\[  \topo^+\longrightarrow \catinf,\quad\quad (X,E)\longmapsto \Fact_{(X,E)}(\icat) \]
carrying a continuous map $f:(X,E)\rightarrow (Y,F)$ of marked spaces to the pushforward $f_{\sharp}$. Similarly, combining Remarks \ref{rmk:pbmultiplicative} and \ref{rmk:pbweiss}, we conclude that for a fixed symmetric monoidal \infcat $\icat^{\otimes}$, we have a functor 
\[  \left(\topo^{+,\simeq}_{\mathsf{open}}\right)^{op}\longrightarrow \catinf,\quad\quad (X,E)\longmapsto \Fact_{(X,E)}(\icat) \]
carrying a $G$-Cartesian open embedding $i:(U,E)\hookrightarrow (Y,F)$ of marked spaces to the pullback $i^*$.
\end{rmk}

\subsubsection{Marked smooth conical manifolds and constructibility}
When the underlying space of a marked space $(X,E)$ on which a factorization algebra is defined is a smooth conical manifold, it is natural to ask that the factorization algebra is \emph{constructible with respect to the stratification}. For a smooth conical manifold $X$, we will demand that the marking is compatible with the stratification, in the sense that the marking is required to be a subset of the 0-dimensional strata of $X$ as in the following definition.
\begin{nota}
For $X$ a smooth conical manifold, we let $X_0\subset X$ denote the union of all 0-dimensional strata. This is a discrete subspace of $X$.    
\end{nota}
\begin{defn}[Marked smooth conical manifolds]\label{defn:markedsmcon}
A \emph{marked smooth conical manifold} is a pair $(X,E)$ of a smooth conical manifold together with a subset $E\subset X_0$ of the set of 0-dimensional strata of $X$. A \emph{conically smooth open embedding} of marked smooth conical manifolds $(X,E)\rightarrow (Y,F)$ is a conically smooth open embedding $f:X\rightarrow Y$ such that $f(E)\subset F$. This yields a category $\mfd^+$ of marked smooth conical manifolds equipped with an evident functor 
\[  G_{\mfd}:\mfd^+\longrightarrow \mfd \]
forgetting the marking.
\end{defn}
\begin{rmk}
The fiber of $G_{\mfd}$ over some marked smooth conical manifold $X$ is the set $P(X_0)$ of subsets of $X_0$. Since $X_0\subset X$ is a discrete subspace, any subset $E\subset X_0$ inherits the discrete topology from $X$; we conclude that $G_{\mfd}$ is a subfunctor of the functor $G:\topo^+\rightarrow\topo$. Note that $G_{\mfd}$ is a Cartesian fibration: since a map $f:X\rightarrow Y$ of $\mfd$ is an open embedding, we have $f(X_0)=Y_0\cap X$ and the preimage of $F\subset Y_0$ under $f$ is the Cartesian lift of $f$ terminating at $F$.
\end{rmk}

In the next section, we will introduce the $\infty$-categorical variant of the category $\mfd^+$ (or rather, of the subcategory $\mfd^{+,\simeq}\subset \mfd^+$ on the Cartesian morphisms), for which we will have to work a bit harder. 
\begin{defn}[Pointless (pre)factorization algebras]
Any smooth conical manifold $X$ determines a marked smooth conical manifold $(X,X_0)$ with the \emph{maximal} marking. We will say that a (pre)factorization algebra on the marked smooth conical manifold $(X,X_0)$ is a \emph{pointless (pre)factorization algebra} on $X$.
\end{defn}

For $(X,E)$ a marked smooth conical manifold, we will write 
\[  \disk_{/(X,E)}\subset \disks_{/(X,E)} \subset \open(X,E)   \]
for the full subposets spanned by those opens that are abstractly isomorphic in the category $\mfd$ to a disk respectively to a disjoint union of such (that is, disjoint subordinate to $X$). Similarly, there are full suboperads 
\[  \disk_{/(X,E)}^{\otimes}\subset \disks_{/(X,E)}^{\otimes} \subset \open(X,E)^{\otimes}   \]
whose colors are those of $\open(X,E)^{\otimes}$ that lie in $\disk_{/(X,E)}$ or $\disks_{/(X,E)}$. We say that a marked inclusion $U\subset V$ of (disjoint unions of) disks is an \emph{isotopy equivalence} if $U$ and $V$ are abstractly isomorphic in the category $\mfd$.
\begin{defn}[Constructibility]
Let $(X,E)$ be marked smooth conical manifold and let $\F$ be a prefactorization algebra on $(X,E)$ valued in a symmetric monoidal \infcat $\icat^{\otimes}$. We will say that $\F$ is \emph{constructible} if for every marked inclusion $U\subset V$ of disks, the induced map
\[  \F(U) \longrightarrow \F(V)\]
is an equivalence; that is, if the restriction 
\[ \F|_{\disk_{/(X,E)}}:\disk_{/(X,E)}\longrightarrow \icat\]
carries isotopy equivalences to equivalences. We let
\[\mathsf{PFact}^{\mathrm{cstr}}_{(X,E)}(\icat)\subset \mathsf{PFact}_{(X,E)}(\icat),\qquad \Fact^{\mathrm{cstr}}_{(X,E)}(\icat)\subset \Fact_{(X,E)}(\icat) \]
denote the full subcategories spanned by constructible prefactorization algebras and constructible factorization algebras respectively.
\end{defn}

\begin{rmk}
Let $f:(X,E)\rightarrow (Y,F)$ be a map of marked spaces, and suppose that both $(X,E)$ and $(Y,F)$ are marked smooth conical manifolds, then the pushforward $f_{\sharp}$ generally does not carry constructible (pre)factorization algebras to constructible (pre)factorization algebras. In contrast, if $i:(U,E)\hookrightarrow (Y,F)$ is a $G_{\mathsf{Mfd}}$-Cartesian conically smooth open embedding, then it is immediate that the pullback $i^*$ does preserve constructibility. It follows that for a fixed symmetric monoidal \infcat $\icat^{\otimes}$, we have a functor
\[  \left(\topo^{+,\simeq}_{\mathsf{open}}\right)^{op}\longrightarrow \catinf,\quad\quad (X,E)\longmapsto \Fact^{\mathrm{cstr}}_{(X,E)}(\icat). \]
In \cite[Theorem 5.3]{KS}, it is shown that this functor is a sheaf (assuming that $\icat^{\otimes}$ is presentably symmetric monoidal and that $(X,E)$ has enough good disks as explained below).
\end{rmk}

When working with constructible factorization algebras, it is natural to restrict to cases in which a factorization algebra is \emph{fully determined} by its value on disks. As far as we are aware, it is not known whether factorization algebras (in the sense of Definition \ref{defn:pointless-fact-algs}) have this property for all marked smooth conical manifolds (in the precise sense articulated in Proposition \ref{prop:restricttodisks} below); however this is true for marked smooth conical manifolds that admit bases of disks that are well-behaved, in the following sense.
\begin{defn}[Enough good disks]\label{defn:factorizingdiskbasis}
Let $(X,E)$ be a marked smooth conical manifold. We say that the pair $(X,E)$ \emph{has enough good (marked) disks} if there is a full subposet $\mathcal{B}\subset \disks_{/(X,E)}$ with the following properties.
\begin{enumerate}[$(1)$]
    \item The elements of $\mathcal{B}$ constitute a basis for the topology of $X$.
    \item $\mathcal{B}$ is stable under disjoint decomposition of disks: if $U$ and $V$ are nonempty disjoint disks and $U\cup V$ lies in $\mathcal{B}$, then both $U$ and $V$ lie in $\mathcal{B}$.
    \item $\mathcal{B}$ is stable under disjoint union: if $U$ and $V$ lie in $\mathcal{B}$ and are disjoint, then $U\cup V$ lies in $\mathcal{B}$.
    \item $\mathcal{B}$ is stable under intersections: if $U=\bigcup_iU_i$ and $V=\bigcup_jV_j$ are disjoint unions of disks and $U,V\in\mathcal{B}$, then for all $i,j$, the intersection $U_i\cap V_j$ is either empty or an element of $\disk_{/(X,E)}$, and the disjoint union $U\cup V=\cup_{i,j}U_i\cap V_j$ is an element of $\mathcal{B}$.
\end{enumerate}
Such a full subposet $\mathcal{B}$ is a \emph{factorizing disk-basis of $(X,E)$}.
\end{defn}
\begin{rmk}\label{rmk:diskbasis}
We might also say that a full subposet $\mathcal{B}\subset \disks_{/(X,E)}$ is a \emph{disk-basis} if only conditions $(1)$, $(2)$ and $(3)$ hold. Any marked smooth conical manifold has a disk-basis: $\disks_{/(X,E)}$ itself is a disk-basis.    
\end{rmk}
\begin{rmk}\label{rmk:factorizingbasiscriterion}
Let $(X,E)$ be a marked smooth conical manifold and let $\mathcal{B}\subset\disk_{/(X,E)}$ be a full subposet. Suppose that the elements of $\mathcal{B}$ constitute a basis of $X$ and that $\mathcal{B}$ is closed under intersections. Then the full subposet $\mathcal{B}'\subset \disks_{/(X,E)}$ spanned by nonempty disjoint unions of elements of $\mathcal{B}$ is a factorizing disk-basis of $(X,E)$. 
\end{rmk}
\begin{rmk}
Let $X$ be a smooth manifold, then $X$ has enough good disks, or said differently, admits a factorizing disk-basis; this is the classical result that smooth manifolds admit good open covers. In fact, we do not know whether there are any (marked) smooth conical manifolds at all that do not admit a factorizing disk-basis.    
\end{rmk}
In \cite{KSW,KS}, the existence of a factorizing disk-basis is a standing assumption. In the coming sections, we will work with very simple smooth conical manifolds for which the existence of a factorizing disk-basis is easy to verify. We have the following result.
\begin{prop}\label{prop:restricttodisks}
Let $(X,E)$ be a marked space and suppose that $(X,E)$ has enough good disks. Let $\icat^{\otimes}$ be a presentably symmetric monoidal \infcat\footnote{In \cite{KSW,KS}, for simplicity it is assumed that (pre)factorization algebras take values in presentably symmetric monoidal \infcats to guarantee the existence of arbitrary (small) operadic left Kan extensions. A more careful analysis of the relevant operadic Kan extensions (like the one considered in Proposition \ref{prop:restricttodisks}) will reveal that it suffices to assume that $\icat$ admits sifted colimits and that the tensor product preserves sifted colimits.} and let 
\[\alg^{\mathfrak{m},\mathrm{Weiss}}_{\disks_{/(X,E)}}(\icat)\subset \alg_{\disks_{/(X,E)}}(\icat)\]
be the full subcategory spanned by those $\icat$-valued $\disks_{/(X,E)}^{\otimes}$-algebras that are multiplicative\footnote{If the source of a coCartesian arrow of $\open(X,E)^{\otimes}$ lies in $\disks_{/(X,E)}^{\otimes}$, then so does its target; we say that a $\disks_{/(X,E)}^{\otimes}$-algebra is multiplicative if it preserves coCartesian morphisms.} and satisfy Weiss codescent\footnote{Let $(X,E)$ be a marked space, then it is not hard to see that the following declaration specifies a Grothendieck topology on $\disks_{/(X,E)}$: let $U$ be a disjoint union of disks, then every sieve on $U$ in $\disks_{/(X,E)}$ (that is, every saturated full subcategory of $\left(\disks_{/(X,E)}\right)_{/U}$) that contains a marked Weiss cover is covering.}. Then restriction along $\disks_{/(X,E)}^{\otimes}\subset \open(X,E)^{\otimes}$ induces an equivalence
\begin{equation}\label{eq:restricttodisks}
\Fact_{(X,E)}(\icat) \overset{\simeq}{\longrightarrow} \alg^{\mathfrak{m},\mathrm{Weiss}}_{\disks_{/(X,E)}}(\icat)
\end{equation}
of \infcatst, with inverse given by operadic left Kan extension. Moreover, for any factorization algebra $\F$ on $(X,E)$, the restriction $\F|_{\open(X,E)}$ to the fiber over $\langle 1\rangle$ is also an \emph{ordinary} left Kan extension of its restriction to $\disks_{/(X,E)}$. 
\end{prop}
This follows from \cite[Proposition 7.10]{KS}, which is itself a minor variation on \cite[Proposition 4.3.2]{KSW}.
\begin{rmk}
The assumption on the existence of enough good disks may be avoided if one is willing to slightly alter the definition of factorization algebras. Let us say that a \emph{hyperfactorization algebra} on a marked space $(X,E)$ is a multiplicative prefactorization algebra whose underlying cosheaf satisfies Weiss \emph{hyper}codescent. Then it can be shown that the equivalence \eqref{eq:restricttodisks} holds for hyperfactorization algebras without any assumption on $(X,E)$ (where, of course, we also require that on the right hand side of the equivalence, we restrict to \emph{hyper}cosheaves)\footnote{We came to this observation in conversation with Anja \v{S}vraka and Tashi Walde.}. We will not pursue this point further in this work.  
\end{rmk}
Let $(X,E)$ be a marked smooth conical manifold and let
\[  \alg^{\mathrm{m},\mathrm{cstr}}_{\disks_{/(X,E)}}(\icat)\subset \alg^{\mathrm{m}}_{\disks_{/(X,E)}}(\icat) \]
be the full subcategory spanned by (multiplicative) $\disks_{/(X,E)}$-algebras for which the restriction of the underlying functor $\disks_{/(X,E)}\rightarrow \icat$ to $\disk_{/(X,E)}$ carries isotopy equivalences to equivalences (it follows right away from multiplicativity that also isotopy equivalences of disjoint unions of disks are carried to equivalences in $\icat$). It is a consequence of Proposition \ref{prop:restricttodisks} that in case $(X,E)$ has enough good disks, the restriction functor 
\[ \Fact_{(X,E)}^{\mathrm{cstr}}(\icat) \overset{\simeq}{\longrightarrow} \alg^{\mathfrak{m},\mathrm{Weiss}}_{\disks_{/(X,E)}}(\icat) \cap \alg^{\mathrm{m},\mathrm{cstr}}_{\disks_{/(X,E)}}(\icat) \]
is an equivalence of \infcatst. However, as asserted by the following result, the full subcategory $\alg^{\mathrm{m},\mathrm{cstr}}_{\disks_{/(X,E)}}(\icat)$ is already contained in $\alg^{\mathfrak{m},\mathrm{Weiss}}_{\disks_{/(X,E)}}(\icat)$.
\begin{prop}[Constructible implies Weiss on disks]\label{prop:constructibleweissondisks}
Let $(X,E)$ be a marked smooth conical manifold and let $\icat^{\otimes}$ be a presentably symmetric monoidal \infcatt. For any $\icat$-valued constructible precosheaf on $\disk_{/(X,E)}$, its left Kan extension to $\open(X,E)$ satisfies Weiss hypercodescent. 
\end{prop}
This is \cite[Theorem 7.33]{KS}\footnote{In this reference, it is assumed that $(X,E)$ is maximally marked, that is, $E=X_0$. While this will always be the case for the conical manifolds under consideration in this article (the stratified cubes introduced in Section \ref{sec:moritacat}), Proposition \ref{prop:constructibleweissondisks} holds without this assumption.}, which is itself a minor variation on \cite[Theorem 5.3.1]{KSW}. It follows that for $\icat^{\otimes}$ a presentably symmetric monoidal \infcat and $(X,E)$ a marked smooth conical manifold that has enough good disks, the restriction along $\disks_{/(X,E)}^{\otimes}\subset\open(X,E)^{\otimes}$ induces an equivalence 
\[ \Fact^{\mathrm{cstr}}_{(X,E)}(\icat)\overset{\simeq}{\longrightarrow} \alg^{\mathrm{m},\mathrm{cstr}}_{\disks_{/(X,E)}}(\icat)   \]
of \infcatst. The objects of $\disks_{/(X,E)}^{\otimes}$ are disjoint unions of disks, so it is natural to expect that a \emph{multiplicative} $\disks_{/(X,E)}^{\otimes}$-algebra is determined by its behavior on the full suboperad $\disk_{/(X,E)}^{\otimes}\subset \disks_{/(X,E)}^{\otimes}$. This is indeed the case.
\begin{prop}[Multiplicative extension to disjoint union completion]
Let $(X,E)$ be a marked smooth conical manifold and let $\icat^{\otimes}$ be a symmetric monoidal \infcatt, then restriction along $\disk_{/(X,E)}^{\otimes}\subset\disks_{/(X,E)}^{\otimes}$ induces an equivalence 
\[  \alg^{\mathfrak{m}}_{\disks_{/(X,E)}}(\icat)\overset{\simeq}{\longrightarrow} \alg_{\disk_{/(X,E)}}(\icat) \]
which restricts to an equivalence
\[  \alg^{\mathfrak{m},\mathrm{cstr}}_{\disks_{/(X,E)}}(\icat)\overset{\simeq}{\longrightarrow} \alg^{\mathrm{cstr}}_{\disk_{/(X,E)}}(\icat). \]
\end{prop}
This is \cite[Proposition 7.13]{KS}, which is itself a minor variation on \cite[Proposition 4.21]{KSW}. Combining the results above, we conclude that for $(X,E)$ a marked smooth conical manifold that has enough good disks, restriction along $\disk_{/(X,E)}^{\otimes}\subset\open(X,E)^{\otimes}$ induces an equivalence\footnote{Again, the condition of there being enough good disks may be removed by working with Weiss hypercosheaves.} 
\begin{equation}\label{eq:restricttodisk} \Fact^{\mathrm{cstr}}_{(X,E)}(\icat)\overset{\simeq}{\longrightarrow} \alg^{\mathrm{cstr}}_{\disk_{/(X,E)}}(\icat).   \end{equation}
It follows formally that $\Fact^{\mathrm{cstr}}_{(X,E)}(\icat)$ is the \infcat of algebras in $\icat$ for the operadic localization of $\disk_{/(X,E)}^{\otimes}$ at the collection of isotopy equivalences. We will make this operadic localization explicit in Section \ref{sec:plessdisks}.

\subsubsection{Sifted colimits of (pre)factorization algebras}\label{sec:siftedcolims}

In the sequel, we will need some results on compatibility of pushforward of factorization algebras with sifted colimits. We collect these here. Recall that a symmetric monoidal \infcat $\icat^{\otimes}$ is \emph{$\otimes$-sifted cocomplete} if $\icat$ admits sifted colimits and the tensor product $\icat\times\icat\rightarrow \icat$ preserves sifted colimits (equivalently, by siftedness, that the tensor product preserves those colimits separately in both variables). 
\begin{prop}\label{prop:factorizationsifted}
Let $\icat^{\otimes}$ be a $\otimes$-sifted cocomplete symmetric monoidal \infcat and
let $(X,E)$ be a marked space. Then the full subcategory $\mathsf{Fact}_{(X,E)}(\icat)\subset\mathsf{PFact}_{(X,E)}(\icat)$ is stable under sifted colimits. If $(X,E)$ is a marked smooth conical manifold, then the full subcategory $\mathsf{Fact}^{\mathrm{cstr}}_{(X,E)}(\icat)\subset\mathsf{PFact}_{(X,E)}(\icat)$ is also stable under sifted colimits.
\end{prop}
\begin{proof}
Let $\mathcal{J}$ be a sifted \infcat and let $f_{(\_)}:\mathcal{J}^{\rhd}\rightarrow \mathsf{PFact}_{(X,E)}(\icat)$ be a colimit diagram. We will show the following.
\begin{enumerate}[$(1)$]
    \item If for each $j\in\mathcal{J}$, the functor $\mathsf{Open}(X,E)\rightarrow \icat$ underlying the prefactorization algebra $f_j$ is a Weiss cosheaf, then the same is true for the prefactorization algebra $f_{\infty}$.
    \item If for each $j\in\mathcal{J}$, the prefactorization algebra $f_j$ is multiplicative, then $f_{\infty}$ is multiplicative.
    \item If $(X,E)$ is a marked smooth conical manifold and for each $j\in \mathcal{J}$, the functor $f_j|_{\disk_{/(X,E)}}:\disk_{/(X,E)}\rightarrow \icat$ carries isotopy equivalences to isotopy equivalences, then the same is true for $f_{\infty}$.
\end{enumerate}
It follows from \cite[Theorem 3.2.3.1]{LurHA} that the functor $\pfact_{(X,E)}(\icat)\rightarrow \fun(\open(X,E),\icat)$ to $\icat$-valued precosheaves on $\open(X,E)$ preserves and detects sifted colimits. Now $(1)$ is a consequence of the fact that cosheaves (for any topology) are stable under colimits in precosheaves. Since restriction of precosheaves along $\disk_{/(X,E)}\subset \open(X,E)$ preserves colimits and equivalences are stable under colimits, we also deduce $(3)$. It remains to prove $(2)$. Let $U,V\subset X$ be disjoint opens of $X$, then we have a commuting diagram 
\[
\begin{tikzcd}
\left(\underset{j\in \mathcal{J}}{\colim} f_j(U)\right) \otimes\left(\underset{j'\in \mathcal{J}}{\colim} f_{j'}(V)\right)\ar[d] & \underset{(j,j')\in \mathcal{J}\times\mathcal{J}}{\colim} \left(f_j(U)\otimes f_{j'}(V)\right)  \ar[d]\ar[l] & \underset{j\in \mathcal{J}}{\colim} \left(f_j(U)\otimes f_{j}(V)\right) \ar[d] \ar[l]\\
 \underset{j\in \mathcal{J}}{\colim}  f_{j}(U\amalg V)\ar[r,equal]& \underset{j\in \mathcal{J}}{\colim}  f_{j}(U\amalg V)\ar[r,equal]&\underset{j\in \mathcal{J}}{\colim}  f_{j}(U\amalg V).
\end{tikzcd}
\]
The left vertical map is the map $f_{\infty}(U)\otimes f_{\infty}(V)\rightarrow f_{\infty}\left(U\amalg V\right)$ that we wish to show is an equivalence. Since $f_j$ is multiplicative for all $j\in\mathcal{J}$, the right vertical map is an equivalence. The top right horizontal map is an equivalence because $\mathcal{J}$ is sifted and the top left horizontal map is an equivalence because the tensor product $\otimes$ is assumed to preserve sifted colimits separately in both variables. It follows that the middle and left vertical maps are equivalences too. 
\end{proof}
\begin{cor}\label{cor:factorizationsifted}
Let $\icat^{\otimes}$ be a $\otimes$-sifted cocomplete symmetric monoidal \infcat and let $f:(X,E)\rightarrow (Y,F)$ be a map of marked spaces. Then the following hold.
\begin{enumerate}[$(1)$]
    \item The pushforward $f_{\sharp}:\Fact_{(X,E)}(\icat)\rightarrow \Fact_{(Y,F)}(\icat)$ preserves sifted colimits. 
    \item If $(X,E)$ and $(F,E)$ are marked smooth conical manifolds and $f_*$ carries constructible factorization algebras to constructible factorization algebras, then $f_{\sharp}:\Fact_{(X,E)}^{\mathrm{cstr}}(\icat)\rightarrow \Fact^{\mathrm{cstr}}_{(Y,F)}(\icat)$ preserves sifted colimits.
\end{enumerate}
\end{cor}
\begin{proof}
It follows from Proposition \ref{prop:factorizationsifted} that sifted colimits of (constructible) factorization algebras are computed in prefactorization algebras, so it suffices to show that the functor $f_*:\pfact_{(X,E)}(\icat)\rightarrow \pfact_{(Y,F)}(\icat)$ preserves sifted colimits. Invoking \cite[Theorem 3.2.3.1]{LurHA} again, it suffices to show that the functor $\fun(\open(X,E),\icat)\rightarrow \fun(\open(Y,F),\icat)$ induced by pulling back marked opens along $f$ preserves sifted colimits, but this functor preserves all colimits.
\end{proof}

\subsection{Pointless disk algebras}\label{sec:plessdisks}

In this section, we extend the technology of \cite{KS} to allow for pointless disk algebras to obtain a localization result generalizing that of \cite{KSW} and \cite{KS}. We define, for each smooth conical manifold $X$ together with a marking $E$ on the set of 0-dimensional strata, an \infop $\infdisk^{\otimes}_{/(X,E)}$. We  show in Corollary \ref{cor:fact=alg} that there is a natural equivalence
\begin{equation}\label{eq:factdisk}
    \mathsf{Fact}^{\mathrm{cstr}}_{(X,E)}(\icat) \xrightarrow{\simeq}  \alg_{\infdisk_{/(X,E)}}(\icat)  
\end{equation} 
for each presentably symmetric monoidal \infcat $\icat$. Specializing to maximally marked conical manifolds, we conclude that pointless constructible factorization algebras valued in $\icat$ are the same as algebras for the \infop $\infdisk^{\otimes}_{/(X,X_0)}$. The equivalence \eqref{eq:factdisk} will have a number of technically and conceptually important consequences; for instance, considering for $(X,E)$ the conical manifold $\R_{\geq0}$ with the cone point $0$ marked, we will demonstrate an explicit equivalence of \infops $\infdisk^{\otimes}_{/(\R_{\geq 0},0)}\rightarrow \mathsf{RMod}^{\otimes}$, identifying $\infdisk^{\otimes}_{/(\R_{\geq 0},0)}$ with the operad governing pairs of an associative algebra together with a (unpointed!) right module. Similarly, this implies that there is an equivalence $\infdisk^{\otimes}_{/(\R,0)}\rightarrow \mathsf{BMod}^{\otimes}$, justifying our approach to the Morita category.
\subsubsection{$\infty$-Categories and \infops of marked disks and manifolds}
To obtain $\infty$-categorical, `topologized' version of the operad of marked disks $\disk^{\otimes}_{/(X,E)}$, we wish to emulate the constructions of \cite{AFT-fh-stratified}: the \infop $\infdisk^{\otimes}_{/(X,E)}$ should be obtained from a subcategory of the slice \infcat $\infmfd^+_{/(X,E)}$ for $\infmfd^+$ the \emph{$\infty$-category} of marked smooth conical manifolds. This section is concerned with the construction of this \infcat and thereafter the \infop $\infdisk^{\otimes}_{/(X,E)}$. 
\begin{rmk}\label{rmk:plessdisknaive}
We would like to say that $\infdisk_{/(X,E)}$ is the wide subcategory of $\infdisk_{/X}$ on those maps of disks $U\rightarrow V$ over $X$ that induce a bijection between the marked 0-dimensional strata of $X$ that are contained in $U$ and the marked 0-dimensional strata of $X$ that are contained in $V$ and a similar description should yield $\infdisk_{/(X,E)}^{\otimes}$ as a wide suboperad of $\infdisk^{\otimes}_{/X}$. This is analogous to how the (ordinary) operad controlling pointless prefactorization algebras was defined, but in the current setting, it is not a priori obvious that the preceding declaration is homotopically meaningful. In other words, we must verify that extracting from a smooth conical manifold its set of 0-dimensional strata is in fact a legitimate functorial operation on the \infcat $\infmfd$. This is what the material below achieves.
\end{rmk}
We start by recalling the excellent formal properties enjoyed by the \infcat $\infbsc$ from \cite{AFT-local-structures}.
\begin{lem}[\cite{AFT-local-structures}]\label{lem:bscproperties}
Let $U=\mathsf{C}(Y)\times\R^j$ and $V=\mathsf{C}(Z)\times\R^i$ be basics and let $f:U\rightarrow V$ be a conically smooth open embedding. Then the following are equivalent
\begin{enumerate}[$(1)$]
    \item The map $f$ is an equivalence in the \infcat $\infbsc$.
    \item There is an equality of depths $\mathsf{depth}(U)=\mathsf{depth}(V)$.
    \item There is an isomorphism between $U$ and $V$ in the category $\bsc$, that is, $i=j$ and there is an isomorphism $Y\cong Z$ in the category $\mfd$. 
    \item The image $f(U)$ intersects the cone locus $\R^i\subset V$.
\end{enumerate}
Moreover, if $f$ is not an equivalence in $\infbsc$, then the inequality $\mathsf{depth}(U)<\mathsf{depth}(V)$ holds.
\end{lem}
\begin{rmk}
As observed in \cite[Example 2.4.7]{AFT-local-structures}, for any compact smooth conical manifold $Z$, the cone $\mathsf{C}(Z)$ has both depth and dimension equal to $\mathsf{dim}(Z)+1$. Also, forming the product of any smooth conical manifold with $\R^i$ increases the dimension by $i$ and does not change the depth. It follows that for $U=\mathsf{C}(Y)\times\R^j$ a basic, we have $\mathsf{dim}(U)-\mathsf{depth}(U)=j$.
\end{rmk}
\begin{cor}
Let $f:U\rightarrow V$ be a morphism in $\infbsc$, then exactly one of the following is true.
\begin{enumerate}[$(1)$]
    \item The map $f$ is an equivalence in the \infcat $\infbsc$.
    \item The inequality $\mathsf{dim}(U)-\mathsf{depth}(U)>\mathsf{dim}(V)-\mathsf{depth}(V)$ holds.
\end{enumerate}
\end{cor}
\begin{proof}
If $f$ induces an equivalence in $\infbsc$, then $U$ and $V$ are abstractly isomorphic in the category $\bsc$ by Lemma \ref{lem:bscproperties} and in particular have the same dimension and depth, so if $(1)$ holds, then $(2)$ is false. Therefore, it suffices to show that in case $(1)$ is false, then $(2)$ holds. We proceed by induction on the depth of $V$.  By Lemma \ref{lem:bscproperties} it is greater or equal to the depth of $U$. If $\mathsf{depth}(V)=\mathsf{depth}(U)=-1$, then both $U$ and $V$ are empty and hence equivalent. Suppose $\mathsf{depth}(V)\geq 0$ and write $V=\mathsf{C}(Z)\times\R^i$ and $U=\mathsf{C}(Y)\times\R^j$ for compact conical smooth manifolds $Z$ and $Y$ (which may be empty) so that $\mathsf{dim}(U)-\mathsf{depth}(U)=j$ and $\mathsf{dim}(V)-\mathsf{depth}(V)=i$. If $f(U)$ intersects the cone-locus of $V$, then $f$ is an equivalence by Lemma \ref{lem:bscproperties}, so we may assume that $f$ factors as
\[ \mathsf{C}(Y)\times\R^j\hooklongrightarrow Z\times \R_{>0} \times\R^{i}  \subset \mathsf{C}(Z)\times\R^i   \]
where the first map is a conically smooth open embedding. Since basics form a basis of any smooth conical manifold, we may assume that $Z=\mathsf{C}(Z')\times\R^{k}$ for some compact smooth conical manifold $Z'$ of strictly lower depth than $Z$ and some integer $k\geq 0$, so that we have a conically smooth open embedding
\[ \mathsf{C}(Y)\times\R^j\hooklongrightarrow \mathsf{C}(Z')\times\R^{k}\times \R_{>0}\times\R^i \cong \mathsf{C}(Z')\times\R^{k+1+i}. \]
Now the inductive hypothesis guarantees that $j\geq k+1+i$, so $j>i$. 
\end{proof}
Recall from \cite[Section 4.3]{AFT-local-structures} the partially ordered set $[\mathcal{B}\mathsf{sc}]$ of equivalence classes of objects of $\mathcal{B}\mathsf{sc}$ with partial order defined by $[V]\leq [U]$ just in case there exists a morphism $U\rightarrow V$ among (any) representatives (note the reversal). Recall that a full subposet $P'\subset P$ is {\em consecutive} if for every $x,y \in P'$ and any $t\in P$ such that $x\leq t\leq y$, then $t\in P'$.
\begin{cor}\label{cor:basicsposet}
The following hold true.
\begin{enumerate}[$(1)$]
    \item Let $[a,b]\subset \Z_{\geq0}$ be an interval, that is, a consecutive subposet. Let $[\infbsc]_{[a,b]}\subset [\infbsc]$ be the full subposet spanned by equivalence classes of basics $[U]$ for which $\mathsf{dim}(U)-\mathsf{depth}(U) \in [a,b]$, then $[\infbsc]_{[a,b]}$ is a consecutive subposet. 
    \item Let $n\in \Z_{\geq 0}$ be a nonnegative integer, then the full subposet $[\infbsc]_{n}\subset [\infbsc]$ spanned by equivalence classes of basics $[U]$ for which $\mathsf{dim}(U)-\mathsf{depth}(U) =n$ has no nonidentity morphisms. 
\end{enumerate}
\end{cor}
 After \cite[Lemma 4.4.4]{AFT-local-structures}, for any smooth conical manifold $X$ and any $x\in X$, the set of those elements $[U]$ of $[\infbsc]$ for which there exists a conically smooth open embedding $f:U\rightarrow X$ for which $x\in f(U)$ contains a unique maximal element $[U_x]$ of depth equal to the local depth $\mathsf{depth}_x(X)$, and the map
 \[ X\longrightarrow   [\infbsc],\quad\quad x\longmapsto [U_x] \]
is a continuous map of topological spaces (where $[\infbsc]$ is endowed with the usual topology in which a set is open just in case it is upward closed). It is an exercise (see \cite[Proposition 4.4.7]{AFT-local-structures}) to show that for any consecutive full subposet $Q\subset [\mathcal{B}\mathsf{sc}]$, there is a simplicial functor 
\begin{equation}\label{eq:stratumfunctor}
  (\_)_{[Q]}:  \mathcal{M}\mathsf{fd}_{\simp} \longrightarrow  \mathcal{M}\mathsf{fd}_{\simp} 
\end{equation}
of locally Kan simplicial categories that carries $X \in  \mathcal{M}\mathsf{fd}_{\simp}$ to the cone in the pullback diagram 
\[
\begin{tikzcd}
 X_{[Q]} \ar[d] \ar[r,hook]& X \ar[d] \\
 Q \ar[r,hook] & {[}\mathcal{B}\mathsf{sc}{]}
\end{tikzcd}
\]
among topological spaces (recall that since $Q\subset  [\mathcal{B}\mathsf{sc}]$ is consecutive, the subspace $X_{[Q]} \subset X$ admits a natural structure of a smooth conical manifold for which the inclusion is conically smooth). The inclusion $X_{[Q]}\hookrightarrow X$ exhibits the former smooth conical manifold as the locus of those points $x$ for which the conically smooth embeddings $U\rightarrow X$ with $[U]\in Q$ whose image contain $x$ form a local basis at $x$. 
\begin{nota}\label{nota:0dimstratafunctor}
Consider the special case of $Q$ being the full subposet $[\mathcal{B}\mathsf{sc}]_0\subset [\mathcal{B}\mathsf{sc}]$ spanned by those basics $U$ for which $\mathsf{dim}(U)-\mathsf{depth}(U) =0$; by Corollary \ref{cor:basicsposet}, this subposet is consecutive and contains no nonidentity morphisms. In this case, the smooth conical manifold $X_{[\mathcal{B}\mathsf{sc}]_0}$ is the union of all 0-dimensional strata of $X$ and we will write $X_0$ for the underlying set of $X_{[\mathcal{B}\mathsf{sc}]_0}$. The only extra data remembered by the simplicial category $\mathcal{M}\mathsf{fd}_{\simp}$ is a choice of disjoint decomposition: the full simplicial subcategory of $\mathcal{M}\mathsf{fd}_{\simp}$ spanned by 0-dimensional smooth conical manifolds is equivalent to the category defined as follows. 
\begin{enumerate}
    \item[$(O)$] Objects are surjective maps $E\rightarrow P$ of countable sets.
    \item[$(M)$] Maps are commuting diagrams
    \[
\begin{tikzcd}
E \ar[d] \ar[r] & F\ar[d] \\
P \ar[r] & Q
\end{tikzcd}
\]
in which the horizontal maps are injective and the vertical ones surjective.
\end{enumerate}
Forgetting the set $P$, we have a simplicial functor
\[ (\_)_0: \mathcal{M}\mathsf{fd}_{\simp} \longrightarrow \set_{\mathrm{inj}} \]
to the category $\set_{\mathrm{inj}}$ of sets and injections between them, carrying the smooth conical manifold $X$ to its set $X_0$ of 0-dimensional strata. 
\end{nota}
The formation of disjoint unions endows both $\mathcal{M}\mathsf{fd}_{\simp}$ and $\mathsf{Set}_{\mathrm{inj}}$ with symmetric monoidal structures and the functor $(\_)_0$ canonically enhances to a symmetric monoidal functor
\[ (\_)_0: \mathcal{M}\mathsf{fd}^{\otimes}_{\simp} \longrightarrow \set_{\mathrm{inj}}^{\otimes} \]
of symmetric monoidal locally Kan simplicial categories. Note that the empty set is a symmetric monoidal unit on both sides (viewed as a stratified space in a trivial fashion on the left hand side) but neither symmetric monoidal structure is coCartesian.
\begin{cons}[The symmetric monoidal \infcat of marked smooth conical manifolds]\label{cons:markedmfd}
Consider the category $\set_{\mathrm{inj}}^{+,\simeq}$ of \emph{marked sets and injections} defined as follows.
\begin{enumerate}
    \item[$(O)$] Objects are pairs $(S,E)$ of a set together with a subset $E\subset S$.
    \item[$(M)$] A morphism between pairs $(S,E)$ and $(T,F)$ is an injective map of sets $f:S\hookrightarrow T$ such that $f(E)=f(S) \cap F$.
\end{enumerate}
Observe that the obvious forgetful functor $\set_{\mathrm{inj}}^{+,\simeq}\rightarrow \set_{\mathrm{inj}}$ is a right fibration. The category $\set_{\mathrm{inj}}^{+,\simeq}$ also admits a symmetric monoidal structure by disjoint union for which the empty pair $(\emptyset,\emptyset)$ is the symmetric monoidal unit, and the functor $\set_{\mathrm{inj}}^{+,\simeq}\rightarrow \set_{\mathrm{inj}}$ is symmetric monoidal. Now we apply the coherent nerve $\ner^{\mathrm{coh}}$ (\cite[Definition 1.1.5.5]{LurHTT}) to the symmetric monoidal simplicial functor $(\_)_0$ and define the symmetric monoidal \infcat $\left(\infmfd^{+,\simeq}\right)^{\otimes}$ of \emph{marked smooth conical manifolds} as the cone in the pullback diagram 
\[
\begin{tikzcd}
\left(\infmfd^{+,\simeq}\right)^{\otimes} \ar[d] \ar[r] & (\set_{\mathrm{inj}}^{+,\simeq})^{\otimes} \ar[d] \\
\infmfd^{\otimes} \ar[r,"(\_)_0"] & \set_{\mathrm{inj}}^{\otimes}
\end{tikzcd}
\]
among symmetric monoidal \infcatst. We will let $\infmfd^{+,\simeq}:=\left(\infmfd^{+,\simeq}\right)^{\otimes}_{\langle 1\rangle}$ denote its underlying \infcat and write objects of $\infmfd^{+,\simeq}$ as pairs $(X,E)$ for $X \in \infmfd$ and $E\subset X_0$ a subset of the 0-dimensional strata of $X$. We have a full subcategory $\infbsc^{+,\simeq}\subset \infmfd^{+,\simeq}$ spanned by those objects for which the underlying smooth conical manifold is a basic. The symmetric monoidal \infcat of \emph{marked conical disks} $\left(\infdisks^{+,\simeq}\right)^{\otimes}$ is the smallest symmetric monoidal subcategory of $\left(\infmfd^{+,\simeq}\right)^{\otimes}$ containing $\infbsc^{+,\simeq}$, and we let $\infdisks^{+,\simeq}:=\left(\infdisks^{+,\simeq}\right)^{\otimes}_{\langle 1\rangle}$ denote its underlying \infcatt; it consists of disjoint unions of objects of $\bsc^{+,\simeq}$.
\end{cons}
\begin{rmk}[Markings as a tangential structure]\label{rmk:markingstangentialstr}
The functor $\infmfd^{+,\simeq}\rightarrow \infmfd$ is a base change of the right fibration $\set^{+,\simeq}_{\mathrm{inj}}\rightarrow \set_{\mathrm{inj}}$ and therefore a right fibration. The associated presheaf is in fact a sheaf. It follows that the right fibration $\infmfd^{+,\simeq}\rightarrow \infmfd$ is determined by its base change along $\infbsc\subset\infmfd$ in the following sense: the base-changed right fibration $\infbsc^{+,\simeq}\rightarrow \infbsc$ is a \emph{tangential structure} or \emph{\infcat of basics} in the sense of \cite[Section 5]{AFT-local-structures} and we recover $\infmfd^{+,\simeq}$ as the \emph{\infcat of $\infbsc^{+,\simeq}$-manifolds} 
\[ \infmfd^{+,\simeq}\simeq \infmfd(\infbsc^{+,\simeq}):=\infmfd\times_{\mathsf{RFib}(\infbsc)}\mathsf{RFib}(\infbsc)_{/\infbsc^{+,\simeq}} \]
introduced in \emph{loc. cit.} Note that the \infcat of basics $\infbsc^{+,\simeq}$ is very simple:
\begin{enumerate}[$(1)$]
    \item Its fiber over a basic $U$ with $\mathsf{dim}(U)-\mathsf{depth}(U)>0$ (that is, $U=\R^i\times \mathsf{C}(Z)$ for some smooth conical manifold $Z$ and some $i>0$) is empty.
    \item Its fiber over a basic $U$ with $\mathsf{dim}(U)-\mathsf{depth}(U)=0$ (that is, $U=\mathsf{C}(Z)$ for some compact smooth conical manifold $Z$) is a two-element set. This set encodes whether the cone point is or isn't marked.
\end{enumerate}
\end{rmk}
Now let $(X,E)$ be a marked smooth conical manifold, then we apply \cite[Corollary 1.20, Notation 1.21]{AFT-fh-stratified} to the symmetric monoidal \infcat $\left(\infmfd^{+,\simeq}\right)^{\otimes}$ and the right fibration 
\[  \infdisks^{+,\simeq}_{/(X,E)} := \infdisks^{+,\simeq}\times_{\infmfd^{+,\simeq}}\infmfd^{+,\simeq}_{/(X,E)} \longrightarrow  \infdisks^{+,\simeq},   \]
using that the symmetric monoidal unit of $\left(\infmfd^{+,\simeq}\right)^{\otimes}$ and $\left(\infdisks^{+,\simeq}\right)^{\otimes}$ is initial to obtain the \infop $(\infdisks^{+,\simeq})^{\otimes}_{/(X,E)}$ with underlying \infcat $\infdisks^{+,\simeq}_{/(X,E)}$ fitting into a pullback diagram
\[
\begin{tikzcd}
     (\infdisks^{+,\simeq})^{\otimes}_{/(X,E)} \ar[d] \ar[r] & \left(\infdisks^{+,\simeq}_{/(X,E)}\right)^{\coprod} \ar[d] \\
     (\infdisks^{+,\simeq})^{\otimes} \ar[r] & (\infdisks^{+,\simeq})^{\coprod}
\end{tikzcd}
\]
of \infopst, where $(\_)^{\coprod}:\catinf\rightarrow \opinf$ is the coCartesian \infop construction of \cite[Section 2.4.3]{LurHA}. Explicitly, a multimorphism
\begin{equation}\label{eq:multimordisk}  \left\{(U_1,S_1)\overset{f_1}{\rightarrow} (X,E)\oplus\ldots \oplus (U_n,S_n)\overset{f_n}{\rightarrow} (X,E)\right\} \longrightarrow \left\{(V,T)\rightarrow (X,E)\right\}\end{equation}
of $\left(\infdisks^{+,\simeq}\right)_{/(X,E)}^{\otimes}$ consists of 
\begin{enumerate}[$(i)$]
    \item a map $f:(U,S) = \coprod_i(U_i,S_i)\rightarrow (V,
    T)$ in the \infcat $\infdisks^{+,\simeq}$ (note that $S=U_0\times_{V_0}T$ by construction of $\infdisks^{+,\simeq}$).
    \item for each $i\in \langle n\rangle^{\circ}$, a homotopy between the map $f_i:(U_i,S_i)\rightarrow (X,E)$ and the composite
    \[  (U_i,S_i)\longrightarrow (U,S)\overset{f}{\longrightarrow} (V,T)\longrightarrow (X,E). \]
    (where again $S=U_0\times_{X_0}E$ and $S_i=(U_i)_0\times_{X_0}E$ by construction of $\infdisks^{+,\simeq}$).
\end{enumerate}

\begin{defn}[$\infty$-Operads of marked disks]
Let $(X,E)$ be a marked smooth conical manifold. The \infop of \emph{marked disks in $(X,E)$}, denoted $\infdisks^{\otimes}_{/(X,E)}$, is the wide suboperad of  $\left(\infdisks^{+,\simeq}\right)_{/(X,E)}^{\otimes}$ obtained by declaring a multimorphism \eqref{eq:multimordisk} to lie in $\infdisks^{\otimes}_{/(X,E)}$ just in case the map in $(i)$ above induces a bijection $S\rightarrow T$ of sets, or in other words, the maps $S\hookrightarrow E$ and $T\hookrightarrow E$ have the same image. We further have a full suboperad $\infdisk^{\otimes}_{/(X,E)}\subset \infdisks^{\otimes}_{/(X,E)}$ spanned by those pairs $\left(\langle n\rangle, (U_1,S_1)\oplus \ldots \oplus  (U_n,S_n)\right)$ for which $U_i$ is a basic (instead of a disjoint union of such) for all $i$. These \infops have underlying \infcats denoted by $\infdisks_{/(X,E)}:=\left(\infdisks^{\otimes}_{/(X,E)}\right)_{\langle 1\rangle}$ and $\infdisk_{/(X,E)}:=\left(\infdisk^{\otimes}_{/(X,E)}\right)_{\langle 1\rangle}$.    
\end{defn}

\begin{rmk}
The right fibration $\infdisks^{+,\simeq}\rightarrow \infdisks$ induces for each marked smooth conical manifold $(X,E)$ a trivial fibration $(\infdisks^{+,\simeq})^{\otimes}_{/(X,E)}\rightarrow \infdisks^{\otimes}_{/X}$. In view of the Remark \ref{rmk:markingstangentialstr}, this is an instance of the general fact that for any \infcat of basics $\mathcal{B}\rightarrow \infbsc$, the associated map of \infops $\left(\infdisks(\mathcal{B})\right)^{\otimes}_{/X}\rightarrow \left(\infdisks\right)^{\otimes}_{/X}$ is a trivial fibration for every $\mathcal{B}$-manifold $X$. In particular, the \infop $\infdisks^{\otimes}_{/(X,E)}$ is equivalent to a wide suboperad of $\infdisks^{\otimes}_{/X}$. Similarly, $\infdisk^{\otimes}_{/(X,E)}$ is equivalent to the wide suboperad of $\infdisk^{\otimes}_{/X}$ informally described in Remark \ref{rmk:plessdisknaive}. Accordingly, we will write objects of $\infdisks^{\otimes}_{/(X,E)}$ or $\infdisk^{\otimes}_{/(X,E)}$ as 
\[ U_1\oplus \ldots\oplus U_n\] with $U_n\in \infdisks_{/X}$ respectively $\infdisk_{/X}$ instead of
\[(U_1,S_1)\oplus \ldots \oplus(U_n,S_n),\] 
since for each $1 \leq i\leq n$, the marking $S_i$ on $U_i$ must be the pullback $(U_i)_0\times_{X_0}E$.
\end{rmk}

\begin{rmk}\label{rmk:suboperad}
Let $(X,E)$ be a marked smooth conical manifold. The \infop $\infdisks^{\otimes}_{/(X,E)}$ is a wide suboperad of $ \infdisks^{\otimes}_{/X}$ in a particularly simple way: the inclusion on spaces of multimorphisms is either the inclusion of the empty space or it is the identity. More precisely, let $U\in\infdisks_{/(X,E)}$ and let $\{V_j\}_{1\leq j\leq n}\in \infdisks_{/(X,E)}$. If the subsets $U_0\cap E\subset E$ and $\coprod_j (V_j)_0 \cap E\subset E$ of $E$ are not the same, the space of multimorphisms 
\[  \mathrm{Mul}_{\infdisks^{\otimes}_{/(X,E)}}(\{V_j\}_{1\leq j\leq n},U) \]
is empty. If these subsets of $E$ are the same, the map 
\[  \mathrm{Mul}_{\infdisks^{\otimes}_{/(X,E)}}(\{V_j\}_{1\leq j\leq n},U) \hooklongrightarrow \mathrm{Mul}_{\infdisks^{\otimes}_{/X}}(\{V_j\}_{1\leq j\leq n},U) \]
is the identity. In particular, if $U\hookrightarrow X$ is a single disk (instead of a disjoint union of such) for which $\mathsf{dim}(U)-\mathsf{depth}(U)=0$ and the map $*\cong U_0\hookrightarrow X_0$ intersects $E$, then for the empty collection $\emptyset$ of objects of $\infdisks_{/(X,E)}$ the inclusion 
\[  \mathrm{Mul}_{\infdisks^{\otimes}_{/(X,E)}}(\emptyset,U) \hooklongrightarrow \mathrm{Mul}_{\infdisks^{\otimes}_{/X}}(\emptyset,U) \]
is the inclusion 
\[ \emptyset \hooklongrightarrow *. \]
\end{rmk}

We will now identify the spaces of multimorphisms of the \infop $\infdisk^{\otimes}_{/(X,E)}$ with configuration spaces.

\begin{obs}[Configuration spaces]\label{obs:configurations}
Let $(X,E)$ be a marked smooth conical manifold. Since $\infdisk_{/(X,E)}\subset \infdisk_{/X}$ is a wide subcategory, the inclusion becomes the identity on maximal subgroupoids, so that we have an identification
\[   \infdisk_{/(X,E)}^{\simeq}\simeq \infdisk_{/X}^{\simeq}\simeq \coprod_{[V]\in [\infbsc]}\Hom_{\infmfd}(V,X)_{\mathsf{Aut}_0(V)} \overset{\simeq}{\longrightarrow} \coprod_{[V]\in[\infbsc]} X_{[V]}, \]
(where we abuse notation and write $X_{[V]}$ for the singular homotopy type of the smooth manifold $X_{[V]}$), given by evaluation at the origin of a disk, by e.g. \cite[Lemma 4.4.8]{AFT-local-structures} (here we write $\Hom_{\infmfd}(V,X)_{\mathsf{Aut}_0(V)}$ for the coinvariants of the action of origin-preserving conically smooth automorphisms of $V$). Any conically smooth open embedding $V\hookrightarrow X$ for $V$ a disk specifies a path component $X_{[V\hookrightarrow X]}\subset X_{[V]}$ determined by the map
\[  \R^{k} \simeq V_{[V]} \longrightarrow X_{[V]} \]
(where $k=\mathsf{dim}(V)-\mathsf{depth}(V)$). Let $U\in \infdisk_{/(X,E)}$, let $n>0$ an integer and let $\vec{V}:=(V_1\hookrightarrow X)\oplus\ldots\oplus (V_n\hookrightarrow X)$ be an object of $\infdisk_{/(X,E)}^{\otimes}$ in the fiber over $\langle n\rangle$. According to Remark \ref{rmk:suboperad}, if $U_0\cap E\subset E$ and $\coprod_{i}(V_i)_0\cap E\subset E$ do not coincide, then the space of multimorphisms
\[\Hom_{\infdisk_{/(X,E),\mathrm{act}}^{\otimes} }(V_1\oplus \ldots\oplus V_n,U)=\mathrm{Mul}_{\infdisk_{/(X,E)}^{\otimes}}\left(\{V_j\}_{1\leq j\leq n},U\right)\]
is empty, and if these subsets of $E$ are the same, then we have 
\[\mathrm{Mul}_{\infdisk_{/(X,E)}^{\otimes}}\left(\{V_j\}_{1\leq j\leq n},U\right) = \mathrm{Mul}_{\infdisk_{/X}^{\otimes}}\left(\{V_j\}_{1\leq j\leq n},U\right).\]
The object $(V_1\hookrightarrow X)\oplus\ldots\oplus (V_n\hookrightarrow X)$ determines a map 
\[ H:\langle n\rangle^{\circ} \longrightarrow \pi_0(\infdisk^{\simeq}_{/X})\cong \coprod_{[V]\in[\infbsc]}\pi_0(X_{[V]})   \]
carrying $j\in \langle n\rangle^{\circ}$ to the path component $X_{[V_j\hookrightarrow X ]} \subset X_{[V_j]}$. For each $V\rightarrow X$ an object of $\infdisk_{/X}$, consider the pullback diagram
\[
\begin{tikzcd}
    U_{[V\hookrightarrow X]} \ar[d,hook] \ar[r,hook] & U_{[V]} \ar[d,hook] \\
    X_{[V\hookrightarrow X]} \ar[r,hook] & X_{[V]}
\end{tikzcd}
\]
of open inclusions of smooth manifolds; note that the right vertical map need not be an inclusion of a single path component, so that the pullback $U_{[V\hookrightarrow X]} $ might consist of multiple components, even though $X_{[V\hookrightarrow X]}$ is connected. Now the space of multimorphisms $\mathrm{Mul}_{\infdisk_{/X}^{\otimes}}\left(\{V_j\}_{1\leq j\leq n},U\right)$ is identified, by evaluation at the center of disks, with (the homotopy type of) the product of ordered configuration spaces
\[ \Hom_{\infdisk_{/X,\mathrm{act}}^{\otimes} }(V_1\oplus \ldots\oplus V_n,U)  \overset{\simeq}{\longrightarrow} \prod_{X_{[V\hookrightarrow X]}\in\pi_0\left(\infdisk^{\simeq}_{/X}\right)} \mathsf{Conf}_{|H^{-1}(X_{[V\hookrightarrow X]})|}(U_{[V\hookrightarrow X]}). \]
Note that $\mathsf{Conf}_{|H^{-1}(X_{[V\hookrightarrow X]})|}(U_{[V\hookrightarrow X]})$ is empty if $\mathsf{dim}(V)-\mathsf{depth}(V)=0$ and $|H^{-1}(X_{[V\hookrightarrow X]})|>1$; in particular, in case $U_0\cap E$ is nonempty, then the space $\mathrm{Mul}_{\infdisk_{/(X,E)}^{\otimes}}\left(\{V_j\}_{1\leq j\leq n},U\right)$ of multimorphisms may only be nonempty if there is \emph{exactly one} $i\in \langle n\rangle^{\circ}$ for which $\mathsf{dim}(V_i)-\mathsf{depth}(V_i)=0$ and the image of $*\cong (V_i)_0\hookrightarrow X_0$ is the marked point $U_0\cap E$.
\end{obs}

The preceding analysis allows us to perform an essential computation that showcases the utility of pointless disk algebras: while $\infdisk^{\otimes}_{/\R_{\geq 0}}$-algebras are associative algebras together with \emph{pointed} right modules, $\infdisk^{\otimes}_{/(\R_{\geq 0},0)}$-algebras are associative algebras together with a right module (without pointing).

\begin{ex}[Pointless disk algebras on the stratified line]\label{ex:stratifiedline}
Consider the marked smooth conical manifold $(\R_{\geq 0},0)$, that is, the cone $\mathsf{C}(*)$ on a point, with the cone point marked. The space $\infdisk_{/(\R_{\geq 0},0)}^{\simeq}$ is discrete and has two equivalence classes of objects, represented by the embeddings
\[  \R_{\geq 0}=\R_{\geq 0}\quad\text{and}\quad \R\cong \R_{>0}\hookrightarrow \R_{\geq 0}, \]
where we have fixed an (orientation preserving) automorphism $\R\cong \R_{>0}$. Since there are no conically smooth open embeddings $\R_{\geq 0}\hookrightarrow \R$, the space of multimorphisms \[\mathrm{Mul}_{\infdisks^{\otimes}_{/(\R_{\geq 0},0)}}(\{Y_i\}_{1\leq i\leq n},\R_{>0}\hookrightarrow \R_{\geq 0})\]
is only nonempty if all $Y_i$ are equal to $\R_{>0}\hookrightarrow \R_{\geq 0}$, in which case Observation \ref{obs:configurations} provides an equivalence
\[ \mathrm{Mul}_{\infdisks^{\otimes}_{/(\R_{\geq 0},0)}}(\{\R_{>0}\hookrightarrow \R_{\geq 0}\}_{1\leq i\leq n},\R_{>0}\hookrightarrow \R_{\geq 0}) \overset{\simeq}{\longrightarrow} \mathsf{Conf}_n(\R_{>0}).  \]
Let $f:\langle n\rangle^{\circ}\hookrightarrow\R_{>0}$ be an injective map, that is, an element of $\mathsf{Conf}_n(\R_{>0})$. Then $f$ determines a linear ordering $\prec$ on the set $\langle n\rangle^{\circ}$ by declaring $i\prec j$ if $f(i)<f(j)$; this yields a continuous map
\[ \varphi_n:\mathsf{Conf}_n(\R_{>0}) \longrightarrow \mathsf{LinOrd}(\langle n\rangle^{\circ})\]
where the codomain is equipped with the discrete topology. The fiber of $\varphi_n$ at each linear ordering can be identified with a convex subset of $\R^n$, so we conclude that $\varphi_n$ is a homotopy equivalence. Since the cone point of $\R_{\geq 0}$ is marked, the space of multimorphisms 
\[\mathrm{Mul}_{\infdisks^{\otimes}_{/(\R_{\geq 0},0)}}(\{Y_i\}_{1\leq i\leq n},\R_{\geq 0}= \R_{\geq 0})\]
is nonempty only if there is \emph{exactly one} $i$ for which $Y_i$ is equal to $\R_{\geq 0}= \R_{\geq 0}$, in which case Observation \ref{obs:configurations} provides an equivalence
\[ \mathrm{Mul}_{\infdisks^{\otimes}_{/(\R_{\geq 0},0)}}(\{\R_{>0}\hookrightarrow \R_{\geq 0}\}_{\langle n\rangle^{\circ} \setminus \{i\}}\cup \{\R_{\geq 0}=\R_{\geq 0}\},\R_{\geq 0}=\R_{\geq 0}) \overset{\simeq}{\longrightarrow} \mathsf{Conf}_{n-1}(\R_{>0}) \times \mathsf{Conf}_1(*).  \]
Repeating the analysis above, we have a homotopy equivalence
\[ \psi_n: \mathsf{Conf}_{n-1}(\R_{>0})\times \mathsf{Conf}_1(*)  \cong \mathsf{Conf}_{n-1}(\R_{>0})\overset{\simeq}{\longrightarrow} \mathsf{LinOrd}(\langle n\rangle^{\circ}\setminus \{i\}). \]
For the multimorphisms from the empty collection, we deduce from  Remark \ref{rmk:suboperad} the equivalences 
\[ \mathrm{Mul}_{\infdisks^{\otimes}_{/(\R_{\geq 0},0)}}(\emptyset,\R_{\geq 0}=\R_{\geq 0})=\emptyset \quad\text{and}\quad \mathrm{Mul}_{\infdisks^{\otimes}_{/(\R_{\geq 0},0)}}(\emptyset,\R_{> 0}\hookrightarrow \R_{\geq 0})\simeq *.  \]
We conclude that $\infdisk_{/(\R_{\geq 0},0)}^{\otimes}$ is equivalent to a discrete operad, and the identification of the spaces of multimorphisms with sets of linear orderings (together with the straightforward verification that composition in $\infdisk_{/(\R_{\geq 0},0)}^{\otimes}$ is compatible with composition of linear orderings) entails that the assignment
\[  \R_{\geq 0}=\R_{\geq0}\longmapsto \mathfrak{m},\quad \quad \R_{>0}\hookrightarrow \R_{\geq0} \longmapsto \mathfrak{a}  \]
together with the maps $\phi_n$ and $\psi_n$ determines an equivalence
\[ \infdisk_{/(\R_{\geq 0},0)}^{\otimes}\overset{\simeq}{\longrightarrow} \mathsf{RM}^{\otimes} \]
of \infopst, where $\mathsf{RM}^{\otimes}$ is the \infop controlling associative algebras together with a right module of \cite[Variant 4.2.1.36]{LurHA}. An entirely parallel argument yields equivalences of \infops 
\[ \infdisk^{\otimes}_{/(\R_{\leq 0},0)}\overset{\simeq}{\longrightarrow} \mathsf{LM}^{\otimes} \quad\text{and}\quad \infdisk_{/(\R,0)}^{\otimes}\overset{\simeq}{\longrightarrow} \mathsf{BM}^{\otimes}  \]
to the \infops of \cite[Definition 4.2.1.1]{LurHA} and \cite[4.3.1.1]{LurHA} respectively. For instance, the \infop $\infdisk^{\otimes}_{/(\R,0)}$ has three equivalence classes of objects represented by the open embeddings $\R_{<0}\hookrightarrow \R$, $\R=\R$ and $\R_{>0}\hookrightarrow \R$, and the equivalence $\infdisk_{/(\R,0)}\simeq \mathsf{BM}^{\otimes}$ is given by the assignment 
\[ \R_{<0}\hookrightarrow \R \longmapsto \mathfrak{a}\_ ,\quad \quad\R=\R\longmapsto \mathfrak{m},\quad \quad \R_{>0}\hookrightarrow \R \longmapsto \mathfrak{a}_+.  \]
Restricting the analysis above to disks of depth 0 recovers the standard equivalence $\infdisk^{\otimes}_{/\R}\simeq \mathsf{Assoc}^{\otimes}$, and the conically smooth open embedding $\R\cong \R_{<0}\hookrightarrow \R_{\leq 0}$ determines the map 
\[  \mathsf{Assoc}^{\otimes} \simeq \infdisk^{\otimes}_{/\R}\longrightarrow \infdisk^{\otimes}_{(\R_{\leq 0},0)} \simeq \mathsf{LM}^{\otimes}   \]
of \cite[Remark 4.2.1.10]{LurHA}. The map $\R\cong \R_{>0}\hookrightarrow \R_{\geq 0}$ induces the obvious variant for right modules and the map $\R\coprod \R\cong \R_{<0}\coprod \R_{>0}$ induces the \infop map 
\[ \mathsf{Assoc}^{\otimes} \boxplus \mathsf{Assoc}^{\otimes}\longrightarrow \mathsf{BM}^{\otimes}  \]
of \cite[Remark 4.3.1.10]{LurHA}, where $\boxplus$ denotes the coproduct of \infopst. 
\end{ex}

\begin{var}\label{var:discreteops}
Construction \ref{cons:markedmfd} admits an obvious 1-categorical analogue, which recovers the categories $\disk_{/(X,E)}$, $\disks_{/(X,E)}$ and the corresponding operads of the previous section, as follows: restricting the symmetric monoidal functor $(\_)_0$ to the underlying category $\mfd$ of the simplicial category of $\infmfd_{\simp}$ yields a symmetric monoidal functor
\[  (\_)_0:\mfd^{\otimes} \longrightarrow \mathsf{Set}^{\otimes}_{\mathrm{inj}} \]
and a symmetric monoidal category $\left(\mfd^{+,\simeq}\right)^{\otimes}$ defined by the pullback diagram
\[
\begin{tikzcd}
\left(\mfd^{+,\simeq}\right)^{\otimes} \ar[d] \ar[r] & (\set_{\mathrm{inj}}^{+,\simeq})^{\otimes} \ar[d] \\
\mfd^{\otimes} \ar[r,"(\_)_0"] & \set_{\mathrm{inj}}^{\otimes}
\end{tikzcd}
\]
for which the initial marked smooth conical manifold $(\emptyset,\emptyset)$ is the unit. We then define the full subcategories $\bsc^{+,\simeq}\subset \mfd^{+,\simeq}$, $\disks^{+,\simeq}\subset \mfd$ and $\left(\disks^{+,\simeq}\right)^{\otimes}\subset \left(\mfd^{+,\simeq}\right)^{\otimes}$, as well as the operads $\left(\disks\right)^{\otimes}_{/(X,E)}$ and $\disks^{\otimes}_{/(X,E)}$ with underlying categories $\disks_{/(X,E)}$ and $\disk_{/(X,E)}$ for any smooth conical marked manifold $(X,E)$ in a fashion parallel to the definitions of Construction \ref{cons:markedmfd}. For each of these, there are evident essentially surjective maps (of \infcats or \infopst) from the sans serif variants defined just now to the calligraphic ones of Construction \ref{cons:markedmfd}. These maps are compatible with the maps forgetting the markings; in particular we have for each marked smooth conical manifold a commuting diagram 
\begin{equation*}
\begin{tikzcd}
 \disks^{\otimes}_{/(X,E)} \ar[d,hook] \ar[r]  &\infdisks^{\otimes}_{/(X,E)}\ar[d,hook] \\
 \left(\disks^{+,\simeq}\right)^{\otimes}_{/(X,E)} \ar[d,"\simeq"] \ar[r]  &\left(\infdisks^{+,\simeq}\right)^{\otimes}_{/(X,E)}\ar[d,"\simeq"] \\
 \disks^{\otimes}_{/X}  \ar[r]  &\infdisks^{\otimes}_{/X} 
\end{tikzcd}
\end{equation*}
of \infops where the upper vertical maps are wide suboperad inclusions and the lower ones are trivial fibrations of \infopst.
\end{var}

\subsubsection{Localizing at isotopy equivalences}

Let $(X,E)$ be a marked smooth conical manifold and let $\mathcal{J}\subset \disk_{/X}$ be the wide subcategory on those conically smooth open embeddings $f:U\hookrightarrow V$ of disjoint unions of basics over $X$ that are carried to an equivalence in $\infdisk_{/X}$ (equivalently, by Lemma \ref{lem:bscproperties} those $f:U\hookrightarrow V$ for which $U$ and $V$ are isomorphic in $\mathsf{Bsc}$). Note that this implies that every map contained in $\mathcal{J}$ is also contained in the subcategory $\disk_{/(X,E)}\subset \disk_{/X}$, so we will allow a small abuse of notation and regard $\mathcal{J}$ as a wide subposet of $\disk_{/(X,E)}$. Our next main result states that localizing the discrete operad of marked disks at the morphisms contained in $\mathcal{J}$ yields the \infop of marked disks; this is a marked variant of results familiar from \cite[Section 5.4.5]{LurHA} and \cite{AFT-fh-stratified}. A precursor of this result was worked out by Eilind Karlsson in her thesis; see \cite[Theorem 7.25]{KS}. 
\begin{thm}[Localizing at isotopy equivalences]\label{thm:markedlocalization}
Let $(X,E)$ be a marked smooth conical manifold. Then the functor 
\[ \disk_{/(X,E)} \longrightarrow \infdisk_{/(X,E)} \]
described in Variant \ref{var:discreteops} carries $\mathcal{J}$ into the maximal subgroupoid of $\infdisk_{/(X,E)}$ and the induced functor
\[  \disk_{/(X,E)}^{\otimes}[\mathcal{J}^{-1}] \longrightarrow  \infdisk^{\otimes}_{/(X,E)}  \]
is an equivalence of \infopst. Moreover, the functor $\disk_{/(X,E)}[\mathcal{J}^{-1}]\rightarrow \infdisk_{/(X,E)}$ induced by localization of the underlying \infcats is also an equivalence.
\end{thm}
Composing with the equivalence \eqref{eq:restricttodisk}, we immediately deduce the following.
\begin{cor}\label{cor:fact=alg}
Let $(X,E)$ be a marked smooth conical manifold that has enough good disks and let $\icat^{\otimes}$ be a presentably symmetric monoidal \infcatt, then there is a canonical equivalence 
\[  \Fact^{\mathrm{cstr}}_{(X,E)}(\icat)\simeq \alg_{\infdisk_{/(X,E)}}(\icat) \]
of \infcatst.
\end{cor}
\begin{cor}
Let $(X,E)$ be a marked smooth conical manifold that has enough good disks and let $\icat^{\otimes}$ be a presentably symmetric monoidal \infcatt. Then $\Fact^{\mathrm{cstr}}_{(X,E)}(\icat)$ is presentable.    
\end{cor}
\begin{proof}
This follows at once from \cite[Corollary 3.2.3.5]{LurHA}.    
\end{proof}
Combining Theorem \ref{thm:markedlocalization} with Example \ref{ex:stratifiedline}, we deduce the following.
\begin{cor}\label{cor:Fact-bimod}
Let $\icat^{\otimes}$ be a presentably symmetric monoidal \infcatt. There are canonical equivalences
\[  \Fact^{\mathrm{cstr}}_{(\R_{\geq 0},0)}(\icat)\simeq \mathsf{RMod}(\icat), \quad \Fact^{\mathrm{cstr}}_{(\R,0)}(\icat)\simeq \mathsf{BMod}(\icat),\quad \Fact^{\mathrm{cstr}}_{(\R_{\leq 0},0)}(\icat)\simeq \mathsf{LMod}(\icat)\]
of \infcatst.
\end{cor}

\begin{rmk}\label{rmk:comparepullback}
Let $i:(U,E)\hookrightarrow (Y,F)$ be a map of $\mfd^{+,\simeq}$ (which we may also regard as a map of $\infmfd^{+,\simeq}$), then the induced functor $\infdisk^{+,\simeq}_{/(U,E)}\rightarrow \infdisk^{+,\simeq}_{/(F,E)}$ determines a morphism $\left(\infdisk^{+,\simeq}\right)^{\otimes}_{/(U,E)}\rightarrow \left(\infdisk^{+,\simeq}\right)^{\otimes}_{/(F,E)}$ of \infops which restricts to a morphism 
\begin{equation}\label{eq:infdiskpfwd}\infdisk^{\otimes}_{/(U,E)} \longrightarrow \infdisk^{\otimes}_{/(Y,F)}  \end{equation}
of \infopst. We similarly have a morphism $\disk^{\otimes}_{/(U,E)} \rightarrow \disk^{\otimes}_{/(Y,F)}$ of operads which is also the restriction of the pushforward $i_*$ of Remark \ref{rmk:pushfwd}, and the diagram 
\begin{equation}\label{eq:operadsquaresopen}
\begin{tikzcd}
   \infdisk^{\otimes}_{/(U,E)}\ar[r] & \infdisk^{\otimes}_{/(F,Y)} \\
   \disk^{\otimes}_{/(U,E)}\ar[r]\ar[u]\ar[d,hook] & \disk^{\otimes}_{/(F,Y)} \ar[u]\ar[d,hook]\\
   \open(U,E)^{\otimes} \ar[r,"i_*"] & \open(Y,F)^{\otimes}
\end{tikzcd}
\end{equation}
of \infops commutes. It follows from Corollary \ref{cor:fact=alg} that for a presentably symmetric monoidal \infcat $\icat^{\otimes}$, the pullback of constructible factorization algebras $i^*:\Fact^{\mathrm{cstr}}_{(Y,F)}(\icat)\rightarrow \Fact^{\mathrm{cstr}}_{(U,E)}(\icat)$ can be identified with the functor $\alg_{\infdisk_{/(Y,F)}}(\icat)\rightarrow \alg_{\infdisk_{/(U,E)}}(\icat)$ that composes with the map \eqref{eq:infdiskpfwd}. In particular, the pullback preserves all limits (\cite[Corollary 3.2.2.5]{LurHA}); in fact, it has a left adjoint by operadic left Kan extension along \eqref{eq:infdiskpfwd}.
\end{rmk}

\begin{rmk}\label{rmk:comparepushwfd}
Let $f:(X,E)\rightarrow (Y,F)$ be a continuous map of marked spaces and suppose that $(X,E)$ and $(Y,F)$ are marked smooth conical manifolds. Suppose that $f^{-1}:\open(Y,F)\rightarrow \open(X,E)$ carries disks to disks, and isotopy equivalences among disks to isotopy equivalences among disks. Then $f^{-1}$ restricts to a map of operads $\disk^{\otimes}_{/(Y,F)}\rightarrow \disk^{\otimes}_{/(X,E)}$ and induces by Theorem \ref{thm:markedlocalization} a map 
\begin{equation}\label{eq:infdiskbp}\infdisk^{\otimes}_{/(Y,F)} \longrightarrow \infdisk^{\otimes}_{/(X,E)}  \end{equation}
of \infops so that the diagram
\begin{equation}\label{eq:squaresoperads}
\begin{tikzcd}
   \infdisk^{\otimes}_{/(Y,F)}\ar[r] & \infdisk^{\otimes}_{/(X,E)} \\
   \disk^{\otimes}_{/(Y,F)}\ar[r]\ar[u]\ar[d,hook] & \disk^{\otimes}_{/(X,E)} \ar[u]\ar[d,hook]\\
   \open(Y,F)^{\otimes} \ar[r,"f^{-1}"] & \open(X,E)^{\otimes}
\end{tikzcd}
\end{equation}
of \infops commutes. It follows from Corollary \ref{cor:fact=alg} that for a presentably symmetric monoidal \infcat $\icat^{\otimes}$, the pushforward of constructible factorization algebras $f_{\sharp}:\Fact^{\mathrm{cstr}}_{(X,E)}(\icat)\rightarrow \Fact^{\mathrm{cstr}}_{(Y,F)}(\icat)$ can be identified with the functor $\alg_{\infdisk_{/(X,E)}}(\icat)\rightarrow \alg_{\infdisk_{/(Y,F)}}(\icat)$ that composes with the map \eqref{eq:infdiskbp}. In particular, the pushforward preserves all limits; in fact, it has a left adjoint given by operadic left Kan extension along \eqref{eq:infdiskbp}.
\end{rmk}
\begin{rmk}
Pulling back opens along the involution $\R\rightarrow \R$, $x\mapsto -x$ carries disks to disks and preserves isotopy equivalences. As an instance of the preceding remark, we obtain equivalences 
\[ \mathsf{RMod}(\icat)\simeq  \Fact^{\mathrm{cstr}}_{(\R_{\geq 0},0)}(\icat)\simeq \Fact^{\mathrm{cstr}}_{(\R_{\leq 0},0)}(\icat)  \simeq  \mathsf{LMod}(\icat). \]
It is not hard to convince oneself upon inspection of the diagram \eqref{eq:squaresoperads}  that this is the equivalence of \cite[Construction 4.6.3.1]{LurHA}, which carries a pair $(A,M)$ of an algebra together with a right module to the pair $(A^{rev},M)$ of the opposite algebra and the corresponding left module.
\end{rmk}

We now turn to the proof of Theorem \ref{thm:markedlocalization}. In view of the extra bookkeeping forced on us by the presence of markings, we will employ the powerful operadic localization criterion afforded by the theory of \emph{operadic approximations}, introduced in \cite[Section 2.3.3]{LurHA}, instead of the somewhat less economical criterion of \cite[Lemma A.4.7]{KSW}. The exact shape of the result that we require was written up by Harpaz (\cite{Harpaz-littledisks}) and more recently by Arakawa-Carmona-Pratali (\cite{ACP-DendroidalRezkNerve}).
\begin{thm}[Operadic approximation]\label{thm:operadic approximation}
Let $f:\iop\rightarrow \mathcal{P}^{\otimes}$ be a map of \infopst. Suppose that $f$ is an \emph{operadic approximation}, that is:
\begin{enumerate}[$(1)$]
    \item The map 
    \[  \Of\longrightarrow \mathcal{P} \]
    has weakly contractible fibers.
    \item For each $X\in \Of$, the map 
    \[ \left(\iop_{\mathrm{act}} \right)_{/X}\longrightarrow \left(\mathcal{P}^{\otimes}_{\mathrm{act}} \right)_{/f(X)} \]
    has weakly contractible fibers.
\end{enumerate}
Then $f$ exhibits $\mathcal{P}^{\otimes}$ as a \emph{universal}\footnote{That is, $f$ remains an operadic localization after pulling back along any map of \infopst.} operadic localization of $\iop$ at all morphisms of $\Of$ that are carried to an equivalence by $\Of\rightarrow\mathcal{P}$\footnote{It seems that this result is more general than \cite[Theorem 2.3.3.23]{LurHA} as Lurie states his operadic localization criterion only for maps $f:\iop\rightarrow\mathcal{P}^{\otimes}$ for which, instead of $(1)$ of Theorem \ref{thm:operadic approximation}, either of the following more restrictive conditions hold. 
\begin{enumerate}
    \item[$(1*)$] $\Of\rightarrow \mathcal{P}$ is an equivalence of \infcatst.
    \item[$(1**)$] $\mathcal{P}$ is a Kan complex and $\Of\rightarrow \mathcal{P}$ exhibits a localization at all arrows.
\end{enumerate}
Neither $(1*)$ nor $(1**)$ apply to localizations of disk operads of general conical manifolds. However, there is a minor error in the proof of \cite[Theorem 2.3.3.23]{LurHA} which makes it appear as though it only applies to the cases $(1*)$ and $(1**)$ while in fact, Lurie's argument does suffice to prove Theorem \ref{thm:operadic approximation}.}. Moreover, the functor $\Of\rightarrow\mathcal{P}$ also exhibits a localization of \infcats at the morphisms of $\Of$ that become an equivalence in $\mathcal{P}$.
\end{thm}

\begin{rmk}
Conditions $(1)$ and $(2)$ in the theorem above admit the following reformulation.
\begin{enumerate}[$(1')$]
    \item The map of spaces 
    \[ |\Of|\longrightarrow \mathcal{P}^{\simeq}  \]
    is an equivalence.
    \item For each $X\in \Of$, each integer $n\geq 0$ and each tuple $\{Y_i\}_{1\leq i\leq n}\in \mathcal{P}^{\times n}$, the map of spaces 
    \[ \left|\left(\Of^{\otimes}_{\mathrm{act}}\right)_{/X}\times_{\mathcal{P}^{\otimes}_{\mathrm{act}}} \left\{ Y_1\oplus \ldots \oplus Y_n \right\} \right| \longrightarrow \mathrm{Mul}_{\mathcal{P}^{\otimes}}\left(\{Y_i\}_{1\leq i\leq n},f(X)\right )  \]
    is an equivalence.
\end{enumerate}
We will apply Theorem \ref{thm:operadic approximation} in this form below.
\end{rmk}

\begin{proof}[Proof of Theorem \ref{thm:markedlocalization}]
The wide subcategory $\mathcal{J}\subset\disk_{/(X,E)}$ contains precisely those morphisms that are carried to an equivalence in $\infdisk_{/(X,E)}$, by definition. It suffices to argue that the functor $\disk^{\otimes}_{/(X,E)} \rightarrow \infdisk^{\otimes}_{/(X,E)}$ is an operadic approximation, that is, we must show the following.
\begin{enumerate}[$(1)$]
    \item The canonical map 
    \[ \disk_{/(X,E)}  \longrightarrow \infdisk_{/(X,E)}   \]
    has weakly contractible fibers.
    \item For each $U\in \disk_{/(X,E)}$, the canonical map 
    \[\left(\disk_{/(X,E),\mathrm{act}}^{\otimes}\right)_{/U} \longrightarrow \left(\infdisk_{/(X,E),\mathrm{act}}^{\otimes}\right)_{/U}\]
    has weakly contractible fibers.
\end{enumerate}
We first show $(1)$. Since $\mathcal{J}=\disk_{/(X,E)}\times_{ \infdisk_{/(X,E)}} \infdisk_{/(X,E)}^{\simeq}$, it suffices to show that the map 
\[ \mathcal{J}\longrightarrow   \infdisk_{/(X,E)}^{\simeq} \]
is a weak homotopy equivalences; in other words, that it exhibits $\infdisk_{/(X,E)}^{\simeq}$ as a localization of $\mathcal{J}$. The space $ \infdisk_{/(X,E)}^{\simeq} $ is the coproduct $\coprod_{[U]\in\infbsc}X_{[U]}$ where $X_{[U]}$ is the locus of those points $x\in X$ for which the conically smooth embeddings $U\hookrightarrow X$ for which the image of the cone point contains $x$ form a local basis at $x$, so it suffices to show that the map 
\[  \mathcal{J}_{[U]}\longrightarrow X_{[U]}    \]
is a localization, where $J_{[U]}\subset \mathcal{J}$ is the full subposet spanned by embedded disks $V\subset X$ such that $V$ is abstractly isomorphic in $\bsc$ to $U$. This map is the right diagonal one in the commuting diagram
\[
\begin{tikzcd}
   & \mathcal{J}_{[U]} \ar[dr] \ar[dl]\\
   \underset{V\in \mathcal{J}_{[U]}}{\colim}V_{[U]} \ar[rr] && X_{[U]}.
\end{tikzcd}
\]
For each $V\in \mathcal{J}_{[U]}$, the space $V_{[U]}$ is equivalent to $\R^k$ for $k=\mathsf{dim}(U)-\mathsf{depth}(U)$ and thus contractible, and so the left diagonal map exhibits a localization. It suffices to argue that the horizontal map is an equivalence. Consider the map of posets
\[  \mathcal{J}_{[U]}\longrightarrow\mathsf{Open}(X_{[U]}),\quad\quad V\longmapsto V_{[U]}. \]
For each $x\in X_{[U]}$, the poset $\mathcal{J}_{[U]x}$ of those $V\subset X$ for which $x\in V_{[U]}$ is cofiltered (since disks form a basis) and therefore weakly contractible. Invoking \cite[Theorem A.3.1]{LurHA}, we conclude that the horizontal map in the diagram above is an equivalence. We show $(2)$. The map in question is the top horizontal one in the triangle
\begin{equation}\label{eq:fibers}
\begin{tikzcd}
  \left(\disk_{/(X,E),\mathrm{act}}^{\otimes}\right)_{/U}  \ar[dr] \ar[rr] &&   \left(\infdisk_{/(X,E),\mathrm{act}}^{\otimes}\right)_{/U} \ar[dl] \\
  & \infdisk_{/(X,E),\mathrm{act}}^{\otimes}.
\end{tikzcd}    
\end{equation}
Since the right vertical map is a right fibration, it suffices to argue that the horizontal map becomes a weak homotopy on the fibers of the diagonal map at each object of $ \infdisk_{/(X,E),\mathrm{act}}^{\otimes}$. We distinguish between two cases.
\begin{enumerate}[$(i)$]
    \item $\mathsf{dim}(U)-\mathsf{depth}(U)=0$ and the point $*\cong U_0\hookrightarrow X_0$ is marked. 
    \item $\mathsf{dim}(U)-\mathsf{depth}(U)=0$ but the point $*\cong U_0\hookrightarrow X_0$ is not marked, or $\mathsf{dim}(U)-\mathsf{depth}(U)>0$.
\end{enumerate}
We will only give the argument for case $(i)$; case $(ii)$ is similar (and corresponds to the operadic localization of the unmarked disk operad, which is standard). Let $\vec{V}:=(V_1\hookrightarrow X)\oplus\ldots\oplus (V_n\hookrightarrow X)$ be an object in the fiber over $\langle n\rangle$. We may assume $n>0$, for otherwise the fibers of the diagonal maps in the diagram \eqref{eq:fibers} are empty (this is specific to the marked situation; in case $(ii)$, the fibers of the diagonal maps would be contractible instead). According to Remark \ref{rmk:suboperad}, if $*\cong U_0\cap E\subset E$ and $\coprod_{i}(V_i)_0 \cap E\subset E$ do not coincide in $E$, then the mapping space
\[\Hom_{\infdisk_{/(X,E),\mathrm{act}}^{\otimes} }(V_1\oplus \ldots\oplus V_n,U)=\mathrm{Mul}_{\infdisk_{/(X,E)}^{\otimes}}\left(\{V_j\}_{1\leq j\leq n},U\right)\]
that is, the fiber of the right diagonal map in \eqref{eq:fibers} at $\vec{V}$, is empty, and if these are the same subset of $E$, then we have 
\[\mathrm{Mul}_{\infdisk_{/(X,E)}^{\otimes}}\left(\{V_j\}_{1\leq j\leq n},U\right) = \mathrm{Mul}_{\infdisk_{/X}^{\otimes}}\left(\{V_j\}_{1\leq j\leq n},U\right).\]
In this case, this space may only be nonempty if there is \emph{exactly one} $i\in \langle n\rangle^{\circ}$ for which $\mathsf{dim}(V_i)-\mathsf{depth}(V_i)=0$ and the conically smooth open embedding
\[ V_i\hooklongrightarrow X \]
yields the same marked point $*\cong (V_i)_0\hookrightarrow X_0$ as $*\cong U_0\hookrightarrow E$. With the notation of Observation \ref{obs:configurations}, the object $(V_1\hookrightarrow X)\oplus\ldots\oplus (V_n\hookrightarrow X)$ determines a map 
\[ H:\langle n\rangle^{\circ} \longrightarrow \pi_0(\infdisk^{\simeq}_{/X})\cong \coprod_{[V]\in[\infbsc]}\pi_0(X_{[V]})   \]
carrying $j\in \langle n\rangle^{\circ}$ to the path component $X_{[V_j\hookrightarrow X ]} \subset X_{[V_j]}$ and the space of multimorphisms $\mathrm{Mul}_{\infdisk_{/X}^{\otimes}}\left(\{V_j\}_{1\leq j\leq n},U\right)$ is identified, by evaluation at the center of disks, with (the homotopy type of) the product of ordered configuration spaces
\[ \Hom_{\infdisk_{/X,\mathrm{act}}^{\otimes} }(V_1\oplus \ldots\oplus V_n,U)  \overset{\simeq}{\longrightarrow} \prod_{X_{[V\hookrightarrow X]}\in\pi_0\left(\infdisk^{\simeq}_{/X}\right)} \mathsf{Conf}_{|H^{-1}(X_{[V\hookrightarrow X]})|}(U_{[V\hookrightarrow X]}). \]
To compute the fiber of the left diagonal map, let $\icatd_{\vec{V}}\subset \left(\infdisk^{\otimes}_{/(X,E)}\right)_{\langle n\rangle}\subset \infdisk_{/(X,E),\mathrm{act}}^{\otimes}$ be the full subcategory of $\left(\infdisk^{\otimes}_{/(X,E)}\right)_{\langle n\rangle}$ spanned by objects $\vec{V}'=(V_1'\hookrightarrow X)\oplus\ldots \oplus (V_n'\hookrightarrow X)$ for which the image of the cone locus of each $V_j'\hookrightarrow X$ intersects the cone locus of the image of $V_j\hookrightarrow X$ (in other words, the full subcategory of $\left(\infdisks_{/(X,E)}^{\otimes}\right)_{\langle n\rangle}$ spanned by those $\vec{V}'$ that are equivalent to $\vec{V}$). For any two such objects $\vec{V}'$ and $\vec{V}''$, we have 
\begin{equation}\label{eq:vmaps}   \Hom_{\left(\infdisk^{\otimes}_{/(X,E)}\right)_{\langle n\rangle}}(V'_1\oplus \ldots\oplus V'_n,V''_1\oplus\ldots\oplus V''_n) \simeq \prod_{j\in\langle n\rangle^{\circ}} \left(V_j'\right)_{[V_j'']} \simeq \prod_{j\in\langle n\rangle^{\circ}} \R^{k_j}\simeq *   \end{equation}
where $k_j=\mathsf{dim}(V_j)-\mathsf{depth}(V_j)$. It follows that the subcategory inclusion $\icatd_{\vec{V}}\subset \infdisk_{/(X,E),\mathrm{act}}^{\otimes}$ is equivalent to the map $*\rightarrow \infdisk_{/(X,E),\mathrm{act}}^{\otimes}$ classifying $\vec{V}$, so we may identify the fiber of the left diagonal map at $\vec{V}$ with the subcategory of $\left(\disk_{/(X,E),\mathrm{act}}^{\otimes}\right)_{/U}$ of those objects and morphisms whose image lies in $\icatd_{\vec{V}}$. Unwinding definitions, this subcategory can be identified with the partially ordered set $P_{\vec{V}}$ of tuples $(V_1',\ldots,V'_n)$ of pairwise disjoint disks inside $U$ for which the cone locus of $V_j'$ lies in the path component $X_{[V_j\hookrightarrow X]}\subset X_{[V_j]}$ \emph{and} the map $\coprod_j (V'_j)_0\hookrightarrow X_0$ has the same image as $*\cong U_0\hookrightarrow E$. If this second clause is satisfied, there is again \emph{exactly one} $i\in \langle n\rangle^{\circ}$ for which $\mathsf{dim}(V_i)-\mathsf{depth}(V_i)=0$ and the conically smooth open embedding
\[ V_i\hooklongrightarrow X \]
yields the same marked point $*\cong (V_i)_0\hookrightarrow E$ as $*\cong U_0\hookrightarrow E$. We wish to show that the induced map 
\[ P_{\vec{V}} \longrightarrow \Hom_{\infdisk^{\otimes}_{/(X,E),\mathrm{act}}}(\vec{V},U) \]
is a weak homotopy equivalence. We assume there is exactly one $V_i$ for which $V_i\hookrightarrow X$ intersects the marked cone point; otherwise both \infcats are empty. The map in question is the upper right diagonal one in the commuting diagram
\[
\begin{tikzcd}
& P_{\vec{V}} \ar[dr] \ar[dl] \\
\underset{V'\in P_{\vec{V}}}{\colim}\Hom_{\left(\infdisk^{\otimes}_{/(X,E)}\right)_{\langle n\rangle}}(\vec{V},\vec{V'}) \ar[rr]\ar[d,"\simeq"] && \Hom_{\infdisk^{\otimes}_{/(X,E),\mathrm{act}}}(\vec{V},U) \ar[d,"\simeq"]\\
 \underset{V'\in P_{\vec{V}}}{\colim}\underset{1\leq j\leq n}{\prod} \left(V'_j\right)_{[V_j]} \ar[rr] &&   \prod_{X_{[V\hookrightarrow X]}\in\pi_0\left(\infdisk^{\simeq}_{/X}\right)} \mathsf{Conf}_{|H^{-1}(X_{[V\hookrightarrow X]})|}(U_{[V\hookrightarrow X]}).
\end{tikzcd}
\]
In the upper left corner of the square, the computation \eqref{eq:vmaps} shows that all mapping spaces appearing in the colimit are contractible, so that the upper left diagonal map can be identified with the localization $P_{\vec{V}}\rightarrow |P_{\vec{V}}|$ at all morphisms. Hence it suffices to argue that the horizontal maps are equivalences of spaces. The association 
\[ (V_1',\ldots,V_n')\longmapsto \underset{1\leq j\leq n}\prod \left(V'_j\right)_{[V_j]} \]
may be regarded as a map 
\[P_{\vec{V}}\longrightarrow \mathsf{Open}\left(\prod_{X_{[V\hookrightarrow X]}\in\pi_0\left(\infdisk^{\simeq}_{/X}\right)} \mathsf{Conf}_{|H^{-1}(X_{[V\hookrightarrow X]})|}(U_{[V\hookrightarrow X]})\right)\]
of partially ordered sets. For each point $x\in \prod_{X_{[V\hookrightarrow X]}\in\pi_0\left(\infdisk^{\simeq}_{/X}\right)} \mathsf{Conf}_{|H^{-1}(X_{[V\hookrightarrow X]})|}(U_{[V\hookrightarrow X]})$, the full subposet $(P_{\vec{V}})_x$ spanned by tuples $ (V_1,\ldots,V_n)$ for which $x\in \prod_{1\leq j\leq n}\left(V'_j\right)_{[V_j]}$ is cofiltered (since disks form a basis) and therefore weakly contractible. Invoking \cite[Theorem A.3.1]{LurHA}, we conclude that the horizontal maps in the diagram above are equivalences. 
\end{proof}

\begin{rmk}\label{rmk:basislocalization}
From the proof of Theorem \ref{thm:markedlocalization}, we deduce right away the following natural generalization: let $(X,E)$ be a marked smooth conical manifold and let $\mathcal{B}\subset \disk_{/(X,E)}$ be a full subposet whose elements are a basis for the topology of $X$ (so that the full subposet $\mathcal{B}'\subset \disks_{/(X,E)}$ whose elements are disjoint unions of those of $\mathcal{B}$ is a \emph{disk-basis} in the sense of Remark \ref{rmk:diskbasis}). Let $\mathcal{J}_{\mathcal{B}}$ be the intersection $\mathcal{B}\cap \mathcal{J}$ and let $\mathcal{B}^{\otimes}\subset \disk_{/(X,E)}^{\otimes}$ be the full suboperad spanned by colors that lie in $\mathcal{B}$, then the maps
\[\mathcal{B}[\mathcal{J}_{\mathcal{B}}^{-1}]\longrightarrow \infdisk_{/(X,E)}  \]
and
\[\mathcal{B}^{\otimes}[\mathcal{J}_{\mathcal{B}}^{-1}]\longrightarrow \infdisk_{/(X,E)}^{\otimes} \]
are equivalences of \infcats and \infops respectively. 
\end{rmk}

\subsection{The higher Morita category}\label{sec:moritacat}

In this section, we give a construction of the factorization higher Morita category. We find it convenient to work in a concrete model for higher categories, namely \emph{(complete) higher Segal objects} in $\widehat{\spa}$, the \infcat of (large) spaces. The higher Morita category comes to us as the \emph{underlying $(\infty,n+1)$-category} of an \emph{$n$-fold category object} of $\catinfh$, the \infcat of (large) \infcatst. All \infcats of composable sequences of morphisms that comprise this $n$-fold category object will be \infcats of pointless constructible factorization algebras on natural stratifications of the $n$-dimensional open unit cube.

\subsubsection{Stratifying posets for cubes}
We begin by describing the partially ordered sets that govern the stratifications of our cubes, as well as the functoriality they exhibit in the (semi-)simplex category.
\begin{nota}
We will denote an element $[k_1,\ldots,k_n]$ of $\simp^{n}=(\simp)^{\times n}$, that is, a tuple of finite linearly ordered sets, collectively by $[\vec{k}]$.
\end{nota}
\begin{defn}
For an element $[\vec{k}]=[k_1,\ldots,k_n]$ of $\simp^n$, we let $P_{[\vec{k}]}$ denote the partially ordered set
\[  P_{[\vec{k}]}:=\mathsf{sd}(\mathsf{Sp}_{k_1})^{op}\times \ldots \times \mathsf{sd}(\mathsf{Sp}_{k_n})^{op}.  \]
Here, $\mathsf{sd}(\_)$ is the subdivision functor and $\mathsf{Sp}_k\subset \Delta^k$ denotes the spine of the $k$-simplex, the union of all \emph{consecutive} 1-simplices in $\Delta^k$. Note that $P_{[\vec{k}]}=P_{[k_1]}\times\ldots \times P_{[k_n]}$ by definition. We will regard this poset as a topological space, so that a set is open if and only if it is upward closed.      
\end{defn}
\begin{rmk}
    In Haugseng's \cite[Definition 5.2]{Haugseng-iteratedspans} the poset $P_{[k]}$ is denoted by (boldface) $\Lambda^k$.
\end{rmk}
We record the following functorialities of the posets $P_{[\_]}$.
\begin{lem}\label{lem:functorposet}
Let $f:[k]\rightarrow [l]$ be a map of finite linear ordinals, then the corresponding map $\Delta^k\rightarrow \Delta^l$ on simplices carries the spine $\mathsf{Sp}_k\subset \Delta^k$ into the spine $\mathsf{Sp}_l\subset \Delta^l$ in the following cases.
\begin{itemize}
    \item The map $f$ is the inclusion of a subinterval, that is, $f(i+1)=f(i)+1$ for $0\leq i\leq k-1$.
    \item The map $f$ is surjective.
\end{itemize}
\end{lem}
The proof is a direct inspection left to the reader. It follows that the association $[\vec{k}]\mapsto P_{[\vec{k}]}$ determines functors
\[ P_{\mathrm{in}}:\simp^n_{\mathrm{in}} \longrightarrow \mathsf{Poset}, \quad \quad  P_{\mathrm{surj}}:\simp^n_{\mathrm{surj}} \longrightarrow \mathsf{Poset}, \]
where $\simp_{\mathrm{in}}\subset \simp$ denotes the subcategory spanned by subinterval inclusions (these correspond to inert maps of $\simpop$, hence the subscript) and $\simp_{\mathrm{surj}}\subset \simp$ denotes the subcategory spanned by surjections. These functors combine into a single one defined on the subcategory of $\simp$ whose morphisms are composites of a surjective map followed by a subinterval inclusion.
\begin{rmk}
The poset $P_{[k]}$ can be explicitly described as follows: denoting the elements of the underlying set of $P_{[k]}$ as
\[ \{\mathfrak{a}_0,\mathfrak{m}_{0,1}, \mathfrak{a}_1,\mathfrak{m}_{1,2},\ldots,\mathfrak{m}_{k-1,k},\mathfrak{a}_k\},\]
the only nontrivial relations are 
\[  \mathfrak{a}_i>\mathfrak{m}_{i,i+1} \quad\text{and}\quad \mathfrak{a}_{i+1}> \mathfrak{m}_{i,i+1} \]
for all $0\leq i\leq k-1$. For instance, $P_{[3]}$ viewed as a category can be depicted as 
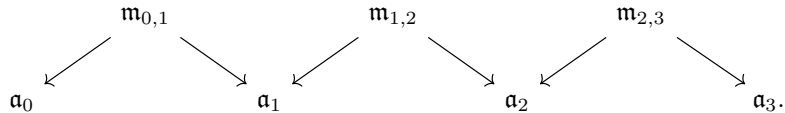
\[
\begin{tikzcd}
& \mathfrak{m}_{0,1} \ar[dr]\ar[dl] &&   \mathfrak{m}_{1,2} \ar[dr]\ar[dl]   && \mathfrak{m}_{2,3} \ar[dr]\ar[dl]   \\
\mathfrak{a}_0 &&\mathfrak{a}_1 &&\mathfrak{a}_2 &&\mathfrak{a}_3. 
\end{tikzcd}
\]
 The functor $P_{\mathrm{in}}$ of Lemma \ref{lem:functorposet} carries a subinterval inclusion $f:[k]\hookrightarrow [l]$ to the map 
\begin{equation*}\label{eq:posetfunctor}
\mathfrak{a}_i\longmapsto \mathfrak{a}_{f(i)},\quad \quad \mathfrak{m}_{i,i+1}\longmapsto \mathfrak{m}_{f(i),f(i)+1}
\end{equation*}  
for $0\leq i\leq k$ in the first case and $0\leq i\leq k-1$ in the second. The functor $P_{\mathrm{surj}}$ carries a surjective map $g:[k]\rightarrow [l]$ to the map 
\begin{equation*}\label{eq:posetfunctor2}
\mathfrak{a}_i\longmapsto \mathfrak{a}_{g(i)},\quad \quad \mathfrak{m}_{i,i+1}\longmapsto \begin{cases}
   \mathfrak{m}_{g(i),g(i)+1}& \text{if } g(i+1)=g(i)+1 \\
    \mathfrak{a}_{g(i)} & \text{if } g(i+1)=g(i) 
\end{cases}
\end{equation*} 
for $0\leq i\leq k$ in the first case and $0\leq i\leq k-1$ in the second. Note that for a \emph{surjective} order-preserving map $g:[k]\rightarrow [l]$, we must have $g(i+1)=g(i)$ or $g(i+1)=g(i)+1$ for all $i$ (so in particular $g(0)=0$), and $g(k)=l$). 
\end{rmk}
\begin{rmk}
As the notation suggests, the elements $\mathfrak{a}_i\in P_{[k]}$ will correspond to strata that will be labeled by algebra objects by a constructible factorization algebra, while the elements $\mathfrak{m}_{i,i+1}$ correspond to strata that will be labeled by bimodule objects.
\end{rmk}

We will also work with the following variant of $P_{[k]}$.
\begin{var}
We let $\overline{P}_{[k]}$ denote the poset obtained from $P_{[k]}$ by adjoining two additional \emph{boundary} elements, denoted $\mathfrak{m}_{-1,0}$ and $\mathfrak{m}_{k,k+1}$, together with the relations 
\[  \mathfrak{a}_0>\mathfrak{m}_{-1,0} \quad\text{and}\quad \mathfrak{a}_k>\mathfrak{m}_{k,k++1}.  \]
Taking products, we similarly have posets $P_{[\vec{k}]}\subset \overline{P}_{[\vec{k}]}$ for all $[\vec{k}]\in \simp^n$.    
\end{var}
For instance, we may depict the category $\overline{P}_{[3]}$ as
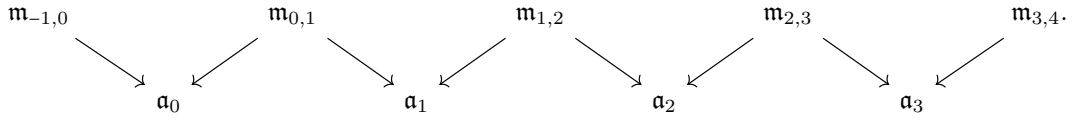
\[
\begin{tikzcd}
\mathfrak{m}_{-1,0}\ar[dr]&& \mathfrak{m}_{0,1} \ar[dr]\ar[dl] &&   \mathfrak{m}_{1,2} \ar[dr]\ar[dl]   && \mathfrak{m}_{2,3} \ar[dr]\ar[dl] && \mathfrak{m}_{3,4}.\ar[dl]  \\
&\mathfrak{a}_0 &&\mathfrak{a}_1 &&\mathfrak{a}_2 &&\mathfrak{a}_3 
\end{tikzcd}
\]
Note that we have a canonical identification 
\[  \overline{P}_{[\vec{k}]} \cong \mathsf{sd}(\mathsf{Sp}_{k_1+1})\times \ldots \times \mathsf{sd}(\mathsf{Sp}_{k_n+1}).   \]
It follows from Lemma \ref{lem:functorposet} that the assignment $[k_1+1,\ldots,k_n+1]\mapsto \overline{P}_{[\vec{k}]}$ determines a functor 
\[ \overline{P}_{\mathrm{surj}}:\simp^n_{\geq 1,\mathrm{surj}}\longrightarrow \mathsf{Poset}  \]
which carries a surjective map $g:[k+1]\rightarrow [l+1]$ to the map 
\begin{equation}\label{eq:posetfunctor2.5}
\mathfrak{m}_{i-1,i}\longmapsto \mathfrak{m}_{g(i)-1,g(i)},\quad \quad \mathfrak{a}_{i}\longmapsto \begin{cases}
   \mathfrak{m}_{g(i)-1,g(i)}& \text{if } g(i+1)=g(i) \\
    \mathfrak{a}_{g(i)} & \text{if } g(i+1)=g(i)+1 
\end{cases}
\end{equation} 
for $0\leq i\leq k+1$ in the first case and $0\leq i\leq k$ in the second.
(We will not need the functor $\overline{P}_{\mathrm{in}}$ defined on the subinterval inclusions). Recall that there is an isomorphism of categories
\[  \simp_{\geq 1,\mathrm{surj}} \cong\simpop_{\mathrm{inj}}, \]
where $\simpop_{\mathrm{inj}}\subset \simpop$ is the subcategory on injective maps. Consequently, we have a functor
\begin{equation}\label{eq:posetfunctor3} \overline{P}_{(\_)}:\simpopnin\longrightarrow \mathsf{Poset}. \end{equation}
For later use, we isolate some convenient properties enjoyed by the maps of posets in the image of this functor. For this, we introduce some terminology.
\begin{defn}[Depth]\label{defn:depth}
For $[k]\in\simp$, we have a map of sets
\[ \mathsf{d}:\overline{P}_{[k]}\longrightarrow [1] \]
defined by $\mathsf{d}(\mathfrak{a}_i)=0$ and $\mathsf{d}(\mathfrak{m}_{i, i+1)})=1$ (in the latter case we also consider $i=-1,k$). Taking products, we also let $\mathsf{d}$ denote the map of sets
\[ \overline{P}_{[\vec{k}]}=\overline{P}_{[k_1]}\times\ldots\times\overline{P}_{[k_n]}\overset{\mathsf{d}^n}{\longrightarrow} [1]^n\overset{+}{\longrightarrow} [n] \]
where the second map carries $(i_1,\ldots,i_n)$ to the sum $\sum_{j=1}^ni_j$.
\end{defn}
\begin{rmk}
The depth function counts how many instances of an element of the form $\mathfrak{m}_{i,i+1}$ there are among a tuple $(X_1,\ldots,X_n)\in P_{[\vec{k}]}$, and similarly for $\overline{P}_{[\vec{k}]}$
\end{rmk}
\begin{rmk}
Note that the depth is a map of posets 
\[ \mathsf{d}:\overline{P}_{[\vec{k}]}^{op}\longrightarrow [n],  \]
that is, the smaller the element in the poset $\overline{P}_{[\vec{k}]}$, the larger the depth.
\end{rmk}

\begin{defn}\label{defn:depthorientationpreserving}
Let $[k],[l]\in\simp$ and let $\alpha:\overline{P}_{[k]}\rightarrow \overline{P}_{[l]}$ be a  map of posets. 
\begin{enumerate}[$(1)$]
    \item We say that $\alpha$ is \emph{depth-preserving} if the following condition is satisfied: for every element of the form $\mathfrak{m}_{i,i+1}\in P_{[k]}$, the element $\alpha(\mathfrak{m}_{i,i+1})$ is of the form $\mathfrak{m}_{j,j+1}$, in other words, $\alpha$ carries depth 1 elements to depth 1 elements.
    \item If $\alpha$ is depth-preserving, we say that $\alpha$ is \emph{orientation-preserving} if the following condition is satisfied: for every $0\leq i\leq k-1$ for which $\alpha(\mathfrak{a}_i)$ is of the form $\mathfrak{a}_j$ (that is, has depth $0$), we have $\alpha(\mathfrak{m}_{i,i+1})=\mathfrak{m}_{j,j+1}$.
    \item If $\alpha$ is depth- and orientation-preserving, we say that $\alpha$ is \emph{boundary-preserving} if $\alpha$ carries the boundary elements $\mathfrak{m}_{-1,0}$ and $\mathfrak{m}_{k,k+1}$ to boundary elements.
\end{enumerate}
\end{defn}
\begin{rmk}\label{rmk:depthorientationcompose}
If $\alpha:\overline{P}_{[k]}\rightarrow \overline{P}_{[l]}$ and $\beta:\overline{P}_{[l]}\rightarrow \overline{P}_{[m]}$ are maps of posets and $\alpha$ and $\beta$ are depth-preserving (orientation-preserving, boundary-preserving), then then $\beta\circ \alpha$ is depth-preserving (orientation-preserving, boundary-preserving) as well. Moreover, the identities are clearly depth-, orientation- and boundary-preserving. It follows that the posets $\left\{\overline{P}_{[k]}\right\}_{[k]\in\simp}$ together with the depth-preserving (orientation-preserving, boundary-preserving) maps form a subcategory of $\mathsf{Poset}$.
\end{rmk}
\begin{cor}\label{cor:posetfunctordepthor}
Let $f:[k+1]\rightarrow  [l+1]$ be a surjective map of finite linear ordinals (corresponding to an \emph{injective} map $[k]\hookleftarrow [l]$ of finite linear ordinals, i.e.~order-preserving map) and let $\alpha:\overline{P}_{[k]}\rightarrow \overline{P}_{[l]}$ be its image under $\overline{P}$. Then $\alpha$ is depth-preserving, orientation-preserving and boundary-preserving. 
\end{cor}
\begin{proof}
It follows immediately from the definition of $\mathsf{d}$ and the formula \eqref{eq:posetfunctor2.5} that $\alpha$ is depth- and orientation-preserving. Since $f$ is surjective, it carries $0$ to $0$ and $k+1$ to $l+1$, so the formula \eqref{eq:posetfunctor2.5} also shows that boundary elements are preserved.
\end{proof}
\begin{rmk}
If $f:[k+1]\hookrightarrow  [l+1]$ is a subinterval inclusion, then the induced map $\overline{P}_{[k]}\rightarrow \overline{P}_{[l]}$ is also depth- and orientation-preserving. In fact, it is not hard to see that the subcategory of $\mathsf{Poset}$ whose objects are the posets $\overline{P}_{[k]}$ and whose morphisms are depth- and orientation-preserving maps is equivalent to the wide subcategory of $\simp_{\geq 1}$ whose morphisms are composites of surjective maps followed by subinterval inclusions.  The subcategory $\simp_{\geq 1,\mathrm{surj}}\subset \simp$ is identified with the category of the posets $\left\{\overline{P}_{[k]}\right\}_{[k]\in\simp}$ with maps that are depth-, orientation- \emph{and} boundary-preserving.
\end{rmk}

\subsubsection{Stratified cubes}

Now we prepare our collection of stratified cubes. These are obtained by stratifying the interval $(0,1)$ with the equidistant stratification by a finite collection of points. These stratifications will be governed by the partially ordered sets introduced just above. 
\begin{defn}[Stratified cubes]\label{defn:stratifiedcube}
Write $\square$ for the open interval $(0,1)$, and $\overline{\square}$ for the closed interval $[0,1]$. For each $[k]$ in $\simp$, consider the continuous map 
\[ \pi_{[k]}: \square\longrightarrow P_{[k]}  \]
defined by formula
\begin{equation}\label{eq:cubestrat}
\pi_{[k]}(x) =\begin{cases}
    \mathfrak{a}_i & \text{if } x\in \left(\frac{i}{k+1},\frac{i+1}{k+1}\right), \\
    \mathfrak{m}_{i,i+1} & \text{if } x=\frac{i+1}{k+1},
\end{cases}
\end{equation} 
for $0\leq i\leq k$. We let $\square_{[k]}$ denote the map $ \pi_{[k]}: \square\rightarrow P_{[k]}$ viewed as a poset-stratified space. For the closed interval, we similarly have a continuous map
\[ \overline{\pi}_{[k]}: \overline{\square}\longrightarrow \overline{P}_{[k]}  \]
defined by the formula \eqref{eq:cubestrat}, now for $-1\leq i\leq k$. Taking products, we also have for each $[\vec{k}]=[k_1,\ldots,k_n]$ the poset-stratified spaces
\[
\pi_{[\vec{k}]}:\square^{n} \longrightarrow P_{[\vec{k}]} \quad\text{and}\quad \overline{\pi}_{[\vec{k}]}:\overline{\square}^{n} \longrightarrow \overline{P}_{[\vec{k}]}
\]
that we denote by $\square^n_{[\vec{k}]}$ and $\overline{\square}^n_{[\vec{k}]}$ respectively. We will call these objects the \emph{stratified} (\emph{open} or \emph{closed}) \emph{cubes}.
\end{defn}

\begin{ex}
For $n=2$, $k_1=4$, and $k_2=3$ the stratified plane $\square^{2}_{[4,3]}$ is depicted in the picture below. 
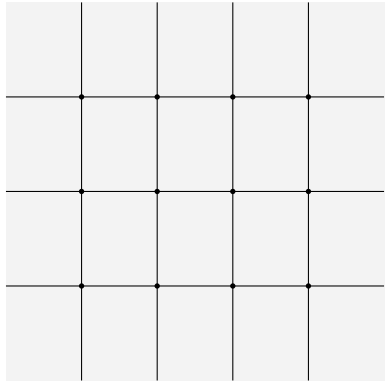
\begin{center}
\begin{tikzpicture}
\fill[ultra nearly transparent] (0,0) -- (5,0) -- (5,5) -- (0,5) -- cycle;

\draw (1, 0) -- (1, 5);
\draw (2, 0) -- (2, 5);
\draw (3, 0) -- (3, 5);
\draw (4, 0) -- (4, 5);

\draw (0, 1.25) -- (5, 1.25);
\draw (0, 2.5) -- (5, 2.5);
\draw (0, 3.75) -- (5, 3.75);

\foreach \x in {1,2,3,4} {
    \foreach \y in {1.25, 2.5, 3.75} {
        \fill (\x,\y) circle (0.1em);
    }
}
\end{tikzpicture}
\end{center}

For $n=3$, $k_1=2$, $k_2=3$, and $k_3=1$ the stratified cube $\square^{3}_{[2,3,1]}$ is depicted in the picture below.  (outer $(0,1)^3$ omited):
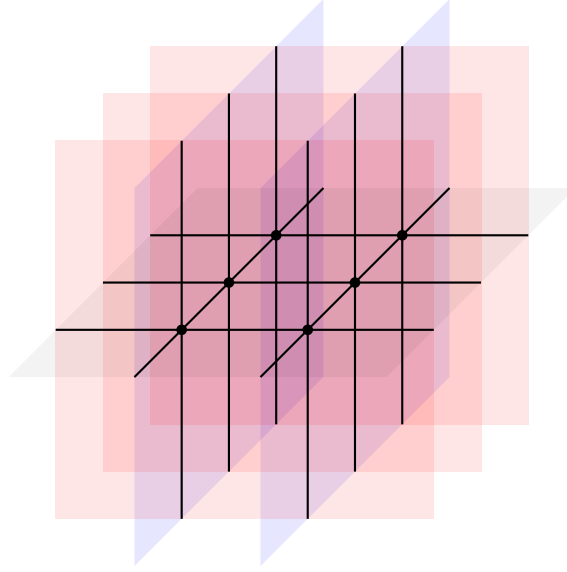
\begin{center}
\begin{tikzpicture}[x={(0.5cm,0.5cm)}, y={(1cm,0cm)}, z={(0cm,1cm)}]

    \def\max{5}
    \def\mid{2.5}
    \def\vlines{1.25, 2.5, 3.75}
    \def\hlines{1.666, 3.333}

    \foreach \y in \hlines {
        \fill[blue, opacity=0.1] (0,\y,0) -- (\max,\y,0) -- (\max,\y,\max) -- (0,\y,\max) -- cycle;
    }

    \foreach \x in \vlines {
        \fill[red, opacity=0.1] (\x,0,0) -- (\x,\max,0) -- (\x,\max,\max) -- (\x,0,\max) -- cycle;
    }

    \fill[ultra nearly transparent] (0,0,\mid) -- (\max,0,\mid) -- (\max,\max,\mid) -- (0,\max,\mid) -- cycle;

    \foreach \y in \hlines {
        \draw[thick, black] (0,\y,\mid) -- (\max,\y,\mid);
    }

    \foreach \x in \vlines {
        \draw[thick, black] (\x,0,\mid) -- (\x,\max,\mid);
    }

    \foreach \x in \vlines {
        \foreach \y in \hlines {
            \draw[thick, black] (\x,\y,0) -- (\x,\y,\max);
            \fill (\x,\y,\mid) circle (0.2em); 
        }
    }

\end{tikzpicture}
\end{center}
\end{ex}

For each $X\in P_{[\vec{k}]}$, we let $\square^n_X\subset \square^n_{[\vec{k}]}$ denote the stratum determined by $X$, so that we have a pullback diagram 
\[
\begin{tikzcd}
    \square^n_X\ar[d] \ar[r,hook] & \square^n \ar[d, "\pi_{[\vec{k}]}"] \\
    \{X\} \ar[r,hook] & P_{[\vec{k}]}
\end{tikzcd}
\]
of topological spaces. Write $X=(X_1,\ldots,X_n)$ for $X_j\in P_{[k_j]}$ and let $\mathfrak{a}(X)\subset \{1,\ldots,n\}$ be the subset of indices $j$ for which $X_j$ has depth 0, so that its complement $\mathfrak{m}(X)=[n]\setminus \mathfrak{a}(X)$ consists of indices $j$ for which $X_j$ has depth 1. For $j\in \mathfrak{a}(X)$, write $\mathfrak{a}_{i_j}$ for $X_j$ and for $j\in \mathfrak{m}$, write $X_j=\mathfrak{m}_{i_j,i_j+1}$ then $\square^n_X\subset \square^n $ is the subspace 
\[\prod_{j\in \mathfrak{a}(X)} \left(\frac{i_j}{k_j+1},\frac{i_j+1}{k_j+1}\right)\times \prod_{j\in \mathfrak{m}(X)} \left\{\frac{i_j+1}{k_j+1} \right\}\subset \prod_{j\in [n]}(0,1),\]
so that $\square^n_X$ can be identified with the cube $\square^{\mathfrak{a}(X)}$. In fact, there is a unique identification that has the following property: it is given by a product 
\[ \square^{\mathfrak{a}(X)}= \prod_{j\in \mathfrak{a}(X)}(0,1)\xrightarrow{\prod_{j\in \mathfrak{a}(X)}f(j)} \prod_{j\in \mathfrak{a}(X)} \left(\frac{i_j}{k_j+1},\frac{i_j+1}{k_j+1}\right)  \]
such that each map $f_j$ is of the form 
\[f_j(x) = c_jx+b_j\]
for some real constants $c_i>0$ and $b_j$. Such a map is a \emph{rectilinear embedding}. For the reader's convenience, we recall this notion.
\begin{defn}[Rectilinear embeddings]
For $S$ a finite set, we will say that a map $f:\square^S\rightarrow\square^S$ is a \emph{rectilinear embedding} if $f$ is a product $\left\{f_j:\square^{\{j\}}\rightarrow \square^{\{j\}}\right\}_{j\in S}$ where each $f_j$ is given by $x\mapsto a_jx+b_j$ for some real constants $a_j>0$ and $b_j$.
\end{defn}
The following basic observation underlies the functoriality of the stratified cubes: associated to every depth- and orientation-preserving map of posets $\alpha: \overline{P}_{[k]}\rightarrow \overline{P}_{[l]}$, there is a unique \emph{collapse-rescale} map. Observe that if $\alpha(\mathfrak{a}_i)$ has depth 1, the space $\square_{\alpha(\mathfrak{a}_i)}$ is a point.
\begin{lem}[Uniqueness and composability of collapse-rescale maps]\label{lem:collaplsrescale}
Let $\alpha:\overline{P}_{[k]}\rightarrow \overline{P}_{[l]}$ be a depth- and orientation preserving map, in the sense of Definition \ref{defn:depthorientationpreserving}. Then the following hold true.
\begin{enumerate}[$(1)$]
    \item There is a unique continuous stratified map $\overline{\square}_{[k]}\rightarrow \overline{\square}_{[l]}$ over $\alpha$ with the following property: 
    \begin{enumerate}
        \item[$(*)$] For each $\mathfrak{a}_i\in P_{[l]}$ for which $\alpha(\mathfrak{a}_i)$ is also of the form $\mathfrak{a}_j$ (that is, has depth 0), the induced map 
    \[  \square_{\mathfrak{a}_i}\longrightarrow \square_{\alpha(\mathfrak{a}_i)} \]
    is a rectilinear embedding.
    \end{enumerate} 
    Let us say that this is the \emph{collapse-rescale map associated to $\alpha$}, and denote it by $\collresc(\alpha)$.
    \item Let $\beta:\overline{P}_{[k]}\rightarrow \overline{P}_{[l]}$ be another depth- and orientation-preserving map of posets, then the composite $ \collresc(\beta)\circ \collresc(\alpha)$ is a map over $\beta\circ \alpha$ (which is again depth-and-orientation preserving; see Remark \ref{rmk:depthorientationcompose}) satisfying condition $(*)$, so it follows from uniqueness that $ \collresc(\beta)\circ \collresc(\alpha)=\collresc(\beta\circ \alpha)$.
    \item The collapse-rescale map associated to the identity $\overline{P}_{[k]}=\overline{P}_{[k]}$ is the identity on $\overline{\square}$.
\end{enumerate}
\end{lem}
\begin{proof}
We prove $(1)$. Let $\alpha:\overline{P}_{[k]}\rightarrow \overline{P}_{[l]}$ be a depth-preserving and orientation-preserving map. We will construct a continuous stratified map $f:\overline{\square}_{[k]}\rightarrow \overline{\square}_{[l]}$ over $\alpha$ satisfying $(*)$; it will be clear from the construction that this is the unique map satisfying $(*)$. Since $\alpha$ is depth-preserving, we have for all $-1\leq i\leq k$ some $-1\leq \varphi(i)\leq l$ such that $\alpha(\mathfrak{m}_{i,i+1})=\mathfrak{m}_{\varphi(i),\varphi(i)+1}$. Since $f$ is a stratified map, we must have 
\[ f\left(\frac{i+1}{k+1}\right) = \frac{\varphi(i)+1}{l+1} \]
for $-1\leq i\leq k$. Now we define $f$ on all intervals of the form $\left(\frac{i}{k+1},\frac{i+1}{k+1}\right)$. If $\alpha(\alpha_i)$ has depth $1$, then we must have $\alpha(\alpha_i)=\mathfrak{m}_{\varphi(i),\varphi(i)+1}$ and we must have 
\[  f(x) = \frac{\varphi(i)+1}{l+1} \text{ for }x\in \left(\frac{i}{k+1},\frac{i+1}{k+1}\right).   \]
If $\alpha(\alpha_i)$ has depth $0$, then $\alpha(\alpha_i)=\mathfrak{a}_{\varphi(i)}$ and $f|_{\left[\frac{i}{k+1},\frac{i+1}{k+1}\right]}$ is a continuous map
\[ \left[\frac{i}{k+1},\frac{i+1}{k+1}\right]\longrightarrow \left[\frac{\varphi(i)}{l+1},\frac{\varphi(i)+1}{l+1}\right]\]
preserving beginning and endpoints, since $\alpha$ is orientation-preserving. There is exactly one such continuous map that restricts to a rectilinear embedding on the interior $\left(\frac{i}{k+1},\frac{i+1}{k+1}\right)$ (and this map is in fact a homeomorphism). These maps $\left\{f|_{\left[\frac{i}{k+1},\frac{i+1}{k+1}\right]}\right\}_{0\leq i\leq k}$ are equal where their domains overlap, which is at most at a single point, so they glue to the desired continuous stratified map. Now $(2)$ follows right away from the fact that continuous stratified maps and rectilinear embeddings compose, and $(3)$ is immediate.
\end{proof}
Let $f:[\vec{k}]\hookleftarrow [\vec{l}]$ be a map of $\simpopnin$, then the associated map $\overline{P}_{[\vec{k}]}\rightarrow \overline{P}_{[\vec{l}]}$ is a product of depth-preserving and orientation-preserving maps of posets; abusing notation, we will write $\collresc(f)$ for the associated product $\overline{\square}^n_{[\vec{k}]}\rightarrow \overline{\square}^n_{[\vec{l}]}$ of collapse-rescale maps. Combining the preceding lemma with Corollary \ref{cor:posetfunctordepthor}, we obtain the following result.
\begin{cor}
There is a functor $\overline{\square}^n_{(\_)}:\simpopnin\rightarrow \mathsf{StTop}$ that carries $[\vec{k}]$ to $\overline{\square}^n_{[\vec{k}]}$ and a map $f:[\vec{k}]\hookleftarrow [\vec{l}]$ to $\collresc(f)$ so that the diagram of categories
\[
\begin{tikzcd}
    & \mathsf{StTop} \ar[d] \\
    \simpopnin \ar[ur,"\overline{\square}^n_{(\_)}"] \ar[r,"\overline{P}_{(\_)}"']& \mathsf{Poset}
\end{tikzcd}
\]
commutes.
\end{cor}
\begin{rmk}\label{rmk:bdypoints}
For any injective map $f:[k]\hookleftarrow [l]$ of linear ordinals, the map $\overline{P}_{[k]}\rightarrow \overline{P}_{[l]}$ is also boundary-preserving. It follows from the proof of Lemma \ref{lem:collaplsrescale} that for any $f:[\vec{k}]\hookleftarrow[\vec{l}]$ of $\simpopnin$, the induced collapse-rescale map $\collresc(f):\overline{\square}^n_{[\vec{k}]}\rightarrow \overline{\square}^n_{[\vec{l}]}$ carries boundary points of $[0,1]^n=\overline{\square}^n$ to boundary points.
\end{rmk}

\subsubsection{Stratified cubes as smooth conical manifolds}

We wish to apply the results from \cite{KS} and the preceding sections to the stratified cubes; for this, we must endow these stratified spaces with smooth conical manifold structures. First, we note that they are conical manifolds.

\begin{lem}\label{lem:cubesconicalmfd}
The stratified open cubes $\square_{[k]}$ and closed cubes $\overline{\square}_{[k]}$ lie in the full subcategory $\mfd_0$ of conical manifolds.
\end{lem}
\begin{proof}
For $\overline{\square}_{[k]}$, we note the following.
\begin{itemize}
\item[]
    \item If $x\in \overline{\square}=[0,1]$ is carried to some $\mathfrak{a}_i$ by $\overline{\pi}_{[k]}$, then $x$ is in the image of a (stratified) open embedding $\R\rightarrow \overline{\square}$.
    \item Consider the disjoint union $*_-\coprod *_+$ of two points as a stratified space over the poset $\mathfrak{a}_-\coprod \mathfrak{a}_+$ (without relations), then the cone \[\mathsf{C}(*_-\coprod *_+) \cong \R^-_{\geq 0} \vee\R^+_{\geq 0}\] 
    is two copies of $\R_{\geq 0}$ glued along the point $0$, lying over the poset 
    \[ (\mathfrak{a}_-\coprod \mathfrak{a}_+)^{\lhd}=\begin{tikzcd}
        & \mathfrak{m}_{-,+} \ar[dr] \ar[dl] \\
        \mathfrak{a}_- && \mathfrak{a}_+
    \end{tikzcd} \]
    in the natural way (note that this poset is $P_{[1]}$). If $x\in \overline{\square}$ is carried to some $\mathfrak{m}_{i,i+1}$ for $0\leq i\leq k-1$ by $\overline{\pi}_{[k]}$, then $x$ is in the image of a stratified open embedding $\mathsf{C}(*_+\coprod *_-)\cong \R_{\geq 0}\vee \R_{\geq 0}\rightarrow \overline{\square}$ carrying the cone point to $x$.
    \item If $x\in \overline{\square}$ is carried to  $\mathfrak{m}_{-1,0}$ or $\mathfrak{m}_{k,k+1}$ by $\overline{\pi}_{[k]}$, then $x$ is in the image of a stratified open embedding $\mathsf{C}(*)\cong \R_{\geq 0}\rightarrow \overline{\square}$ carrying the cone point to $x$.
\end{itemize}
Thus $\overline{\square}_{[k]}$ and its interior ${\square}_{[k]}$ are conical manifolds, since they admits a basis by images of stratified open embeddings from topological basics. The stratified open and closed cubes $\square^n_{[\vec{k}]}$ and $\overline{\square}^n_{[\vec{k}]}$ are also conical manifolds as the full subcategory $\mathsf{Mfd}_0\subset \mathsf{StTop}_{\mathsf{open}}$ is stable under products, by \cite[Corollary 3.4.10]{AFT-local-structures}.
\end{proof}
\begin{cor}\label{cor:openmap}
The maps $\pi_{[\vec{k}]}$ and $\overline{\pi}_{[\vec{k}]}$ of Definition \ref{defn:stratifiedcube} are open maps for any $[\vec{k}]\in \simp^n$.   
\end{cor}
\begin{proof}
    This is an immediate consequence of Lemma \ref{lem:structuremapisopen}.
\end{proof}
\begin{rmk}
As the stratified cubes are conical manifolds, the results of \cite[Section 2.4]{AFT-local-structures} allow one to speak of the \emph{depth} of strata of the stratified cubes. This is the same as the notion of Definition \ref{defn:depth}.  
\end{rmk}

We are now ready to endow the cubes $\overline{\square}^n_{[\vec{k}]}$ 
with conically smooth structures. 
\begin{cons}[Atlases on cubes]\label{cons:cubeatalas}
Let $[k]\in\simp$, then we define a set $\mathcal{A}_{\overline{\square}_{[k]}}$ of stratified open embeddings to $\overline{\square}_{[k]}$ as follows.
\begin{itemize}
    \item For all $0\leq i\leq k$, all smooth open embeddings 
    \[ \R\longrightarrow \left(\frac{i}{k+1},\frac{i+1}{k+1}\right) \subset \overline{\square}_{[k]}   \]
    are in $\mathcal{A}_{\overline{\square}_{[k]}}$.
    \item For the disjoint union $*_-\coprod *_+$ of two points, consider $\mathsf{C}(*_-\coprod *_+)$ as stratified by the poset $(\mathfrak{a}_-\coprod\mathfrak{a}_+)^{\lhd}$ as in  Lemma \ref{lem:cubesconicalmfd}. We have a canonical homeomorphism 
    \[ \mathsf{C}(*_+\coprod *_-)\cong \R^-_{\geq 0} \vee \R^+_{\geq 0}  \cong \R \]
    that is the identity on $\R_{\geq 0}^+$ and carries $x\in \R_{\geq 0}^-$ to $-x$. Then for all $1\leq i\leq k$, all stratified open embeddings 
    \[  \mathsf{C}(*_-\coprod *_+)\cong \R\longrightarrow  \left(\frac{i-1}{k+1},\frac{i+1}{k+1}\right) \subset \overline{\square}_{[k]}   \]
    where the second map is a smooth open embedding carrying the origin to $\frac{i}{k+1}$ are in $\mathcal{A}_{\overline{\square}_{[k]}}$.
    \item Consider $\mathsf{C}(*)=\R_{\geq 0}$, the cone on a point, then all stratified open embeddings 
    \[  \R_{\geq 0} \longrightarrow \left[0,\frac{1}{k+1}\right)\subset \overline{\square}_{[k]},\quad\quad \R_{\geq 0} \longrightarrow \left(\frac{k}{k+1},1\right]\subset \overline{\square}_{[k]}   \]
    in which the map is smooth\footnote{Here by smooth we mean \emph{Seeley smooth}, that is, the function extends to a smooth function on some neighborhood $(-\epsilon,\infty)$ of $\R_{\geq 0}$.} carrying $0$ to $0$ and $0$ to $1$ respectively, are in $\mathcal{A}_{\overline{\square}_{[k]}}$.
\end{itemize}
This collection $\mathcal{A}_{\overline{\square}_{[k]}}$ satisfies the conditions of \cite[Definition 3.2.10]{AFT-local-structures} of being an atlas; we regard $\overline{\square}_{[k]}$ as endowed with the smooth conical manifold structure by the atlas $\mathcal{A}_{\overline{\square}_{[k]}}$. Omitting the last of the three types of charts described above endows $\square_{[k]}$ with the structure of a smooth conical manifold such that the map $\square_{[k]}\subset \overline{\square}_{[k]}$ is a conically smooth open embedding. Through \cite[Lemma 3.4.8, Corollary 3.4.9, Corollary 3.4.10]{AFT-local-structures}, we have smooth conical manifold structures on the products 
\[  \square_{[\vec{k}]}^n\quad\text{and}\quad \overline{\square}_{[\vec{k}]}^n\]
for all $[\vec{k}]\in \simp^n$.
\end{cons}

\begin{rmk}\label{Rem:not_con_smooth}
It is evident that the three basics appearing in Construction \ref{cons:cubeatalas} are the only ones that admit a stratified open embedding to the cubes $\overline{\square}_{[k]}$; however, it is not the case that the stratified open embeddings specified in Construction \ref{cons:cubeatalas} are the only ones contained in the \emph{maximal} atlas determined by $\mathcal{A}_{\overline{\square}_{[k]}}$. To see this, observe that a continuous stratified map 
\[f:\mathsf{C}(*\coprod *)\cong \R_{\geq0}\vee \R_{\geq 0} \longrightarrow \mathsf{C}(*\coprod *)\cong \R_{\geq0}\vee \R_{\geq 0} \]
that carries the cone point to the cone point determines two maps $f_1,f_2:\R_{\geq 0}\rightarrow\R_{\geq 0}$ such that $f_1(0)=f_2(0)=0$, and $f$ is a conically smooth just in case $f_1|_{\R_{>0}}$ and $f_2|_{\R_{>0}}$ are smooth maps and the one-sided derivatives of $f_1$ and $f_2$ at $0$ coincide, by definition of \emph{conically smooth along} (\cite[Definition 3.1.4]{AFT-local-structures}). In particular, the derivative of $f$ is merely required to be continuous at $0$. Precomposing a chart $\mathsf{C}(*\coprod *)\rightarrow \overline{\square}_{[k]}$ in $\mathcal{A}_{\overline{\square}_{[k]}}$ with such a conically smooth open embedding $f$ whose derivative at the origin is not smooth yields an element in the maximal atlas determined by $\mathcal{A}_{\overline{\square}_{[k]}}$ that is not in $\mathcal{A}_{\overline{\square}_{[k]}}$ itself. A similar comment applies to the third of the three charts described in Construction \ref{cons:cubeatalas}. 
\end{rmk}
\begin{warn}\label{warn:consmf}
The collapse-rescale maps of Lemma \ref{lem:collaplsrescale} are generally \emph{not} conically smooth for the smooth conical manifold structures of Construction \ref{cons:cubeatalas}, since the continuity condition on the derivatives at the cone point is often violated. This is the case, for instance, in the situation of  Remark \ref{Rem:not_con_smooth}, when one of the maps determined by $f$, say $f_2$, is constant at $0$ (collapsing one of the copies of $\R_{\geq 0}$ to $0$) and $f_1$ is an orientation-preserving linear diffeomorphism of $\R_{\geq 0}$ (rescaling the other copy of $\R_{\geq 0}$). There does not seem to be a straightforward way to remain in the conically smooth category while maintaining the simple functoriality of the collapse-rescale maps. 
\end{warn}

\begin{nota}
Let $[\vec{k}]\in \overline{P}_{[\vec{k}]}$. When we consider $\square_{[\vec{k}]}^n$ and $\overline{\square}_{[\vec{k}]}^n$ as \emph{marked} smooth conical manifolds, we will always equip these spaces with the \emph{maximal} marking. We let $E_{[\vec{k}]}\subset \square^n_{[\vec{k}]}$ denote this maximal marking, the subset of points of depth $n$, which is also the union of all the 0-dimensional strata. Sometimes we will denote this marking simply by $E$ to avoid cluttering up the notation.   
\end{nota}

We now exhibit factorizing disk-basis for the stratified cubes (recall Definition \ref{defn:factorizingdiskbasis}), that is, the stratified cubes have enough good disks. We will say that an open $U\subset \overline{\square}^n$ is an \emph{open subcube} if it is an open inclusion of the form 
\[ \prod_{i=1}^n U_i\subset \prod_{i=1}^n [0,1] \]
where each $U_i\subset [0,1]$ is an open interval (which could be of the form $(a,b)$, $[0,a)$, $(b,1]$ or $[0,1]$). Through Lemma \ref{lem:structuremapisopen}, these open sets inherit the structure of stratified spaces, and through Construction \ref{cons:cubeatalas}, they inherit the structure of smooth conical manifolds (since atlases pull back along stratified open embeddings).
\begin{lem}\label{lem:disklikecubes}
Let $U\subset \overline{\square}^n_{[\vec{k}]}$ be an open subcube, then the following are equivalent.
\begin{enumerate}[$(1)$]
    \item The stratified space $U$ with its induced smooth conical manifold structure is abstractly isomorphic to a disk.
    \item The poset $\overline{\pi}_{[\vec{k}]}(U)$ has a minimal element.
    \item The poset $\overline{\pi}_{[\vec{k}]}(U)$ contains an element $X$ such that $\overline{\pi}_{[\vec{k}]}(U)$ is the upward closure of $\{X\}$ in $\overline{P}_{[\vec{k}]}$.
\end{enumerate}
\end{lem}
\begin{proof}
Any disk has the property that its stratifying poset has a minimal element, by definition, so $(1)\Rightarrow (2)$ holds. Suppose $\overline{\pi}_{[\vec{k}]}(U)$ has a minimal element $X$, then the upward closure of $\{X\}$ is contained in $\overline{\pi}_{[\vec{k}]}(U)$ since the latter set is upward closed by Corollary \ref{cor:openmap} and contains $X$. Since $X$ is the minimal element of $\overline{\pi}_{[\vec{k}]}(U)$, the other inclusion also holds: every element of $\overline{\pi}_{[\vec{k}]}(U)$ is larger than $X$ and thus lies in the upward closure. The implication $(3)\Rightarrow (2)$ is obvious. Now we show $(2)\Rightarrow (1)$. Since $U=U_1\times \ldots \times U_n$ for $U_i\in \overline{\square}_{[k_i]}$ an open interval, we can identify $U$ with $U_1\times \ldots \times U_n$ and $\pi_{[\vec{k}]}(U)$ with $\pi_{[k_1]}(U_1)\times\ldots\times  \pi_{[k_n]}(U_n)$. Since products of disks are disks, it suffices to show that each $U_i$ is abstractly isomorphic to a disk. A product of posets has a minimal element if and only if all factors have a minimal element, so we reduce to the case $n=1$. By definition of the maps $\overline{\pi}_{[k]}$, an open interval of $[0,1]$ has a minimal element if and only if it is the image of one of the charts of Construction \ref{cons:cubeatalas}.
\end{proof}

\begin{defn}[Disk-like cubes]\label{defn:disklikecubes}
We will say that an open subcube $U\subset \overline{\square}_{[\vec{k}]}^n$ is \emph{disk-like} if any one of the three equivalent conditions on $U$ of Lemma \ref{lem:disklikecubes} is satisfied for $U$. We let \[\cube_{\overline{\square}^n_{[\vec{k}]}}\subset \open\left(\overline{\square}^n_{[\vec{k}]}\right)\]
denote the collection of disk-like open subcubes, together with the empty set.
\end{defn}
\begin{rmk}
Note that $\cube_{\overline{\square}^n_{[\vec{k}]}}$ depends on $[\vec{k}]$; the larger the integers $k_1,\ldots,k_n$, the fewer open subcubes are disk-like.
\end{rmk}
The following is an immediate consequence of the proof of Lemma \ref{lem:disklikecubes}.
\begin{cor}
Let $U\subset V$ be an inclusion of disk-like cubes of $\overline{\square}^n_{[\vec{k}]}$. Then the following are equivalent.
\begin{enumerate}[$(1)$]
    \item The inclusion is an isotopy equivalence of $\infdisk$.
    \item The posets $\overline{\pi}_{[\vec{k}]}(U)$ and $\overline{\pi}_{[\vec{k}]}(V)$ have the same minimal element.
    \item The posets $\overline{\pi}_{[\vec{k}]}(U)$ and $\overline{\pi}_{[\vec{k}]}(V)$ are equal.
\end{enumerate}
\end{cor}

\begin{lem}
The collection $\cube_{\overline{\square}^n_{[\vec{k}]}}$ is a basis for the topology on $\overline{\square}^n$ stable under intersections.
\end{lem}
\begin{proof}
Clearly, every point of $\overline{\square}_{[\vec{k}]}^n$ is contained in a disk-like cube. The intersection of two open subcubes $U,V$ is either empty or an open subcube, so it suffices to show that $\pi_{[\vec{k}]}(U\cap V)$ has a minimal element. As in the proof of Lemma \ref{lem:disklikecubes}, we reduce to the case of $n=1$. We conclude by noting that for the intersection of two intervals $U,V$ that do not share a point in the 0-dimensional stratum, the intersection $U\cap V$ is either empty or an interval containing no 0-dimensional strata, so that the underlying poset is a singleton in either case. 
\end{proof}
Let $\mathsf{diskCubes}_{\overline{\square}^n_{[\vec{k}]}}\subset \mathsf{Disks}_{/\overline{\square}^n_{[\vec{k}]}}$ be the full subposet spanned by disjoint unions of disk-like cubes, then by Remark \ref{rmk:factorizingbasiscriterion}, we have the following consequence of the preceding lemma.
\begin{cor}\label{cor:factorizingbasisbycubes}
The full subposet $\mathsf{diskCubes}_{\overline{\square}^n_{[\vec{k}]}}\subset \mathsf{Disks}_{/\overline{\square}^n_{[\vec{k}]}}$ is a factorizing disk-basis of $\overline{\square}^n_{[\vec{k}]}$ (and also of the \emph{marked} smooth conical manifold $(\overline{\square}^n_{[\vec{k}]},E_{[\vec{k}]})$), and the same holds for the open stratified cubes.
\end{cor}

The following corollary is a consequence of Remark \ref{rmk:basislocalization} 

\begin{cor}\label{cor:cubelocalization}
The \infopt-maps 
\[  \cube_{\square_{[\vec{k}]}}^{\otimes} \longrightarrow \infdisk_{/\square_{[\vec{k}]}}^{\otimes}\quad \text{and} \quad \cube_{\overline{\square}_{[\vec{k}]}}^{\otimes} \longrightarrow \infdisk_{/\overline{\square}_{[\vec{k}]}}^{\otimes},  \]
and
\[  \cube_{(\square_{[\vec{k}]},E_{[\vec{k}]})}^{\otimes} \longrightarrow \infdisk_{/(\square_{[\vec{k}]},E_{[\vec{k}]})}^{\otimes}\quad \text{and}\quad  \cube_{(\overline{\square}_{[\vec{k}]},E_{[\vec{k}]})}^{\otimes} \longrightarrow \infdisk_{/(\overline{\square}_{[\vec{k}]},E_{[\vec{k}]})}^{\otimes}  \]
exhibit operadic localizations at the isotopy equivalences.
\end{cor}

\subsubsection{Assembling the Morita category}
We now define the non-unital higher Morita category as an $n$-fold semisimplicial object. Just above, we constructed a functor $\simpopnin\rightarrow \mathsf{StTop}$ to the category of stratified topological spaces. The composite $\simpopnin\rightarrow \mathsf{Top}$ forgetting the stratification lifts to the category $\mathsf{Top}^+$ of \emph{marked} topological spaces as follows: for $[\vec{k}]\in \overline{P}_{[\vec{k}]}$, we let $E_{[\vec{k}]}\subset \square^n_{[\vec{k}]}$ be the subset of points of depth $n$ (which is also the union of the 0-dimensional strata) then we have a functor
\[ \simpopnin\longrightarrow \mathsf{Top}^+,\quad\quad [\vec{k}]\longmapsto \left(\overline{\square}^n_{[\vec{k}]},E_{[\vec{k}]}\right).  \]
Now recall the functor 
\[ \mathsf{Top}^+\longrightarrow \mathsf{Op}^{op},\quad \quad (X,E)\longmapsto \open(X,E)^{\otimes}  \]
carrying a marked space to its (1-)operad of marked opens and a map $f:(X,E)\rightarrow (Y,F)$ of marked opens to the evident pullback map $f^{-1}:\open(Y,F)^{\otimes}\rightarrow \open(X,E)^{\otimes}$.  
\begin{lem}
For any map $f:[\vec{k}]\hookleftarrow [\vec{l}]$, the induced functor
\[ \collresc(f)^{-1}:\open\left(\overline{\square}^n_{[\vec{l}]},E_{[\vec{l}]}\right)^{\otimes} \longrightarrow \open\left(\overline{\square}^n_{[\vec{k}]},E_{[\vec{k}]}\right)^{\otimes} \]
carries the full suboperad $\open\left(\square^n_{[\vec{l}]},E_{[\vec{l}]}\right)^{\otimes}\subset \open\left(\overline{\square}^n_{[\vec{l}]},E_{[\vec{l}]}\right)^{\otimes} $ into the full suboperad $\open\left(\square^n_{[\vec{k}]},E_{[\vec{k}]}\right)^{\otimes}\subset \open\left(\overline{\square}^n_{[\vec{k}]},E_{[\vec{k}]}\right)^{\otimes} $.
\end{lem}
\begin{proof}
This is simply the assertion that $\collresc(f)$ carries boundary points of $\overline{\square}^n_{[\vec{k}]}$ to boundary points of $\overline{\square}^n_{[\vec{l}]}$, as explained in Remark \ref{rmk:bdypoints}.    
\end{proof}
\begin{rmk}\label{rmk:collrescomposite}
Let $f:[k]\hookleftarrow [l]$ be an injective map, then the induced map $\collresc(f)^{-1}:\open\left(\square_{[l]}\right)^{\otimes}\rightarrow \open\left(\square_{[k]}\right)^{\otimes}$ admits the following description: let $t_{\mathrm{min}}=\collresc(f)^{-1}(0)$ and  $t_{\mathrm{max}}=\collresc(f)^{-1}(1)$, then $\collresc(f)$ carries the interval $(t_{\mathrm{min}},t_{\mathrm{max}})$ into $(0,1)$ and $\collresc(f)^{-1}$ restricted to $\open\left(\square_{[l]}\right)^{\otimes}$ is the composite 
\[ \open\left(\square_{[l]}\right)^{\otimes}\longrightarrow\open\left((t_{\mathrm{min}},t_{\mathrm{max}})\right)^{\otimes}\longrightarrow \open\left(\square_{[k]}\right)^{\otimes}  \]
of the map pulling back opens along $\collresc(f)|_{(t_{\mathrm{min}},t_{\mathrm{max}})}$ followed by the inclusion $i$ of opens of $(t_{\mathrm{min}},t_{\mathrm{max}})$ into $(0,1)$, so that the induced map on factorization algebras is the composite $\left(\collresc(f)|_{(t_{\mathrm{min}},t_{\mathrm{max}})}\right)_{\sharp}\circ i^*$. The situation in higher dimensions is the same and follows upon taking products.
\end{rmk}
It follows that we have a functor
\[   \simpopnin \longrightarrow \mathsf{Op}^{op},\quad \quad [\vec{k}]\longmapsto \open\left(\square^n_{[\vec{k}]},E_{[\vec{k}]}\right)^{\otimes} \]
to the opposite of the ordinary category of (colored) operads that carries a map $f:[\vec{k}]\hookleftarrow [\vec{l}]$ to $\collresc(f)^{-1}$. Let $\icat^{\otimes}$ be a symmetric monoidal \infcatt, then we define a functor 
\[ \pfact_{\square^n_{(\_)}}(\icat):\simpopnin\longrightarrow \catinfh\]
by composing the functor $\open\left(\square^n_{(\_)},E_{(\_)}\right)^{\otimes}:\simpopnin\rightarrow \mathsf{Op}^{op}$ with the functor
\[ \mathsf{Op}^{op}\subset \opinf^{op}\xrightarrow{\alg_{(\_)}}\catinfh \]
taking algebra objects in $\icat$. (The inclusion is the inclusion of ordinary operads into the \infcat of \infopst.)
\begin{lem}
For any map $f:[\vec{k}]\hookleftarrow [\vec{l}]$, the induced pushforward functor
\[ \collresc(f)_{\sharp}:\pfact_{(\square^n_{[\vec{k}]},E_{[\vec{k}]})}(\icat) \longrightarrow  \pfact_{(\square^n_{[\vec{l}]},E_{[\vec{l}]})}(\icat)\]
carries constructible pointless factorization algebras on $\square^n_{[\vec{k}]}$ to constructible pointless factorization algebras on $\square^n_{[\vec{l}]}$.
\end{lem}
\begin{proof}
This is the content of Lemma 8.19 of \cite{KS}.
\end{proof}
\begin{cor}\label{cor:factorizationsifted2}
For any map $f:[\vec{k}]\hookleftarrow [\vec{l}]$, the induced pushforward functor
\[ \collresc(f)_{\sharp}:\pfact_{(\square^n_{[\vec{k}]},E_{[\vec{k}]})}(\icat) \longrightarrow  \pfact_{(\square^n_{[\vec{l}]},E_{[\vec{l}]})}(\icat)\]
on constructible pointless factorization algebras preserves sifted colimits.
\end{cor}
\begin{proof}
This follows from Proposition \ref{prop:factorizationsifted}; the same proof as the one of Corollary \ref{cor:factorizationsifted} applies.    
\end{proof}
\begin{nota}
We let $\mathcal{K}$ denote the collection of all small \emph{sifted} \infcats and denote by $\catinfh\left(\mathcal{K}\right)\subset \catinfh$ the \infcat of (large) \infcats that admits sifted colimits and functors among them that preserve sifted colimits (c.f. \cite[Definition 4.8.1.1]{LurHA}).
\end{nota}
The preceding results yield a functor
\[ \Fact^{\mathrm{cstr}}_{(\square^n_{(\_)},E_{(\_)}) }(\icat):\simpopnin\longrightarrow \catinfh\left(\mathcal{K}\right), \]
that is, an $n$-fold semisimplicial \infcatt.
\begin{prop}\label{prop:moritasegal}
Suppose that $\icat^{\otimes}$ is presentably symmetric monoidal.
Then the $n$-fold semisimplicial \infcat $\Fact^{\mathrm{cstr}}_{(\square^n_{(\_)},E_{(\_)}) }(\icat)$ satisfies the Segal condition; that is, $\Fact^{\mathrm{cstr}}_{(\square^n_{(\_)},E_{(\_)}) }(\icat)$ is an $n$-fold non-unital category object of $\catinfh$ (in fact, of $\catinfh\left(\mathcal{K}\right)$).
\end{prop}
\begin{proof}
For any $1\leq i\leq n$, fix an element $[\vec{k}']=[k_1,\ldots, k_{i-1}, k_{i+1},\ldots k_n]$. For $k$ another nonnegative integer, write $[\vec{k}]=[k_1,\ldots, k_{i-1},k, k_{i+1},\ldots k_n]$ and set the shorthand notation
\begin{equation}\label{eqn:short_Fact}
F_k\coloneqq \Fact^{\mathrm{cstr}}_{(\square^n_{[\vec{k}]},E_{[\vec{k}]})}(\icat).
\end{equation}
We need to check that for any integer $n>0$, we have an equivalence
\[F_{n+k}\longrightarrow F_n\times_{F_0}  F_k\]
for the maps induced by the commutative square in~$\simp_{\mathrm{inj}}$,
\[\begin{tikzcd}
{[n+k]} & {[k]} \arrow{l} \\
{[n]} \arrow{u} & {[0]}  \arrow{u} \arrow{l}\,
\end{tikzcd}\]
determined by the inclusions $[0,n]\subset [0,n+k]$ and $[n,n+k]
\subset [0,n+k]$. For $l<m\in \mathbb{N}$ we denote by $I_{l,m}$ the interval $(l, m)\subset \mathbb{R}$ viewed as a stratified space \[(l, m)\supset \{l+1,l+2,\ldots, m-1\}\]
and equipped with the maximal marking, so that $I_{l,m}$, as a marked stratified space, is isomorphic to the 1-dimensional open cube $\square_{[m-l-1]}$ by a rescaling and translating. Then we have a cube
\begin{center}
\adjustbox{scale=0.8,center}{%
 \cdcubeNA[small]
    {F_{n+k}}{F_k}
    {F_n}{F_0}
    {\Fact^{\mathrm{cstr}}_{I_{0,n+k+1}\times \square^{n-1}_{[\vec{k}']}}(\icat)}{\Fact^{\mathrm{cstr}}_{I_{n,n+ k+1}\times \square^{n-1}_{[\vec{k}']}}(\icat)}
    {\Fact^{\mathrm{cstr}}_{I_{0,n+1}\times \square^{n-1}_{[\vec{k}']}}(\icat)}{\Fact^{\mathrm{cstr}}_{I_{n,n+1}\times \square^{n-1}_{[\vec{k}']}}(\icat)}
}
\end{center}
where the maps in the bottom square are restrictions and the vertical maps are pushforwards along the homeomorphisms that scale by a factor of $n+k+1$, $n+1$, and $k+1$, respectively, along with  a translation by $n$ in the case of $(n,n+ k+1)$ and $(n,n+1)$ (we have suppressed the marking from the notation in the bottom, but all factorization algebras are pointless). The bottom square is a pullback by 
\cite[Theorem 5.3]{KS}, which says that constructible pointless factorization algebras can be glued from an open cover. Since the vertical maps are all equivalences, we conclude that the top square also is a pullback square, as desired.
\end{proof}

We now promote the assignment $\Fact^{\mathrm{cstr}}_{(\square^n_{(\_)},E_{(\_)})}$ to an $n$-fold semisimplicial \emph{symmetric monoidal} \infcatt; this will induce the symmetric monoidal structure on the higher Morita category. To achieve this, we aim to enhance the functor $\open \left(\square^n_{(\_)},E_{(\_)}\right)^{\otimes}:\simpopnin\rightarrow \mathsf{Op}^{op}$ to a functor 
\begin{equation}\label{eq:symmonenhance}\simpopnin\times \fin\longrightarrow \mathsf{Op}^{op},\quad\quad ([\vec{k}],\langle m\rangle )\longmapsto \open \left((\square^n_{[\vec{k}]})^{\coprod \langle m\rangle^{\circ} },E_{[\vec{k}]}^{\coprod \langle m\rangle^{\circ} }\right)^{\otimes}.\end{equation}
First, we describe the action of this functor on $\fin$. For $(X,E)$ a marked space, consider the functor 
\[   \fin \longrightarrow \mathsf{Op}^{op},\quad\quad \open\left(X^{\coprod \langle m\rangle^{\circ}},E^{\coprod \langle m\rangle^{\circ}} \right)^{\otimes}\]
that carries a map $g:\langle m\rangle\rightarrow \langle r\rangle$ to the map
\[\open\left(X^{\coprod \langle r\rangle^{\circ}},E^{\coprod \langle r\rangle^{\circ}} \right)^{\otimes}\longrightarrow \open\left(X^{\coprod \langle m\rangle^{\circ}},E^{\coprod \langle m\rangle^{\circ}} \right)^{\otimes},\quad\quad \coprod_{j\in \langle r\rangle^{\circ}}U_j\longmapsto \coprod_{i\in \langle m\rangle^{\circ} } U_{g(i)}.  \]
The following is an immediate check.
\begin{lem}
Let $f:[\vec{k}]\hookleftarrow [\vec{l}]$ be a map of $\simpopnin$ and let $g:\langle m\rangle\rightarrow \langle r\rangle$ be a map of $\fin$, then the diagram of operads
\begin{equation}\label{eq:operadscommute}
\begin{tikzcd}
    \open \left((\square^n_{[\vec{l}]})^{\coprod \langle r\rangle^{\circ} },E_{[\vec{l}]}^{\coprod \langle r\rangle^{\circ} }\right)^{\otimes} \ar[d,"\left(\rho(f)^{\coprod\langle r\rangle^{\circ}}\right)^{-1}"] \ar[r] &\open \left((\square^n_{[\vec{l}]})^{\coprod \langle m\rangle^{\circ} },E_{[\vec{l}]}^{\coprod \langle m\rangle^{\circ} }\right)^{\otimes}\ar[d,"\left(\rho(f)^{\coprod\langle m\rangle^{\circ}}\right)^{-1}"]\\
    \open \left((\square^n_{[\vec{k}]})^{\coprod \langle r\rangle^{\circ} },E_{[\vec{k}]}^{\coprod \langle r\rangle^{\circ} }\right)^{\otimes}\ar[r] & \open \left((\square^n_{[\vec{k}]})^{\coprod \langle m\rangle^{\circ} },E_{[\vec{k}]}^{\coprod \langle m\rangle^{\circ} }\right)^{\otimes}
\end{tikzcd}
\end{equation}
commutes, where the horizontal maps are defined just above and the vertical maps pull back along coproducts of collapse-rescale maps.
\end{lem}
It follows right away that carrying the pair of maps $(f,g)$ to either of the two (identical) composites in the diagram \eqref{eq:operadscommute} defines the functor \eqref{eq:symmonenhance}. Composing with the functor $\alg_{(\_)}(\icat):\mathsf{Op}^{op}\subset \opinf^{op} \rightarrow \catinf$ yields a functor 
\[ \simpopnin\times \fin\longrightarrow \catinfh,\quad\quad ([\vec{k}],\langle m\rangle)\longmapsto \Fact^{\mathsf{cstr}}_{(\square_{[\vec{k}]}^{\coprod \langle m\rangle^{\circ}},E_{[\vec{k}]}^{\coprod \langle m\rangle^{\circ} })}(\icat).\]
If follows from \cite[Proposition 5.11]{KS} (which is itself a minor variation on \cite[Corollary 6.2.3]{KSW}) that the adjoint functor 
\begin{equation}\label{eq:symmonenhance2}
    \simpopnin\longrightarrow \fun\left(\fin,\catinfh\right),\quad\quad [k]\longmapsto  \Fact^{\mathsf{cstr}}_{(\square_{[\vec{k}]},E_{[\vec{k}]})}(\icat)^{\otimes}
\end{equation}  
factors through the full subcategory spanned by commutative monoid objects, that is, symmetric monoidal \infcatst.

\subsubsection{Morphisms and composites in the higher Morita category}

Let $\icat^{\otimes}$ be a presentably symmetric monoidal \infcatt. It follows from Corollary \ref{cor:Fact-bimod} that 
\[ \alg_1(\icat)_1 = \Fact^{\mathsf{cstr}}_{(\square_{[1]},E_{[1]})}(\icat) \simeq \mathsf{BMod}(\icat),  \]
and combining Remarks \ref{rmk:comparepullback} and \ref{rmk:collrescomposite}, we see that the product of collapse-rescale maps associated to $[1]\hookleftarrow [0]\cong\{0\}$ and $[1]\hookleftarrow [0]\cong \{1\}$ yields the functor
\[ \mathsf{BMod}(\icat)\longrightarrow \alg(\icat)\times\alg(\icat),\quad\quad (A,M,B)\longmapsto (A,B).  \]
induced by the map
\[ \mathsf{Assoc}^{\otimes}\boxplus \mathsf{Assoc}^{\otimes}\longrightarrow \mathsf{BM}^{\otimes},\qquad (\mathfrak{a},\mathfrak{a})\longmapsto \{\mathfrak{a}_-,\mathfrak{a}_+\}. \]
We now interpret morphism categories in the \emph{higher} nonunital category $\Fact^{\mathrm{cstr}}_{(\square^n_{(\_)},E_{(\_)}) }(\icat)$ in terms of \emph{iterated} bimodules. The technical crux is the following pointless version of the stratified additivity theorem of Carmona-\v{S}vraka \cite[Theorem A]{Carmona-Svraka}. 
\begin{prop}[Pointless additivity]\label{prop:markedadditivity}
Let $\icat^{\otimes}$ be a presentably symmetric monoidal \infcat and let $[\vec{k}]=[k_1,\ldots,k,\ldots,k_n]\in \simpopnin$. Denote $[\vec{k}']=[k_1,\ldots,\widehat{k},\ldots,k_n] \in \simp^{op}_{n-1}$ the tuple obtained by removing $k$. Then pushforward along the projection $\square^{n}_{[\vec{k}]}\rightarrow \square_{[k]}$ induces an equivalence
    \[  \Fact^{\mathrm{cstr}}_{(\square^n_{[\vec{k}]},E_{[\vec{k}]})}(\icat) \xrightarrow{\simeq} \Fact^{\mathrm{cstr}}_{(\square_{[k]},E_{[k]})}\left(\Fact^{\mathrm{cstr}}_{(\square_{[\vec{k}']},E_{[\vec{k}']})}(\icat)\right) \]
of \infcatst.
\end{prop}
A proof will appear in upcoming joint work of the first two authors with Anja \v{S}vraka \cite{ScStSv}. Invoking Corollary \ref{cor:Fact-bimod}, we deduce a canonical equivalence 
\[  \Fact^{\mathrm{cstr}}_{(\square^n_{[\vec{1}]},E_{[\vec{1}]})}(\icat) \simeq \mathsf{BMod}\left(\mathsf{BMod}((\ldots\mathsf{BMod}(\icat))\right)  \]
where on the right hand side we take $n$-fold iterated bimodules. Now we identify composition in the higher Morita category. Let $\alpha:[2]\hookleftarrow \{0,2\}\cong [1]$ be the map that induces composition of a pair of composable morphisms, with associated collapse-rescale map 
\[  \collresc(\alpha): \square_{[2]} \longrightarrow \square_{[1]}. \]
Let $\F$ be pointless constructible factorization algebra on $\square_{[2]}$, corresponding to a collection $(A_0,A_1,A_2)$ of associative algebras and a pair $(M_{01},M_{12})$ of composable bimodules in the Morita category, by Proposition \ref{prop:moritasegal} and Corollary \ref{cor:Fact-bimod}. The pushforward $\collresc(\alpha)_{\sharp}(\F)$ corresponds to the composite. It follows that we may identify the underlying object in $\icat$ of this composite $M_{02}$ with the value of $\F$ any interval $(a,b)\subset \square_{[1]}$ containing $1/2$. Moreover, the inclusion 
\[(c,d)\cup (a,b)\subset (c,d)\]
for $0<c<d<a$ and the inclusion 
\[(a,b)\cup (e,f)\subset (a,f)\] 
for $b<e<f<1$ exhibit the left and right action of the associative algebras $\collresc(\alpha)_{\sharp}(\F)\left((0,1/2)\right)=\F((0,1/3))=A_{0}$ and $\collresc(\alpha)_{\sharp}(\F)((1/2,1))=\F((2/3,1))=A_2$ on $M_{02}$. Following \cite[Example 6.14]{KSW}, we compute the value of the pushforward $\collresc(\alpha)_{\sharp}(\F)((0,1))=\F((0,1))$, by a convenient choice of Weiss cover. Let $\mathcal{K}$ be the set of nonempty compact subsets of $(1/3,2/3)$ with finitely many connected components, and let $\mathcal{W}\subset \open\left(\square_{[2]},E\right)$ be the full subposet spanned by opens of the form $(0,1)\setminus K\subset (0,1)$ for ${K\in \mathcal{K}}$. This subposet generates a sieve $\overline{\mathcal{W}}$ which, as one readily verifies, is a covering sieve for the Weiss topology on $(\square_{[2]},E)$, so that the canonical map
\[  \underset{V\in \overline{\mathcal{W}}}{\colim} \F(V) \overset{\simeq}{\longrightarrow} \F((0,1))  \]
is an equivalence. Since $\mathcal{W}$ is stable under intersections, the inclusion $\mathcal{W}\subset \overline{\mathcal{W}}$ is colimit cofinal. Furthermore, the map
\[ \mathcal{W} \cong \mathcal{K}^{op} \xrightarrow{\pi_0} \simpop \]
where $\pi_0(K)$ for $K \in\mathcal{K}$ inherits the order from $\R$, exhibits a localization (see \cite[Proposition A.2.2]{KSW}), and since $\F$ is constructible and multiplicative, the restriction $\F|_{\mathcal{W}}$ factors through this localization. Consequently, we have an equivalence 
\[ \underset{[n]\in \simpop}{\colim} \F((0,1)\setminus (K_0\cup \ldots \cup K_{n+1})) \simeq \F((0,1))  \]
for $K_0\cup \ldots \cup K_{n+1}$ some compact subset with $n+1$ connected components, written in ascending order. Note that 
\[(0,1)\setminus (K_0\cup \ldots \cup K_{n+1}) = V_0\cup \ldots \cup V_{n+1}  \]
for $V_0$ and $V_{n+1}$ open intervals containing $1/3$ and $2/3$ respectively, and $V_i$ an open interval contained in $ (1/3,2/3)$ for all $1\leq i\leq n$. Using multiplicativity of $\F$, we deduce equivalences 
\[ \F((0,1))\simeq \underset{[n]\in \simpop}{\colim} \F(V_0)\otimes \ldots \otimes \F(V_{n+1}) \simeq \underset{[n]\in \simpop}{\colim} M_{01} \otimes A_1^{\otimes n}\otimes M_{12} = M_{01}\otimes_{A_1}M_{12},   \]
that is, composition in the Morita category is given by the relative tensor product of bimodules, at least on the level of the underlying objects in $\icat$. We will need a refinement of this computation that identifies $\collresc(\alpha)_{\sharp}$ with the relative tensor product \emph{as a functor among \infcats of bimodules}.
\begin{prop}\label{prop:relativetensorprod}
Let $\icat^{\otimes}$ be a presentably symmetric monoidal \infcatt. The functor 
\[  \mathsf{BMod}(\icat)\times_{\mathsf{Alg}(\icat)}\mathsf{BMod}(\icat) \simeq \Fact_{(\square_{[2]},E_{[2]})}^{\mathrm{cstr}}(\icat)\xrightarrow{\rho(\alpha)_{\sharp}} \Fact_{(\square_{[1]},E_{[1]})}^{\mathrm{cstr}}(\icat) \simeq \mathsf{BMod}_{\icat},\]
where the first equivalence comes from Proposition \ref{prop:moritasegal} and the second map is the pushforward $\rho(\alpha)_{\sharp}$ of the collapse-rescale map for the injection $\alpha:[2]\hookleftarrow [1]\cong \{0,2\}$, is equivalent to the relative tensor product functor of \cite[Example 4.4.2.1]{LurHA}.
\end{prop}
A proof will appear in upcoming joint work of the first two authors with Anja \v{S}vraka \cite{ScStSv}. As a consequence of this result, we can promote the non-unital higher Morita category to a unital one. Let $\icatd$ be an \infcat that admits finite limits, then recall from \cite{Haugseng_quasi} the subcategory  
\[ \cat_n^{\mathrm{qu}}(\icatd)\subset \cat_n^{\mathrm{nu}}(\icatd) \]
of \emph{quasi-unital} $n$-fold category objects and \emph{quasi-unital} functors among them in $\icatd$. By \cite[Corollary 4.15]{Haugseng_quasi}, the forgetful functor 
\[ \cat_n(\icatd)\longrightarrow \cat_n^{\mathrm{nu}}(\icatd) \]
induced by the subcategory inclusion $\simpopnin\subset \simp^{n,op}$ is an equivalence onto $\cat_n^{\mathrm{qu}}(\icatd)$. We now verify that $\Fact^{\mathrm{cstr}}_{(\square^n_{(\_)},E_{(\_)}) }(\icat)$ lies in this subcategory.
\begin{prop}\label{prop:moritaquasiunit}
Suppose that $\icat^{\otimes}$ is presentably symmetric monoidal.
Then the $n$-fold non-unital category object $\Fact^{\mathrm{cstr}}_{(\square^n_{(\_)},E_{(\_)}) }(\icat)$ is quasi-unital.
\end{prop}
\begin{proof}
Recall the notation \eqref{eqn:short_Fact}.
Following \cite{Haugseng_quasi}
we need to find a quasi-unit, that is, a map $u\colon F_0 \to F_1$ and a commutative diagram
\begin{center}
\begin{tikzcd}
 F_0 \arrow{r}{u} & F_1 \arrow[bend right, "d_0"']{l}\arrow[bend left, "d_1"]{l}
\end{tikzcd}
\end{center}
such that the composite
\[ F_1 \xleftarrow{\simeq} F_0\times_{F_0} F_1 \xrightarrow{u\times \Id} F_1\times_{F_0} F_1 \xrightarrow{\simeq} F_2 \xrightarrow{d_1} F_1\]
is equivalent to the identity, and similarly for the morphism with $u$ on the other side. Invoking the pointless additivity equivalence \ref{prop:markedadditivity}, we reduce to proving the condition in question for $ \Fact^{\mathrm{cstr}}_{(\square^1_{k},E_{[k]})}(\icat)$. Under the equivalences
\[ \Fact_{\square_{[0]}}^{\mathrm{cstr}}(\icat)\simeq \alg(\icat),\quad\quad \Fact_{(\square_{[1]},E_{[1]})}^{\mathrm{cstr}}(\icat)\simeq \mathsf{BMod}(\icat), \]
the quasi-unit map $u$ is the map induced by the forgetful functor $\mathsf{BM}^{\otimes}\rightarrow \mathsf{Assoc}^{\otimes}$ (carrying both $\mathfrak{m}$ and $\mathfrak{a}$ to $\mathfrak{a}$). Now we invoke Proposition \ref{prop:relativetensorprod} and \cite[Proposition 4.4.3.16]{LurHA}.
\end{proof}
Via the equivalence
\[ \cat_n(\icatd) \simeq \cat^{\mathrm{qu}}_n(\icatd)  \]
for any \infcat $\icatd$ that admits finite limits, we will regard $\Fact^{\mathrm{cstr}}_{\square_{(\_)}^n}(\icat)$ as an $n$-fold category object of $\catinfh$. 

\subsubsection{From category objects to Segal objects and univalent completion}
Our next goal is to extract from this $n$-fold category object an underlying $n$-fold \emph{Segal} object, following the discussion immediately below \cite[Definition 4.11]{Haugseng-iteratedspans}. Let $n>1$ and let $\icatd$ be an \infcat that admits finite limits. There is a fully faithful inclusion
\[  \mathsf{Seg}_n(\icatd)\hooklongrightarrow \mathsf{Cat}(\mathsf{Seg}_{n-1}(\icatd)).  \]
We wish to exhibit a right adjoint to this inclusion, realizing $\mathsf{Seg}_n(\icatd)$ as a colocalization of $\mathsf{Cat}(\mathsf{Seg}_{n-1}(\icatd))$. For this, we introduce some notation. Let $i_1$ denote the inclusion 
\[ i_1: \{[0]\}\times \simp^{n-1,op} \hooklongrightarrow \simp^{n,op},   \]
then the right adjoint $i_{1*}$ to the pullback $i_1^*:\fun(\simp^{n,op},\icatd)\rightarrow \fun(\simp^{n-1,op},\icatd)$ carries an $(n-1)$-fold simplicial object $Y=Y_{\bullet\ldots\bullet}$ to the $n$-fold simplicial object $i_{1*}(Y)$ satisfying 
\[  i_{1*}(Y)_{n\bullet\ldots\bullet} = Y_{\bullet\ldots\bullet}^{\times (n+1)}. \]
Now consider, for $X_{\bullet\ldots\bullet}\in \mathsf{Cat}(\mathsf{Seg}_{n-1}(\icatd))\subset \fun(\simp^{n,op},\icatd)$, the pullback diagram
\begin{equation}\label{eq:segalcoreflection}
\begin{tikzcd}
  U_{\mathsf{Seg}_n}(X) \ar[r] \ar[d] & X_{\bullet\ldots\bullet}\ar[d] \\
  i_{1*}X_{0\ldots0} \ar[r] & i_{1*}X_{0\bullet\ldots\bullet},
\end{tikzcd}
\end{equation}
where the right vertical map is the unit of the adjunction $(i_1^*\adj i_{1*})$ and the lower horizontal map is the image under $i_1$ of the degeneracy map $X_{0\ldots0}\rightarrow X_{0\bullet\ldots\bullet}$. The following is the content of \cite[Lemma 4.17]{Haugseng-iteratedspans}.
\begin{lem}
The upper horizontal map 
\[ U_{\mathsf{Seg}_n}(X) \longrightarrow X_{\bullet\ldots\bullet} \]
in the diagram \ref{eq:segalcoreflection} exhibits $U_{\mathsf{Seg}_n}(X)$ as a $\mathsf{Seg}_n(\icatd)$-colocalization of $X$. 
\end{lem}
It follows that the collection of these colocalization maps assemble into a right adjoint $U_{\mathsf{Seg}_n}$. As right adjoints preserve limits, the functor $U_{\mathsf{Seg}_n}$ also induces a right adjoint to the inclusion
\[  \cat_k(\mathsf{Seg}_n(\icatd)) \hooklongrightarrow  \cat_{k+1}(\mathsf{Seg}_{n-1}(\icatd)) \]
for all integers $k>1$. Now the inclusion $\mathsf{Seg}_n(\icatd)\subset \mathsf{Cat_n}(\icatd)$ admits a right adjoint $U_{\mathsf{Seg}}^n$ by composing the right adjoints of the composite of inclusions
\[ \mathsf{Seg}_n(\icatd)\hooklongrightarrow \cat(\mathsf{Seg}_{n-1}(\icatd)) \hooklongrightarrow   \cat_2(\mathsf{Seg}_{n-2}(\icatd)) \hooklongrightarrow \ldots\hooklongrightarrow \mathsf{Cat_n}(\icatd).   \]
\begin{rmk}
The inclusion $\mathsf{Seg}_n(\icatd)\subset \mathsf{Cat_n}(\icatd)$ and its adjoint $U^n_{\mathsf{Seg}}:\mathsf{Cat_n}(\icatd)\rightarrow \mathsf{Seg}_n(\icatd)$ both preserve products; it follows that this adjunction determines an adjunction between symmetric monoidal $n$-fold Segal objects and symmetric monoidal $n$-fold category objects.
\end{rmk}
\begin{rmk}
Consider the case $n=2$. The \infcat $\Fact^{\mathrm{cstr}}_{(\square^2_{[1,1]},E_{[1,1]})}$ of $(1,1)$-morphisms of the $2$-fold category object $\Fact^{\mathrm{cstr}}_{(\square^2_{(\_)},E_{(\_)})}$ is the \infcat of constructible factorization algebras on the picture on the left below.

\vspace{0.3cm}
{\centering \begin{tikzpicture} 
\filldraw[opacity=0.1] (1,0) -- (0,0) -- (0,2)--(1,2);
\filldraw[opacity=0.1] (1,0) -- (2,0) -- (2,2)--(1,2);
\draw[thick] (1,0) -- (1,2);
\draw[thick] (0,1) -- (2,1);
\fill (1,1) circle (0.06);
\end{tikzpicture} \hspace{1.5cm}\begin{tikzpicture} 
\filldraw[opacity=0.1] (1,0) -- (0,0) -- (0,2)--(1,2);
\filldraw[opacity=0.1] (1,0) -- (2,0) -- (2,2)--(1,2);
\draw[thick] (1,0) -- (1,2);
\fill (1,1) circle (0.06);
\end{tikzpicture} \par  }
\vspace{0.2cm}
\noindent 
The \infcat of 2-morphisms of $U^2_{\mathsf{Seg}}\left(\Fact^{\mathrm{cstr}}_{(\square^2_{(\_)},E_{(\_)})}\right)$, that is, the pullback 
\[  \Fact^{\mathrm{cstr}}_{\square^2_{[1,1]}}(\icat)\times_{\Fact^{\mathrm{cstr}}_{\square^2_{[0,1]}}(\icat)^{\times 2}}\Fact^{\mathrm{cstr}}_{\square^2_{[0,0]}}(\icat)^{\times 2}\hooklongrightarrow \Fact^{\mathrm{cstr}}_{\square^2_{[1,1]}}(\icat) \]
may be considered as the subcategory of $\Fact^{\mathrm{cstr}}_{\square^2_{[1,1]}}(\icat)$ of pointless factorization algebras that are constructible with respect to the coarser stratification depicted on the right. We could have directly defined an $n$-fold Segal object in $\catinfh$ (instead of an $n$-fold category object) by considering pointless constructible factorization algebras on generalizations of the illustration on the right (that is, stratifications of the cube by \emph{affine flags}), and this $n$-fold Segal object would have been equivalent to $U^n_{\mathsf{Seg}}\left(\Fact^{\mathrm{cstr}}_{(\square^n_{(\_)},E_{(\_)})}(\icat)\right)$.
This would have been closer to the original pointed construction in \cite{Scheimbauer}.
\end{rmk}

Now we turn to \emph{univalent completion}: recall that the inclusion 
\[ \mathsf{Gpd}(\spa)\subset\mathsf{Seg}(\spa)  \]
of groupoid objects into category objects in spaces (that is, Segal \emph{spaces}) admits a right adjoint $\iota$; a Segal space $X_{\bullet}$ is \emph{complete} if $\iota X_{\bullet}$ is (essentially) constant. The inclusion 
\[ \mathsf{CSS}(\spa)\subset \mathsf{Seg}(\spa)  \]
of complete Segal spaces is a strongly reflective subcategory; we let $L_{\mathsf{CSS}}$ denote the left adjoint, the \emph{univalent completion}. We similarly have the full subcategory 
\[ \mathsf{CSS}_n(\spa)\subset \mathsf{Seg}_n(\spa)  \]
of \emph{complete $n$-fold Segal spaces} admitting the following equivalent characterizations: an $n$-fold Segal space $\icate=\icate_{\bullet\ldots\bullet}$ is \emph{complete} if 
\begin{enumerate}
    \item The $n$ Segal spaces $\icate_{\bullet0\ldots 0}$, $\icate_{1\bullet 0\ldots 0}$, ..., $\icate_{1\ldots 1\bullet}$ are all complete. 
    \item The Segal space $\icate_{\bullet 0\ldots 0}$ is complete, and for any pair of objects $X,Y\in\icate$, the $(n-1)$-fold Segal spaces $\Hom^1_{\icate}(X,Y)$ are complete.  
\end{enumerate}
Note that the second condition is inductive. The equivalence of the two conditions can be found, for instance, in \cite[Lemma 2.8]{JFS} or \cite[Section 7]{Haugseng-iteratedspans}. The inclusion $ \mathsf{CSS}_n(\spa)\subset \mathsf{Seg}_n(\spa) $ is also a strongly reflective localization; we let $L_{\mathsf{CSS}_n}$ denote the left adjoint.
\begin{rmk}
The univalent completion is compatible with the formation of mapping objects, in the following sense: let $\icate$ be an $n$-fold Segal space and let $X,Y$ be parallel pair of $(k-1)$-morphisms for $1\leq k\leq n$. Then the canonical map
\[ L_{\mathsf{CSS}_{n-k}}\left(\Hom_{\icate}^k(X,Y)\right)  \longrightarrow \Hom^k_{L_{\mathsf{CSS}_n}(\icate)}\left(X,Y\right)  \]
is an equivalence of (complete) $(n-k)$-fold Segal spaces (where on the right hand side we abuse notation and identify $X$ and $Y$ with their image under the completion unit map $\icate\rightarrow L_{\mathsf{CSS}_n}(\icate)$). For a proof, see e.g. \cite[Lemma 5.50]{HaugsengEn}.
\end{rmk}
\begin{rmk}
The inclusion $\mathsf{CSS}_n(\spa)\subset \mathsf{Seg}_n(\spa)$ and its adjoint $L_{\mathsf{CSS}_n}$ both preserve products; it follows that this adjunction determines an adjunction between symmetric monoidal $n$-fold Segal objects and symmetric monoidal complete $n$-fold category objects. 
\end{rmk}
\begin{rmk}
As proven in \cite{BarSP-unicity}, we are justified in writing $\mathsf{Cat}_{(\infty,n)}$ for the \infcat $\mathsf{CSS}_n(\spa)$.
\end{rmk}
\begin{defn}[The higher Morita category]\label{defn:moritacategory}
Let $\icat^{\otimes}$ be a presentably symmetric monoidal \infcatt, and consider the $n$-fold category object \[\Fact^{\mathrm{cstr}}_{(\square^n_{(\_)},E_{(\_)})} \in \cat_n(\catinfh) \simeq \cat_n(\mathsf{CSS}(\widehat{\spa}))\subset   \cat_{n+1}(\widehat{\spa}), \]
which we may consider as an $(n+1)$-fold category object of $\widehat{\spa}$, the \infcat of large spaces. We will write 
\[  \mor_n(\icat) := U^{n+1}_{\mathsf{Seg}}\left(\Fact^{\mathrm{cstr}}_{(\square^n_{(\_)},E_{(\_)})}\right)\in \mathsf{Seg}_n\left(\widehat{\spa}\right) \]
for the coreflection onto $(n+1)$-fold Segal spaces. We will write
\[  \umor_n(\icat) := L_{\mathsf{CSS}_{n+1}}\left(\mor_n(\icat)\right) \subset \mathsf{CSS}_{n+1}\left(\widehat{\spa}\right) = \widehat{\cat}_{(\infty,n+1)} \]
which comes equipped with a completion map 
\[  \mor_n(\icat)\longrightarrow \umor_n(\icat) \]
of $(n+1)$-fold Segal spaces. Since both univalent completion and the coreflection $U^{n+1}_{\mathsf{Seg}}$ preserve products and therefore (commutative) monoid objects, the functor \eqref{eq:symmonenhance2} endows the $(n+1)$-fold Segal object $\mor_n(\icat)$ and the $(\infty,n+1)$-category $\umor_n(\icat)$ with a symmetric monoidal structure, and the completion map 
\[  \mor_n(\icat)\longrightarrow \umor_n(\icat) \]
canonically enhances to a map of symmetric monoidal $(n+1)$-fold Segal objects.
\end{defn}

\begin{rmk}
As we are interested in dualizability and invertibility, univalent completeness is not a particularly relevant property for us in this work: the existence of duals and adjoints can be detected before univalent completion. More precisely, the following hold true.
\begin{enumerate}[$(1)$]
    \item An object of $\mor_n(\icat)$ admits a dual if and only if its image in $\umor_n(\icat)$ admits a dual.
    \item A $k$-morphism of $\mor_n(\icat)$ for $1\leq k\leq n$ admits an adjoint if and only if its image in $\umor_n(\icat)$ admits an adjoint.
    \item Every object and every $k$-morphism for $1\leq k\leq n$ of $\umor_n(\icat)$ is equivalent to one such in the image of the completion map $\mor_n(\icat)\rightarrow \umor_n(\icat)$.
\end{enumerate}
In particular, while the objects of $\umor_n(\icat)$ are not locally constant factorization algebras on $\R^n$ on the nose like those of $\mor_n(\icat)$, we may always find an equivalence to an object in the image of the completion map $\mor_n(\icat)\rightarrow \umor_n(\icat)$. The assertions $(1)$ and $(2)$ are an immediate consequence of the fact that duals and adjoints are detected in various homotopy 2-categories, and the formation of homotopy 2-categories is an operation on 2-fold Segal spaces insensitive to univalent completion. 
\end{rmk}

\begin{rmk}
As is evident from its construction, the $1\leq k\leq n$-morphisms and the $(n+1)$-morphisms of $\mor_n(\icat)$ are of quite different nature. We will make use of this distinction in Section \ref{sec:reg-bimod} on regular bimodules to come; there, it will be essential to \emph{not} colocalize all the way to $(n+1)$-fold Segal spaces, but to remember the extra structure afforded by the presence of maps among constructible factorization algebras that are allowed to be noninvertible away from the 0-dimensional strata. 
\end{rmk}

\begin{var}[The pointed higher Morita category]
The preceding constructions and results can evidently be emulated with constructible factorization algebras on the stratified cubes \emph{without markings}. This yields an $n$-fold category object $\Fact^{\mathrm{cstr}}_{\square^n_{(\_)}}$, together with a natural transformation  
\[ \Fact^{\mathrm{cstr}}_{\square^n_{(\_)}}\longrightarrow  \Fact^{\mathrm{cstr}}_{(\square^n_{(\_)},E_{(\_)})}.  \]    
We write $\mor^{\mathrm{ptd}}_n(\icat)$ for the $(n+1)$-fold Segal space obtained from $\Fact^{\mathrm{cstr}}_{\square^n_{(\_)}}$ and $\umor_n^{\mathrm{ptd}}(\icat)$ for its univalent completion. In this Morita category, all $n$-morphisms are bimodules that come equipped with pointings (and $(n+1)$-morphisms preserve them). The natural transformation above induces a functor 
\[ \umor^{\mathrm{ptd}}_n(\icat)\longrightarrow\umor_n(\icat)  \]
of symmetric monoidal $(\infty,n+1)$-categories that forgets the pointing.
\end{var}
\begin{rmk}
The construction of $\Fact^{\mathrm{cstr}}_{(\square^n_{(\_)},E_{(\_)}) }(\icat)$ is clearly functorial in $\icat^{\otimes}$, that is, we have a functor
\[\Fact^{\mathrm{cstr}}_{(\square^n_{(\_)},E_{(\_)}) }(\_): \simpopnin\times  \calg(\prl)\longrightarrow \catinfh\left(\mathcal{K}\right) \]
which, by Proposition \ref{prop:moritasegal} is adjoint to a functor
\[  \Fact^{\mathrm{cstr}}_{(\square^n_{(\_)},E_{(\_)}) }(\_):\calg(\prl)\longrightarrow \cat^{\mathrm{nu}}_n(\mathcal{K}) \]
taking values in the \infcat of $n$-fold non-unital category objects in $\catinfh\left(\mathcal{K}\right)$. 
It is immediate from the description of the quasi-units of Proposition \ref{prop:moritaquasiunit} that for a symmetric monoidal functor $\icat^{\otimes}\rightarrow\icatd^{\otimes}$, the induced natural transformation $\Fact^{\mathrm{cstr}}_{(\square^n_{(\_)},E_{(\_)}) }(\icat)\rightarrow \Fact^{\mathrm{cstr}}_{(\square^n_{(\_)},E_{(\_)}) }(\icatd)$ is a quasi-unital map, so we have a unique lift to a functor
\[  \Fact^{\mathrm{cstr}}_{(\square^n_{(\_)},E_{(\_)}) }(\_): \calg(\prl)\longrightarrow \cat_n(\catinfh\left(\mathcal{K}\right)).  \] 
Upon coreflecting to $(n+1)$-fold Segal spaces and univalent completion, we obtain a functor
\[ \umor_n(\_): \calg(\prl)\longrightarrow \widehat{\mathsf{Cat}}_{(\infty,n)}, \]
in fact (taking into account symmetric monoidal structures), a functor
\[ \umor_n(\_)^{\otimes}: \calg(\prl)\longrightarrow \calg\left(\widehat{\mathsf{Cat}}_{(\infty,n)}\right), \]
\end{rmk}

\subsubsection{Comparison functors}
Let $\icat^{\otimes}$ be a presentably symmetric monoidal \infcat and let $X,Y$ be a parallel pair of $(k-1)$-morphisms in $\umor_n(\icat)$ for $1\leq k\leq n-1$, so that we may consider the $(\infty,n-k+1)$-category $\Hom^{k}_{\umor_n(\icat)}(X,Y)$ of $k$-morphisms. There is a natural symmetric monoidal comparison functor 
\begin{equation}\label{eq:comparisonfunctor}  \Hom^{k}_{\umor_n(\icat)}(X,Y)\longrightarrow \umor_{n-k}(\icat)   \end{equation}
corresponding to only remembering the deepest stratum, so that every $k$-morphism of $\umor_n(\icat)$ determines an \emph{underlying} object of in $\umor_{n-k}(\icat)$, and each $(k+1)$-morphism determines an \emph{underlying} $1$-morphisms of $\umor_{n-k}(\icat)$, and so on. In Section \ref{sec:lifting_lemma}, we will show that these functors detect adjointability. We now construct these comparison functors.

\begin{cons}
Consider the map of posets 
\[  \overline{\gamma}:\overline{P}_{[1]}\longrightarrow \overline{P}_{[1]} \]
that carries all of $\overline{P}_{[1]}$ into $\mathfrak{m}_{01}$. This is a depth-and orientation-preserving (but not boundary-preserving) map of posets, so by Lemma \ref{lem:collaplsrescale}, there is a unique collapse-rescale map $\collresc(\gamma)$ associated to it; this is simply the map 
\[ \collresc(\overline{\gamma}): \overline{\square}_{[1]} \longrightarrow \{1/2\} \subset \overline{\square}_{[1]}. \]
We will write $\collresc(\gamma)$ for the restriction of $\collresc(\overline{\gamma})$ to $\square_{[1]}$. Taking the $k$-fold product, we have a collection of continuous stratified map
\[ \pi:\square_{[1,\ldots,1,\bullet,\ldots,\bullet]}\longrightarrow \square_{[\bullet,\ldots,\bullet]} \]
compatible with the collapse-rescale maps in the last $(n-k)$-coordinates, that is, a natural transformation of semisimplical stratified spaces. We obtain natural transformation
\[ \pi_*:\Fact_{(\square^n_{[1,\ldots,1,\bullet,\ldots,\bullet]},E)}^{\mathrm{cstr}}(\icat) \longrightarrow \Fact^{\mathrm{cstr}}_{(\square^{n-k}_{[\bullet,\ldots,\bullet]},E)}(\icat), \]
a map of $(n-k)$-fold category objects of $\catinfh$. For any parallel pair of $(k-1)$-morphisms $(X,Y)$, the composite
\[ \Hom^k_{\mor_n(\icat)}(X,Y)\subset  \Fact_{(\square^n_{[1,\ldots,1,\bullet,\ldots,\bullet]},E)}^{\mathrm{cstr}}(\icat) \longrightarrow \Fact^{\mathrm{cstr}}_{(\square^{n-k}_{[\bullet,\ldots,\bullet]},E)}(\icat) \]
(viewed as a map of $(n-k+1)$-fold category objects of $\widehat{\spa}$) factors through $\mor_{n-k}(\icat)$ which yields the functor 
\[  \pi_*: \Hom^{k}_{\mor_n(\icat)}(X,Y)\longrightarrow \mor_{n-k}(\icat) \]
The functor \eqref{eq:comparisonfunctor} is the univalent completion of this functor. As the symmetric monoidal structure on \infcats of factorization algebras of \eqref{eq:symmonenhance2} is compatible with the collapse-rescale maps, the functor $\pi_*$ is naturally symmetric monoidal.
\end{cons}

\subsubsection{Regular bimodules}\label{sec:reg-bimod}
Let $\icat^{\otimes}$ be a symmetric monoidal category. Suppose we are given maps $f:A\rightarrow A'$ and $g:B\rightarrow B'$ of associative algebra objects in $\icat$. Then we have a functor
\begin{equation}\label{eq:bimodpullback}
     {}_{A'}\mathsf{BMod}_{B'}(\icat) \longrightarrow {}_A\mathsf{BMod}_B(\icat)
\end{equation}    
that carries an $A'$-$B'$-bimodule $M$ to the following  $A$-$B$-bimodule. Its  underlying object in $\icat$ is also $M$ with action map given by (suppressing the associators of the tensor product)
\[ A\otimes M\otimes B \xrightarrow{f\otimes\mathrm{id}_M\otimes{g}} A'\otimes M\otimes B'\longrightarrow M ,\]
where the second map is the action of $A'$ and $B'$ on $M$. This assignment is manifestly compatible with composition of maps of associative algebras, that is, there is a functor 
\begin{equation}\label{eq:bimodfunctor} \mathsf{Alg}(\icat)^{op}\times\mathsf{Alg}(\icat)^{op}\longrightarrow \mathsf{Cat},\quad \quad (A,B)\longmapsto {}_A\mathsf{BMod}_B(\icat).  \end{equation}
carrying the pair of maps $(f,g)$ to the map \eqref{eq:bimodpullback}.
\begin{rmk}
It is not so straightforward to establish the $\infty$-\emph{categorical} analogue of this functor directly. However, via the Grothendieck construction, the existence of the functor \eqref{eq:bimodfunctor} (more precisely, the pseudofunctor associated to this functor) can equivalently be expressed as follows.
\begin{enumerate}[$(1)$]
    \item The forgetful functor 
    \[ p:\mathsf{BMod}(\icat)\longrightarrow \mathsf{Alg}(\icat)\times \mathsf{Alg}(\icat),\quad\quad (A,M,B)\longmapsto (A,B) \]
    is a \emph{Cartesian fibration}.
    \item A map $(A,N,B)\rightarrow (A',M,B')$ is $p$-Cartesian just in case the underlying map $N\rightarrow M$ in $\icat$ is an isomorphism. 
\end{enumerate}
Replacing $\icat^{\otimes}$ with a symmetric monoidal \infcatt, the assertions $(1)$ and $(2)$ above are established by \cite[Corollary 4.3.3.3]{LurHA}, so that unstraightening provides the $\infty$-categorical version of the functor \eqref{eq:bimodfunctor}.
\end{rmk}
The functor ${}_{A'}\mathsf{BMod}_{B}(\icat) \rightarrow {}_A\mathsf{BMod}_B(\icat) $ induced by a map $f:A\rightarrow A'$ can be interpreted as a map 
\[  \Hom^1_{\mor_1(\icat)}(A',B)\longrightarrow \Hom^1_{\mor_1(\icat)}(A,B) \]
of categories of 1-morphisms in the (valent) Morita $2$-category. In fact, this map is the composition $\_\circ M_f$ with some 1-morphism $M_f\in \Hom^1_{\mor_1(\icat)}(A,A')$, that is, some bimodule $M_f\in {}_A\mathsf{BMod}_{A'}(\icat)$. It is not hard to give a description of this 1-morphism.

\begin{defn}[Regular bimodules]
For $C$ an associative algebra in a symmetric monoidal category $\icat$, will write ${}_CC_C$ for $C$ regarded as a $C$-$C$-bimodule, that is, the image of $C$ under the functor $\mathsf{Alg}(\icat)\rightarrow \mathsf{BMod}(\icat)$ induced by the forgetful map $\mathsf{BM}^{\otimes}\rightarrow \mathsf{Assoc}^{\otimes}$ of operads. For any pair of maps $f:C\rightarrow D$, $g:C\rightarrow E$ of algebra objects, we will write ${}_DC_{E}$ for the image of ${}_CC_C$ under the functor ${}_C\mathsf{BMod}_{C}(\icat)\rightarrow {}_D\mathsf{BMod}_{E}(\icat)$ induced by the pair $(f,g)$ via the assignment \eqref{eq:bimodfunctor}. For the case of $g$ being the identity $C=E$, we will say that ${}_D{C}_C$ is the \emph{right regular bimodule} associated to $f$. For the case of $f$ being the identity $C=D$, we will say that ${}_CC_E$ is the \emph{left regular bimodule} associated to $g$.
\end{defn}
Now the desired map $\Hom^1_{\mor_1(\icat)}(A',B)\rightarrow \Hom^1_{\mor_1(\icat)}(A,B)$ of categories of 1-morphisms in the Morita 2-category is given by composing with the right regular bimodule ${}_AA'_{A'}$ associated to $f:A\rightarrow A'$. The right and left regular bimodules comprise a very special class of morphisms in the Morita category: they are adjoint to one another. More explicitly, the canonical map 
\[  {}_{A'}A'_{A}\otimes_{A}{}_{A}A'_{A'}\longrightarrow {}_{A'}A'_{A'}\otimes_{A'} {}_{A'}A'_{A'}\simeq A' \]
of $A'$-$A'$-bimodules exhibits ${}_{A'}A'_A$ as right adjoint to ${}_AA'_{A'}$ (so the \emph{left} regular bimodule is the \emph{right} adjoint and vice versa\footnote{Recall that the tensor product $\otimes$ composes modules from left to right, while the composition symbol $\circ$ composes from right to left.}).

From now on $\icat^{\otimes}$ will be a presentably symmetric monoidal \infcatt. Our goal in this section is to establish the appropriate analogues of these results for $\bb{E}_n$-algebras within our factorization model for the $\bb{E}_n$-Morita category valued in $\icat$. The existence and formal properties of these regular factorization bimodules feature prominently in our discussion of adjointable morphisms to come. We start by constructing the functor \eqref{eq:bimodfunctor} in its fibrational form for all categorical levels of the higher Morita category. For this, it will be crucial to remember the maps of constructible factorization algebras that are noninvertible away from the 0-dimensional strata.
\begin{nota}
In the remainder of this section, \emph{and in this section alone}, we will write $\mor_n(\icat)$ for the $n$-fold Segal object of $\catinfh$ defined by
\[ U^n_{\mathsf{Seg}}\left(\Fact^{\mathrm{cstr}}_{(\square^n_{(\_)},E_{(\_)})}\right).  \]
\end{nota}
\begin{nota}
Let $\icatd$ be an $n$-fold Segal object in $\catinf$. We will write $\mathsf{Ob}(\icatd)$ for the \infcat $\icatd_{0\ldots 0}$ of objects of $\icatd_{\bullet\ldots\bullet}$. For $k\geq 1$, we will write $\mathsf{Ar}_k(\icatd)$ for the \infcat $\icatd_{k0\ldots 0}$ of composable sequences of 1-morphisms of $\icatd$ of length $k$.
\end{nota}
In the following proposition and below, we will for the sake of notational brevity speak of `$(-1)$-morphisms', and write $\Hom^0_{\mor_n(R,S)}(R,S)$ for the $n$-fold Segal object $\mor_{n}(\icat)$ itself, when $R,S$ are `$(-1)$-morphisms'. 
\begin{prop}\label{prop:sourcetargetCartesian}
For $0\leq k\leq n-1$, let $R,S$ be a parallel pair of $(k-1)$-morphisms in $\mor_n(\icat)$, so that we may consider the $(n-k)$-fold Segal object $\Hom^k_{\mor_n(\icat)}(R,S)$ of $k$-morphisms between $R$ and $S$ in $\catinfh$. Then the forgetful functor 
\[   \ev_{\{0,1\}}:\mathsf{Ar}_1\left(\Hom^k_{\mor_n(\icat)}(R,S)\right)\longrightarrow \mathsf{Ob}\left(\Hom^k_{\mor_n(\icat)}(R,S)\right)\times \mathsf{Ob}\left(\Hom^k_{\mor_n(\icat)}(R,S)\right) \]
evaluating source and target is a Cartesian fibration. Moreover, a morphism is $ \ev_{\{0,1\}}$-Cartesian just in case the underlying morphism of $\icat$ (obtained by evaluating at a disk-like open subcube containing the deepest stratum) is an equivalence.
\end{prop}
\begin{rmk}
In the statement of the proposition above, we note that we have subcategory inclusion 
\[\mathsf{Ar}_1\left(\Hom^k_{\mor_n(\icat)}(R,S)\right)\subset \Fact^{\mathrm{cstr}}_{(\square^n_{[1\ldots 10\ldots 0]},E)}(\icat)\] 
(for $E$ the maximal marking, which is only nonempty if there are no 0's, that is, if $k=n$), and there is a unique isotopy equivalence class of disk-like open subcubes containing the deepest stratum. Evaluating on any open subcube in this equivalence class yields the functor 
\[ \mathsf{Ar}_1\left(\Hom^k_{\mor_n(\icat)}(R,S)\right)\subset \Fact^{\mathrm{cstr}}_{(\square^n_{[1\ldots 10\ldots 0]},E)}(\icat)\longrightarrow \icat\]    
that detects the Cartesian morphisms in the proposition above.
\end{rmk}
We will prove this momentarily. Observe that the fiber of the functor $\ev_{\{0,1\}}$ in the statement of the proposition above at a pair $(A,B)$ is the \infcat $\mathsf{Ob}\left(\Hom^{k+1}_{\mor_n(\icat)}(A,B)\right)$ of $(k+1)$-morphisms from $A$ to $B$. Unstraightening the Cartesian fibration of Proposition \ref{prop:sourcetargetCartesian} then yields a functor
\[ \mathsf{Ob}\left(\Hom^{k+1}_{\mor_n(\icat)}(\_,\_)\right):\mathsf{Ob}\left(\Hom^k_{\mor_n(\icat)}(R,S)\right)^{op}\times \mathsf{Ob}\left(\Hom^k_{\mor_n(\icat)}(R,S)\right)^{op} \longrightarrow \catinfh.   \]
Let $f:A\rightarrow B$ be a morphism of $\mathsf{Ob}\left(\Hom^k_{\mor_n(\icat)}(R,S)\right)$ so that we have Cartesian transport functors 
\[  (f,\mathrm{id}_B)_*:\mathsf{Ob}\left(\Hom^{k+1}_{\mor_n(\icat)}(B,B)\right)\longrightarrow  \mathsf{Ob}\left(\Hom^{k+1}_{\mor_n(\icat)}(A,B)\right)\]
and 
\[  (\mathrm{id}_B,f)_*:\mathsf{Ob}\left(\Hom^{k+1}_{\mor_n(\icat)}(B,B)\right) \longrightarrow \mathsf{Ob}\left(\Hom^{k+1}_{\mor_n(\icat)}(B,A)\right). \]
\begin{defn}
We write ${}_AB_{B}$ for the image of the identity on $B$ under the functor $(f,\mathrm{id}_B)_*$ and we write ${}_BB_A$ for the image of the identity on $B$ under the functor $(\mathrm{id}_B,f)_*$. The $(k+1)$-morphism ${}_AB_{B}$ is the \emph{right regular (factorization) bimodule} associated to $f$ and the $(k+1)$-morphism ${}_BB_A$ is the \emph{left regular (factorization) bimodule} associated to $f$.
\end{defn}
We now state the main result of this section on the formal properties of regular bimodules. Just below, we give a more digestible summary of what is asserted in this proposition.
\begin{prop}\label{prop:regularbimodajoint}
Let $0\leq k\leq n-1$ and let $R,S$ be a parallel pair of $(k-1)$-morphisms of $\mor_n(\icat)$. Let $f:A\rightarrow B$ be a map of $\mathsf{Ob}\left(\Hom^k_{\mor_n(\icat)}(R,S)\right)$, then the following hold true. 
\begin{enumerate}[$(1)$]
    \item For any $Z\in \Hom_{\mor_n(\icat)}^k(R,S)$, the Cartesian transport map 
    \[  (f,\mathrm{id}_Z)_*:\mathsf{Ob}\left(\Hom^{k+1}_{\mor_n}(\icat)(B,Z)\right) \longrightarrow \mathsf{Ob}\left(\Hom^{k+1}_{\mor_n}(\icat)(A,Z)\right)  \]
    is equivalent to precomposition with ${}_AB_B$, and the Cartesian transport map 
    \[  (\mathrm{id}_Z,f)_*: \mathsf{Ob}\left(\Hom^{k+1}_{\mor_n}(\icat)(Z,B)\right) \longrightarrow \mathsf{Ob}\left(\Hom^{k+1}_{\mor_n}(\icat)(Z,A)\right)  \]
    is equivalent to postcomposition with ${}_BB_A$.
    \item If $g:B\rightarrow C$ be another map of $\mathsf{Ob}\left(\Hom^k_{\mor_n(\icat)}(R,S)\right)$, then the composite $ {}_BC_{C}\circ {}_AB_B$ in $\mor_n(\icat)$ is equivalent to the right regular bimodule ${}_AC_{C}$ associated with $g\circ f$. Similarly, if $h:D\rightarrow A$ is a map of $\mathsf{Ob}\left(\Hom^k_{\mor_n(\icat)}(R,S)\right)$, then the composite $ {}_AA_{D}\circ {}_BB_A$ in $\mor_n(\icat)$ is equivalent to the left regular bimodule ${}_DD_{B}$ associated with $f\circ h$.  
    \item Let $T$ be another $(k-1)$-morphism of $\mor_n(\icat)$ parallel to $R$ and $S$ and let $f':A'\rightarrow B'$ be a morphism of $\mathsf{Ob}\left(\Hom^{k}_{\mor_n(\icat)}(S,T)\right)$, so that we may consider the composite $f'\circ f\in\mathsf{Ob}\left(\Hom^{k}_{\mor_n(\icat)}(R,T)\right)$. Then the horizontal composition functor
    \[  \circ_h: \mathsf{Ar}_1\left(\Hom^{k}_{\mor_n(\icat)}(R,S)\right) \times  \mathsf{Ar}_1\left(\Hom^{k}_{\mor_n(\icat)}(S,T)\right)\longrightarrow  \mathsf{Ar}_1\left(\Hom^{k}_{\mor_n(\icat)}(R,T)\right)  \]
    carries the pair $({}_AA_B,{}_{A'}B'_{B'})$ of right regular bimodules to a morphism equivalent to the right regular bimodule of the composite $f'\circ f$. Similarly, the horizontal composition ${}_BB_A\circ_h {}_{B'}B'_{A'}$ of the left regular bimodules is equivalent to the left regular bimodule of the composite $f'\circ f$.
    \item The $(k+1)$-morphism ${}_AB_{B}$ is left adjoint to ${}_BB_{A}$ in $\mor_n(\icat)$. 
    \item Let $\eta_f:\mathrm{id}_A={}_BB_{B}\rightarrow {}_BB_{A}\circ {}_AB_B$ and $\epsilon_f:{}_AB_B\circ {}_BB_{A}\rightarrow \mathrm{id}_B$ be the unit respectively the counit $(k+2)$-morphisms of the adjunction of $(2)$. Then there are maps \[\widehat{\eta}_f:{}_AA_{A}\longrightarrow {}_BB_{A}\circ {}_AB_B\in \mathsf{Ob}\left(\Hom^{k+1}_{\mor_n(\icat)}(A,A)\right)\]
    and
    \[\widehat{\epsilon}_{f}:{}_AB_B\circ {}_BB_{A}\longrightarrow {}_BB_B\in\mathsf{Ob}\left(\Hom^{k+1}_{\mor_n(\icat)}(B,B)\right) \] 
    such that $\eta_f$ is equivalent to the right regular bimodule of $\widehat{\eta}_f$ and $\epsilon_f$ is equivalent to the right regular bimodule of $\widehat{\epsilon}_f$. Moreover, the maps $\widehat{\eta}_f$ and $\widehat{\epsilon}_{f}$ satisfy the following snake identities: the composite
    \[  {}_AB_B \xrightarrow{\mathrm{id}_{{}_AB_B}\circ\widehat{\eta}_f}  {}_AB_B\circ {}_BB_A\circ  {}_AB_B  \xrightarrow{\widehat{\epsilon}_f\circ\mathrm{id}_{{}_AB_B}} {}_AB_B  \]
    in $\mathsf{Ob}\left(\Hom^{k+1}_{\mor_n(\icat)}(A,B)\right)$ is equivalent to the identity, and the composite 
    \[  {}_BB_A \xrightarrow{\widehat{\eta}_f\circ \mathrm{id}_{{}_BB_A}}  {}_BB_A\circ {}_AB_B\circ  {}_BB_A  \xrightarrow{\mathrm{id}_{{}_BB_A}\circ\widehat{\epsilon}_f} {}_BB_A  \]
    is equivalent to the identity in $\mathsf{Ob}\left(\Hom^{k+1}_{\mor_n(\icat)}(B,A)\right)$. 
\end{enumerate}
\end{prop}
\begin{rmk}\label{rmk:regularbimodfunctor}
The assertion $(2)$ states that the formation of regular bimodules is compatible with vertical composition. This is a shadow of the following fact: there are functors
\[ \mathsf{Ob}\left(\Hom^k_{\mor_n(\icat)}(R,S)\right)\longrightarrow  \Hom^k_{\mor_n(\icat)}(R,S),\quad   \mathsf{Ob}\left(\Hom^k_{\mor_n(\icat)}(R,S)\right)^{op}\longrightarrow  \Hom^k_{\mor_n(\icat)}(R,S)\]
of $(n-k)$-fold Segal objects in $\catinfh$ (where the domains are considered as Segal objects in $\spa$) that are the identity on objects and carry maps to right respectively left regular bimodules. We do not need the full strength of this assertion here, so we do not prove it.
\end{rmk}
\begin{rmk}
The assertion $(3)$ states that the formation of regular bimodules is compatible with horizontal composition.     
\end{rmk}
\begin{rmk}
The assertions $(4)$ and $(5)$ can be summarized as follows: the right regular bimodule ${}_AB_B$ of any map $f:A\rightarrow B$ of $\bb{E}_k$-algebras internal to $R$-$S$-bimodules has a right adjoint given by the left regular bimodule ${}_BB_A$. Moreover, the unit and counit of the adjunction are not merely $\bb{E}_{k+2}$-algebras internal to $B$-$B$-bimodules respectively $A$-$A$-bimodules, but right regular bimodules associated to maps of $\bb{E}_{k+1}$-algebras (internal to $B$-$B$-bimodules respectively $A$-$A$-bimodules), and the snake identities are satisfied already at the level of maps of $\bb{E}_{k+1}$-algebras.
\end{rmk}
Proposition \ref{prop:regularbimodajoint} and the following corollaries will be essential ingredients in the proof of our main theorem, specifically in the proof of the lifting-of-adjoints argument of Section \ref{sec:lifting_lemma}.
\begin{cor}\label{cor:regbimodconservative}
Let $f:A\rightarrow B$ be a map of $\mathsf{Ob}\left(\Hom^k_{\mor_n(\icat)}(R,S)\right)$. Then precomposing with the right regular bimodule ${}_AB_B$ is a conservative functor and postcomposing with the left regular bimodule ${}_BB_A$ is a conservative functor.
\end{cor}
\begin{proof}
According to $(1)$ of Proposition \ref{prop:regularbimodajoint}, composition with ${}_AB_B$ can be identified with the Cartesian transport $(f,\mathrm{id})_*$. Since the $\ev_{\{0,1\}}$-Cartesian morphisms of $\mathsf{Ar}_1\left(\Hom^k_{\mor_n(\icat)}(R,S)\right)$ are precisely those that become an equivalence in $\icat$ via evaluating at disk-like open subcubes containing the deepest stratum, the Cartesian transport commutes with evaluation at a disk-like cube containing the deepest stratum. Since this evaluation functor $\mathsf{Ob}\left(\Hom^{k+1}_{\mor_n(\icat)}(X,Y)\right)\rightarrow \icat$ is conservative for any pair $X,Y$, we conclude that $(f,\mathrm{id})_*$ is conservative. The proof for the left regular bimodule is similar. 
 \end{proof}
\begin{cor}\label{cor:regular-bimod-dualizability}
Let $f:A\rightarrow B$ be a map of $\mathsf{Ob}\left(\Hom^k_{\mor_n(\icat)}(R,S)\right)$. Then the right regular bimodule ${}_AB_B$ is $(n-k+1)$-times right adjointable in $\mor_n(\icat)$.     
\end{cor}
\begin{proof}
It follows from $(4)$ of Proposition \ref{prop:regularbimodajoint} that the right regular bimodule is right adjointable. It follows from $(5)$ that the units and counits are again right regular bimodules. We conclude by applying $(4)$ and $(5)$ repeatedly, increasing the categorical level of the units and counits until we have a collection of units and counits that are $(n+1)$-morphisms. 
\end{proof}
\begin{rmk}
Note that in \cite{GS}, it is proven by geometric manipulations that \emph{any} $k$-morphism in $\mor_n(\icat)$ is $(n-k)$-times right \emph{and} left adjointable. For instance, a 1-morphism $Q:C\rightarrow D$ of $\mor_2(\icat)$ (that is, a $\bb{E}_1$-algebra internal to $C$-$D$-bimodules) has a right adjoint, and this right adjoint is informally given by the \emph{opposite} $\bb{E}_1$-algebra $Q^{rev}$ (see the recollection of Section \ref{sec:n-duals}). This may appear to be in conflict with what is asserted above regarding the right adjoint to the right regular bimodule ${}_AB_B$ being $B$ with $B$ acting on itself on the left and $A$ via $f:A\rightarrow B$ on the right, but this apparent discrepancy is resolved by noting that $B$ is equivalent to $B^{rev}$ (in general in $\Z$ non-equivalent ways) since it is an $\bb{E}_2$-algebra. Proposition \ref{prop:regularbimodajoint} is a very different construction of adjointable morphisms than the one explored in \cite{GS}. In particular, it is also quick to see that regular bimodules have right adjoints in Haugseng's combinatorial model \cite{HaugsengEn} for $\mor_n$, where the geometric reasoning of \cite{GS} cannot a priori be applied; in fact, the arguments of this section are significantly \emph{easier} in the combinatorial model. In that setting, it is not hard to write down the functors of Remark \ref{rmk:regularbimodfunctor}, for example.
\end{rmk}
We proceed with the proof of Proposition \ref{prop:sourcetargetCartesian}, which is a straightforward consequence of additivity (Proposition \ref{prop:markedadditivity}) and Lurie's analysis in \cite[Section 4.3]{LurHA}. To ease the notational burden, we will write
\[  \icatd^m_{X,Y} := \Hom^m_{\mor_n(\icat)}(X,Y)  \]
for the $(n-m)$-fold Segal object of $m$-morphisms in $\catinfh$ between a parallel pair of $(m-1)$-morphisms $X,Y$ of $\mor_n(\icat)$. We will require a more explicit description of these \infcatst. For $1\leq k\leq n$ integers, let $i_k$ denote the inclusion
\[ i_k: \simp^{{k-1},op}\times \{[0]\} \times \simp^{{n-k},op}\hooklongrightarrow \simpopn  \]
whose associated adjunction $(i_{k}^*\adj i_{k*})$ yields a counit transformation
\[ \epsilon_{k}: \fun\left(\simpopn,\catinfh\right)\longrightarrow \fun\left(\Delta^1, \fun(\simpopn,\catinfh)\right).  \]
Composing these counit transformations yields a functor
\[ \epsilon^n:\fun\left(\simpopn,\catinfh\right)\longrightarrow \fun\left((\Delta^1)^{\times n}, \fun(\simpopn,\catinfh)\right).  \]
\begin{ex}
Let $X_{\bullet\bullet}\in \fun\left(\simp^{2,op},\catinfh\right)$ be a bisimplicial (large) \infcatt, then applying $\epsilon^2$ to $X_{\bullet\bullet}$ yields a square of bisimplicial \infcats given at $[n,m]$ by 
\[
\begin{tikzcd}
X_{nm} \ar[d] \ar[r] & X_{0m}^{\times (n+1)} \ar[d] \\
X_{n0}^{\times (m+1)} \ar[r] & X_{00}^{\times (n+m+2)}.    
\end{tikzcd}
\]    
\end{ex}
It will be useful to identify $(\Delta^1)^{\times n}$ with (the nerve of) the poset $P^n:=P(\{1,\ldots,n\})$ of subsets of $\{1,\ldots,n\}$ partially ordered by inclusion, and we will denote $P^n_0= P^n \setminus \{\emptyset\}$. For $X_{\bullet\ldots\bullet}$ an $n$-fold simplicial object, we will write $\epsilon^n_0(X)$ for the restriction $\epsilon^n(X)|_{P_0^n}$ of $\epsilon^n(X):P^n\rightarrow \fun\left(\simpopn,\catinfh \right)$. Now suppose that $\icate_{\bullet\ldots\bullet}$ is an $n$-fold Segal object of $\catinfh$, then to give a parallel pair of $(n-1)$-morphisms $X,Y$ (whose domain and codomain are a parallel pair of $(n-2)$-morphisms, and so on) is to give an extension
\begin{equation}\label{eq:extension}
\begin{tikzcd}
P_0^n \times \{1\} \ar[r,"\epsilon_0^n(\icate)_{1\ldots 1}"] \ar[d,hook] &[2em] \catinfh\\
P_{0}^n \times\Delta^1 \ar[ur,"\overline{\epsilon_0^n(\icate)_{1\ldots 1}}^{(X,Y)}"',dotted]
\end{tikzcd}
\end{equation}
that carries $P_0^n\times \{0\}$ into the terminal \infcatt, that is, a map $*\rightarrow \lim_{P^n_0}\epsilon_0^n(\icate)_{1\ldots 1}$. To see this, we note that the inclusions $\{i_k\}_{1\leq k\leq n}$ also induce adjunctions $\{(i_{k!}\adj i_k^*)\}_{1\leq k\leq n}$ whose units we may compose to obtain a functor
\[ \eta^n:\fun\left(\simpopn,\catinfh\right)\longrightarrow \fun\left(P^n, \fun(\simpopn,\catinfh)\right)  \]
and to give an object of $\lim_{P^n_0}\epsilon_0^n(\icate)_{1\ldots 1}$ is to give a map
\[  \underset{P^n\setminus \{1,\ldots,n\}}{\colim} \eta^n\left(C_n\right) \longrightarrow \icate_{\bullet \ldots\bullet} \]
where $C_n$ is the $n$-cell, the $n$-fold simplicial object representable by $[1,\ldots,1]$. Reflecting the colimit onto complete $n$-fold Segal objects recovers $\del C_n$, the walking parallel pair of $n$-morphisms. It follows that the \infcat $\Hom^n_{\icate}(X,Y)$ of $n$-morphisms among a parallel pair of $(n-1)$-morphisms $X,Y$ admits the following description: the inclusions
\[ P^n\hooklongrightarrow P_0^{n+1},\qquad S\longmapsto S\cup \{n+1\}\] 
and
\[ P_0^n\times \Delta^1\hooklongrightarrow P_0^{n+1},\qquad(S,i)\longmapsto \begin{cases}
   S\cup \{n+1\} & \text{if } i=1, \\
    S & \text{if }i=0 .  
\end{cases}  \]
determine an equivalence $ P^n_0\times\Delta^1\coprod_{P^n_0\times\{1\}}P^n\simeq P_0^{n+1}$ of \infcatst, and we can identify the \infcat $\Hom^n_{\icate}(X,Y)$ with the limit of diagram 
\[ P^{n+1}_0\longrightarrow \catinfh, \]
obtained by amalgamating $\epsilon^n(\icate)_{1\ldots 1}$ and the extension \eqref{eq:extension}. This limit sits as the initial vertex of the cube $P^{n+1}\cong \left(P^{n+1}_0\right)^{\lhd}$.
\begin{ex}
Let $\icate_{\bullet\bullet}\in \fun\left(\simp^{2,op},\catinfh\right)$ be a 2-fold category object of $\catinfh$ and suppose we are given objects $R,S \in X_{00}$ and two objects $A,B$ in $X_{10}$ whose images under $X_{10}\rightarrow X_{00}\times X_{00}$ are both $(R,S)$. Then the \infcat $\Hom^2_{\icate}(A,B)$ fits into a limit diagram 
\[
\begin{tikzcd}
& \Hom^2_{\icate}(A,B)\arrow[rr] \arrow[dd] \ar[dl] & & \{R,S\} \arrow[dl]\arrow[dd] \\
\icate_{11}  \arrow[rr,crossing over] \arrow[dd] & & \icate_{01}^{\times 2}   \\
&\{A,B\}\arrow[dl]  \arrow[rr] & & \{R,S\} \arrow[dl]\\
\icate_{10}^{\times 2}\arrow[rr] & & \icate_{00}^{\times 4}\arrow[from=uu, crossing over]
\end{tikzcd}
\]
of \infcatst.
\end{ex}
\begin{proof}[Proof of Proposition \ref{prop:sourcetargetCartesian}]
The discussion above identifies the map 
\[ \mathsf{Ar}_1\left(\icatd^k_{R,S}\right) \longrightarrow \mathsf{Ob}\left(\icatd^k_{R,S}\right)\times \mathsf{Ob}\left(\icatd^k_{R,S}\right)  \]
with the limit of a diagram 
\[  F:P^{k+1}_0 \times \Delta^1\longrightarrow \catinfh \]
obtained by amalgamating the image under $\epsilon^k$ at $[1,\ldots,1]\in \simp^{k,op}$ of the map \begin{equation}\label{eq:epsilonk}\Fact^{\mathrm{cstr}}_{\square^n_{[\bullet,\ldots,\bullet,1,0,\ldots,0]}}(\icat) \longrightarrow   \Fact^{\mathrm{cstr}}_{\square^n_{[\bullet,\ldots,\bullet,0,0,\ldots,0]}}(\icat)\times\Fact^{\mathrm{cstr}}_{\square^n_{[\bullet,\ldots,\bullet,0,0,\ldots,0]}}(\icat)\end{equation}
of $k$-fold category objects with the extension \eqref{eq:extension} determined by the parallel pair of $(k-1)$-morphisms $R,S$ (here and below in this proof, the marking on the stratified cubes are suppressed from the notation, but all factorization algebras are \emph{pointless}). To show that the limit is a Cartesian fibration, it suffices to argue the following.
\begin{enumerate}[$(1)$]
    \item For every $S\in P^{k+1}_0$, the functor $F(S,0)\rightarrow F(S,1)$ is a Cartesian fibration.
    \item For every inclusion $S'\subset S$, the upper horizontal functor in the square
    \[
    \begin{tikzcd}
    F(S',0) \ar[r] \ar[d] & F(S,0) \ar[d] \\
    F(S',1) \ar[r] & F(S,1)
    \end{tikzcd}    
    \]
    carries Cartesian morphisms to Cartesian morphisms.
\end{enumerate}
Then the limit $F(\emptyset)$ is a Cartesian fibration, and a morphism in the limit is Cartesian if and only if its image in each $F(S)$ for $S\in P^{m+1}_0$ is a Cartesian. Restricted to $P^k_0\times \{0\}\subset P^{k+1}_0$, the diagram $F$ is constant on the equivalence $*\rightarrow *$, so it suffices to show that $F$ restricted to $P^k\hookrightarrow P^{k+1}_0$ has the required property. The diagram $F|_{P^k}$ is the result of applying the functor $\epsilon^k$ to the map \eqref{eq:epsilonk} of $k$-fold category objects and evaluating at $[1,\ldots,1]\in\simp^{k,op}$. Via Proposition \ref{prop:markedadditivity} and Corollary \ref{cor:Fact-bimod}, the map \eqref{eq:epsilonk} is identified with the map 
\begin{equation}\label{eq:bmodev}
\mathsf{BMod}\left(\Fact^{\mathrm{cstr}}_{\square^{n-1}_{[\bullet,\ldots,\bullet,0,\ldots,0]}}(\icat)\right)
\longrightarrow \mathsf{Alg}\left(\Fact^{\mathrm{cstr}}_{\square^{n-1}_{[\bullet,\ldots,\bullet,0,\ldots,0]}}(\icat)\right)\times \mathsf{Alg}\left(\Fact^{\mathrm{cstr}}_{\square^{n-1}_{[\bullet,\ldots,\bullet,0,\ldots,0]}}(\icat)\right),
\end{equation}
which is a Cartesian fibration upon evaluating at any object of $\simp^{k,op}$, by \cite[Corollary 4.3.3.3]{LurHA}, and Cartesian morphisms are precisely those that become an equivalence upon evaluating at $\mathfrak{m}\in \mathsf{BM}$. The diagram $F|_{P^k}:P^k\rightarrow \fun(\Delta^1,\catinfh)$ carries every object to a product of maps of the form \eqref{eq:bmodev} evaluated at some object of $\simp^{k,op}_{\leq 1}$, and it carries morphisms to maps induced by $[0]\cong \{0\}\subset [1]$ and $[0]\cong \{1\}\subset [1]$, which are easily seen to preserve Cartesian arrows. Now an arrow in the limit is Cartesian if the corresponding arrow in $\Fact^{\mathrm{cstr}}_{[1,\ldots,1,0,\ldots,0]}(\icat)$ is an equivalence, but by construction of the diagram $F$, such an arrow becomes an equivalence when evaluated on all disk-like open subcubes \emph{except} for disk-like open subcubes containing the deepest stratum. So such an arrow in the limit is Cartesian precisely if it evaluates to an equivalence on disk-like open subcubes that contain the deepest stratum.
\end{proof}

The Segal condition for $\mor_n(\icat)$ now implies the following.

\begin{cor}\label{cor:cartfib}
Let $n\geq 1$ be an integer and let
\[\ev_{\{0,\ldots,n\}}:
\mathsf{Ar}_n\left(\icatd_{R,S}^k\right)\longrightarrow \mathsf{Ob}\left(\icatd_{R,S}^k\right)\times\ldots\times \mathsf{Ob}\left(\icatd_{R,S}^k\right) 
\]
be the forgetful functor evaluating on all $n+1$ objects of a composable chain of morphisms of length $n$. Then $\ev_{\{0,\ldots,n\}}$ is a Cartesian fibration. Moreover, a morphism is $\ev_{\{0,\ldots,n\}}$-Cartesian just in case the underlying morphism in the $n$-fold product of $\icat$ (obtained by evaluating on all the disk-like open subcubes containing a deepest stratum) is an equivalence.
\end{cor}

\begin{cor}\label{cor:subintcartesian}
Let $[l,l+1,\ldots,m]\subset [n]$ be the inclusion of a subinterval. Then the induced functor 
\[ \ev_{[l,l+1,\ldots,m]}:\mathsf{Ar}_n\left(\icatd^k_{R,S}\right) \longrightarrow \mathsf{Ar}_{m-l}\left(\icatd^k_{R,S}\right)  \]
is a Cartesian fibration.
\end{cor}
\begin{proof}
The Segal condition guarantees a pullback diagram 
\[
\begin{tikzcd}
\mathsf{Ar}_n\left(\icatd^k_{R,S}\right) \ar[d,"\ev_{[l,l+1,\ldots,m]}"'] \ar[r] &[2em] \mathsf{Ar}_{l}\left(\icatd^k_{R,S}\right) \times\mathsf{Ar}_{n-m}\left(\icatd^k_{R,S}\right) \ar[d,"(\ev_l{,}\ev_0)"] \\
\mathsf{Ar}_{m-l}\left(\icatd^k_{R,S}\right) \ar[r,"\ev_0\times\ev_{m-l}"] &[2em] \mathsf{Ob}\left(\icatd^k_{R,S}\right)\times \mathsf{Ob}\left(\icatd^k_{R,S}\right)
\end{tikzcd}
\]
of \infcatst, where the upper horizontal functor is induced by the subintervals $[0,\ldots,l]\subset [n]$ and $[m,\ldots,n]\subset [n]$.
\end{proof}
\begin{lem}\label{lem:composition-preserve-Cartesianmorphisms}
The horizontal functor in the diagram 
\[
\begin{tikzcd}
    \mathsf{Ar}_2\left(\icatd_{R,S}^k\right) \ar[rr,"\ev_{[0,2]}"] \ar[dr,"\ev_{\{0{,}2\}}"'] && \mathsf{Ar}_1\left(\icatd_{R,S}^k\right) \ar[dl,"\ev_{\{0{,}1\}}"] \\
    &\mathsf{Ob}\left(\icatd_{R,S}^k\right)^{\times 2} 
\end{tikzcd}
\]
induced by the inclusion $\Delta^{\{0,2\}}\subset\Delta^2$ carries $\ev_{\{0,2\}}$-Cartesian morphisms to $\ev_{\{0,1\}}$-Cartesian morphisms. 
\end{lem}
\begin{proof}
It follows from Corollary \ref{cor:cartfib} that if a morphism of $\mathsf{Ar}_2\left(\icatd_{R,S}^k\right)$ is $\ev_{\{0,2\}}$-Cartesian, then it evaluates to an equivalence on disk-like open subcubes of the form $(0,1)^{k-1}\times (a,b)\times (0,1)^{n-k}$ for
\begin{itemize}
    \item $\{1/3\}\subset (a,b)\subset (0,2/3)$,
    \item $\{2/3\}\subset (a,b)\subset (1/3,1)$,
    \item $(a,b)\subset (1/3,2/3)$.
\end{itemize}
A morphism of $\mathsf{Ar}_1\left(\icatd_{R,S}^k\right)$ is $\ev_{0,1}$-Cartesian just in case it evaluates to an equivalence on the disk-like open subcube $(0,1)^n$ itself. We are therefore reduced to proving the following assertion: let $\alpha:\F\rightarrow \mathcal{G}$ be a map of constructible pointless factorization algebras on $\square^n_{[1,\ldots,1,2,0,\ldots,0]}$. Suppose that $\alpha$ is an equivalence on the three types of disk-like open subcubes described above. Then $\alpha$ is an equivalence on $(0,1)^n$. This follows from the fact that the value of both $\F$ and $\mathcal{G}$ on $(0,1)^n$ can be recovered as a colimit of tensor products of their values on these three types of opens (and this colimit is the relative tensor product), as follows from the discussion directly above Proposition \ref{prop:relativetensorprod}. 
\end{proof}

We are now equipped to give the proof of Proposition \ref{prop:regularbimodajoint}, which is somewhat tedious, but almost entirely formal. 

\begin{proof}[Proof of $(1)$, $(2)$ and $(3)$ of Proposition \ref{prop:regularbimodajoint}]
We prove $(1)$. We show that precomposing with the right regular bimodule of $f:A\rightarrow B$ induces the Cartesian transport functor $(f,\mathrm{id}_Z)_*$; the proof that postcomposition with the left regular bimodule induces the other Cartesian transport is the same. Consider the diagram 
\[
\begin{tikzcd}
    \mathsf{Ar}_2\left(\icatd_{R,S}^k\right) \ar[r,"\ev_{[0,2]}"] \ar[d,"\ev_{\{0{,}1{,}2\}}"'] &[2em]\mathsf{Ar}_1\left(\icatd_{R,S}^k\right) \ar[d,"\ev_{\{0{,}1\}}"] \\
    \mathsf{Ob}\left(\icatd_{R,S}^k\right)^{\times 3}\ar[r,"\ev_0\times\ev_2"]&[2em]\mathsf{Ob}\left(\icatd_{R,S}^k\right)^{\times 2}. 
\end{tikzcd}
\]
Lemma \ref{lem:composition-preserve-Cartesianmorphisms} guarantees that the upper horizontal carries $\ev_{\{0,1,2\}}$-Cartesian lifts of the triple of maps 
\[(f:A\rightarrow B,\mathrm{id}_B:B\rightarrow B,\mathrm{id}_Z:Z\rightarrow Z)\]
to $\ev_{\{0,1\}}$-Cartesian lifts of the pair of maps $(f,\mathrm{id}_Z)$; it follows that upon taking fibers of the vertical functors, we have a commuting diagram
\[
\begin{tikzcd}
    \mathsf{Ob}\left(\icatd_{B,B}^{k+1}\right) \times  \mathsf{Ob}\left(\icatd_{B,Z}^{k+1}\right) \ar[r,"\circ"] \ar[d] &  \mathsf{Ob}\left(\icatd_{B,Z}^{k+1}\right)\ar[d] \\
      \mathsf{Ob}\left(\icatd_{A,B}^{k+1}\right)\times  \mathsf{Ob}\left(\icatd_{B,Z}^{k+1}\right) \ar[r,"\circ "]&  \mathsf{Ob}\left(\icatd_{A,Z}^{k+1}\right). 
\end{tikzcd}
\]
of \infcats in which the vertical functors are the Cartesian transport functors along $(f,\mathrm{id}_B,\mathrm{id}_Z)$ and $(f,\mathrm{id}_Z)$ and the horizontal functors are composition in the higher Morita category. By definition, the Cartesian transport on the left carries the identity on $B$ to the right regular bimodule ${}_AB_{B}$, so the diagram above restricts to a commuting square
\[
\begin{tikzcd}
 \mathsf{Ob}\left(\icatd_{B,Z}^{k+1}\right) \ar[r,"\circ \mathrm{id}_B"] \ar[d,equal] &[2em]  \mathsf{Ob}\left(\icatd_{B,Z}^{k+1}\right) \ar[d]\\ \mathsf{Ob}\left(\icatd_{B,Z}^{k+1}\right) \ar[r,"\circ {}_AB_{B}"]&[2em]  \mathsf{Ob}\left(\icatd_{A,Z}^{k+1}\right). 
\end{tikzcd}
\]
Since the upper horizontal map is an equivalence, we deduce that the left vertical map and lower horizontal map are equivalent, as desired. We now show $(2)$; again, we treat only the case of right regular bimodules. Let be given maps $f:A\rightarrow B$ and $g:B\rightarrow C$ of $\icatd^k_{R,S}$. The right regular bimodule ${}_AC_C$ is the image of the identity under the composite 
\[  \mathsf{Ob}\left(\icatd^{k+1}_{C,C}\right) \xrightarrow{(g,\mathrm{id}_C)_*} \mathsf{Ob}\left(\icatd^{k+1}_{B,C}\right)\xrightarrow{(f,\mathrm{id}_C)_*}  \mathsf{Ob}\left(\icatd^{k+1}_{A,C}\right),   \]
since $(f,\mathrm{id}_C)_*\circ (g,\mathrm{id}_C)_*\simeq (g\circ f,\mathrm{id}_C)_*$ by Proposition \ref{prop:sourcetargetCartesian}. Then we have a chain of equivalences
\[ {}_AC_C=(g\circ f,\mathrm{id}_C)_*(\mathrm{id}_C)\simeq (f,\mathrm{id}_C)_*\circ (g,\mathrm{id}_C)_*(\mathrm{id}_C)= (f,\mathrm{id}_C)_*({}_BC_C) \simeq  {}_BC_C \circ {}_AB_B \]
with the last one provided by $(1)$. We proceed with $(3)$. Let $T$ be another $(k-1)$-morphism with the same domain and codomain as $R,S$. Let $A',B'\in\mathsf{Ob}\left(\icatd^{k}_{S,T}\right)$ and let $f':A\rightarrow B$ be morphism in this \infcatt. Write $\icatd^{k}_{R,S,T}$ for the $(n-k)$-fold Segal object of $\catinfh$ of pairs of composable morphisms, which is equivalent to $\icatd_{R,S}^k\times\icatd_{S,T}^k$. The regular bimodules ${}_AB_B$ and ${}_{A'}B'_{B'}$ determine an object of $\mathsf{Ar}_1\left( \icatd^{k}_{R,S,T}\right)$ and we must verify that the functor
\[  \mathsf{Ar}_1\left( \icatd^{k}_{R,S,T}\right)\longrightarrow \mathsf{Ar}_1\left( \icatd^{k}_{R,T}\right) \]
carries this object to the right regular bimodule of the composite $f'\circ f$. This amounts to the following assertion: the horizontal functor 
\[
\begin{tikzcd}
\mathsf{Ar}_1\left( \icatd^{k}_{R,S,T}\right)\ar[r]\ar[d,"\ev_{\{0,1\}}\times\ev_{\{0,1\}}"'] & \mathsf{Ar}_1\left( \icatd^{k}_{R,T}\right) \ar[d,"\ev_{\{0,1\}}"]\\
\mathsf{Ob}\left(\icatd_{R,S}\right)^{\times 2} \times \mathsf{Ob}\left(\icatd_{S,T}\right)^{\times 2} \ar[r] & \mathsf{Ob}\left(\icatd_{R,T}\right)^{\times 2} 
\end{tikzcd}
\]
carries Cartesian morphisms for the left vertical functor to Cartesian morphisms for the right vertical functor. This is a consequence of the fact that composition in the Morita category is given by the relative tensor product. More precisely, since Cartesian morphisms in the two top \infcats are detected by evaluating at disk-like open subcubes containing the deepest strata, we reduce to the following assertion: let $\alpha:\F\rightarrow \mathcal{G}$ be map of a pointless constructible factorization algebra on $\square^2_{[1,\ldots ,1,2,1,0,\ldots, 0]}$ and suppose that $\alpha$ is an equivalence on disk-like open subcubes of the form $(0,1)^{k-2}\times (a,b)\times (0,1)^{n-k+1}$ for 
\begin{itemize}
    \item $\{1/3\}\subset (a,b)\subset (0,2/3)$,
    \item $\{2/3\}\subset (a,b)\subset (1/3,1)$,
    \item $(a,b)\subset (1/3,2/3)$.
\end{itemize}
Then $\alpha$ is an equivalence on $(0,1)^n$. This follows from the fact that the value of both $\F$ and $\mathcal{G}$ on $(0,1)^n$ can be recovered as a colimit of tensor products of their values on these three types of opens (and this colimit is the relative tensor product), as follows from the discussion directly above Proposition \ref{prop:relativetensorprod}. 
\end{proof}

We now proceed with $(4)$ and $(5)$ of Proposition \ref{prop:regularbimodajoint}. We start by producing the units and counits for the adjoints. For $A\in \mathsf{Ob}(\icatd^k_{R,S})$, we let ${}_AA_A A_A$ denote the image of $A$ under the degeneracy map $\mathsf{Ob}\left(\icatd^k_{R,S}\right)\rightarrow \mathsf{Ar}_2\left(\icatd^k_{R,S}\right)$. The functor
\[ \ev_{\{0,2\}}:\mathsf{Ar}_2\left(\icatd^k_{R,S}\right)\longrightarrow \mathsf{Ob}\left(\icatd^k_{R,S}\right)\times \mathsf{Ob}\left(\icatd^k_{R,S}\right)\]
is a Cartesian fibration, so for $f:A\rightarrow B$ a map in $\mathsf{Ob}(\icatd^k_{R,S})$, we have an $\ev_{\{0,2\}}$-Cartesian lift 
\[{}_AB_B{}B_A \overset{\alpha}{\longrightarrow} {}_BB_BB_B \] 
of the map $(f:A\rightarrow B,f:A\rightarrow B)$. Note that the functors
\[  \ev_{[0,1]}: \mathsf{Ar}_2\left(\icatd^k_{R,S}\right)\longrightarrow \mathsf{Ar}_1\left(\icatd^k_{R,S}\right),\quad \quad  \ev_{[1,2]}: \mathsf{Ar}_2\left(\icatd^k_{R,S}\right)\longrightarrow \mathsf{Ar}_1\left(\icatd^k_{R,S}\right) \]
induced by the map $\Delta^1= \Delta^{\{0,1\}}\subset \Delta^2$ respectively $\Delta^1= \Delta^{\{1,2\}}\subset \Delta^2$ 
carry ${}_AB_B{}B_A$ to ${}_AB_B$ and ${}_BB_A$ respectively, since these functors carry $\ev_{\{0,2\}}$-Cartesian edges to $\ev_0$-Cartesian edges respectively $\ev_1$-Cartesian edges, as one readily verifies. It follows that the functor 
\[  \ev_{[0,2]}: \mathsf{Ar}_2\left(\icatd^k_{R,S}\right)\longrightarrow \mathsf{Ar}_1\left(\icatd^k_{R,S}\right) \]
induced by the map $\Delta^{1}=\Delta^{\{0,2\}}\subset \Delta^2$ carries ${}_AB_B{}B_A$ to the composite ${}_BB_A\circ {}_AB_B$. There is a dotted diagonal map
\[
\begin{tikzcd}
&  {}_AB_B{}B_A\ar[dr,"\alpha"] \\
 {}_AA_A{}A_A\ar[rr,"f"'] \ar[ur,dotted,"\widetilde{\eta}_f"] && {}_BB_B{}B_B
\end{tikzcd}
\]
resulting in 2-simplex lying over the 2-simplex
\[
\begin{tikzcd}
&  (A,A)\ar[dr,"f"] \\
 (A,A)\ar[rr,"f"] \ar[ur,equal] && (B,B)
\end{tikzcd}
\]
in $\mathsf{Ob}\left(\icatd^k_{R,S}\right)\times \mathsf{Ob}\left(\icatd^k_{R,S}\right)$, since $\alpha$ is $\ev_{\{0,2\}}$-Cartesian. The image of $\widetilde{\eta}$ under $\ev_{[0,2]}$ yields a map 
\[ \widehat{\eta}_f: \mathrm{id}_A={}_AA_A \longrightarrow {}_BB_A\circ {}_AB_B\]
of $\mathsf{Ar}_1\left(\icatd^k_{R,S}\right)$
which we claim corresponds to the the unit of an adjunction. The counit is a bit easier to construct: the functor 
\[  \ev_{1}:\mathsf{Ar}_2\left( \icatd_{R,S}^k\right)\longrightarrow \mathsf{Ob}\left(\icatd^k_{R,S}\right) \]
is a Cartesian fibration, so for $f:A\rightarrow B$ a map of $\mathsf{Ob}\left(\icatd^k_{R,S}\right)$, we have an $\ev_1$-Cartesian lift 
\[  {}_BB_AB_B\overset{\widetilde{\epsilon}_f}{\longrightarrow}{}_BB_BB_B.   \]
Then the image of $\widetilde{\epsilon}_f$ under $\ev_{[0,2]}$ yields a map 
\[ \widehat{\epsilon}_f:{}_AB_B\circ{}_BB_A\longrightarrow {}_BB_B=\mathrm{id}_B  \]
of $\mathsf{Ar}_1\left(\icatd^k_{R,S}\right)$ which we claim corresponds to the counit of an adjunction.

\begin{proof}[Proof of $(4)$ and $(5)$ of Proposition \ref{prop:regularbimodajoint}]
It follows from $(2)$ and $(3)$ (compatibility of right regular bimodules with horizontal and vertical composition) that it suffices to show $(5)$; the snake identity asserted by $(5)$ for $\widehat{\eta}_f$ and $\widehat{\epsilon}_f$ will then imply the snake identity for the associated right regular bimodules. We verify that the composite 
\[  {}_AB_B \longrightarrow  {}_AB_B\circ {}_BB_A\circ  {}_AB_B  \longrightarrow {}_AB_B  \]
is equivalent to the identity. To see this, we realize this composite as the result of composing a certain map
\[  {}_AA_AA_AB_B \longrightarrow  {}_AB_BB_AB_B\longrightarrow {}_AB_B B_B B_B  \]
of $\mathsf{Ar}_3\left(\icatd^k_{R,S}\right)$ we construct below. In line with the notation introduced above, we will write ${}_AA_AA_AA_A$ for the image of $A$ under the degeneracy map $\mathsf{Ob}\left(\icatd^k_{R,S}\right)\rightarrow \mathsf{Ar}_3\left(\icatd^k_{R,S}\right)$.
\begin{itemize}
    \item The object ${}_AA_AA_AB_B$ is the domain of an $\ev_{[0,1,2]}$-Cartesian lift
    \[  {}_AA_AA_AB_B\overset{\beta}{\longrightarrow} {}_BB_BB_BB_B \]
    of the map $f:{}_AA_AA_A\rightarrow {}_BB_BB_B$ terminating at ${}_BB_BB_BB_B$, where $\ev_{[0,1,2]}$ is the functor 
    \[ \ev_{[0,1,2]}: \mathsf{Ar}_{3}\left(\icatd^{k}_{R,S}\right) \longrightarrow  \mathsf{Ar}_{2}\left(\icatd^{k}_{R,S}\right)  \]
    induced by the map $\Delta^2=\Delta^{\{0,1,2\}}\subset \Delta^3$. Let $\ev_{[0,3]}:\mathsf{Ar}_{3}\left(\icatd^{k}_{R,S}\right) \rightarrow  \mathsf{Ar}_{1}\left(\icatd^{k}_{R,S}\right) $ be the functor induced by the map $\Delta^1=\Delta^{\{0,3\}}\subset\Delta^3$, then applying $\ev_{[0,3]}$ to ${}_AA_AA_AB_B$ yields the composite $ {}_AB_B\circ {}_AA_A\circ {}_AA_A $.
    \item The object $ {}_AB_BB_AB_B$ is the domain of an $\ev_{\{0,2\}}$-Cartesian lift
    \[  {}_AB_BB_AB_B\overset{\gamma}{\longrightarrow} {}_BB_BB_BB_B \]
    of the map $f:(A,A)\rightarrow (B,B)$ of $\mathsf{Ob}\left(\icatd^k_{R,S}\right)\times \mathsf{Ob}\left(\icatd^k_{R,S}\right)$. In particular, applying $\ev_{[0,3]}$ to this object yields the composite ${}_AB_B\circ {}_BB_A\circ {}_AB_B$. 
     \item The object $ {}_AB_BB_BB_B$ is the domain of an $\ev_{0}$-Cartesian lift
     \[  {}_AB_BB_BB_B\overset{\delta}{\longrightarrow} {}_BB_BB_BB_B \]
     of the map $f:A\rightarrow B$. In particular, applying $\ev_{[0,3]}$ to this object yields the composite $ {}_BB_B\circ {}_BB_B\circ {}_AB_B$.
     \item The first map ${}_AA_AA_BB_B \rightarrow  {}_AB_BB_AB_B$ is the dotted lift in the 2-simplex
     \[
\begin{tikzcd}
&   {}_AB_BB_AB_B\ar[dr,"\gamma"] \\
{}_AA_AA_AB_B\ar[rr,"\beta"] \ar[ur,dotted] && {}_BB_B{}B_BB_B
\end{tikzcd}
\]
lying over the 2-simplex
\[
\begin{tikzcd}
&  (A,A)\ar[dr,"f"] \\
 (A,A)\ar[rr,"f"] \ar[ur,equal] && (B,B)
\end{tikzcd}
\]
which exists up to a contractible space of choices by virtue of $\gamma$ being $\ev_{0,2}$-Cartesian. Applying $\ev_{[0,3]}$ to this map yields the horizontal composite of the unit $\widehat{\eta}_f:{}_AA_A\rightarrow {}_BB_A\circ {}_AB_B$ with the identity on ${}_AB_B$.
\item The second map $ {}_AB_BB_AB_B \rightarrow  {}_AB_BB_BB_B$ is the dotted lift in the 2-simplex
     \[
\begin{tikzcd}
&    {}_AB_BB_BB_B\ar[dr,"\delta"] \\
{}_AB_BB_AB_B\ar[rr,"\gamma"] \ar[ur,dotted] && {}_BB_B{}B_BB_B
\end{tikzcd}
\]
lying over the 2-simplex
\[
\begin{tikzcd}
&  A\ar[dr,"f"] \\
 A\ar[rr,"f"] \ar[ur,equal] && B,
\end{tikzcd}
\]
which exists up to a contractible space of choices by virtue of $\delta$ being $\ev_{0}$-Cartesian. Applying $\ev_{[0,3]}$ to this map yields the horizontal composite of the identity on ${}_AB_B$ with the counit $\widehat{\epsilon}_f: {}_AB_B\circ {}_BB_A\rightarrow {}_BB_B$.
\end{itemize}
We will be done once we argue that applying $\ev_{[0,3]}$ to the composite 
\[  {}_AA_AA_AB_B \longrightarrow {}_AB_B B_B B_B  \]
yields the identity map. Since we have by construction a 2-simplex
\[
\begin{tikzcd}
& {}_AB_BB_BB_B\ar[dr,"\delta"] \\
 {}_AA_AA_AB_B\ar[rr,"\beta"] \ar[ur] && {}_BB_BB_BB_B,
\end{tikzcd}
\]
it suffices to argue the following.
\begin{enumerate}[$(1)$]
    \item Applying $\ev_{[0,3]}$ to the map $\beta:{}_AA_AA_AB_B \rightarrow {}_BB_BB_BB_B$ yields an $\ev_{0}$-Cartesian map ${}_AB_B\rightarrow {}_BB_B$ over $f:A\rightarrow B$.
    \item Applying $\ev_{[0,3]}$ to the map $\delta:{}_AB_BB_BB_B \rightarrow {}_BB_BB_BB_B$ yields an $\ev_{0}$-Cartesian map ${}_AB_B\rightarrow {}_BB_B$ over $f:A\rightarrow B$.
    \item Applying $\ev_{0}$ to the map $ {}_AA_AA_AB_B\rightarrow {}_AB_BB_BB_B$ yields a map equivalent to the identity on $A$.
\end{enumerate}
We note that $(3)$ is true by construction of the composite 
\[  {}_AA_AA_AB_B \longrightarrow  {}_AB_BB_AB_B\longrightarrow {}_AB_B B_B B_B  \]
since the functor $\ev_0$ carries both maps to the identity $A=A$ by stipulation. Similarly, applying $\ev_0$ to the maps $\beta$ and $\delta$ yields $f:A\rightarrow B$ by stipulation, so it suffices to argue that $\ev_{[0,3]}$ carries $\beta$ and $\delta$ to $\ev_0$-Cartesian morphisms. It follows from Lemma \ref{lem:composition-preserve-Cartesianmorphisms} that the functor $\ev_{[0,3]}$ carries $\ev_0$-Cartesian morphisms of $\mathsf{Ar}_3\left(\icatd^k_{R,S}\right)$ to $\ev_0$-Cartesian morphisms of $\mathsf{Ar}_1\left(\icatd^k_{R,S}\right)$, so the image of $\delta$ is indeed $\ev_0$-Cartesian. To see that $\ev_{[0,3]}(\beta)$ is $\ev_0$-Cartesian, it suffices to show that it induces an equivalence upon evaluating at disk-like open subcubes containing the deepest stratum, since $\ev_3(\beta)$ is an equivalence by construction. The composite
\[  \mathsf{Ar}_1\left(\icatd^k_{R,S}\right)\times_{\mathsf{Ob}\left(\icatd^k_{R,S}\right)}\mathsf{Ar}_1\left(\icatd^k_{R,S}\right)\times_{\mathsf{Ob}\left(\icatd^k_{R,S}\right)}\mathsf{Ar}_1\left(\icatd^k_{R,S}\right)\simeq \mathsf{Ar}_3\left(\icatd^k_{R,S}\right) \xrightarrow{\ev_{[0,3]}}\mathsf{Ar}_1\left(\icatd^k_{R,S}\right) \longrightarrow \icat   \]
where the second functor evaluates at the deepest stratum is equivalent to the relative tensor product (Proposition \ref{prop:relativetensorprod}), so we complete the proof by observing that in the composite
\[ \bb{1}_{\icat}\otimes \bb{1}_{\icat}\otimes B\longrightarrow   A\otimes_AA\otimes_AB\longrightarrow  B\otimes_BB\otimes_BB \]
the first map and the composition are equivalences.
\end{proof}

\section{$(n+1)$-Dualizability and invertibility in the $\E_n$-Morita category}
In this section, we prove our main theorem, which we state here in its precise form for the reader's convenience.
\begin{thm}\label{thm:mainthm}
Let $\icat$ be a presentably symmetric monoidal \infcatt. Then any $\bb{E}_n$-algebra $A$ in $\mor_n(\icat)$ (or in $\umor_n(\icat)$) is $(n+1)$-dualizable if and only if for all $0\leq k\leq n$, the object $A$ is left dualizable as a left module over the factorization homology 
\[  \int_{S^{k-1}\times \R^{n-k+1}}A. \]
\end{thm}

The proof of this theorem will consist of two parts, a geometric part (\ref{sec:geometric_part}) and an argument for lifting adjoints (\ref{sec:lifting_lemma}). 
The lifting-of-adjoints argument in essence states that to find an adjoint of an $n$-morphism in the higher Morita category it is enough to find an adjoint for the underlying bimodule in the usual Morita category.

This allows to reduce the theorem to a question about finding adjoints for specific bimodules in the usual Morita category. These bimodules are found inductively using a geometric argument and excision of factorization homology.

\subsection{Lifting adjoints}\label{sec:lifting_lemma}
Our goal in this subsection is to prove the following result upon which our main theorem relies. In the setting of tensor categories, the analog of this result is Proposition 5.17, 5.18, and Lemma 5.19 in \cite{BJS}.

\begin{thm}[Lifting-of-adjoints argument]\label{thm:main-lifting-lemma}
Let $\F$ be a constructible pointless factorization algebra on $\square^n_{[\vec{1}]}$, representing an $n$-morphism
\[  M\overset{\F}{\longrightarrow}N  \]
between $(n-1)$-morphisms $M,N:A\rightarrow B$ in the higher Morita category, where $A$ and $B$ are a parallel pair of $(n-2)$-morphisms $\mor_n(\icat)$. Then $\F$, as a 1-morphism of the 2-fold Segal space $\Hom_{\mor_n(\icat)}^{n-1}(A,B)$, admits a left adjoint if and only if the constructible factorization algebra $\pi_{*}\F$ on $\square_{[1]}$ admits a left adjoint as a 1-morphism of $\mor_1(\icat)$.
\end{thm}
\begin{rmk}
In this section, $\icat^{\otimes}$ is a presentably symmetric monoidal \infcat so that we may consider the $(\infty,n+1)$-category $\umor_n(\icat)$ as defined in Section \ref{sec:moritacat}. However, the arguments for lifting adjoints in this section are written in sufficient generality so as to also apply to Haugseng's higher Morita category of \cite{HaugsengEn}. In fact, the arguments to come that are specific to the Morita category are \emph{easier} in Haugseng's combinatorial model. While it is the case that the (pointless) factorization Morita category of Section \ref{sec:moritacat} and the combinatorial Morita category of \cite{HaugsengEn} are equivalent, a precise comparison is beyond the scope of this article, so we content ourselves with stating the following theorem for the factorization version. The comparison between Haugseng's model and the factorization model will be taken up by the first two authors with Anja \v{S}vraka \cite{ScStSv}. 
\end{rmk}

\begin{rmk}
The lifting-of-adjoints argument is superficially similar to a result of Lurie asserting that an ordinary $A$-$B$-bimodule $M$ (for $A$ and $B$ $\bb{E}_1$-algebras) admits a left dual \emph{as an $A$-$B$-bimodule} if and only if $M$ admits a left dual \emph{as a left $A$-module} (see \cite[Proposition 4.6.2.13]{LurHA}), that is, $M$ admits a left dual if and only if composing $M$ (as an $A$-$B$-bimodule) with the left regular bimodule ${}_BB_{\bb{1}_{\icat}}$ admits a left dual (as a left $A$-module). The proof of Theorem \ref{thm:main-lifting-lemma} also involves composing with regular bimodules to verify dualizability of $n$-morphisms of $\mor_n(\icat)$, but differs crucially in two ways due to the presence of higher categorical dimensions.
\begin{enumerate}[$(1)$]
    \item We compose with regular (factorization) bimodules living in different categorical levels than the $n$-morphism itself. 
    \item We compose with regular bimodules \emph{on both sides}.
\end{enumerate}
\end{rmk}
The lifting argument requires a few categorical preliminaries. We will set these up first and then give the proof of Theorem \ref{thm:main-lifting-lemma} for the case of $n=2$ and end this section with the general case. We start with the following reformulation of the notion of an adjointable morphism in an $(\infty,2)$-category in terms of representable functors. In this section, it will be convenient to argue in a model-independent fashion with $(\infty,k)$-categories, so we use the univalent objects $\umor_n(\icat)$ instead of $\mor_n(\icat)$.
\begin{lem}[Criterion for adjoints]\label{lem:adjointcriterion}
Let $\icatd$ be an $(\infty,2)$-category and let $g:X\rightarrow Y$ be a 1-morphism in $\icatd$. Then $g$ admits a left adjoint if and only if the following conditions are satisfied.
\begin{enumerate}[$(1)$]
    \item For every object $Z\in\icatd$, the functor 
    \[  \Hom_{\icatd}(Z,X) \xrightarrow{g\circ\_} \Hom_{\icatd}(Z,Y)\]
    admits a left adjoint.
    \item For each 1-morphism $h:Z'\rightarrow Z$ of $\icatd$, the induced diagram 
    \[
    \begin{tikzcd}
    \Hom_{\icatd}(Z,X)  \ar[d,"\_\circ h"] \ar[r,"g\circ\_"] & \Hom_{\icatd}(Z,Y) \ar[d,"\_\circ h"]\\
    \Hom_{\icatd}(Z',X) \ar[r,"g\circ \_"] & \Hom_{\icatd}(Z',Y)
    \end{tikzcd}
    \]
    among \infcats is horizontally left adjointable.
\end{enumerate}
\end{lem}
\begin{rmk}[Mates and adjointability]
We recall that a square diagram 
\[
\begin{tikzcd}
 \icatd \ar[r,"G_{\icatd}"] \ar[d,"L"] & \icatd' \ar[d,"L'"] \\
 \icate\ar[r,"G_{\icate}"] & \icate'
\end{tikzcd}
\]
of \infcats is \emph{horizontally left adjointable} if the functors $G_{\icatd}$ and $G_{\icate}$ admit left adjoints $F_{\icatd}$ and $F_{\icate}$ respectively, and the \emph{mate} transformation 
\[ F_{\icate}\circ L'\longrightarrow  G_{\icate}\circ L' \circ G_{\icatd} \circ F_{\icatd} \simeq F_{\icate} \circ G_{\icate}\circ L\circ F_{\icatd} \longrightarrow L\circ F_{\icatd}\]
is an equivalence, where the first map is induced by the unit of $(F_{\icatd}\adj G_{\icat})$ and the second by the counit of $(F_{\icate}\adj G_{\icate})$; this is also known as the \emph{Beck-Chevalley condition}. There are three obvious variant notions: that of \emph{vertical} left adjointability and horizontal and vertical \emph{right} adjointability.
\end{rmk}
\begin{rmk}\label{rmk:mates}
More generally, any lax or oplax square in an $(\infty,2)$-category in which two parallel (vertical or horizontal) morphisms have left or right adjoints, has an associated \emph{mate}. Exchanging (op)lax naturality squares with their mates gives equivalences of $(\infty,2)$-categories of functors and (op)lax natural transformations that are componentwise right (left) adjoints between them; see Theorem E in \cite{AGH-laxstraigthening} for the general $(\infty,2)$-categorical statement. 
\end{rmk}
\begin{proof}[Proof of Lemma \ref{lem:adjointcriterion}]
This is well-known and straightforward to prove by hand. It also follows right away from the combination of \cite[Corollary 5.2.10]{AGH-laxstraigthening} and Yoneda's lemma for $(\infty,2)$-categories.
\end{proof}
\begin{rmk}\label{rmk:flippedadjointcrit}
In Lemma \ref{lem:adjointcriterion}, swapping every instance of `left' with `right' and vice versa yields another true statement.   
\end{rmk}
\begin{cor}\label{cor:inducedadjointability}
Let $G:\icatd\rightarrow \icate$ be a functor of $(\infty,2)$-categories. Suppose that $G$ admits a left adjoint $F$. Then for every 1-morphism $g:X\rightarrow Y$ of $\icatd$ that admits a left adjoint and each object $Z\in \icatd$, the induced square
\[
\begin{tikzcd}
  \Hom_{\icatd}(Z,X) \ar[d,"G_{Z,X}"] \ar[r,"g\circ \_"] &[2em]   \Hom_{\icatd}(Z,Y) \ar[d,"G_{Z,Y}"] \\
  \Hom_{\icate}(G(Z),G(X)) \ar[r,"F(g)\circ \_"] &[2em]   \Hom_{\icate}(G(Z),G(Y))
\end{tikzcd}
\]
of \infcats is horizontally left adjointable. 
\end{cor}
\begin{proof}
The square above is equivalent to the square
\[
\begin{tikzcd}
  \Hom_{\icatd}(Z,X) \ar[d,"\_\circ \epsilon_Z"] \ar[r,"g\circ \_"] &   \Hom_{\icatd}(Z,Y) \ar[d,"\_ \circ \epsilon_Z"] \\
  \Hom_{\icatd}(FG(Z),X) \ar[r,"g\circ \_"] &   \Hom_{\icatd}(FG(Z),Y),
\end{tikzcd}
\]
for $\epsilon_Z$ the counit $FG(Z)\rightarrow Z$, so we conclude by applying Lemma \ref{lem:adjointcriterion}.
\end{proof}
\begin{rmk}\label{rmk:flippedinducedadjointability}
In Corollary \ref{cor:inducedadjointability}, if we assume instead that $g$ admits a \emph{right} adjoint, then the square is horizontally right adjointable.
\end{rmk}

\begin{lem}\label{lem:adjointcons}
Let be given a commuting diagram 
\[
\begin{tikzcd}
    \icatd \ar[r,"G"] \ar[d,"U_{\icatd}"] & \icate \ar[d,"U_{\icate}" '] \\
    \icatd_0 \ar[r,"G_0"] \ar[u,"L_{\icatd}" , bend left] & \icate_0 \ar[u,"L_{\icate}" ', bend right] \ar[l,"F_0" , bend left]
\end{tikzcd}
\]
of straight arrows among \infcats in which $U_{\icatd}$ and $U_{\icate}$ have left adjoints $L_{\icatd}$ and $L_{\icate}$ respectively, and $G_0$ has a left adjoint $F_0$. Suppose that the following conditions are satisfied.
\begin{enumerate}[$(1)$]
    \item The \infcat $\icatd$ admits geometric realizations of simplicial objects.
    \item The adjunction $(L_{\icate}\adj U_{\icate})$ is monadic; that is, $U_{\icate}$ is conservative and $\icate$ admits geometric realizations of $U_{\icate}$-split simplicial objects, which are preserved by $U_{\icate}$.
\end{enumerate}
Then $G$ admits a left adjoint. Suppose furthermore that the following condition is satisfied.
\begin{enumerate}[$(3)$]
\item The functor $U_{\icatd}$ preserves geometric realizations. 
\end{enumerate}
Let $F$ denote a left adjoint to $G$, then the square is horizontally left adjointable if and only if the composite 
\[ F_0\circ U_{\icate}\circ L_{\icate}\longrightarrow U_{\icatd}\circ F\circ L_{\icate} \]
of $L_{\icate}$ with the mate is an equivalence.
\end{lem}

\begin{proof} 
We first show that $G$ has a left adjoint assuming $(1)$ and $(2)$. Let $\mathfrak{X}\subset \icate$ be the full subcategory spanned by objects $E$ for which the functor
\[  \icatd \longrightarrow \spa,\quad \quad D\longmapsto \Hom_{\icate}\left(E,G(D)\right)  \]
is corepresentable by an object of $\icatd$; we wish to show that $\mathfrak{X}=\icate$. Note that $\mathfrak{X}$ is the full subcategory spanned by those objects of $\icate$ that are carried to a representable functor by the restricted (opposite) Yoneda embedding
\[  \icate\overset{j}{\hooklongrightarrow}\pshv(\icate^{op})^{op} \xrightarrow{G^*} \pshv(\icatd^{op})^{op}.  \]
Since $\icatd$ admits geometric realizations of simplicial objects and the (opposite) Yoneda embedding $j:\icatd\rightarrow \pshv(\icatd^{op})^{op}$ preserves all colimits that exist in $\icatd$, the full subcategory of $\pshv(\icatd^{op})^{op}$ spanned by representables is stable under geometric realizations. Since the composite $G^*\circ j$ above preserves all colimits that exist in $\icate$, it follows that the full subcategory $\mathfrak{X}\subset\icate$ is closed under the formation of geometric realizations. By \cite[Proposition 4.7.3.14]{LurHA} and the assumption that $(L_{\icate}\adj U_{\icate})$ is monadic, every object of $\icate$ is obtained as a geometric realization of a $U_{\icate}$-split simplicial object for which each term lies in the image of $L_{\icate}$. Hence, it suffices to show that $\mathfrak{X}$ contains the image of $L_{\icate}$. But for $E_0\in \icate_0$, we have equivalences 
\begin{align*}
    \Hom_{\icate}\left(L_{\icate}(E_0),G(D)\right) &\simeq  \Hom_{\icate_0}\left(E_0,U_{\icate}G(D)\right) \\
    &\simeq \Hom_{\icate_0}\left(E_0,G_0U_{\icatd}(D)\right) \\
     &\simeq \Hom_{\icatd_0}\left(F_0(E_0),U_{\icatd}(D)\right) \\
      &\simeq \Hom_{\icatd}\left(L_{\icatd}F_0(E_0),D\right) 
\end{align*}
natural in $D$, so that the functor $\Hom_{\icat}(L_{\icate}(E_0),G(\_))$ is corepresented by $L_{\icatd}F_0(E_0)$. Now assume that additionally, $U_{\icatd}$ preserves geometric realizations. We show $(3)$. Since $U_{\icate}$ preserves geometric realizations of $U_{\icate}$-split simplicial objects, the collection of objects $E\in \icate$ for which the mate
\[  F_0U_{\icate}(E)\longrightarrow U_{\icatd}F(E)  \]
is an equivalence is stable under geometric realizations of $U_{\icate}$-split simplicial objects. Since every object of $E$ is a geometric realization of a $U_{\icate}$-split simplicial object for which each term is of the form $L_{\icate}(E_0)$ for some $E_0\in\icate_0$, we conclude.
\end{proof}

Let us now detail the proof of Theorem \ref{thm:main-lifting-lemma} for $n=2$, that is, for the $(\infty,3)$-category $\umor_2(\icat)$; the proof for the general case will also involve an inductive procedure and will be given at the end of this section. Let $\F:M\rightarrow N$ be a 2-morphism in $\umor_2(\icat)$ between 1-morphisms $M,N:A\rightarrow B$ for $A$ and $B$ constructible factorization algebras on $\square^2$. Suppose that the underlying $1$-morphism $\pi_*(\F)$ in $\umor_1(\icat)$ admits a left adjoint $\pi_*(\F)^L$. Write $\icatd=\Hom^1_{\umor_2(\icat)}(A,B)$ and $\icate=\umor_1(\icat)$, then according to Lemma \ref{lem:adjointcriterion}, we must show that for every map $h:Z'\rightarrow Z$ in $\Hom^1_{\umor_2(\icat)}(A,B)$, the top square in the diagram
\begin{equation}\label{eq:homcube}
\adjustbox{scale=0.75}{
\begin{tikzcd}
& \Hom_{\icatd}(Z,M) \arrow[rr,"\F\circ "] \arrow[dd] \ar[dl,"\circ h"'] & & \Hom_{\icatd}(Z,N) \arrow[dl,"\circ h"]\arrow[dd] \\
\Hom_{\icatd}(Z',M)  \arrow[rr, crossing over,"\F\circ "{xshift=12pt}] \arrow[dd] & & \Hom_{\icatd}(Z',N)   \\
& \Hom_{\icate}(\pi_*(Z),\pi_*(M))\arrow[dl,"\circ \pi_*(h) "'] \arrow[rr,"\pi_*(\F)\circ "{xshift=-18pt}] & & \Hom_{\icate}(\pi_*(Z),\pi_*(N)) \arrow[dl,"\circ  \pi_*(h)"]\\
\Hom_{\icate}(\pi_*(Z'),\pi_*(M))\arrow[rr,"\pi_*(\F)\circ "] & & \Hom_{\icate}(\pi_*(Z'),\pi_*(N))\arrow[from=uu, crossing over]\\
\end{tikzcd}}
\end{equation}
among \infcats is horizontally left adjointable, where the downwards maps are induced by the functor $\pi_*:\Hom^1_{\umor_2(\icat)}(A,B)\rightarrow \umor_1(\icat)$. To see that the top horizontal maps have left adjoints under the assumption that the bottom horizontal maps have left adjoints, we apply Lemma \ref{lem:adjointcons}. Since the vertical functors are induced by pushing forward pointless constructible factorization algebras along the projections $\square^2_{[1,n]}\rightarrow \square_{[n]}$ onto the second coordinate, it is straightforward to see (using the results of Sections \ref{sec:siftedcolims}) that conditions $(1)$, $(2)$ and $(3)$ of Lemma \ref{lem:adjointcons} are satisfied; in particular, all vertical functors are monadic right adjoints. Let $F^L_{Z}$ and $F^L_{Z'}$ be the resulting left adjoints of the two top horizontal maps, then to complete the proof, we must show that the mate transformation
\[   F^L_{Z'}\circ (\_\circ h) \longrightarrow (\_\circ h) \circ F^L_Z  \]
of the top square is an equivalence. Another application of Lemma \ref{lem:adjointcriterion} tells us that the corresponding mate transformation 
\[  (\pi_*(\F)^L\circ \_)\circ (\_\circ \pi_*(h)) \longrightarrow (\_\circ \pi_*(h)) \circ  (\pi_*(\F)^L\circ \_) \]
of the \emph{bottom} square is an equivalence. Since the vertical functors are in particular conservative, it will suffice to show that the vertical functors carry the mate of the upper square to the mate of the lower square. For this, it is enough to argue that the mates of the \emph{front} and the \emph{back} squares are equivalences; this is a consequence of the following lemma and its corollary.
\begin{lem}[Horizontal pasting of vertical mates]\label{lem:pastingofmates}
Let
\[
\begin{tikzcd}
    \icat \ar[d,"U_{\icat}"] \ar[r,"f"] & \icatd \ar[d,"U_{\icatd}"] \ar[r,"h"] & \icate \ar[d,"U_{\icate}"]\\
    \icat'\ar[r,"f'"] &  \icatd' \ar[r,"h'"] & \icate'
\end{tikzcd}
\]
be a composite $\Delta^2\times\Delta^1\rightarrow \catinf$ of squares of \infcatst. Suppose that the vertical functors have left adjoints $L_{\icat}$, $L_{\icatd}$ and $L_{\icate}$. Then the mate of the outer rectangle is a composite
\[ L_{\icate}\circ h'\circ  f'\longrightarrow h\circ L_{\icatd}\circ f' \longrightarrow h\circ f\circ L_{\icat}   \]
of the mates of the two squares. 
\end{lem}
This is easy to prove, and a well-known functoriality property of mates; in particular, it is a very simple case of the equivalence described in Remark \ref{rmk:mates} taking mates of naturality squares. The following is immediate.
\begin{cor}\label{cor:pastingofmates}
In the situation of Lemma \ref{lem:pastingofmates}, the following hold.
\begin{enumerate}[$(1)$]
    \item If the left square and the right square are vertically left adjointable, then the outer rectangle is vertically left adjointable.
    \item If the right square and the outer rectangle are vertically left adjointable \emph{and} the functor $h$ is conservative, then the left square is vertically left adjointable.
\end{enumerate}
\end{cor}
We will momentarily also have need of another pasting property of mates involving squares in which instead the \emph{horizontal} functors have adjoints.
\begin{lem}[Horizontal pasting of horizontal mates]\label{lem:pastingofmates2}
Let
\[
\begin{tikzcd}
    \icat \ar[d,"f"] \ar[r,"U"] & \icatd \ar[d,"h"] \ar[r,"G"] & \icate \ar[d,"k"]\\
    \icat'\ar[r,"U'"] &  \icatd' \ar[r,"G'"] & \icate'
\end{tikzcd}
\]
be a composite $\Delta^2\times\Delta^1\rightarrow \catinf$ of squares of \infcatst. Suppose that the horizontal functors have left adjoints $L$, $F$, and $L'$ and $F'$. Then the mate of the outer rectangle is a composite
\[  L'\circ F'\circ k\longrightarrow L'\circ h\circ F\longrightarrow f\circ L\circ F   \]
of the mates of the two squares. 
\end{lem}
\begin{cor}\label{cor:pastingofmates2}
In the situation of Lemma \ref{lem:pastingofmates2}, the following hold.
\begin{enumerate}[$(1)$]
    \item If the left square and the right square are horizontally left adjointable, then the outer rectangle is horizontally left adjointable.
    \item If the left square and the outer rectangle are horizontally left adjointable \emph{and} the functor $L'$ is conservative, then the right square is horizontally left adjointable.
\end{enumerate}
\end{cor}

We now show that the top square of \eqref{eq:homcube} is horizontally left adjointable in case that the bottom square and the front and back squares are left adjointable: apply $(1)$ of Corollary \ref{cor:pastingofmates} to the composite of the back and bottom square to conclude that this composed square is horizontally left adjointable. Since this square is equivalent to the composite of the top and front square, we conclude by $(2)$ of Corollary \ref{cor:pastingofmates} that the top square is horizontally left adjointable as well. \\
So, we must show that mates of the front and back squares of \eqref{eq:homcube} are equivalences, that is, for any $Z\in\Hom^1_{\umor_2(\icat)}(A,B)$, the square
\begin{equation}\label{eq:homsquare}
\begin{tikzcd}
\Hom_{\Hom^1_{\umor_2(\icat)}(A,B)}(Z,M) \arrow[r,"\F\circ "] \arrow[d,"\pi_{*(Z,M)}"']  &[2em] \Hom_{\Hom^1_{\umor_2(\icat)}(A,B)}(Z,N)\arrow[d,"\pi_{*(Z,N)}"] \\
 \Hom_{\umor_1(\icat)}(\pi_*(Z),\pi_*(M)) \arrow[r,"\pi_*(\F)\circ "] &[2em] \Hom_{\umor_1(\icat)}(\pi_*(Z),\pi_*(N))
\end{tikzcd}
\end{equation}
of \infcats is horizontally left adjointable.
\begin{rmk}\label{rmk:extendfactorizationalg}
According to Theorem \ref{thm:main-lifting-lemma} and Lemma \ref{lem:adjointcriterion}, the upper horizontal left adjoint $F^L_Z$ must be given by composing with a morphism $\F^L\in \Hom_{\icatd}(N,M)$. The horizontal left adjointability then asserts that that this morphism is carried to $\pi_*(\F)^L$ by the functor $\pi_*:\Hom_{\umor_2(\icat)}^1(A,B)\rightarrow \umor_1(\icat)$. In the language of factorization algebras, we are given a pointless constructible factorization algebra as on the left below, together with a left adjoint $\pi_*(\F)^L{}$ of the projection to second coordinate. 
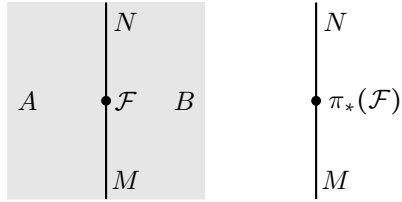
\begin{equation*}
\begin{tikzpicture}[scale=1.3]
\filldraw[opacity=0.1] (1,0) -- (0,0) -- (0,2)--(1,2);
\filldraw[opacity=0.1] (1,0) -- (2,0) -- (2,2)--(1,2);
\draw[thick] (1,0) -- (1,2);
\fill (1,1) circle (0.05);
\node at (1.2,0.2) {$M$};
\node at (1.2,1.8) {$N$};
\node at (1.2,1) {$\F$};
\node at (0.2,1) {$A$};
\node at (1.8,1) {$B$};
\end{tikzpicture} \qquad \qquad\begin{tikzpicture}[scale=1.3]
\draw[thick] (1,0) -- (1,2);
\fill (1,1) circle (0.05);
\node at (1.2,0.2) { $M$};
\node at (1.2,1.8) {$N$};
\node at (1.5,1) {$\pi_*(\F)$};
\end{tikzpicture} 
\end{equation*}
This left adjoint is another pointless constructible factorization algebra $\pi_*(\F)^L$ on $\square_{[1]}$ whose value on $(0,1/2)$ is $N$ and whose value on $(1/2,1)$ is $M$ as depicted on the left below, and the aforementioned results assert that this one-dimensional factorization algebra can be extended to a pointless constructible factorization algebra on $\square^2_{[1,1]}$ whose value on $(0,1/2)\times (0,1)$ is $A$ and whose value on $(1/2,1)\times (0,1)$ is $B$, as depicted on the right below, so that $\pi_*(\F^{L})\simeq \pi_*(\F)^L$.
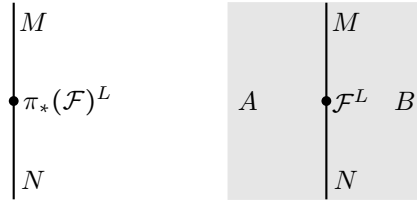
\begin{equation*}
\begin{tikzpicture}[scale=1.3]
\draw[thick] (1,0) -- (1,2);
\fill (1,1) circle (0.05);
\node at (1.2,0.2) {$N$};
\node at (1.2,1.8) {$M$};
\node at (1.55,1) {$\pi_*(\F)^L$};
\end{tikzpicture} 
\qquad  \qquad \begin{tikzpicture}[scale=1.3]
\filldraw[opacity=0.1] (1,0) -- (0,0) -- (0,2)--(1,2);
\filldraw[opacity=0.1] (1,0) -- (2,0) -- (2,2)--(1,2);
\draw[thick] (1,0) -- (1,2);
\fill (1,1) circle (0.05);
\node at (1.2,0.2) {$N$};
\node at (1.2,1.8) {$M$};
\node at (1.25,1) {$\F^L$};
\node at (0.2,1) {$A$};
\node at (1.8,1) {$B$};
\end{tikzpicture}
\end{equation*} 
\end{rmk}

To prove the horizontal left adjointability of the square \eqref{eq:homsquare}, we aim to apply the last criterion of Lemma \ref{lem:adjointcons}. To this end, it is convenient to have a better understanding of the left adjoints of the vertical maps of the diagram \eqref{eq:homcube}. For this, we observe an equivalence 
\[  \umor_1(\icat)\simeq \Hom^1_{\umor_2(\icat)}(\mathbb{1}_{\icat},\mathbb{1}_{\icat}) = \Omega_{\mathbb{1}_{\icat}}(\umor_2(\icat))  \]
of $(\infty,2)$-categories (the higher Morita categories form a \emph{categorical spectrum}). Assume for simplicity that $B=\mathbb{1}_{\icat}$; the general case amounts to simply repeating the argument below. The equivalence $\umor_1(\icat)\simeq \Hom^1_{\umor_2(\icat)}(\mathbb{1}_{\icat},\mathbb{1}_{\icat})$ then entails that the functor
\[ \pi_*: \Hom^1_{\umor_2(\icat)}(A,\mathbb{1}_{\icat}) \longrightarrow  \umor_1(\icat)  \simeq \Hom^1_{\umor_2(\icat)}(\mathbb{1}_{\icat},\mathbb{1}_{\icat}) \]
is given by $\_\circ {}_{\bb{1}_{\icat}}A_{A}$, precomposing with the right regular factorization bimodule induced by the map $\eta_A:\mathbb{1}_{\icat}\rightarrow A$ of $\bb{E}_2$-algebras\footnote{Recall that technically, objects in the univalent completion $\umor_n(\icat)$ are not locally constant factorization algebras on $\R^n$ on the nose, like in $\mor_n(\icat)$. However, every object of $\umor_n(\icat)$ is equivalent to one in the image of the completion $\mor_n(\icat)\rightarrow \umor_n(\icat)$.}. The 1-morphism ${}_{\bb{1}_{\icat}}A_{A}$ has a right adjoint ${}_AA_{\bb{1}_{\icat}}$, the left regular bimodule ($(4)$ and $(5)$ of Proposition \ref{prop:regularbimodajoint}). It follows that the functor $\_\circ {}_{\bb{1}_{\icat}}A_A$ has a \emph{left} adjoint $\_\circ {}_AA_{\bb{1}_{\icat}}$. For any $Z\in \Hom^1_{\umor_2(\icat)}(A,\bb{1}_{\icat})$, the counit 2-morphism $\epsilon_Z:(\_\circ {}_AA_{\bb{1}_{\icat}})\circ ( \_\circ {}_{\bb{1}_{\icat}}A_{A})(Z)\rightarrow Z$ is again a right regular bimodule, induced by a map of $\bb{E}_1$-algebras (in $A$-$\bb{1}_{\icat}$-bimodules); this follows from $(3)$ of Proposition \ref{prop:regularbimodajoint}, since the counit map $\epsilon_Z$ is the horizontal composite of the counit of $\left({}_{\bb{1}_{\icat}}A_A\adj {}_{A}A_{\bb{1}_{\icat}}  \right)$ with the identity on $Z$. In fact, it's not hard to see that the counit $\epsilon_Z$ is the right regular bimodule of the action map 
\[  A\otimes \pi_*(Z)\longrightarrow Z \]
of $\E_1$-algebras in left $A$-modules (note that $\pi_*(Z)$ is just $Z$ with the left $A$-module structure forgotten). For ease of notation, we will use this identification below, but our argument only uses that $\epsilon_Z$ is a right regular bimodule. For any $W\in \umor_1(\icat)$, $(3)$ of Proposition \ref{prop:regularbimodajoint} guarantees that also the unit 2-morphism $\eta_W:W\rightarrow ( \_\circ {}_{\bb{1}_{\icat}}A_{A})\circ (\_\circ {}_AA_{\bb{1}_{\icat}})(W)$ is a right regular bimodule, in this case associated to the obvious map 
\[ \bb{1}_{\icat}\otimes W\longrightarrow \pi_*(A\otimes W) \]
of $\bb{E}_1$-algebras. Now consider the diagram \eqref{eq:homcube} for the case of $h:Z'\rightarrow Z$ being the counit $\epsilon_Z:A\otimes \pi_*(Z)\rightarrow Z$:
\begin{equation}\label{eq:homcube2}
\adjustbox{scale=0.68}{
\begin{tikzcd}
& \Hom_{\icatd}(Z,M) \arrow[rr,"\F\circ "] \arrow[dd] \ar[dl,"\circ \epsilon_Z"'] & & \Hom_{\icatd}(Z,N) \arrow[dl,"\circ \epsilon_Z ",red]\arrow[dd,red] \\
\Hom_{\icatd}(A\otimes \pi_*(Z),M)  \arrow[rr, crossing over,"\F\circ "{xshift=12pt}] \arrow[dd] & & \Hom_{\icatd}(A\otimes \pi_*(Z),N)   \\
& \Hom_{\icate}(\pi_*(Z),\pi_*(M))\arrow[dl,"\circ \pi_*(\epsilon_Z) "'] \arrow[rr,"\pi_*(\F)\circ "{xshift=-18pt}] & & \Hom_{\icate}(\pi_*(Z),\pi_*(N)) \arrow[dl,"\circ  \pi_*(\epsilon_Z) "]\\
\Hom_{\icate}(\pi_*(A\otimes \pi_*(Z)),\pi_*(M))\arrow[rr,"\pi_*(\F)\circ "] & & \Hom_{\icate}(\pi_*(A\otimes \pi_*(Z)),\pi_*(N)).\arrow[from=uu, crossing over]\\
\end{tikzcd}}
\end{equation}
Note that in this diagram, the \emph{top} square and the \emph{back} square are equivalent (but the \emph{bottom} square and the \emph{front} square are not). We wish to show that the back square is horizontally left adjointable; let its mate be denoted by $\alpha$. We claim that it suffices to show that the front square is horizontally left adjointable. To see this, we argue as follows. Since $\epsilon_Z$ is a right regular bimodule, it has a right adjoint $\epsilon_Z^R$, so it follows from Corollary \ref{cor:inducedadjointability} (applied to $\icatd^{1-op}$ and $\icate^{1-op}$!) that the left and right side squares are horizontally left adjointable. According to Lemma \ref{lem:adjointcons}, to show that the mate $\alpha$ of the back square is an equivalence, it suffices to show that the composite of the \emph{left adjoint} of the vertical red arrow with the mate $\alpha$ is an equivalence. Since the vertical red arrow is equivalent to the diagonal red arrow, this is the same as showing that the composite
\[ \alpha \circ (\_\circ \epsilon_Z^R)  \]
is an equivalence. Since the mate of the right side square is an equivalence, it suffices to show that the composite of the mate of the right side square with the mate of the back square is an equivalence. By Lemma \ref{lem:pastingofmates2}, this composite of mates is also the composite of the mates of the front square and the left side square. Since the mate of the left side square is an equivalence, we indeed reduce to proving that the mate of the front square is an equivalence. Note that this front square is again an instance of the diagram \eqref{eq:homsquare}, with $Z$ in that diagram replaced with $A\otimes \pi_*(Z)$. Write $W=\pi_*(Z)$, then front square of \eqref{eq:homcube2} can be identified with the top one in the diagram 
\[
\begin{tikzcd}
\Hom_{\icatd}(A\otimes W,M)  \arrow[r,"\F\circ "] \arrow[d,"\circ \epsilon_{A\otimes W}"]  & \Hom_{\icatd}(A\otimes W,N)  \ar[d,"\circ \epsilon_{A\otimes W}"] \\    \Hom_{\icatd}(A\otimes \pi_*(A\otimes W),M)\arrow[r,"\F\circ "]\ar[d] & \Hom_{\icatd}(A\otimes \pi_*(A\otimes  W),N) \ar[d]\\
\Hom_{\icatd}(A\otimes W,M)\arrow[r,"\F\circ "] &  \Hom_{\icatd}(A\otimes W,N) 
\end{tikzcd}
\]
in which the top vertical maps are given by precomposing with the counit $\epsilon_{A\otimes W}$ and the lower vertical maps are given by precomposing with the image under $A\otimes (\_)$ of the unit $\eta_W$ of $W$. The vertical composites are equivalences by the triangle identity, so invoking $(2)$ of Corollary \ref{cor:pastingofmates} again, it suffices to show that the lower square is horizontally left adjointable and that the lower vertical maps are conservative. This lower square can be identified with the square 
\[
\begin{tikzcd}
   \Hom_{\icate}( \pi_*(A\otimes W),\pi_*(M))\arrow[r,"\pi_*(\F)\circ "]\ar[d,"\circ \eta_W"] &[2em] \Hom_{\icate}(\pi_*(A\otimes W),\pi_*(N)) \ar[d,"\circ \eta_W"]\\
\Hom_{\icate}( W,\pi_*(M))\arrow[r,"\pi_*(\F)\circ "] &[2em]  \Hom_{\icate}( W,\pi_*(N))   
\end{tikzcd}
\]
in which the vertical functors are given by precomposing with the unit of $W$. This square is left adjointable, and the vertical functors are conservative, as they are given by precomposing with a right regular bimodule (Corollary \ref{cor:regbimodconservative}). This concludes the proof of Theorem \ref{thm:main-lifting-lemma} for the Morita $(\infty,3)$-category. 

\begin{proof}[Proof of Theorem \ref{thm:main-lifting-lemma}]
The proposition follows immediately from the conjunction of the following statements $(*_n)$ for $n\geq 2$.
\begin{enumerate}
    \item[$(*_n)$] An $n$-morphism in the higher Morita $(\infty,n+1)$-category $\umor_n(\icat)$ admits a left adjoint if and only if the underlying $(n-1)$-morphism in the $(\infty,n)$-category $\umor_{n-1}(\icat)$ admits a left adjoint.
\end{enumerate}
Let $R,S$ be objects of $\umor_n(\icat)$ and let $\pi_*$ denote the functor
\[ \Hom^1_{\umor_n(\icat)}(R,S) \longrightarrow \umor_{n-1}(\icat).  \]
This functor is equivalent to the composite
\[ \Hom^1_{\umor_n(\icat)}(R,S) \longrightarrow \Hom^1_{\umor_n(\icat)}(R,\bb{1}_{\icat}) \longrightarrow  \Hom^1_{\umor_n(\icat)}(\bb{1}_{\icat},\bb{1}_{\icat}) \simeq \umor_{n-1}(\icat) \]
induced by the maps $\bb{1}_{\icat}\rightarrow R$ and $\bb{1}_{\icat}\rightarrow S$, so it suffices to show the following.
\begin{enumerate}[$(a)$]
    \item The functor $\Hom^1_{\umor_n(\icat)}(R,S) \rightarrow \Hom^1_{\umor_n(\icat)}(R,\bb{1}_{\icat})$ induced by $\bb{1}_{\icat}\rightarrow S$ detects whether an $(n-1)$-morphism has a left adjoint.
    \item The functor $\Hom^1_{\umor_n(\icat)}(R,\bb{1}_{\icat}) \rightarrow \Hom^1_{\umor_n(\icat)}(\bb{1}_{\icat},\bb{1}_{\icat})$ induced by $\bb{1}_{\icat}\rightarrow R$ detects whether an $(n-1)$-morphism has a left adjoint.
\end{enumerate}
The proofs are essentially identical; we only do $(b)$. So assume that $S=\bb{1}_{\icat}$. Let $A,B$ be $(n-2)$-morphisms among $(n-3)$-morphisms $C,D$ among... among the objects $R$ and $\bb{1}_{\icat}$. Let $\pi_{*(A,B)}$ denote the induced functor
\[ \pi_{*(A,B)}:\Hom^{n-1}_{\umor_n(\icat)}(A,B) \longrightarrow \Hom^{n-2}_{\umor_{n-1}(\icat)}(\pi_*(A),\pi_*(B))  \]
on $(\infty,2)$-categories between $(n-1)$-morphisms and $(n-2)$-morphisms. We must show that a 1-morphism in $\Hom^{n-1}_{\umor_n(\icat)}(A,B)$ admits a left adjoint if and only if its image under $\pi_{*(A,B)}$ admits a left adjoint. We claim that it suffices to show the following.
\begin{enumerate}
    \item[$(\bullet)$] There exists an object $A'$ of $\Hom_{\umor_{n}(\icat)}^{n-2}(C,D)$ and a map $A'\rightarrow A$ of $\bb{E}_{2}$-algebras internal to $C$-$D$-bimodules such that the functor 
    \[  \pi_{*(A,B)}: \Hom_{\umor_n(\icat)}^{n-1}(A,B)\longrightarrow  \Hom_{\umor_{n-1}(\icat)}^{n-2}(\pi_*(A),\pi_*(B))  \]
    of $(\infty,2)$-categories is equivalent to the functor 
    \[  (\_) \circ {}_{A'}A_A:\Hom_{\umor_n(\icat)}^{n-1}(A,B)\longrightarrow  \Hom_{\umor_{n}(\icat)}^{n-1}(A',B)  \]
    that precomposes with the right regular bimodule ${}_{A'}A_A$.
\end{enumerate}
To see that the assertion $(\bullet)$ implies the proposition, note that repeating the argument below Remark \ref{rmk:extendfactorizationalg} verbatim allows one to conclude that a 1-morphism of $\Hom_{\umor_n(\icat)}^{n-1}(A,B)$ admits a left adjoint if and only if its image under the functor $(\_)\circ {}_{A'}A_A$ composing with the right regular bimodule admits a left adjoint. We thereby reduce to proving $(\bullet)$. This follows from repeatedly applying the following observation: let $X,Y$ be a parallel pair of $k$-morphisms of $\umor_n(\icat)$ for $0\leq k\leq n-2$ and let ${}_{X'}X_X:X'\rightarrow X$ be a right regular bimodule, inducing a functor
\[ G:=(\_)\circ {}_{X'}X_X: \Hom^{k+1}_{\umor_n(\icat)}(X,Y)\longrightarrow \Hom^{k+1}_{\umor_n(\icat)}(X',Y)\]
with \emph{left} adjoint $L:=(\_)\circ {}_{X}X_{X'}$.
Let $V,W\in\Hom^{k+1}_{\umor_n(\icat)}(X,Y)$, then the following hold.
\begin{enumerate}[(1)]
    \item The induced functor 
    \[  G_{V,W}:\Hom^{k+2}_{\umor_n(\icat)}(V,W)\longrightarrow \Hom^{k+2}_{\umor_n(\icat)}(G(V),G(W))  \]
    is equivalent to the functor
     \[  \Hom^{k+2}_{\umor_n(\icat)}(V,W)\longrightarrow \Hom^{k+2}_{\umor_n(\icat)}(LG(V),W)  \]
     induced by precomposing with the counit $LG(V)\rightarrow V$.
     \item The counit $LG(V)\rightarrow V$ is again a right regular bimodule. 
\end{enumerate}
The assertion $(1)$ is a standard immediate consequence of the definition of an adjunction. For $(2)$, we note that by $(4)$ and $(5)$ of Proposition \ref{prop:regularbimodajoint}, the counit of the adjunction $\left({}_{X'}X_X\adj {}_XX_{X'}\right)$ is a right regular bimodule. Then the map $LG(V)\rightarrow V$ is obtained by horizontally composing this counit with the identity on $V$, so it follows from  $(3)$ of Proposition \ref{prop:regularbimodajoint} that this map is also a right regular bimodule. 
\end{proof}

\subsection{$n$-dualizability in the $\E_n$-Morita category}\label{sec:n-duals}

In \cite{GS} $n$-dualizability of any $\E_n$-algebra was established by explicitly constructing unit and counit data. 
In light of the lifting-of-adjoints argument developed in Section \ref{sec:lifting_lemma}, we wish to compute the underlying bimodules of the unit and counit morphisms witnessing the top adjunctions. We revisit the constructions in \cite{GS} with this in mind. We first recall the construction of the dual of an $\E_n$-algebra. 

\begin{cons}\label{cons:duals}
Let 
\[r:(0,1)\longrightarrow(0,1),\quad\quad  x\longmapsto 1-x\] 
be the reflection of the interval. Let $\rev  := r \times \Id: (0,1) \times (0,1)^{n-1} \to (0,1) \times (0,1)^{n-1}$ be reflection in the first interval. For a locally constant factorization algebra $\F$ on $(0,1)^n$, we write $\Ss^{\rev}:=\rev_{\sharp}(\F)$ to be the locally constant factorization algebra given by the pushforward of $\F$ along $\rev$. Now consider the functions
\[f^1:(0,1)\longrightarrow (0,1),\qquad x \longmapsto 1/2\sin(\pi x),\]
and $f^0 = 1-f^1$. We will call them {\em fold maps}. For a factorization algebra $\F$ on $(0,1)^n$, let $\F^{>}$ be the factorization algebra given by the pushforward of $\F$ along $f^1\times\Id_{(0,1)^{n-1}}$. If $\F$ is locally constant, then $\F^{>}$ is constructible with respect to the stratification \[\pi_{[1,\vec{0}]}:(0,1)^n\longrightarrow P_{[1]}\times P_{[\vec{0}]}\]
and thus defines a 1-morphism in $\mor_n(\icat)$. The source of $\F^{>}$ can be identified with $\F^\rev \otimes \F$ and the target is unit $\bb{1}_{\icat}$. Similarly, let $\F^{<}$ be the pushforward of $\F$ along $f^0\times \Id_{(0,1)^{n-1}}$. For $\F$ locally constant, $\F^{<}$ is identified with a 1-morphism from $\bb{1}_{\icat}$ to $\F \otimes \F^{\rev}$ in $\mor_n(\icat)$.
\end{cons}

\begin{prop}[\cite{GS}]\label{prop:En-alg-dual}
Let $\F$ be a locally constant factorization algebra on $(0,1)^n$ valued in $\icat$. Then $\F$ is 1-dualizable as an object in $\mor_n(\icat)$ with dual $\F^{\rev}$ and evaluation and coevaluation 1-morphisms given by
\begin{align*}
    \ev_{\F} = {\F^{>}} \qquad \text{and} \qquad \coev_{\F} = {\F^{<}}.
\end{align*}
\end{prop}

Next, we recall the construction of the adjoints for a $k$-morphism $\F$ in $\mor_n(\icat)$. We first consider the case $n=2$.

\begin{cons}\label{cons:adjoints}
Let 
\[\inv: (0,1)^2\longrightarrow (0,1)^2,\qquad  (x,y) \longmapsto(1-x,1-y) \] denote the inversion. Let $\F$ be a locally constant factorization algebra on $(0,1)^2$. Since $\inv$ is orientation preserving, there is an equivalence $\inv_*\F \simeq \F$. In particular, $\inv_*\F$ can be identified with $\F$ by either a half-rotation clockwise or a half-rotation anticlockwise. Let $\F$ be a $1$-morphism in $\mor_2(\icat)$ from $R$ to $S$, then $\inv_*\F$ defines a 1-morphism from $\inv_*S$ to $\inv_*R$. Using the half-rotation clockwise convention one can identify $\inv_*\F$ with a $k$-morphism from $S$ to $R$, which we denote by $\F^{(1/2)}$. Similarly, using the half-rotation anticlockwise convention one can identify $\inv_*\F$ with a $k$-morphism from $S$ to $R$, which we denote by $\F^{(-1/2)}$. Pictorially, $\F$ is a constructible factorization algebra on the stratified space
$$
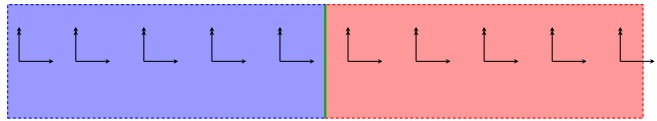
\begin{tikzpicture}[>={Stealth[length=0.5mm, width=0.5mm]},scale=1.5]
\draw[red,dash pattern=on 1pt off 1pt] (0,0) -- (2.8,0) -- (2.8,1) -- (0,1)--cycle;
\filldraw[red,opacity=0.4] (0,0) -- (2.8,0) -- (2.8,1) -- (0,1)--cycle;

\draw[->] (0.2,0.5) --  + (0:{0.3});
\draw[->] (0.8,0.5) --  + (0:{0.3});
\draw[->] (1.4,0.5) --  + (0:{0.3});
\draw[->] (2,0.5) --  + (0:{0.3});
\draw[->] (2.6,0.5) --  + (0:{0.3});

\draw[->>] (0.2,0.5) --  + (90:{0.3});
\draw[->>] (0.8,0.5) --  + (90:{0.3});
\draw[->>] (1.4,0.5) --  + (90:{0.3});
\draw[->>] (2,0.5) --  + (90:{0.3});
\draw[->>] (2.6,0.5) --  + (90:{0.3});

\begin{scope}[xshift=-2.8cm]
\draw[blue,dash pattern=on 1pt off 1pt] (0,0) -- (2.8,0) -- (2.8,1) -- (0,1)--cycle;
\filldraw[blue,opacity=0.4] (0,0) -- (2.8,0) -- (2.8,1) -- (0,1)--cycle;

\draw[->] (2.4,0.5) --  + (0:{0.3});
\draw[->] (1.8,0.5) --  + (0:{0.3});
\draw[->] (1.2,0.5) --  + (0:{0.3});
\draw[->] (0.6,0.5) --  + (0:{0.3});
\draw[->] (0.1,0.5) --  + (0:{0.3});

\draw[->>] (2.4,0.5) --  + (90:{0.3});
\draw[->>] (1.8,0.5) --  + (90:{0.3});
\draw[->>] (1.2,0.5) --  + (90:{0.3});
\draw[->>] (0.6,0.5) --  + (90:{0.3});
\draw[->>] (0.1,0.5) --  + (90:{0.3});
\end{scope}

\draw[thick,green!70!black] (0,0)--(0,1);
\end{tikzpicture}$$
where the standard framing from $\R^2$ is depicted. The blue region corresponds to the factorization algebra $R$ and the red region corresponds to the factorization algebra $S$. The 1-morphism $\F^{(1/2)}$ is the corresponding constructible factorization algebra on
$$
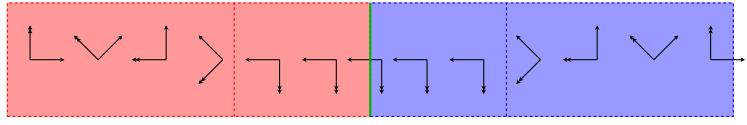
\begin{tikzpicture}[>={Stealth[length=0.5mm, width=0.5mm]},scale=1.5]
\draw[red,dash pattern=on 1pt off 1pt] (0,0) -- (3.2,0) -- (3.2,1) -- (0,1)--cycle;
\filldraw[red,opacity=0.4] (0,0) -- (3.2,0) -- (3.2,1) -- (0,1)--cycle;

\draw[->] (0.2,0.5) --  + (0:{0.3});
\draw[->] (0.8,0.5) --  + (45:{0.3});
\draw[->] (1.4,0.5) --  + (90:{0.3});
\draw[->] (1.9,0.5) --  + (135:{0.3});
\draw[->] (2.4,0.5) --  + (180:{0.3});
\draw[->] (2.9,0.5) --  + (180:{0.3});

\draw[->>] (0.2,0.5) --  + (90:{0.3});
\draw[->>] (0.8,0.5) --  + (135:{0.3});
\draw[->>] (1.4,0.5) --  + (180:{0.3});
\draw[->>] (1.9,0.5) --  + (225:{0.3});
\draw[->>] (2.4,0.5) --  + (270:{0.3});
\draw[->>] (2.9,0.5) --  + (270:{0.3});

\begin{scope}[xshift=3.2cm]
\draw[blue,dash pattern=on 1pt off 1pt] (0,0) -- (3.2,0) -- (3.2,1) -- (0,1)--cycle;
\filldraw[blue,opacity=0.4] (0,0) -- (3.2,0) -- (3.2,1) -- (0,1)--cycle;

\draw[->] (3,0.5) --  + (0:{0.3});
\draw[->] (2.5,0.5) --  + (45:{0.3});
\draw[->] (2,0.5) --  + (90:{0.3});
\draw[->] (1.5,0.5) --  + (135:{0.3});
\draw[->] (1,0.5) --  + (180:{0.3});
\draw[->] (0.5,0.5) --  + (180:{0.3});
\draw[->] (0.1,0.5) --  + (180:{0.3});

\draw[->>] (3,0.5) --  + (90:{0.3});
\draw[->>] (2.5,0.5) --  + (135:{0.3});
\draw[->>] (2,0.5) --  + (180:{0.3});
\draw[->>] (1.5,0.5) --  + (225:{0.3});
\draw[->>] (1,0.5) --  + (270:{0.3});
\draw[->>] (0.5,0.5) --  + (270:{0.3});
\draw[->>] (0.1,0.5) --  + (270:{0.3});
\end{scope}

\draw[red,dash pattern=on 1pt off 1pt] (2,0) -- (2,1);
\draw[blue,dash pattern=on 1pt off 1pt] (4.4,0) -- (4.4,1);
\draw[thick,green!70!black] (3.2,0)--(3.2,1);
\end{tikzpicture}$$
The inner third of the diagram is $\inv_*\F$ which is indicated by the inverted page framing. The rotation of the framings in the left and right thirds of the diagram indicate the clockwise identification of $\inv_*R$ and $\inv_*S$ with $R$ and $S$. Note that the clockwise convention requires the framing to rotate clockwise as it goes from $\inv_*S$ to $S$ and similarly for $R$.
\end{cons}

\begin{cons}
For $\F$ a $k$-morphism from $R$ to $S$ in $\mor_n(\icat)$, there is an analogous story where $\inv$ is applied to the $k$-th and $(k+1)$-st coordinates. The rotations identifying $\inv_*R$ with $R$ and $\inv_*S$ with $S$ are also applied to the $k$-th and $(k+1)$-st coordinates. Let $\F^{(1/2)}$  and $\F^{(-1/2)}$ denote the corresponding $k$-morphisms from $S$ to $R$, again using a clockwise and anticlockwise convention respectively.
\end{cons}

\begin{nota}
For $m$ an odd integer, we use the notation $\F^{(m/2)}$ to indicate that $m/2$ clockwise rotations have been used to identify $\inv_*R$ and $\inv_*S$ with $R$ and $S$. 
\end{nota}

The morphisms $\F^{(1/2)}$ and $\F^{(-1/2)}$ are respectively the left and right adjoint for $\F$. We now recall the construction of the unit and counit $(k+1)$-morphisms.

\begin{cons}\label{cons:units}
Fix a diffeomorphism $\varphi$ of $(0,1)^2$ that in a neighborhood of the line $x=1/2$ has the form
\begin{equation}\label{eq:twist-formula}
    (x,y) \longmapsto (1/2 - 1/2(1 - x)\cos(\pi y), 1 - (1-x)(\sin(\pi y)).
\end{equation}
Pictorially, this map looks as follows:
\begin{equation*}
  \begin{tikzpicture}[baseline=0.9cm]
    \draw[dashed] (0,0) rectangle (2,2);
    \draw[thick] (1,0) -- (1,2);
    \fill[opacity=0.3] (0.7,0) rectangle (1.3,2);
    \filldraw (1,1) circle (1.4pt);
    \draw[->] (1,1) -- (1.3,1);
    \draw[->>] (1,1) -- (1,1.3);
\end{tikzpicture}
\xrightarrow{\varphi}
\begin{tikzpicture}[baseline=0.9cm]
    \draw[dashed] (0,0) rectangle (2,2);
    \draw[thick] (0.5,2) arc [start angle=180, end angle=360, x radius=0.5, y radius=1];
    \filldraw (1,1) circle (1.4pt);
    \fill[opacity=0.3] (0.3,2) arc [start angle=180, end angle=360, x radius=0.7, y radius=1.2] (1.7,2) -- (1.3,2) arc [start angle=0, end angle=-180, x radius=0.3, y radius=0.8] (0.7,2) -- (0.3,2);
    \draw[->] (1,1) -- (1,1.3);
     \draw[->>] (1,1) -- (0.7,1);
\end{tikzpicture}
\end{equation*}
Define the following variants of $\varphi$:
\begin{alignat*}{3}
    \varphi^0 &:=  \varphi \qquad &\text{and} \qquad \varphi^{1} &:= \inv \circ \varphi \circ \inv ,\\
    \psi^0 &:=  \varphi \circ \inv \qquad &\text{and} \qquad \psi^{1} &:=  \inv \circ \varphi \ .
    \end{alignat*}
We will call these the {\em bend maps}. Let $\varphi_k:(0,1)^n\to (0,1)^n$ denote $\varphi$ applied to the $k$-th and $k+1$-st coordinates and similarly for its variants.
\end{cons}
Let $\beta:[7]\hookleftarrow [1]$ be the map defined by $\beta(0)=0$ and $\beta(1)=7$, so that the associated collapse-rescale map $\collresc(\beta):(0,1)\rightarrow (0,1)$ carries the interval $[1/8,7/8]$ into $1/2$. Let $\F$ be a $k$-morphism in $\mor_n(\icat)$ for some $1\leq k < n$. Let $\collresc_{k} : (0,1)^{n} \to (0,1)^{n}$ be the collapse-rescale map $\Id^{k-1} \times \collresc(\beta)  \times \Id^{n-k}$. Note that with this choice of collapse-rescale map, the image of the hyperplane $\{x_k=\frac12\}$ under $\varrho_k\circ \varphi_k$ is contained in $\{x_k=\frac12\}$. Similarly, pushing forward $\F$ along $\varphi_k^0$, $\varphi_k^1$, $\psi_k^0$ or $\psi_k^1$, and then along the collapse-rescale map $\collresc_{k}$ produces constructible factorization algebras on $(0,1)^n$ with the standard $(k+1)$-stratification.

\begin{prop}[\cite{GS}]\label{prop:En-adjoint-data}
Let $1\leq k \leq n-1$, and let ${\F}$ be a $k$-morphism in $\mor_n(\icat)$ from $\Ss$ to $\Tt$. Then the left adjoint for $\F$ is given by $\F^{(-1/2)}$ the unit and counit are given by the $(k+1)$-morphisms $(\collresc_{k} \circ \psi^0_k)_{\sharp} \F$ and $(\collresc_{k} \circ \psi^1_k)_{\sharp} \F$ respectively. The right adjoint for $\F$ is given by $\F^{(1/2)}$ and the unit and counit are given by $(\collresc_{k} \circ \varphi^0_k)_{\sharp} \F$ and $(\collresc_{k} \circ \varphi^1_k)_{\sharp} \F$ respectively.
\end{prop}

\begin{nota}
The notation is chosen so that $\varphi$ corresponds to right adjunction data, $\psi$ corresponds to left adjunction data, a superscript $0$ corresponds to a unit, and a superscript $1$ corresponds to a counit.
\end{nota}

\begin{rmk}
In \cite{GS}, the unit and counit are constructed by pushing forward along $\varphi_k^0$, $\varphi_k^1$, $\psi_k^0$ or $\psi_k^1$ and then along the collapse-rescale maps $\collresc_{k}$ and $\collresc_{k+1}$. However, given the prescription of $\varphi$ in \eqref{eq:twist-formula}, we only require a push-forward along $\collresc_{k}$ to recover a constructible factorization algebra.
\end{rmk}

\subsection{Proof of the conjecture}\label{sec:geometric_part}

Let $\F \in \mor_{n}(\icat)$. We wish to compute the underlying bimodules of the unit and counit $n$-morphisms witnessing the $n$-dualizability for $\F$. We will do so by determining the underlying bimodule of the unit and counit $k$-morphisms witnessing the $k$-dualizability of $\F$, by induction on $k$.
Throughout we will denote, for $S\subset\{1,\ldots, n\}$, by $\pi_S:(0,1)^n \to (0,1)^S$ the projection onto the  coordinates specified by $S$.

\begin{defn}
Let $\F$ be a $k$-morphism in $\mor_{n}(\icat)$. Pushing forward $\F$ along $\pi_{k,...,n}$ gives a constructible factorization algebra on $\square_{[1]} \times \square_{[\vec{0}]}^{n-k}$, which we call the \emph{underlying bimodule} of $\F$.
\end{defn}

\begin{cons}\label{cons:right-adj-data}
For each binary word -that is, a string of 0's and 1's- of length $k \in \{1,...,n\}$, define a $k$-morphism $u^w$ in $\mor_n(\icat)$ as follows: let $u^0 := \coev_\F$ and $u^1:= \ev_\F$ as given in Proposition \ref{prop:En-alg-dual}. For any binary word $w'$ of length less than $n$, iteratively define $u^{w',0}$ and $u^{w',1}$ to be the unit and counit of the adjunction $u^{w'} \dashv (u^{w'})^R$ respectively, as constructed in Proposition \ref{prop:En-adjoint-data}.
\end{cons}

The units and counits alternatively can be obtained by first folding and bending the plane and only collapsing at the end, rather than the above iterative step. To see this,  for each binary word $w$ (i.e. a string of 0's and 1's) of length $k$, define the function $\varphi^{w}:(0,1)^k \to (0,1)^k$ as the composite
\begin{equation*}
\begin{tikzcd}
    \varphi^w: (0,1)^k \arrow[r,"f^{w_1}"] & (0,1)^k \arrow[r,"\varphi^{w_2}_{1}"] & (0,1)^k \arrow[r,"\varphi^{w_3}_{2}"] &  \cdots \cdots \arrow[r,"\varphi^{w_k}_{k-1}"] & (0,1)^k 
\end{tikzcd}
\end{equation*}
using the fold and bend maps from Construction \ref{cons:duals} and Construction \ref{cons:units}.

\begin{prop}\label{prop:equiv-of-bimods}
Let $w$ be a binary word of length $k$. The underlying bimodule of $u^w$ is equivalent to $ (\pi_{k,...,n} \circ (\varphi^w \times \Id^{n-k}))_\sharp \F $.
\end{prop}

\begin{proof}
The case $k=1$ is true by the construction of duals in Proposition \ref{prop:En-alg-dual}. We proceed by induction. Suppose $w= w',w_k$ and assume that the proposition is true for $w'$.
By definition the following diagram commutes
\begin{equation}\label{diag:comm-morse}
    \begin{tikzcd}[row sep=small,column sep=small]
	{(0,1)^k} && {(0,1)} \\
	\\
	{(0,1)^2} && {(0,1)^2 \,.} 
	\arrow["{g^w}", from=1-1, to=1-3]
	\arrow["{g^{w'}\times\Id}"', from=1-1, to=3-1]
	\arrow["{\varphi^{w_k}}"', from=3-1, to=3-3]
	\arrow["{\pi_2}"', from=3-3, to=1-3]
\end{tikzcd}
\end{equation}
It follows from the construction of adjoints in Proposition \ref{prop:En-adjoint-data} that $u^w$ is given by pushing forward of $u^{w'}$ along the map
\begin{equation*}
    (0,1)^n \xrightarrow{\varphi^{w_k}_{k-1}} (0,1)^n \xrightarrow{\collresc_{k-1}} (0,1)^n \,.
\end{equation*}
Since $\pi_{k,...,n} \circ \collresc_{k-1} = \pi_{k,...,n}$, it follows that $(\pi_{k,...,n})_{\sharp}u^w = (\pi_{k,...,n} \circ \varphi^{w_k}_{k-1})_{\sharp} u^{w'}$ which is equivalent, by the inductive hypothesis and the commutative diagram \eqref{diag:comm-morse}, to $(g^{w}\times \Id^{n-k})_{\sharp}\F$.
\end{proof}

By Proposition \ref{prop:equiv-of-bimods}, to understand the unit and counit $n$-morphisms witnessing the $n$-dualizability of $\F$, we need to understand the maps $\pi_{n} \circ \varphi^w:(0,1)^n\to (0,1)$ for $w$ a binary word of length $n$. We do so inductively using Morse theory. The following proposition is the key to the inductive geometric argument.

\begin{prop}\label{prop:Morse-1}
Let $w$ be a binary word of length $k$ for $1\leq k \leq n$. The map $g^w:=  \pi_k \circ \varphi^w : (0,1)^k \to (0,1)$ is a Morse function with a single critical point at $(1/2,...,1/2)$ of index $|w|$.
\end{prop}

\begin{proof}
The case $k=1$ follows from the fact that $f^1 = 1/2\sin(\pi x)$ has a single critical point of index 1 at $1/2$. We proceed by induction. Suppose that the proposition is true for a binary word $w'$ of length $k-1$ and let $w = w',w_k$ for some $w_k \in \{0,1\}$. 
Recall the commutative diagram \eqref{diag:comm-morse}.
Since $\varphi^{w_k}$ is a diffeomorphism, the critical points of $g^w$ can only occur at points of $(0,1)^k$ for which $g^{w'}\times \Id$ is not submersive. By the inductive hypothesis, this occurs only on the line $\{(1/2,...,1/2)\}\times (0,1)$. There are two possibilities to consider: either $w_k =0$ or $w_k=1$. Suppose that $w_k=0$, then in a neighborhood of the line $(1/2,...,1/2)\times (0,1)$, the map $g^{w}$ is given by
\begin{align*}
    g^w(\mathbf{x},y) =  1-(1-g^{w'}(\mathbf{x}))\sin(\pi y)
\end{align*}
whose derivative is $(dg^{w'}(\mathbf{x})\sin(\pi y), -\pi(1-g^{w'}(\mathbf{x}))\cos(\pi y))$, which vanishes only when $y=1/2$ and $\mathbf{x}=(1/2,...,1/2)$. At this point the Hessian is given by
\begin{equation*}
    \begin{bmatrix}
         H_{g^{w'}}(1/2,...,1/2) & 0 \\
        0 & \frac{\pi^2}{2} \\
    \end{bmatrix}
\end{equation*}
which has index $|w'| = |w|$. Similarly, if $w_k=1$, then in a neighborhood of the line $(1/2,...,1/2)\times (0,1)$, the map $g^{w_k}$ is given by
\begin{align*}
    g^w(\mathbf{x},y) =  g^{w'}(\mathbf{x})\sin(\pi y)
\end{align*}
whose derivative is $(dg^{w'}(\mathbf{x})\sin(\pi y), \pi g^{w'}(\mathbf{x})\cos(\pi y))$, which vanishes only when $y=1/2$ and $\mathbf{x}=(1/2,...,1/2)$. At this point the Hessian is given by
\begin{equation*}
    \begin{bmatrix}
         H_{g^{w'}}(1/2,...,1/2) & 0 \\
        0 & \frac{-\pi^2}{2} \\
    \end{bmatrix}
\end{equation*}
which has index $|w'| + 1 = |w|$.
\end{proof}

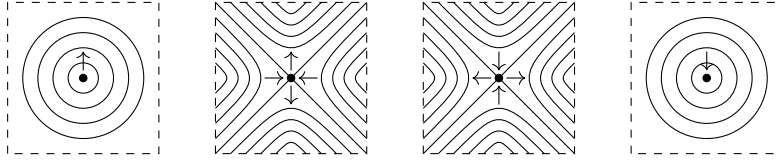
\begin{figure}
    \centering
    \begin{tikzpicture}[baseline=0.9cm]
    \draw[dashed] (0,0) rectangle (2,2);
    \filldraw (1,1) circle (1.4pt);
    \foreach \y in {1,...,4}{
    \draw (1,1) circle (\y*0.2cm);
    }
    \draw[->] (1,1.1) -- (1,1.35);
    \end{tikzpicture}
    \hspace{5mm}
    \begin{tikzpicture}[baseline=0.9cm]
    \draw[dashed] (0,0) rectangle (2,2);
    \filldraw (1,1) circle (1.4pt);
    \draw (0,2) -- (2,0);
    \foreach \y in {1,...,4}{
    \draw (0,{2-\y*0.2}) .. controls ({1-\y*0.2},1) .. (0,{\y*0.2});
    \draw (2,{2-\y*0.2}) .. controls ({1+\y*0.2},1) .. (2,{\y*0.2});
    \draw ({2-\y*0.2},2) .. controls (1,{1+\y*0.2}) .. ({\y*0.2},2);
    \draw ({2-\y*0.2},0) .. controls (1,{1-\y*0.2}) .. ({\y*0.2},0);
    }
    \draw[->] (1,1.1) -- (1,1.35);
    \draw[->] (1,0.9) -- (1,0.65);
    \draw[->] (0.65,1) -- (0.9,1);
    \draw[->] (1.35,1) -- (1.1,1);
    \draw (0,0) -- (2,2);
    \end{tikzpicture}
    \hspace{5mm}
    \begin{tikzpicture}[baseline=0.9cm]
    \draw[dashed] (0,0) rectangle (2,2);
    \filldraw (1,1) circle (1.4pt);
    \foreach \y in {1,...,4}{
    \draw (0,{2-\y*0.2}) .. controls ({1-\y*0.2},1) .. (0,{\y*0.2});
    \draw (2,{2-\y*0.2}) .. controls ({1+\y*0.2},1) .. (2,{\y*0.2});
    \draw ({2-\y*0.2},2) .. controls (1,{1+\y*0.2}) .. ({\y*0.2},2);
    \draw ({2-\y*0.2},0) .. controls (1,{1-\y*0.2}) .. ({\y*0.2},0);
    }
    \draw[<-] (1,1.1) -- (1,1.35);
    \draw[<-] (1,0.9) -- (1,0.65);
    \draw[<-] (0.65,1) -- (0.9,1);
    \draw[<-] (1.35,1) -- (1.1,1);
    \draw (0,2) -- (2,0);
    \draw (0,0) -- (2,2);
    \end{tikzpicture}
    \hspace{5mm}
    \begin{tikzpicture}[baseline=0.9cm]
    \draw[dashed] (0,0) rectangle (2,2);
    \filldraw (1,1) circle (1.4pt);
    \foreach \y in {1,...,4}{
    \draw (1,1) circle (\y*0.2cm);
    }
    \draw[<-] (1,1.1) -- (1,1.35);
    \end{tikzpicture}
    \caption{The four Morse foliations in a neighborhood of $(1/2,1/2)$ corresponding, from left to right, to the functions $g^{0,0}$, $g^{1,0}, g^{0,1}$ and $g^{1,1}$.}
    \label{fig:Morse-fols}
\end{figure}

This brings us to a preliminary characterization of $(n+1)$-dualizability in the higher Morita category~$\umor_{n}(\icat)$.

\begin{prop}\label{prop:first-n+1-dual}
Let $\F \in \mor_{n}(\icat)$, then $\F$ is $(n+1)$-dualizable in $\umor_{n}(\icat)$ if and only if for all binary words of length $w$, the 1-morphism in $\mor_{1}(\icat)$ given by $g^w_{\sharp}\F$ is left adjointable.
\end{prop}

\begin{proof}
In \cite{Araujo-thesis, SS25}, various characterizations of $(n+1)$-dualizability in an $(\infty, n+1)$-category are given in terms of existence of one-sided adjoints. 
In our situation, the results of \cite{Araujo-thesis} imply that a locally constant factorization algebra $A$ is $(n+1)$-dualizable in $\umor_n(\icat)$ if and only if $u^w$ admits a left adjoint for all binary words $w$ of length $n$.
Applying Theorem \ref{thm:main-lifting-lemma}, $u^w$ admits a left adjoint in $\umor_n(\icat)$ if and only if the underlying bimodule of $u^w$ admits a left adjoint in $\mor_1(\icat)$. By Proposition \ref{prop:equiv-of-bimods}, the underlying bimodule of $u^w$ is equivalent to $g^w_{\sharp}\F$.
\end{proof}

\begin{rmk}
Whether the underlying bimodule of $u^w$ admits a left adjoint depends only on the number of 1's in $w$, since by the Morse lemma, all Morse functions with a given index are equivalent to the standard Morse function in a neighborhood of the critical point.
\end{rmk}

We will now use Lurie's equivalence
\begin{equation}\label{eqn:Lurie_E_n}
   \alg_{\E_n}(\icat) \xrightarrow{\simeq} \Fact^{\mathrm{cstr}}_{\mathbb{R}^n}(\icat) \,, 
\end{equation}
sending an $\E_n$-algebra $A$ to the locally constant factorization algebra $U \mapsto \int_U A$; see \cite[Section 5.4.2]{LurHA} (which may be regarded as a special case of Corollary \ref{cor:fact=alg}). Using this to translate $\F$ to an $\E_n$-algebra, the respective bimodules will arise as the factorization homology of the Morse handles we just identified.
That is, we now interpret the bimodules $g^{w}_{\sharp}\F$ in terms of the factorization homology of $A$.

Moreover, given a 1-morphism $\M$ in $\mor_1(\icat)$, the underlying left module of $\M$ can be extracted as follow: first construct a  constructible factorization algebra on $(0,1/2]$ by pushing forward $\M$ along 
\[c_L : (0,1) \longrightarrow (0,1/2], \quad t \longmapsto \min\{t,1/2\}.\]
Recall from Corollary \ref{cor:Fact-bimod} that 
\begin{equation}\label{eq:module-equiv}
    \Fact^{\mathrm{cstr}}_{((0,1/2],1/2)}(\icat) \simeq \mathsf{LMod}^{\otimes}(\icat) \,,
\end{equation}
the $\infty$-category of associative algebra objects in $\icat$ along with an unpointed left module. Thus the constructible factorization algebra $(c_{L})_{\sharp}\M$ is equivalent to an algebra in $\icat$ equipped with a left module. Similarly, the underlying right module of $\M$ is the pushforward of $\M$ along \[c_{R}: (0,1) \longrightarrow [1/2,1), \quad t \longmapsto \max\{t,1/2\},\] and is equivalent to an algebra in $\icat$ equipped with a right module.  For $\M$ an $n$-morphism in $\mor_{n}(\icat)$ the underlying left and right modules are given by taking the underlying left and right module of $(\pi_{n})_{\sharp}\M$.

For a binary word $w$ of length $k$ let $|w| = \sum_{i=1}^k w_i$ be the number of 1's in $w$.

\begin{prop}\label{prop:mod-formula}
For each binary word $w$ of length $n$ the underlying left module of $u^w$ is equivalent to $A$ with the canonical action  of the factorization homology
\begin{equation*}
    \int_{S^{|w|-1} \times \R \times \R^{n-|w|}} A \,.
\end{equation*}
The underlying right module of $u^w$ is equivalent to $A$ with the canonical action  of the factorization homology
\begin{equation*}
    \int_{\R^{|w|} \times S^{n - |w| - 1 } \times \R} A \,.
\end{equation*}
\end{prop}

\begin{proof}
By Proposition \ref{prop:equiv-of-bimods} the underlying bimodule of $u^w$ is equivalent to $(g^w)_{\sharp}\F$. By Proposition \ref{prop:Morse-1} $g^w:(0,1)^n \to (0,1)$ is a Morse function with a single critical point at $(1/2,...,1/2)$ with value $1/2$ and index $|w|$.

By the Morse Lemma there exists $\epsilon > 0$ and a diffeomorphism of $(g^{w})^{-1}(1/2-\epsilon,1/2+\epsilon)$ with an open neighborhood of $(0,...,0)$ in $\R^{|w|}\times \R^{n-|w|}$ identifying $g^{w}$ with the standard Morse function of index $|w|$. Thus, in a neighborhood of the critical value $1/2$ the underlying left module of $g^{w}_\sharp\F$ is given by the factorization homology of $A$ over $\R^{|w|}\times \R^{n-|w|}$ with left action by the factorization homology of $A$ over  $S^{|w|-1}\times \R \times \R^{n-|w|}$, while the underlying right module is given by the factorization homology of $A$ over $\R^{|w|}\times \R^{n-|w|}$ with right action by the factorization homology of $A$ over $\R^{|w|}\times S^{n-|w|} \times \R$. The result follows since a constructible factorization algebra on $\square_{[1]}$ is determined up to equivalence by its value in an arbitrarily small open neighborhood of $1/2$.
\end{proof}

\begin{rmk}
By analyzing the Hessian of the Morse function $g^{w}$, one can identify $S^{|w|-1}$ with a sphere in the subspace $\langle x_i| w_i =1 \rangle \cong \R^{|w|}$ and $S^{n - |w| - 1 } $ with a sphere in the subspace $\langle x_i | w_i = 0 \rangle \cong \R^{n-|w|}$. The component $\R$ corresponds to the radial component of these spheres. See Figure \ref{fig:Morse-fols} for an illustration in dimension two.
\end{rmk}

\begin{rmk}
We comment on a potential source of ambiguity regarding the `canonical' module structure of $A$ over its factorization homologies. If $A$ is a locally constant factorization algebra on $\R^n$ valued in a presentably symmetric monoidal \infcat $\icat^{\otimes}$, then there is a manifest left action of $\int_{S^{k-1}}A$ on $A$, which arises as follows. The decomposition $\R^n=\R\times \R^{k-1}\times \R^{n-k}$ and the spherical coordinate system on $\R^{k-1}$ determine a homeomorphism 
\[  \R^n = \R^{k}\times \R^{n-k}\xrightarrow{\cong} \mathsf{C}(S^{k-1})\times \R^{n-k} \]
that restricts to the map $\{0\}\times \R^{n-k} \xrightarrow{\mathrm{id}}\{*\}\times \R^{n-k}$ with $*$ the cone point of $\mathsf{C}(S^{k-1})$, and to the map
\[ (u,v)\longmapsto \left(|u|,\frac{u}{|u|},v\right)\]
away from $\{0\}\times \R^{n-k}$. Taking the cone on the map $S^{k-1}\rightarrow *$ determines a continuous stratified map 
\[ p:\mathsf{C}(S^{k-1})\times \R^{n-k} \longrightarrow \mathsf{C}(*)\times \R^{n-k}\longrightarrow \mathsf{C}(*) = \R_{\geq 0} \xrightarrow{x\mapsto -x}  \R_{\leq 0} \]
where the second map is the obvious projection. The pushforward $p_{\sharp}A$ is a constructible factorization algebra on $\R_{\geq 0}$ which exhibits $A$ as a left module for $\int_{S^{k-1}}A$. However, the Morse-theoretic arguments involved in the construction of duality data for locally constant factorization algebras in the higher Morita category produce an a priori \emph{different} left $\int_{S^{k-1}}A$-module structure on $A$: let $f:\R^n\rightarrow \R$ be the standard Morse function of index $k$, that is, for $x=(x_1,\ldots,x_n)$ the map $f$ is given by the formula
\[f(x)= -x_1^2- \ldots -x^2_k + x_{k+1}^2 + \ldots +x_n^2\]
and let $q$ be the composite 
\[  q:\R^n\xrightarrow{f}\R\xrightarrow {\mathrm{min\{x,0\}}}\R_{\leq 0}, \]
then the underlying left module structure that Proposition \ref{prop:mod-formula} refers to corresponds to the constructible factorization algebra $q_{\sharp}A$ on $\R_{\leq 0}$. Fortunately, it is not hard to exhibit an equivalence of constructible factorization algebras $p_{\sharp}A\simeq q_{\sharp}A$ on $\R_{\leq 0}$; indeed, consider the continuous map $h:\R^k\times \R^{n-k}\rightarrow \R^k\times \R^{n-k}$ defined by the formula
\[  h(u,v)=\begin{cases}
     \left(-\frac{u}{|u|}f(u,v),v\right) & \text{if } f(u,v)<0, \\
    (0,v) & \text{if }f(u,v)\geq 0,  \end{cases} \]
and observe the commuting diagram 
\[
\begin{tikzcd}
\R^n \ar[dr,"q"'] \ar[rr,"h"] && \R^n \ar[dl,"p"] \\
& \R_{\leq 0}.
\end{tikzcd}
\]
One may verify that the pushforward $h_{\sharp}$ on factorization algebras descends to an endofunctor on \emph{constructible} factorization algebras equivalent to the identity, so that the desired equivalence $p_{\sharp}A\simeq q_{\sharp}A$ follows from the functoriality of pushing forward factorization algebras.
\end{rmk}

We now complete the proof of the main theorem.

\begin{proof}[Proof of Theorem \ref{thm:mainthm}]
Let $A$ be an $\E_n$-algebra and $\F$ the corresponding locally constant factorization algebra on $(0,1)^n$ given by \eqref{eqn:Lurie_E_n}. By Proposition \ref{prop:first-n+1-dual}, $\F$ is $(n+1)$-dualizable if and only if $g^w_{\sharp}\F$ admits a left adjoint as a 1-morphism in $\mor_{1}(\icat)$ for all binary words $w$ of length $n$. It follows from Proposition \ref{prop:relativetensorprod} that composition in $\mor_1(\icat)$ is equivalent to the relative tensor product of bimodules. Recall from \cite[Proposition 4.6.2.13]{LurHA} that a bimodule $\leftindex_{A} M_B$ in $\mor_1(\icat)$ admits a left adjoint if and only if $M$ is dualizable over $A$. By Proposition \ref{prop:mod-formula}, $g^w_{\sharp}\F$ admits a left adjoint if and only if $A$ is dualizable over 
\begin{equation*}
    \int_{S^{|w|-1}\times \R \times \R^{n-|w|}} A.
\end{equation*}
This is true for all binary words $w$ of length $n$ if and only if it is true for all $|w| \in \{0,...,n\}$.
\end{proof}

\subsection{A pictorial argument via excision}\label{sec:geometric_part_old}

In the proof of Theorem \ref{thm:mainthm} we keep track of the local behavior of the units and counits by analyzing the corresponding functions and using results from Morse theory.
Here, we provide a pictorial argument accompanying and illustrating the formal one of the previous section.
We showcase the geometric core of the argument by the excision of factorization homologies occurring at each step.
We again recover that the units and counits witnessing the $n$-dualizability for $A$ are given by the action of the factorization homology of $A$ over certain spheres on $A$. However, this argument doesn't keep track of the explicit actions of the these factorization homologies on $A$; this was kept track of by the local structure of the Morse functions. Nonetheless, we find this geometric argument enlightening and include it here for comparison. We revisit Proposition \ref{prop:mod-formula} and identify the appearing modules.

\begin{cons}
Fix $\epsilon \in (0,1/2)$. Given a binary word $w$ of length at most $n$, define $S^{w} :=\emptyset$ if $w$ is all zeros, otherwise define
\begin{align}
    S^w :=& \left\{(x_1,...,x_n) \in \R^{n} : 1-\epsilon < \sum_{w_i=1} x_i^2 <  1+\epsilon \right\} 
\end{align}
with canonical $n$-framings coming from $\R^n$. Let $\gamma = |w| = |\{i:w_i=1\}|$ and note that as framed $n$-manifolds $S^{w} \cong S^{\gamma-1} \times \R^{n-\gamma}\times \R$, where the last $\R$ component corresponds to the radial direction. Let $\overline{w}$ denote $w$ with zeros and ones exchanged. The manifolds $S^w$ and $S^{\overline{w}}$ are submanifolds of
\begin{align}
    D^w :=& \left\{ (x_1,...,x_n) \in \R^{n} : \sum_{w_i=0} x_i^2 < 1 + \epsilon , \sum_{w_i=1} x_i^2 < 1+\epsilon \right\} 
\end{align}
Hence, the factorization homologies $\int_{S^w} A$ and $\int_{S^{\overline{w}}} A$ have canonical actions on $\int_{D^{w}} A \simeq A$. Note that $D^{w} \cong D^{\gamma} \times D^{k-\gamma} \times \R^{n-k}$.
\end{cons}

We now give an argument identifying the algebras acting on the left and right of the underlying bimodule of $u^w$. However, this argument does not recover the actions of these algebras on $\int_{D^w} A \simeq A$, which are given in the stronger Proposition \ref{prop:mod-formula}.

\begin{prop}\label{prop:bimod-formula}
For each binary word $w$ of length at most $n$, there exists identifications of the source and target of the underlying bimodule of $u^w$ with the algebras
\begin{equation}
    \int_{S^{w}} A \qquad \text{and} \qquad {\int_{S^{\overline{w}} } A}\label{eq:bimod-formula}
\end{equation}
respectively.
\end{prop}

\begin{proof}
We proceed by induction on the length of $w$. For $w$ of length 1, the definitions of $\ev_A$ and $\coev_A$ from Proposition \ref{prop:En-alg-dual} are readily seen to agree with \eqref{eq:bimod-formula}. 

Consider the composition of the diffeomorphism $\varphi:(0,1)^2\to (0,1)^2$ of Construction \ref{cons:adjoints} used to construct the unit of the adjunction $\M \dashv \M^{R}$ for a 1-morphism $\M$ in $\mor_2(\icat)$, and the projection $\pi_2:(0,1)^2 \to(0,1)$ onto the second coordinate:
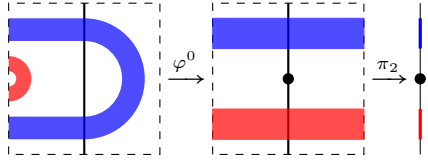
\begin{equation*}
\begin{tikzpicture}[baseline=0.9cm]
    \draw[dashed] (0,0) rectangle (2,2);
    \draw[thick] (1,2) -- (1,0);
    \fill[blue,opacity=0.7] (0,0.2) -- (1,0.2) arc[start angle = -90, end angle = 90,radius = 0.8]  
    -- (0,1.8) -- (0,1.5) -- (1,1.5) arc[start angle = 90, end angle = -90,radius = 0.5] 
    -- (0,0.5) -- (0,0.2);
    \fill[red,opacity=0.7] (0,0.7) arc[start angle = -90, end angle = 90,radius = 0.3]  
    -- (0,1.1)  arc[start angle = 90, end angle = -90,radius = 0.1]  -- (0,0.7) ;
\end{tikzpicture}
\xrightarrow{\varphi^0}
\begin{tikzpicture}[baseline=0.9cm]
    \draw[dashed] (0,0) rectangle (2,2);
    \draw[thick] (1,2) -- (1,0);
    \node[fill,circle,inner sep=1.5pt] at (1,1) {};
    \filldraw[red,opacity=0.7] (0,0.2) rectangle (2,0.6);
    \filldraw[blue,opacity=0.7] (0,1.4) rectangle (2,1.8);
\end{tikzpicture}
\xrightarrow{\pi_2}
\begin{tikzpicture}[baseline=0.9cm]
    \draw (1,2) -- (1,0);
    \node[fill,circle,inner sep=1.5pt] at (1,1) {};
    \draw[red,very thick] (1,0.2) -- (1,0.6);
    \draw[blue,very thick] (1,1.4) -- (1,1.8);
\end{tikzpicture}
\end{equation*}
The preimage of the red region lies entirely in the region to the left of the stratum $\{1/2\}\times (0,1)\subset (0,1)^2$. The source of the unit is therefore identified with the source of $\F$. The preimage of the blue region can be identified with the gluing of two regions $U_1$ and $U_2$ glued across their intersection:
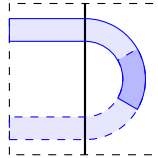
\begin{equation*}
\begin{tikzpicture}[baseline=0.9cm]
    \draw[dashed] (-1,-1) rectangle (1,1);
    \draw[thick] (0,1) -- (0,-1);
    \fill[blue,opacity=0.1] (-1,-0.8) -- (0,-0.8) arc[start angle = -90, end angle = 90,radius = 0.8]  
    -- (-1,0.8) -- (-1,0.5) -- (0,0.5) arc[start angle = 90, end angle = -90,radius = 0.5] 
    -- (-1,-0.5) -- (-1,-0.8);
    \draw[blue,dashed] (-1,-0.8) -- (0,-0.8) arc[start angle = -90, end angle = 30,radius = 0.8]  
    -- (30:{0.5}) arc[start angle = 30, end angle = -90,radius = 0.5] -- (-1,-0.5) -- (-1,-0.8);
    \draw[blue] (-1,0.8) -- (-1,0.5) -- (0,0.5) arc[start angle = 90, end angle = -30,radius = 0.5] 
    -- (-30:{0.8}) arc[start angle = -30, end angle = 90,radius = 0.8] -- (-1,0.8) ;
    \fill[blue,opacity=0.2] (30:{0.5}) arc[start angle = 30, end angle = -30,radius = 0.5] 
    -- (-30:{0.8}) arc[start angle = -30, end angle = 30,radius = 0.8] -- (30:{0.5}) ;
\end{tikzpicture}
\end{equation*}
We therefore have that the target of the unit is identified with $\M(U_1 \sqcup_{U_1 \cap U_2} U_2)$.

We now proceed to the inductive step. Fix $k<n$ and assume that Proposition \ref{prop:bimod-formula} holds for all binary words $w'$ of length $w' \leq k$. Let $w$ be a binary word of length $k+1$. There are two cases to check; either $w = w'0$ or $w=w'1$ for some word $w'$ of length $k$. By the inductive hypothesis the underlying bimodule of $u^{w'}$ is $A \simeq \int_{D^w} A$ with source and target algebras given by
\begin{equation}
    \int_{S^{w'}} A \qquad \text{and} \qquad {\int_{S^{\overline{w'}}} A}
\end{equation}
respectively.

Suppose that $w = w'0$ so that $u^w$ is the unit of the adjunction $u^{w'} \dashv (u^{w'})^R$. By the gluing arguments above the underlying bimodule of $u^w$ is $A \simeq \int_{D^w} A$ with source and target algebras
\begin{equation}
    \int_{S^{w'}} A \qquad \text{and} \qquad \int_{D^{w'} \coprod_{S^{\overline{w'}}} D^{w'}} A \simeq \int_{S^{\overline{w}}} A  \label{eq:case1-1}
\end{equation}
respectively. The last equivalence uses the diffeomorphism
\begin{equation*}
    D^{w'} \coprod_{S^{\overline{w'}}} D^{w'}\cong S^{w'0} = S^{\overline{w}}
\end{equation*}
given by decomposing $S^{w'0}$ into an upper and lower hemisphere glued across the equator $S^{w'}$.

Suppose that $w = w'1$ so that $u^w$ is the counit of the adjunction $u^{w'} \dashv (u^{w'})^R$. The counit of the adjunction is produced by pushing forward along $\inv \circ \varphi \circ \inv$ applied to the appropriate coordinates. This exchanges the roles played by the red and blue regions in the gluing argument. It follows that the underlying bimodule of $u^w$ is $A \simeq \int_{D^w} A$ with source and target algebras
\begin{equation}
     \int_{D^{w'} \coprod_{S^{w'}} D^{w'}} A \simeq \int_{S^{w}} A    \qquad \text{and} \qquad  \int_{S^{\overline{w'}}} A \label{eq:case1-3}
\end{equation}
respectively, where we have used the diffeomorphism
\begin{equation}
      D^{w'} \coprod_{S^{w'}} D^{w'} \cong S^{{w'}1}
\end{equation}
again coming from a decomposition into hemispheres.
\end{proof}

\begin{ex}
We illustrate Proposition \ref{prop:bimod-formula} in the case $n=2$. Let $\F \in \mor_2(\icat)$ and let $\ev_{\F}$ be the evaluation 1-morphism given in Proposition \ref{prop:En-alg-dual}. It is constructed using the fold map $f^1 \times \Id:(0,1)^2 \to (0,1)^2$. Let $u_{\ev}$ be the unit of the adjunction $\ev_{\F}\dashv \ev_{\F}^R$. It is constructed using the diffeomorphism $\varphi^0$ of Construction \ref{cons:units}. By Proposition \ref{prop:bimod-formula} the underlying bimodule for $u_{\ev}$ is $\int_{D^{1} \times D^{1}} A $ with source and target algebras
\begin{equation*}
    \int_{S^0 \times \R \times D^1}A \qquad \text{and} \int_{D^1 \times S^0 \times \R}A 
\end{equation*}
respectively. To illustrate these identifications, the following diagram depicts the value of the underlying bimodule of $u_{\ev}$ on standard disks in terms of the value of $\F$ on their preimage in $(0,1)^2$ -- note that we are reading the pictures from right to left in that the a region in a picture is the preimage of the picture to its right:
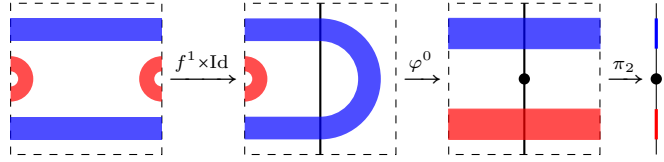
\begin{equation*}
  \begin{tikzpicture}[baseline=0.9cm]
    \draw[dashed] (0,0) rectangle (2,2);
    \fill[blue,opacity=0.7] (0,0.2) rectangle (2,0.5);
    \fill[blue,opacity=0.7] (0,1.5) rectangle (2,1.8);
     \fill[red,opacity=0.7] (0,0.7) arc[start angle = -90, end angle = 90,radius = 0.3]  
    -- (0,1.1)  arc[start angle = 90, end angle = -90,radius = 0.1]  -- (0,0.7) ;
     \fill[red,opacity=0.7] (2,0.7) arc[start angle = 270, end angle = 90,radius = 0.3]  
    -- (2,1.1)  arc[start angle = 90, end angle = 270,radius = 0.1]  -- (2,0.7) ;
\end{tikzpicture}
\xrightarrow{f^1 \times \Id}
\begin{tikzpicture}[baseline=0.9cm]
    \draw[dashed] (0,0) rectangle (2,2);
    \draw[thick] (1,2) -- (1,0);
    \fill[blue,opacity=0.7] (0,0.2) -- (1,0.2) arc[start angle = -90, end angle = 90,radius = 0.8]  
    -- (0,1.8) -- (0,1.5) -- (1,1.5) arc[start angle = 90, end angle = -90,radius = 0.5] 
    -- (0,0.5) -- (0,0.2);
    \fill[red,opacity=0.7] (0,0.7) arc[start angle = -90, end angle = 90,radius = 0.3]  
    -- (0,1.1)  arc[start angle = 90, end angle = -90,radius = 0.1]  -- (0,0.7) ;
\end{tikzpicture}
\xrightarrow{\varphi^0}
\begin{tikzpicture}[baseline=0.9cm]
    \draw[dashed] (0,0) rectangle (2,2);
    \draw[thick] (1,2) -- (1,0);
    \node[fill,circle,inner sep=1.5pt] at (1,1) {};
    \filldraw[red,opacity=0.7] (0,0.2) rectangle (2,0.6);
    \filldraw[blue,opacity=0.7] (0,1.4) rectangle (2,1.8);
\end{tikzpicture}
\xrightarrow{\pi_2}
\begin{tikzpicture}[baseline=0.9cm]
    \draw (1,2) -- (1,0);
    \node[fill,circle,inner sep=1.5pt] at (1,1) {};
    \draw[red,very thick] (1,0.2) -- (1,0.6);
    \draw[blue,very thick] (1,1.4) -- (1,1.8);
\end{tikzpicture}
\end{equation*}

Let $v_{\ev}$ be the counit of the adjunction $\ev_{\F} \dashv \ev_{\F}^R$. By Proposition \ref{prop:bimod-formula} the underlying bimodule for $v_{\ev}$ is $\int_{D^2} A \simeq A$ with source and target algebras
\begin{equation*}
    \int_{\emptyset} A \simeq \unit \qquad \text{and} \qquad \int_{S^1 \times \R} A \ .
\end{equation*}
The following diagram depicts the value of the underlying bimodule of $v_{\ev}$ in terms of the value of $\F$ on their preimage in $(0,1)^2$:
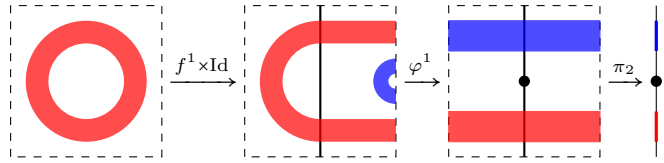
\begin{equation*}
  \begin{tikzpicture}[baseline=0.9cm]
    \draw[dashed] (0,0) rectangle (2,2);
    \fill[red, even odd rule, opacity=0.7]
    (1,1) circle (0.8)
    (1,1) circle (0.5);
\end{tikzpicture}
\xrightarrow{f^1 \times \Id}
\begin{tikzpicture}[baseline=0.9cm]
    \draw[dashed] (0,0) rectangle (2,2);
    \draw[thick] (1,2) -- (1,0);
    \fill[red,opacity=0.7] (2,0.2) -- (1,0.2) arc[start angle = 270, end angle = 90,radius = 0.8]  
    -- (2,1.8) -- (2,1.5) -- (1,1.5) arc[start angle = 90, end angle = 270,radius = 0.5] 
    -- (2,0.5) -- (2,0.2);
    \fill[blue,opacity=0.7] (2,0.7) arc[start angle = 270, end angle = 90,radius = 0.3]  
    -- (2,1.1)  arc[start angle = 90, end angle = 270,radius = 0.1]  -- (2,0.7) ;
\end{tikzpicture}
\xrightarrow{\varphi^{1}}
\begin{tikzpicture}[baseline=0.9cm]
    \draw[dashed] (0,0) rectangle (2,2);
    \draw[thick] (1,2) -- (1,0);
    \node[fill,circle,inner sep=1.5pt] at (1,1) {};
    \filldraw[red,opacity=0.7] (0,0.2) rectangle (2,0.6);
    \filldraw[blue,opacity=0.7] (0,1.4) rectangle (2,1.8);
\end{tikzpicture}
\xrightarrow{\pi_2}
\begin{tikzpicture}[baseline=0.9cm]
    \draw (1,2) -- (1,0);
    \node[fill,circle,inner sep=1.5pt] at (1,1) {};
    \draw[red,very thick] (1,0.2) -- (1,0.6);
    \draw[blue,very thick] (1,1.4) -- (1,1.8);
\end{tikzpicture}
\end{equation*}

Now compare the underlying bimodule of $u_{\ev}$ with the underlying bimodule of $v_{\coev}$, the counit of the adjunction $\coev_{\F} \dashv \coev_{\F}^R$  which is described by the diagram
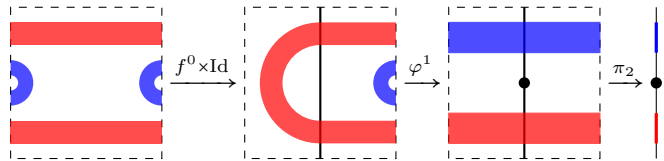
\begin{equation*}
  \begin{tikzpicture}[baseline=0.9cm]
    \draw[dashed] (0,0) rectangle (2,2);
     \fill[red,opacity=0.7] (0,0.2) rectangle (2,0.5);
    \fill[red,opacity=0.7] (0,1.5) rectangle (2,1.8);
     \fill[blue,opacity=0.7] (0,0.7) arc[start angle = -90, end angle = 90,radius = 0.3]  
    -- (0,1.1)  arc[start angle = 90, end angle = -90,radius = 0.1]  -- (0,0.7) ;
     \fill[blue,opacity=0.7] (2,0.7) arc[start angle = 270, end angle = 90,radius = 0.3]  
    -- (2,1.1)  arc[start angle = 90, end angle = 270,radius = 0.1]  -- (2,0.7) ;
\end{tikzpicture}
\xrightarrow{f^0 \times \Id}
\begin{tikzpicture}[baseline=0.9cm]
    \draw[dashed] (0,0) rectangle (2,2);
    \draw[thick] (1,2) -- (1,0);
    \fill[red,opacity=0.7] (2,0.2) -- (1,0.2) arc[start angle = 270, end angle = 90,radius = 0.8]  
    -- (2,1.8) -- (2,1.5) -- (1,1.5) arc[start angle = 90, end angle = 270,radius = 0.5] 
    -- (2,0.5) -- (2,0.2);
    \fill[blue,opacity=0.7] (2,0.7) arc[start angle = 270, end angle = 90,radius = 0.3]  
    -- (2,1.1)  arc[start angle = 90, end angle = 270,radius = 0.1]  -- (2,0.7) ;
\end{tikzpicture}
\xrightarrow{\varphi^{1}}
\begin{tikzpicture}[baseline=0.9cm]
    \draw[dashed] (0,0) rectangle (2,2);
    \draw[thick] (1,2) -- (1,0);
    \node[fill,circle,inner sep=1.5pt] at (1,1) {};
    \filldraw[red,opacity=0.7] (0,0.2) rectangle (2,0.6);
    \filldraw[blue,opacity=0.7] (0,1.4) rectangle (2,1.8);
\end{tikzpicture}
\xrightarrow{\pi_2}
\begin{tikzpicture}[baseline=0.9cm]
    \draw (1,2) -- (1,0);
    \node[fill,circle,inner sep=1.5pt] at (1,1) {};
    \draw[red,very thick] (1,0.2) -- (1,0.6);
    \draw[blue,very thick] (1,1.4) -- (1,1.8);
\end{tikzpicture}
\end{equation*}
The actions on the left and right are the reverse of those for $u_{\ev}$. This compares with the Hessian of the Morse functions  $g^{0,1}$ and $g^{1,0}$; the indices are same, but different coordinates correspond to the stable and unstable manifolds of the critical point.
\end{ex}

\subsection{Invertibility in the $\E_n$-Morita category}

We now turn our attention to \cite[Conjecture 1.9]{BJSS}, which we can deduce easily from Theorem \ref{thm:mainthm} and our proof thereof.
\begin{cor}\label{cor:invertibility}
An $\bb{E}_n$-algebra $A$ is invertible if, and only if, it is $(n+1)$-dualizable and the canonical maps
\begin{equation}
    \int_{S^{k-1}\times \mathbb{R}^{n-k+1}} A \longrightarrow Z_{n-k}(A) \label{eq:invert-maps}
\end{equation}
are equivalences for $k=0,\ldots n$.
\end{cor}

\begin{proof}
Recall that an invertible object in a symmetric monoidal $(\infty,N)$-category is $m$-dualizable for every $m$. Therefore, $(n+1)$-dualizablity of $A$ is a necessary condition for invertibility. Assuming that $A$ is $(n+1)$-dualizable, we show that invertibility of $A$ is equivalent to the invertibility of the maps \eqref{eq:invert-maps} for  $k=0,\ldots n$.

First, recall the following facts about invertibility in a symmetric monoidal $(\infty,N)$-category $\mathcal{E}$. Let $f$ be a $k$-morphism in $\mathcal{E}$, and $f^R$ a right adjoint for $f$, if the unit and counit of the adjunction $f \dashv f^R$ are invertible, then $f^R$ is an inverse for $f$. For an $(n+1)$-dualizable object $X$ in $\mathcal{E}$, it therefore suffices to check that the $m$-morphisms witnessing the adjunction data are invertible for a fixed $m \leq n+1$.

We now specialize to the case of an $\E_n$-algebra $A$ in $\mor_n(\icat)$. By the preceding paragraph, it suffices to show that the $n$-morphism $u^w$ given in Construction \ref{cons:right-adj-data} is invertible for all binary words $w$ of length $n$. By Theorem \ref{thm:main-lifting-lemma}, the $n$-morphism $u^w$ is invertible if and only if the underlying bimodule of $u^w$ is invertible as a 1-morphism in $\mor_{1}(\icat)$. By Proposition \ref{prop:mod-formula} the source and target of the underlying bimodule of $u^w$ are given by the $\E_1$-algebras
\begin{equation*}
    B_k := \int_{S^{k-1}\times \R \times D^{n-k}} A  \qquad \text{and} \qquad B_{n-k}:= \int_{D^k \times S^{n-k-1} \times \R } A .
\end{equation*}
where $k = |w|$. Moreover, under the equivalence
\begin{equation*}
    \Hom_{\mor_{1}(\icat)}(B_k,B_{n-k}) \simeq {}_{B_k}\mathsf{BMod}_{B_{n-k}}(\icat)
\end{equation*}
of Corollary \ref{cor:Fact-bimod}, $u^w$ corresponds to $A$ with its canonical $(B_k,B_{n-k})$-bimodule structure. It follows that the underlying bimodule of $u^w$ is invertible if and only if $A$ is invertible as a $(B_k,B_{n-k})$-bimodule. It is a standard result (we provide a proof for completeness in Lemma \ref{lem:invert-bimod-cond}) that, for $\E_1$-algebras $B$ and $C$, a $(B,C)$-bimodule $M$ is invertible if and only if the canonical maps
\begin{alignat*}{3}
    B &\longrightarrow \mathrm{End}_{C}(M) \qquad &\text{and} \qquad  C^{rev} &\longrightarrow  \mathrm{End}_{B}(M) 
\end{alignat*}
are equivalences. Applying this result to the case at hand, we see that $A$ is invertible as a $(B_k,B_{n-k})$-bimodule if and only if the canonical maps
\begin{equation*}
    B_k \longrightarrow \mathrm{End}_{B_{n-k}}(A) \qquad B_{n-k}^{rev} \longrightarrow \mathrm{End}_{B_{k}}(A)
\end{equation*}
are equivalences. Since this must be true for $k=0,...,n$, we recover the desired criteria.
\end{proof}

\begin{lem}\label{lem:invert-bimod-cond}
Let $B$ and $C$ be $\E_1$-algebras in $\icat$. Let $_{B}M_C \in {}_B\mathsf{BMod}_C(\icat)$ be a left dualizable $(B,C)$-bimodule. Then $_{B}M_C$ is invertible as an $(B,C)$-bimodule if and only if the maps
\begin{alignat*}{3}
    \theta_{B} : B &\longrightarrow \mathrm{End}_{C}(M) \qquad &\text{and} \qquad  \theta_{C}: C^{rev} &\longrightarrow \mathrm{End}_{B}(M)
\end{alignat*}
are invertible.
\end{lem}

\begin{proof}
Since $M$ is left dualizable there exists a $(C,B)$-bimodule $N$ and bimodule morphisms
\begin{align*}
    &\varepsilon  : \ {}_{C}N_B \otimes_{B} {}_BM_{C} \ \longrightarrow \  {}_{C}C_{C} \\
    &\eta : \  {}_{B}B_{B} \ \longrightarrow \  {}_{B} M_C \otimes_C {}_CN_{B} 
\end{align*}
satisfying the usual snake identities. The map
\begin{align*}
    \Psi : \  {}_{B}M_C \otimes_C {}_CN_B \ &\longrightarrow \ \mathrm{End}_{C}(M)
\end{align*}
adjoint to 
\[  {}_BM_C\otimes_C {}_CN_B \otimes_B {}_BM_C \xrightarrow{\mathrm{id}\otimes \epsilon} {}_BM_C  \]
is an equivalence of $(B,B)$-bimodules. The snake identity of the adjunction ensure that the diagram 
\begin{equation*}
    \begin{tikzcd}
	{{}_{B}B_{B}} && {{}_{B}M \otimes_{C} N_B} \\
	\\
	&& {\mathrm{End}_{C}(M)}
	\arrow["\eta", from=1-1, to=1-3]
	\arrow["{\theta_{B}}"', from=1-1, to=3-3]
	\arrow["\Psi", from=1-3, to=3-3]
\end{tikzcd}
\end{equation*}
commutes up to equivalence. Hence, $\eta$ is an equivalence if and only if $\theta_B$ is an equivalence. Similarly, one can show that $\varepsilon$ is an equivalence if and only if $\theta_{C}$ is an equivalence. The result follows since $M$ is invertible if and only if $\varepsilon$ and $\eta$ are equivalences.
\end{proof}

\section{Generalizing to relative dualizability}\label{sec:relative}

In this section we generalize our main theorem, Theorem \ref{thm:mainthm}, to a relative version.

\subsection{Statement of the relative dualizbility theorem}
A 1-morphism $M:A\to B$ in $\mor_n(\icat)$ is a constructible (pointless for $n=1$) factorization algebra on the stratified disk $\square_{[1]} \times \square^{n-1}$, or, equivalently due to Corollary \ref{cor:fact=alg}, a pointless disk algebra; that is, an algebra for $\infdisk^{\otimes}_{/(X,E)}$.

As a stratified space $\square_{[1]} \times \square^{n-1}$ is equivalent to $\R^n$ with codimension one stratum $\{0\} \times \R^{n-1}$ separating two codimension zero strata. We will label the codimension zero stratum $\R_{<0} \times \R^{n-1}$ by $A$ and the codimension zero stratum $\R_{>0} \times \R^{n-1}$ by $B$:
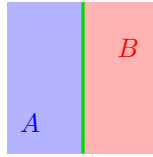
$$
\begin{tikzpicture}[scale=1, baseline=(current bounding box.base)]
\fill[color=blue,opacity=0.3] (-1,0) -- (0,0) -- (0,2) -- (-1,2) -- cycle;
\fill[color=red,opacity=0.3] (0,0) -- (1,0) -- (1,2) -- (0,2) -- cycle;
\draw[very thick, green!85!black] (0,0) -- (0,2);
\draw (-0.7, 0.4) node[color=blue] {$A$} ;
\draw (0.6, 1.4) node[color=red] {$B$} ;
\end{tikzpicture}
$$

We  want to assign values in $\mathcal{C}$ to certain $n$-framed manifolds $X$ equipped with a codimension 1-stratum $X_1 \subset X$ and a labeling of the connected components of $X \backslash X_1$ by elements of the set $\{A,B\}$, such that every point $x\in X_1$ admits a neighborhood diffeomorphic to $\square_{[1]} \times \square^{n-1}$ as stratified disks. For simplicity, we call such manifolds \textit{$n$-framed manifolds with interface}.

In dimension $n>1$ pointless disk algebras are the same as disk algebras, hence we can evaluate the factorization homology for smooth conical manifolds with tangential structures from \cite{AFT-fh-stratified} with coefficient $M:A\to B$ on any $n$-framed manifold with interface.

For $n=1$ the only appearing manifold with interface is the interval with a single marked point to which we assign $M$.

\begin{nota}
Given a 1-morphism $M:A\to B$ in $\mor_n(\icat)$ and an $n$-framed manifold $\widetilde{X}$ with interface, denote by
\begin{equation*}
    \int_{\widetilde{X}} M_{A\to B}
\end{equation*}
the factorization homology of $M$ on $\widetilde{X}$ for $n>1$ and $M$ for $n=1$ and $\widetilde{X}=\square_{[1]}$.
\end{nota}

We now define $n$-framed manifolds  with interface required to describe dualizability conditions for a 1-morphism in $\mor_n(\icat)$. 

\begin{defn}
    \label{cons:stratified-spheres}
For $1\leq k \leq n$, let $\widetilde{\R}^k$ denote $\R^k$ with the structure of a $k$-framed manifold with interface given by the codimension one submanifold $\{x_1=0\}$, and labeling region $\{x_1<0\}$ by $A$ and region $\{x_1>0\}$ by $B$. The submanifolds $D^k$ and $S^{k-1}$ inherit an interface from $\R^k$. Denote the corresponding $k$-framed manifolds with interface by $\widetilde{D}^{k}$ and $\widetilde{S}^{k-1}$. Fix $\delta \in (0,1/2)$. We define the following $k$-framed submanifolds of $\widetilde{D}^k$ and $\widetilde{S}^{k-1}$:
\begin{alignat*}{3}
    \widetilde{D}^k_+ &= \widetilde{D}^k \cap \{x_1 > -\delta\} &\qquad  \widetilde{D}^k_- &= \widetilde{D}^k \cap \{x_1 < \delta \} \\
    \widetilde{S}^{k-1}_+ &= \widetilde{S}^{k} \cap \{x_1 > -\delta \} &\qquad  \widetilde{S}^{k-1}_- &= \widetilde{S}^{k-1} \cap \{x_1 < \delta \}.
\end{alignat*} 
\end{defn}

For $k=1,2,3$, the stratified manifold $\widetilde{D}^k_+$ is given by
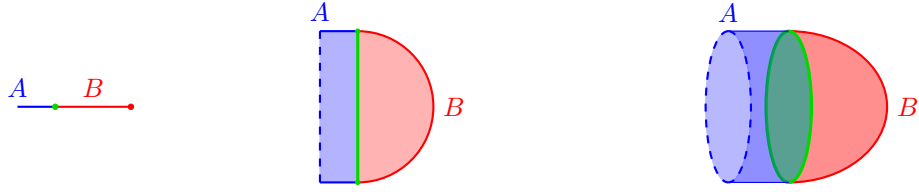
\begin{equation*}
\begin{tikzpicture}
\draw[thick,red] (1,0) -- (0,0);
\draw[thick,blue] (0,0) -- (-0.5,0);
\node[above,red] at (0.5,0) { $B$};
\node[above,blue] at (-0.5,0) { $A$};
\filldraw[red]  (1,0) circle (1pt);
\filldraw[green!85!black]  (0,0) circle (1pt);

\begin{scope}[xshift=4cm]
\filldraw[red,opacity=0.3] (0,1) arc (90:-90:1);
\draw[thick,red] (0,1) arc (90:-90:1) node[midway,right,red] {$B$} ;

\filldraw[blue,opacity=0.3] (0,-1) -- (-0.5,-1) -- (-0.5,1) -- (0,1);
\draw[thick,blue] (0,1) -- (-0.5,1) node[above,blue] {$A$};;
\draw[thick,blue] (0,-1) -- (-0.5,-1);
\draw[thick,blue,dashed] (-0.5,-1) -- (-0.5,1);

\draw[very thick, green!85!black] (0,1) -- (0,-1);
\filldraw[green!85!black]  (0,1) circle (0.8pt);
\filldraw[green!85!black]  (0,-1) circle (0.8pt);
\end{scope}

\begin{scope}[xshift=10cm]

    \filldraw[blue,opacity=0.2] (-1.1,1) -- (-0.3,1) arc (90:270:0.3 and 1) -- (-1.1,-1) arc (270:90:0.3 and 1) ;
    \filldraw[red,opacity=0.2]  (-0.3,1) arc (90:-90:1.3 and 1) arc (-90:-270:0.3 and 1);
    \draw[thick,blue,dashed]  (-1.1,1) arc (90:270:0.3 and 1);
    \draw[very thick, green!85!black] (-0.3,1) arc (90:270:0.3 and 1);

    \filldraw[green,opacity=0.5] (0,0) arc (0:360:0.3 and 1); 
    \filldraw[blue,opacity=0.1] (-0.8,0) arc (0:360:0.3 and 1);

    \filldraw[blue,opacity=0.3] (-1.1,1) -- (-0.3,1) arc (90:-90:0.3 and 1) -- (-1.1,-1) arc (-90:90:0.3 and 1) ;
    \filldraw[red,opacity=0.3]  (-0.3,1) arc (90:-90:1.3 and 1) arc (-90:90:0.3 and 1);
    \draw[thick,blue,dashed] (-1.1,1) arc (90:-90:0.3 and 1);
    \draw[blue] (-1.1,1) node[above,blue] {$A$} -- (-0.3,1) ;
    \draw[blue] (-1.1,-1) -- (-0.3,-1);

    \draw[thick,red]  (-0.3,1) arc (90:-90:1.3 and 1) node[midway,right,red] {$B$};
    \draw[very thick, green!85!black] (-0.3,1) arc (90:-90:0.3 and 1);

\end{scope}

\end{tikzpicture}
\end{equation*}

The product $\widetilde{S}^{k-1}_{\pm} \times \R^{n-k+1}$ admits the canonical structure of an $n$-framed manifold with interface. Since $\widetilde{S}^{k-1}_{\pm}$ is the boundary of $\widetilde{D}^k_{\pm}$, for $M:A\to B$ a 1-morphism in $\mor_{n}(\icat)$, the factorization homology of $M$ over $\widetilde{S}^{k-1}_{\pm} \times \R^{n-k+1}$ acts canonically on the factorization homology
\begin{align*}
    \int_{\widetilde{D}^{k}_{\pm} \times \R^{n-k}} M_{A\to B} \simeq \int_{\widetilde{\R}^n} M_{A\to B} \simeq M \ .
\end{align*}

We now state the main result of this section; we defer its proof to Section \ref{sec:relative_proof}.

\begin{thm}\label{thm:main-oplax-thm}
Let $\icat$ be a presentably symmetric monoidal \infcatt. Let $M:A\to B$ be a 1-morphism in $\mor_n(\icat)$.
 \begin{enumerate}
 \item {\em (``even'' = oplax case)}  The 1-morphism $M$ is $n$-times right adjointable in $\mor_{n}(\icat)$ (we also call this $R^n$-adjointable; see Definition \ref{def:n-times-adj}) if and only if for $0\leq k \leq n-1$ the object $M$ is right dualizable as a right module over the factorization homologies 
\[  \int_{\widetilde{S}^{k}_{+} \times \R^{n-k}} M_{A\to B} \, .\] 

\item {\em (``odd'' case)}
On the other hand, $M$ is $R^{n-1}L$-adjointable in $\mor_{n}(\icat)$ (see Definition \ref{def:n-times-adj}) if and only if for $0\leq k \leq n-1$ the object $M$ is left dualizable as a left module over the factorization homologies 
\[  \int_{\widetilde{S}^{k}_{-}\times \R^{n-k}} M_{A\to B} \, . \] 

 \end{enumerate}
\end{thm}

\begin{rmk}
Theorem \ref{thm:main-oplax-thm} implies Theorem \ref{thm:mainthm}. Since $\mor_{n}(\icat) \simeq \Hom_{\mor_{n+1}(\icat)}(\unit,\unit)$, an $\E_{n}$-algebra $A$ is equivalently an endomorphism of the unit in $\mor_{n+1}(\icat)$. Moreover, $A$ is $n$-dualizable as an $\E_{n}$-algebra if and only if it is $n$-times right dualizable as a 1-morphism in $\mor_{n+1}(\icat)$ if and only if it is $R^{n-1}L$-dualizable as a 1-morphism in $\mor_{n+1}(\icat)$. The factorization homologies for a 1-morphism $A:\unit \to \unit$ in $\mor_{n+1}(\icat)$ over the stratified half spheres satisfy
\begin{equation*}
    \int_{\widetilde{S}^{k}_{-}\times \R^{n-k+1}} A_{\unit\to \unit} \ \simeq\ \int_{S^{k-1} \times \R^{n-k+1}} A  \ \simeq \ \int_{\widetilde{S}^{k}_{+}\times \R^{n-k+1}} A_{\unit \to \unit},
\end{equation*}
where the middle integral uses the factorization homology of $A$ as an $\E_{n}$-algebra, thus recovering Theorem \ref{thm:mainthm}.

\end{rmk}

\begin{ex}
Consider first the case that $n=1$, then $M:A\to B$ is a constructible factorization algebra on $\square_{[1]}$. The stratified manifold $\tilde{S}^{0}_+$ is the point $1\in \tilde{\R}$ which is labeled by $B$, and ${D}^{0}_+$ is the half-closed interval $(-\delta,1] \subset \widetilde{\R}$. The action of the factorization homology $\int_{\tilde{S}^{0}_+ \times \R} M_{A\to B} = \int_{\R} B \simeq B$ on $\int_{D^{0}_+} M_{A\to B} \simeq M$ is therefore the defining action of $B$ on $M$. The ``even'' case of Theorem \ref{thm:main-oplax-thm} corresponds to the statement that $M$ admits a right adjoint if and only if $M$ is dualizable over $A$.

Similarly the factorization homology $\int_{\tilde{S}^{0}_- \times \R} M_{A\to B} = \int_{\R} A \simeq A$ which acts canonically on $\int_{D^{0}_-} M_{A\to B} \simeq M$. The ``\text{odd}'' case of Theorem \ref{thm:main-oplax-thm} then recovers the statement that $M$ admits a left adjoint if and only if $M$ is dualizable over $A$.
\end{ex}

\begin{ex}
We give an illustration of Theorem \ref{thm:main-oplax-thm} for a 1-morphism $M:A\to B$ in $\mor_2(\icat)$. In this case the morphism $M$ admits a right adjoint and the unit and counits of the adjunction are given by pushing forward along the maps $\varphi^0$ and $\varphi^1$ of Construction \ref{cons:units} followed by a collapse-rescale map. The 1-morphism $M$ is a constructible factorization algebra on
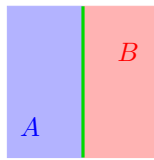
\[\begin{tikzpicture} 
\filldraw[blue,opacity=0.3] (1,0) -- (0,0) -- (0,2)--(1,2);
\filldraw[red,opacity=0.3] (1,0) -- (2,0) -- (2,2)--(1,2);
\draw[very thick, green!85!black] (1,0) -- (1,2);
\node[blue] at (0.3,0.4) {$A$};
\node[red] at (1.6,1.4) {$B$};
\end{tikzpicture}\]
The image of the regions $A$ and $B$ under $\varphi^0$ and $\varphi^1$ are given by
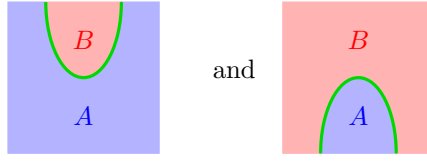
\[\begin{tikzpicture}[baseline = 1cm]
\filldraw[blue,opacity=0.3] (0.5,2) -- (0,2) -- (0,0) -- (2,0) -- (2,2) -- (1.5,2) arc (0:-180:0.5 and 1);
\filldraw[red,opacity=0.3] (0.5,2) -- (1.5,2) arc (0:-180:0.5 and 1);
\draw[very thick, green!85!black] (1.5,2) arc (0:-180:0.5 and 1);
\node[blue] at (1,0.5) {$A$};
\node[red] at (1,1.5) {$B$};
\end{tikzpicture} \qquad \text{and} \quad \begin{tikzpicture}[baseline = 1cm]
\filldraw[blue,opacity=0.3](0.5,0) -- (1.5,0)arc (0:180:0.5 and 1);
\filldraw[red,opacity=0.3]  (0.5,0) -- (0,0) -- (0,2) -- (2,2) -- (2,0) -- (1.5,0) arc (0:180:0.5 and 1);
\draw[very thick, green!85!black] (1.5,0) arc (0:180:0.5 and 1);
\node[blue] at (1,0.5) {$A$};
\node[red] at (1,1.5) {$B$};
\end{tikzpicture} \]
respectively. The source 1-morphisms of the unit and counit of the adjunction are given by the factorization homology of $M$ on
\begin{equation*}
    \begin{tikzpicture}[baseline=0.13cm]
\filldraw[blue,opacity=0.3] (0,0) rectangle (2,0.5);
\end{tikzpicture} \ \cong\  \widetilde{S}^0_{-} \times \R^2 \qquad \text{and} \qquad  
\begin{tikzpicture}[baseline=0.13cm]
\filldraw[red,opacity=0.3] (0,0) rectangle (0.5,0.5);
\filldraw[red,opacity=0.3] (1.5,0) rectangle (2,0.5);
\filldraw[blue,opacity=0.3] (0.5,0) rectangle (1.5,0.5);
\draw[very thick, green!85!black] (0.5,0) -- (0.5,0.5);
\draw[very thick, green!85!black](1.5,0) -- (1.5,0.5);
\end{tikzpicture}\  \cong \widetilde{S}^1_{-} \times \R
\end{equation*}
respectively. Whereas the target 1-morphisms are given by the factorization homology of $M$ on 
\begin{equation*}
    \  \begin{tikzpicture}[baseline=0.13cm]
\filldraw[blue,opacity=0.3] (0,0) rectangle (0.5,0.5);
\filldraw[blue,opacity=0.3] (1.5,0) rectangle (2,0.5);
\filldraw[red,opacity=0.3] (0.5,0) rectangle (1.5,0.5);
\draw[very thick, green!85!black] (0.5,0) -- (0.5,0.5);
\draw[very thick, green!85!black](1.5,0) -- (1.5,0.5);
\end{tikzpicture} \ \cong\  \widetilde{S}^1_{+} \times \R \qquad \ \text{and} \qquad  
\begin{tikzpicture}[baseline=0.13cm]
\filldraw[red,opacity=0.3] (0,0) rectangle (2,0.5);
\end{tikzpicture}\  \cong \widetilde{S}^0_{+} \times \R^2
\end{equation*}
respectively. Note that $M$ is 2-times right adjointable if and only if the unit and counit of the right adjunction admit right adjoints. Whereas, $M$ is $RL$-adjointable if and only if the unit and counit of the right adjunction admit left adjoints.
\end{ex}

\subsection{Relative field theories and partial dualizability}

In \cite{JFS}, building on the ideas of \cite{ST,FTrel}, a notion of relative (also called twisted) topological field theory was defined. We recall this notion and the corresponding classification via the Cobordism Hypothesis. In light of this classification, Theorem \ref{thm:main-oplax-thm} gives rise to relative field theories via the Cobordism Hypothesis.

Let $\mathcal{E}$ be a symmetric monoidal $(\infty,N)$-category. An $n$-dimensional framed \emph{oplax relative field theory} is a symmetric monoidal functor
\begin{align*}
    \mathcal{R} : \Bord_{n}^{\fr} \longrightarrow \mathcal{E}^\to \ ,
\end{align*}
where $\mathcal{E}^\to$ is the oplax arrow category of $\mathcal{E}$, defined in \cite{JFS}. In particular, the objects of $\mathcal{E}^\to$ are the 1-morphisms in $\mathcal{E}$. The oplax arrow category $\mathcal{E}^{\to}$ admits source and target functors $S,T\colon \mathcal{E}^\to \to \mathcal{E}$, and for any closed $k$-dimensional $n$-framed manifold $Y$  (i.e. an object in $\Omega^{k}\Bord_n^{\fr}$), an oplax relative field theory assigns a $(k+1)$-morphism
\begin{equation*}
    \mathcal{R}(Y) \colon (S\circ \mathcal{R})(Y) \longrightarrow (T\circ \mathcal{R})(Y)
\end{equation*}
in $\mathcal{E}$. In this sense, $\mathcal{R}$ is an $n$-dimensional field theory defined relative to (or twisted by) the absolute $n$-dimensional framed theories $S\circ \mathcal{R} :\Bord_{n}^{\fr} \to \mathcal{E}$ and $T\circ \mathcal{R} :\Bord_{n}^{\fr} \to \mathcal{E}$.

\begin{rmk}
Often the term \emph{relative} implies that composing with either $S$ or $T$ recovers the trivial theory. To reduce terminology, we also use relative in the two sided setting where neither composite is required to be trivial.
\end{rmk}

Recall that a 1-morphism $f$ is $n$-times right adjointable if $f$ admits a right adjoint, and inductively, for $1\leq k\leq n-1$, the units and counits witnessing the right adjunctions of the $k$-morphisms themselves admit right adjoints. In \cite{JFS}, it is shown that the $n$-dualizable objects in $\mathcal{E}^\to$ are the 1-morphisms $f:x\to y$ in $\mathcal{E}$ such that $x$ and $y$ are $n$-dualizable as objects in $\mathcal{E}$, and $f$ is $n$-times right adjointable as a 1-morphism in $\mathcal{E}$. 

Consider now the case that $\mathcal{E} = \umor_n(\icat)$. Let $M:A\to B$ be a 1-morphism in $\mor_n(\icat)$. Since $\umor_n(\icat)$ is $n$-dualizable by \cite{GS}, the requirement that $A$ and $B$ are $n$-dualizable is automatic. Applying the cobordism hypothesis therefore gives the following classification of oplax relative field theories.

\begin{prop}[\cite{JFS}]\label{prop:oplax-classification}
Assuming the Cobordism Hypothesis, an $n$-dimensional framed oplax relative field theory $\mathcal{R} : \Bord_{n}^{\fr} \to \umor_{n}(\icat)^\to$
is equivalent to an $n$-times right adjointable morphism $M$ in $\umor_{n}(\icat)$.
\end{prop}

In \cite{JFS}, a notion of \emph{lax} relative field theory was also provided. It is shown in \cite{SS25}, that this is related to the notion of oplax relative field theory by replacing the target $\mathcal{E}^{\to}$ with the target $((\mathcal{E}^{(\text{odd-op})})^{\to})^{(\text{odd-op})}$ where odd-op denotes the $(\infty,N)$-category with the direction of odd morphisms reversed. Moreover, \cite{SS25} show that the notion of oplax and lax relative field theories are equivalent for $n$ even, and there is another distinct version of relative field theory in these dimensions. We now recall the various notions of partial dualizability introduced in \cite{SS25}, which correspond to different notions of relative field theories.

\begin{defn}\label{def:dexterity-fn}
A \textit{dexterity function of length $n$} is a function
\begin{equation*}
    a^n:\{1,...,n\} \to \{L,R\}.
\end{equation*} 
\end{defn}

We often represent a dexterity function by a string; e.g. $RL$ is the dexterity function given by $1\mapsto R$ and $2\mapsto L$. For a dexterity function $a^n$ of length $n$ and $k< n$, we denote by $a^{n}_{-k}$ the dexterity function of length $(n-k)$ defined by
\begin{align*}
    a^{n}_{-k}(i) := a^n(i+k).
\end{align*}
This notation is convenient in expressing the following definition of higher adjointability.

\begin{defn}\label{def:n-times-adj}
Let $\mathcal{E}$ be an $(\infty,N)$-category. Let $a^n$ be a dexterity function of length $n \geq 1$. Let $f$ be a $k$-morphism in $\mathcal{E}$. Then we say that $f$ is \textit{$a^n$-adjointable} if
\begin{equation*}
    \begin{cases} \text{$f$ is right-adjointable},  &\text{if $a^n(1)$ =R}  \\
     \text{$f$ is left-adjointable},  &\text{if $a^n(1)$ =L} 
    \end{cases}
\end{equation*}
and if $n>1$, the unit and counit of the adjunction are themselves $a^{n}_{-1}$-adjointable.
\end{defn}

\begin{rmk}
Let $R^n$ be the constant ambidexterity function $R^n(i) = R$ of length $n$, then the notion of $R^n$-adjointable recovers the notion of $n$-times right adjointable. Similarly, letting $L^n$ be the constant ambidexterity function $L^n(i) = L$ of length $n$, then the notion of $L^n$-adjointable recovers the notion of $n$-times left adjointable.
\end{rmk}

The following theorem is one of the main results of \cite{SS25}.

\begin{thm}[\cite{SS25}]\label{lem:even-even}
Let $\mathcal{E}$ be an $(\infty,N)$-category and let $f$ be a $k$-morphism in $\mathcal{E}$. Let $a^n$ and $b^n$ be two dexterity functions of length $n$ such that
\begin{align*}
    |(a^n)^{-1}(R)| \equiv |(b^n)^{-1}(R)| \quad \text{mod 2},
\end{align*}then $f$ is $a^n$-adjointable if and only if it is $b^n$-adjointable.
\end{thm}

There are therefore only two inequivalent notions of partial $n$-dualizability for a $k$-morphism. We choose as representatives the dexterity functions $R^n$ and $R^{n-1}L$. By Proposition \ref{prop:oplax-classification} a 1-morphism $M$ in $\mor_n(\icat)$ that is $R^n$-adjointable defines an $n$-dimensional framed oplax relative field theory. Following \cite{SS25} we also refer to this case as ``even'' since the number of $L$'s in any equivalent dexterity function is always zero mod 2. In contrast, a 1-morphism $M$ in $\mor_n(\icat)$ that is $R^{n-1}L$-adjointable defines an $n$-dimensional framed \emph{odd-relative field theory} via \cite{SS25} and the Cobordism Hypothesis.

With this set-up, we can reformulate our main Theorem \ref{thm:main-oplax-thm} in terms of relative field theories.
\begin{cor}\label{cor:main-oplax-thm}
Let $\icat$ be a presentably symmetric monoidal \infcatt. Let $M:A\to B$ be a 1-morphism in $\mor_n(\icat)$. Assuming the Cobordism Hypothesis,
 \begin{enumerate}
 \item {\em (``even'' = oplax case)}  the 1-morphism $M$ gives an ``even''=oplax $n$-dimensional framed TFT if and only if for $0\leq k \leq n-1$ the object $M$ is right dualizable as a right module over the factorization homologies 
\[  \int_{\widetilde{S}^{k}_{+} \times \R^{n-k}} M_{A\to B} \,.\] 
\item {\em (``odd'' case)}
On the other hand, the 1-morphism $M$ gives an odd $n$-dimensional framed TFT if and only if for $0\leq k \leq n-1$ the object $M$ is left dualizable as a left module over the factorization homologies 
\[  \int_{\widetilde{S}^{k}_{-}\times \R^{n-k}} M_{A\to B} \,. \] 
 \end{enumerate}
\end{cor}

\subsection{Morse theory in codimension one}

The proof of Theorem \ref{thm:main-oplax-thm} will use similar arguments to the proof of Theorem \ref{thm:mainthm}. In place of a Morse function, we will use the notion of a Morse function on a manifold with interface. Morse theory for stratified manifolds is a well established subject, see for example the textbook \cite{GM}. A manifold with interface is a very simple example of a stratified manifold with a codimension one stratification coming from an embedded submanifold. The Morse theory of a manifold with interface is therefore very similar to Morse theory for manifolds with boundary developed in \cite{JR72,Bra74,Haj}. We recall the definition and main results of stratified Morse theory applied to the special case of a manifold with interface.

\begin{defn}\label{def:strat-morse}
Let $\widetilde{X}$ be an $n$-framed manifold with interface whose underlying $n$-framed manifold is $X$. A \emph{stratified Morse function} on $\widetilde{X}$ is a Morse function $f:X \to \R$ satisfying the following two additional conditions:
\begin{enumerate}
    \item The restriction $f_1:=f|_{X_1}:X_1 \to \R$ to the interface $X_1$ is a Morse function.
    \item There are no critical points of $f$ in a neighborhood of $X_1$.
\end{enumerate}
We write $\tilde{f}:\widetilde{X} \to \R$ to denote the stratified Morse function and $f:X\to \R$ to denote the underlying Morse function.
\end{defn}

The critical points of $\tilde{f}:\widetilde{X}\to \R$ come in two families. An \emph{interior critical point} of $\tilde{f}$ is a critical point of $f:X\to \R$ in the usual sense. By condition two, interior critical points belong to $X \backslash X_1$.  A \emph{codimension one critical point} of $\tilde{f}$ is a critical point of $f_1$. The codimension one critical points further separate into two distinct families. Let $x\in X_1$ be a codimension one critical point and let $\hat{n} \in T_xX$ be a outwards normal vector to region $A$ (inwards normal to region $B$). Since $f$ has no critical points in a neighborhood of $X_1$, the derivative $df_x(\hat{n})$ is non-zero. We call a codimension one critical point \emph{positive} if $df_x(\hat{n}) > 0$ and \emph{negative} if $df_x(\hat{n}) < 0$. Note that this does not depend on the specific choice of $\hat{n}$. The \emph{index} of a codimension one critical point $x\in X_1$ is the index of $f_1:X_1 \to \R$ at $x$.

The Morse lemma for codimension one critical points is a straightforward extension of the standard Morse lemma. A proof in the setting of manifolds with boundary appeared in \cite{JR72,Bra74}. 

\begin{lem}\label{lem:strat-Morse-lemma}
Let $\tilde{f}:\widetilde{X}\to \R$ be a stratified Morse function. Let $x\in \widetilde{X}$ be a codimension one critical point of index $k$. Then there exists local coordinates $(x_1,...,x_n)$ at $x$ such that $(0,x_{2},...,x_{n}) \subset X_1$, the coordinate $x_1$ is outwards normal to region $A$, and $\tilde{f}$ is given by
\begin{align*}
 \tilde{f}(x_1,...,x_n) = \tilde{f}(x) \pm x_1 - \sum_{i=2}^{k+1} x_i^2 + \sum_{i=k+2}^{n} x_i^2 ,
\end{align*}
where $\pm$ corresponds to whether $x$ is a positive ($+$) or negative ($-$) codimension one critical point.
\end{lem}

\begin{rmk}
We choose $x_1$ for the coordinate normal to the interface since the standard $n$-framing on $\widetilde{\R}^n \cong \square_{[1]} \times \square^{n-1}$ has the 1-st component of the framing normal to the interface.
\end{rmk}

\begin{rmk}
There is a second approach to Morse theory for manifolds with boundary (and therefore also interfaces) in which one drops condition (2) of Definition \ref{def:strat-morse}, and instead requires that $\nabla f$ is always tangent to $X_1$ (this requires a choice of metric on $X$). In this case, $\pm x_1$ is replaced by $\pm x_1^2$ in the local coordinate form, see \cite{BNR} for details. It will be important for us that we use Definition \ref{def:strat-morse}, since the handle attachments we require are those of \cite{GM} which correspond to the local form of Lemma \ref{lem:strat-Morse-lemma}.
\end{rmk}

Following \cite{GM} (see also \cite{Haj}), the stratified handle attachment corresponding to a positive codimension one critical point of index $\gamma$ is built as follows. Start with the product
\begin{equation*}
    \tilde{I} \times D^{\gamma} \times D^{n-\gamma-1}
\end{equation*}
where $\tilde{I} = \square_{[1]}$ is the standard stratified interval. This product is a manifold with corners, the faces are given by $\pt_{A} \times D^{\gamma} \times D^{n-\gamma-1}$, $\pt_{B} \times D^{\gamma} \times D^{n-\gamma-1}$, $\tilde{I} \times S^{\gamma-1}\times D^{n-\gamma-1}$ and $\tilde{I} \times D^{\gamma}\times S^{n-\gamma-2}$. By smoothing the corners $\pt_A \times D^{\gamma} \times S^{n-\gamma-2}$ and $\pt_B \times S^{\gamma-1}\times D^{n-\gamma-1}$, the handle attachment  is given by interpreting $\tilde{I} \times D^{\gamma} \times D^{n-\gamma}$ as a handle from 
\begin{align*}
      (\pt_{A} \times D^{\gamma} \times D^{n-\gamma-1} )  \bigcup_{\pt_A \times D^{\gamma} \times S^{n-\gamma-2}} (\tilde{I} \times D^{\gamma}\times S^{n-\gamma-2} )\ \cong  \widetilde{S}^{n-\gamma-1}_- \times D^{\gamma}
\end{align*}
to 
\begin{align*}
    ( \tilde{I} \times S^{\gamma-1}\times D^{n-\gamma-1} ) \bigcup_{\pt_B \times S^{\gamma-1}\times D^{n-\gamma-1}} (\pt_{B} \times D^{\gamma} \times D^{n-\gamma-1} )  \cong \widetilde{S}^{\gamma}_+ \times D^{n-\gamma-1} \ .
\end{align*}
The identifications on the right are given by applying the canonical diffeomorphisms coming from gluing the product $\tilde{I} \times S^k$ with a disc $D^k$ for some $k$, either on the right (the disk is labeled by $B$), or on the left (the disc is labeled by $A$):
\begin{equation*}
     (\pt_A \times D^{n-\gamma -1}) \bigcup_{\pt_A \times S^{n-\gamma-2}} (\tilde{I} \times S^{n-\gamma-2})  \cong \widetilde{S}^{n-\gamma-1}_-
\end{equation*}
and 
\begin{align*}
    ( \tilde{I} \times S^{\gamma-1} ) \bigcup_{\pt_B \times S^{\gamma-1}} (\pt_{B} \times D^{\gamma}  )  \cong \widetilde{S}^{\gamma}_+ \ .
\end{align*}
For $\gamma=1,2$, the compositions $( \tilde{I} \times S^{\gamma-1} ) \bigcup_{\pt_B \times S^{\gamma-1}} (\pt_{B} \times D^{\gamma}  )$ are given by
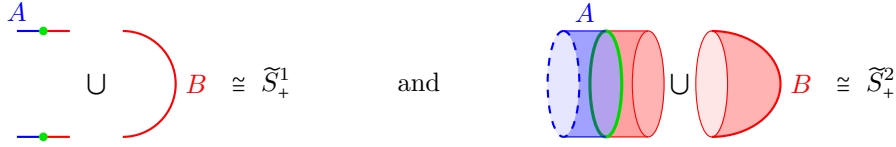
\begin{equation*}
\begin{tikzpicture}[scale=0.7,baseline=-0.1cm]
\draw[thick,red] (1.5,1) arc (90:-90:1) node[midway,right,red] {$B$};

\node at (1,0) {$\bigcup$};

\draw[thick,blue] (0,1) -- (-0.5,1)  node[above,blue] {$A$};
\draw[thick,blue] (0,-1) -- (-0.5,-1);
\draw[thick,red] (0,1) -- (0.5,1);
\draw[thick,red] (0,-1) -- (0.5,-1);

\node[
    draw=green!90!black,
    fill=green!90!black,
    circle,
    inner sep=1pt
  ] at (0,1) {};
  \node[
    draw=green!90!black,
    fill=green!90!black,
    circle,
    inner sep=1pt
  ] at (0,-1) {};
\end{tikzpicture} \ \cong \ \widetilde{S}^{1}_+ \qquad \qquad \text{and} \qquad \qquad 
\begin{tikzpicture}[scale=0.7,baseline=-0.1cm]
    \draw[thick,blue,dashed]  (-1.1,1) arc (90:270:0.3 and 1);
    \draw[very thick, green!85!black] (-0.3,1) arc (90:270:0.3 and 1);
    \draw[red]  (0.5,1) arc (90:270:0.3 and 1);

    \filldraw[blue,opacity=0.1] (-1.1,1) -- (-0.3,1) arc (90:270:0.3 and 1) -- (-1.1,-1) arc (-90:-270:0.3 and 1) ;
    \filldraw[red,opacity=0.1] (0.5,1) -- (-0.3,1) arc (90:270:0.3 and 1) -- (0.5,-1) arc (-90:-270:0.3 and 1) ;
    
    \filldraw[blue,opacity=0.3] (-1.1,1) -- (-0.3,1) arc (90:-90:0.3 and 1) -- (-1.1,-1) arc (-90:90:0.3 and 1) ;
    \filldraw[red,opacity=0.3] (0.5,1) -- (-0.3,1) arc (90:-90:0.3 and 1) -- (0.5,-1) arc (-90:90:0.3 and 1) ;
    \draw[red]  (0.5,1) arc (90:-90:0.3 and 1);

    \draw[thick,blue,dashed] (-1.1,1) arc (90:-90:0.3 and 1);
    \draw[red] (0.5,1) -- (-0.3,1);
    \draw[red] (0.5,-1) -- (-0.3,-1);
    \draw[blue] (-1.1,1) -- (-0.3,1) node[midway,above,blue] {$A$};;
    \draw[blue] (-1.1,-1) -- (-0.3,-1);
    \draw[very thick, green!85!black] (-0.3,1) arc (90:-90:0.3 and 1);
    
    \node at (1.1,0) {$\bigcup$};

    \filldraw[red,opacity=0.3]  (1.7,1) arc (90:-90:1.3 and 1)  arc (-90:90:0.3 and 1);
    \filldraw[red,opacity=0.1] (2,0) arc (0:360:0.3 and 1);
    \draw[thick,red]  (1.7,1) arc (90:-90:1.3 and 1) node[midway,right,red] {$B$};
    \draw[red]  (1.7,1) arc (90:-90:0.3 and 1);
    \draw[red]  (1.7,1) arc (90:270:0.3 and 1);
    
\end{tikzpicture} \ \cong \ \widetilde{S}^{2}_+
\end{equation*}

Giving $\tilde{I} \times D^{\gamma} \times D^{n-\gamma}$ the standard $n$-framing coming from $\R^{n}$, the resulting $n$-framings on $\widetilde{S}^{n-\gamma-1}_- \times D^{\gamma}$ and $\widetilde{S}^{\gamma}_+ \times D^{n-\gamma-1}$ are those coming from considering them as a submanifold of the boundary of a stratified disk, as defined in Construction \ref{cons:stratified-spheres}.

\begin{rmk}
For a negative codimension one critical point, the convention for reading $\tilde{I}\times D^{\gamma}\times D^{n-\gamma}$ as a handle is different. In particular the gluing of faces should be done across the other shared corners. We will only need to use the positive codimension one critical points.   
\end{rmk}

\subsection{Relative dualizability theorem}\label{sec:relative_proof}

We now prove Theorem \ref{thm:main-oplax-thm} classifying relative dualizability in the higher Morita category. The proof proceeds analogously to the proof of Theorem \ref{thm:mainthm} with the Morse theoretic arguments replaced by stratified Morse theoretic arguments.

\begin{cons}
Recall the maps from Construction \ref{cons:units}. For each binary word $w$ of length $k \in \{0,...,n-1\}$, define a map $\phi^w:(0,1)^{k+1} \to (0,1)^{k+1}$ to be the composite
\begin{align*}
    (0,1)^{k+1} \xrightarrow{\varphi^{w_1}_1} (0,1)^{k+1} \xrightarrow{\varphi^{w_2}_2} \cdots \xrightarrow{\varphi^{w_k}_k} (0,1)^{k+1}.
\end{align*}
and let $h^{w}:= \pi_{k+1}\circ \phi^w:(0,1)^k \to (0,1)$. Note that $\phi^{\emptyset} =  h^{\emptyset}= \Id:(0,1)\to (0,1)$.
\end{cons}

\begin{prop}\label{prop:Morse-strat}
As a manifold with interface, let $(0,1)^{k+1} = \square_{[1]} \times \square^{k}$. The function $h^{w}: (0,1)^{k+1} \to (0,1)$ is a stratified Morse function with a single positive codimension one critical point of index $|w|$ at $(1/2,....,1/2)$. 
\end{prop}

\begin{proof}
The proof is similar to that of Proposition \ref{prop:Morse-1}. We proceed by induction on the length of $w$. If $w=\emptyset$, the map $\Id:\square_{[1]} \to (0,1)$ has no interior critical points since it is submersive everywhere, it has a codimension one critical point at $1/2 \in (0,1)$ which is easily seen to be positive with respect to the standard stratification on $\square_{[1]}$. For the inductive step suppose that $w = w',w_k$ for $w'$ a word of length $k-1$ and suppose the proposition is true for $w'$. By definition we have a commutative diagram
\begin{equation*}
    \begin{tikzcd}[row sep=small,column sep=small]
	{(0,1)^{k+1}} && {(0,1)} \\
	\\
	{(0,1)^2} && {(0,1)^2}
	\arrow["{h^w}", from=1-1, to=1-3]
	\arrow["{h^{w'}\times\Id}"', from=1-1, to=3-1]
	\arrow["{\varphi^{w_k}}"', from=3-1, to=3-3]
	\arrow["{\pi_2}"', from=3-3, to=1-3]
\end{tikzcd}
\end{equation*}
By the inductive hypothesis, $h^{w'}$ is submersive everywhere, and since $\varphi^{w_k}$ is a diffeomorphism, so is $h^{w}$. Hence $h^w$ has no interior critical points.

To determine the codimension one critical points, consider $h^{w}_1:(0,1)^k \to (0,1)$, the restriction of $h^{w}$ to $\{1/2\}\times (0,1)^k \subset \square_{[1]} \times (0,1)^k$. By definition of $\varphi^{0}$ and $\varphi^1$, for words of length one we have that $h^{0}_1$ is the function $ y \mapsto 1-1/2\sin(\pi y)$ and $h_1^1$ is the function $y \mapsto 1/2\sin(\pi y)$. The same arguments as in the proof of Proposition \ref{prop:Morse-1} show that $h^{w}_1$ has a single critical point at $(1/2,...,1/2)$ of index $|w|$. 

It remains to show that the codimension one critical point of $h^{w}$ at $(1/2,....,1/2)$ is positive. Let $\hat{n} =\partial_1$ be the normal at $(1/2,...,1/2)$, by definition of $\varphi^{w_k}$
\begin{align*}
    dh^{w}(\hat{n}) =  dh^{w'}(\hat{n})\sin(\pi/2) = dh^{w'}(\hat{n})
\end{align*}
By the inductive hypothesis $dh^{w'}(\hat{n}) > 0$, hence the codimension one critical point remains positive.
\end{proof}

\begin{cons}
Let $M$ be a 1-morphism in $\mor_n(\icat)$. For each binary word $w$ of length $k \in \{0,...,n-1\}$, we define a $(k+1)$-morphism $v^{w}$ in $\mor_n(\icat)$ as follows. For the empty word, set $v^{\emptyset} = M$. For any binary word $w'$ with positive length less than $n-1$, iteratively define $v^{w'0}$ and $v^{w'1}$ to be the unit and counit of the adjunction $v^{w'} \dashv (v^{w'})^R$ respectively, as given in Proposition \ref{prop:En-adjoint-data}. In particular, $v^{0}$ and $v^{1}$ are the unit and counit of the adjunction $\F \dashv \F^R$.
\end{cons}

We now prove our main result on relative dualizability.

\begin{proof}[Proof of Theorem \ref{thm:main-oplax-thm}]
By \cite{JFS,SS25}, the 1-morphism $M$ is $R^{n-1}L$-adjointable in $\mor_n(\icat)$ if and only if $v^w$ admits a left adjoint for all binary words $w$ of length $n-1$. By Theorem \ref{thm:main-lifting-lemma} $v^w$ admits a left adjoint in $\mor_n(\icat)$ if and only if the underlying bimodule of $v^w$ admits a left adjoint in $\mor_1(\icat)$. By the same arguments as Proposition \ref{prop:equiv-of-bimods}, the underlying bimodule of $v^w$ is equivalent to $(\pi_{k,...,n} \circ (\phi^w))_\ast M = h^{w}_{\ast} M$. By Proposition \ref{prop:Morse-strat} $h^w$ is a stratified Morse function with a single positive codimension one critical point of index $|w|$. By the stratified Morse lemma, we have that $M$ is $R^{n-1}L$-adjointable if and only if $M$ is dualizable over
\begin{equation*}
    \int_{\widetilde{S}^{|w|}_{-}\times \R^{n-|w|}} M_{A\to B}
\end{equation*}
for all binary words $w$ of length $n-1$. Similarly, $M$ is $R^{n}$-adjointable if and only if $v^w$ admits a right adjoint for all binary words $w$ of length $n-1$ if and only if $M$ is dualizable over
\begin{equation*}
    \int_{\widetilde{S}^{n-|w|-1}_+ \times \R^{|w|+1}} M_{A\to B}
\end{equation*}
for all binary words $w$ of length $n-1$. Since $w$ is ranging over all binary words of length $n-1$, $|w|$ takes all values from $0$ to $n-1$.
\end{proof}

\subsection{Neumann and Dirichlet theories}

\begin{prop}
A 1-morphism $M : \unit \to B$ in $\mor_n(\icat)$ is $R^{n-1}L$-adjointable if and only if $M$ is $n$-dualizable as an object in $\mor_{n-1}(\icat)$. Similarly, a 1-morphism $N : A \to \unit$ in $\mor_n(\icat)$ is $R^{n}$-adjointable if and only if $N$ is $n$-dualizable as an object in $\mor_{n-1}(\icat)$.
\end{prop}

\begin{proof}
Let $N:A  \to \unit$ be a 1-morphism with trivial target. We will apply excision to the diffeomorphism
\begin{align*}
    ( \tilde{I} \times S^{k-1}\times \R^{n-k} ) \bigcup_{\pt_B \times S^{k-1}\times \R \times \R^{n-k}} (\pt_{B} \times D^{k} \times \R^{n-k} )  \cong \widetilde{S}^{k}_+  \times \R^{n-k} \ .
\end{align*}
As the target of $N$ is trivial the factorization homologies over $\pt_B \times S^{k-1}\times \R \times \R^{n-k}$ and $\pt_{B} \times D^{k} \times \R^{n-k} $ are trivial, since they are manifolds entirely labeled by $B$. Excision of factorization homology yields
\begin{equation*}
    \int_{\widetilde{S}^{k}_{+} \times \R^{n-k}} N_{A\to \unit} \simeq \int_{\tilde{I} \times S^{k-1}\times \R^{n-k} } N_{A\to \unit} \otimes_{ \unit}   \unit \simeq \int_{\tilde{I} \times S^{k-1}\times \R^{n-k} } N_{A\to \unit} 
\end{equation*}
By definition of stratified factorization homology on the product $\tilde{I} \times S^{k-1}\times \R^{n-k}$ we have
\begin{equation*}
    \int_{\tilde{I} \times S^{k-1}\times \R^{n-k} } N_{A\to \unit} \simeq \int_{S^{k-1}\times \R^{n-k} } N
\end{equation*}
where the second term is using factorization homology of $N$ as an $\E_{n-1}$-algebra. The action of $\int_{S^{k-1}\times \R^{n-k}} N$ on $N$ is seen to agree with the canonical action since it arises from the canonical action of $\int_{\tilde{I} \times S^{k-1}\times \R^{n-k}} N$ on  $\int_{\tilde{I} \times D^{k-1}\times \R^{n-k}} N$. By Theorem \ref{thm:main-oplax-thm}, $N:A\to \unit$ is $R^n$-adjointable if and only if $N$ is dualizable over $\int_{S^{k -1} \times \R^{n-k}} N$ for $0\leq k\leq n-1$. By Theorem \ref{thm:mainthm}, this is equivalent to $N$ being $n$-dualizable as an object in $\mor_{n-1}(\icat)$. The result for a 1-morphism $M:\unit \to B$ is analogous.
\end{proof}

Using Proposition \ref{prop:oplax-classification} we obtain the following.

\begin{cor}[Relative Neumann theory]
Let $N \colon A \to \unit$ be a 1-morphism in $\mor_n(\icat)$ such that $N$ is $n$-dualizable as an object in $\mor_{n-1}(\icat)$. Assuming the Cobordism Hypothesis, there exists an $n$-dimensional framed oplax relative field theory $$\mathcal{R}_N : \Bord_{n}^{\fr} \longrightarrow  \umor_{n}(\icat)^\to \ $$
with $\mathcal{R}_{N}(\pt^+) =  N\colon A \to \unit$.
\end{cor}

The fact that the regular 1-morphism $A\colon \unit \to A$ is $R^n$-adjointable was proven in Corollary \ref{cor:regular-bimod-dualizability}. Using Proposition \ref{prop:oplax-classification} we obtain the following.

\begin{cor}[Relative Dirichlet theory]
Assuming the Cobordism Hypothesis, there exists an $n$-dimensional framed oplax relative field theory $$\mathcal{R}_{A} : \Bord_{n}^{\fr} \longrightarrow \umor_{n}(\icat)^\to\ $$
with $\mathcal{R}(\pt^+) =  A\colon \unit \to A$.
\end{cor}

\begin{rmk}
It is interesting also to consider the statement of Theorem \ref{thm:main-oplax-thm} in the case that $M$ is the regular 1-morphism $A\colon\unit \to A$. Applying  stratified factorization homology for the regular 1-morphism $A\colon \unit \to A$, we have
\begin{equation*}
    \int_{\widetilde{S}^{k}_+  \times \R^{n-k}} A_{\unit \to A} \simeq \int_{D^{k} \times \R^{n-k}} A  \simeq A
\end{equation*}
for all $0\leq \gamma \leq n-1$. Moreover, the right action of  $\int_{\widetilde{S}^{k}_+  \times \R^{n-k}} A_{\unit \to A} \simeq A$ on  $\int_{\widetilde{D}^{k+1}_+ \times \R^{n-k-1}} A_{\unit \to A} \simeq A$ gives $A$ the structure of the regular $A$-module for all  $0\leq \gamma \leq n-1$.  In other words, all the unit and counit $n$-morphisms witnessing the $n$-dualizability for $A:\unit \to A$ have underlying bimodule the regular $A$-module. This was proven directly in Section \ref{sec:reg-bimod}.
\end{rmk}

\bibliographystyle{alpha}
\bibliography{mybib}

\end{document}